%% file: cef.tex
\documentclass[a4paper,11pt, final ]{amsart}
\pdfoutput=1
\usepackage[utf8x]{inputenc}
\usepackage{amsmath, amsthm, amsfonts,stmaryrd, amssymb,epsfig,enumerate, wrapfig, scalerel, tikz,  color, tensor, accents, fancybox, rotating, comment, afterpage, changepage, colortbl,  pdfpages, datetime, fancyhdr, float}
\usepackage{hyperref}
\usepackage{cleveref}
\hypersetup{
    pdftoolbar=true,        
    pdfmenubar=true,        
    pdffitwindow=false,     
    pdfstartview={FitH},    
    pdfauthor={Louis-Hadrien Robert and Emmanuel Wagner},     
    pdfsubject={},   
    pdfcreator={Louis-Hadrien Robert and Emmanuel Wagner},   
    pdfproducer={Louis-Hadrien Robert and Emmanuel Wagner}, 
    pdfkeywords={}, 
    pdfnewwindow=true,      
    colorlinks=true,       
    linkcolor=red,          
    citecolor=teal,        
    filecolor=magenta,      
    urlcolor=violet,          
    linkbordercolor=red,
    citebordercolor=teal,
    urlbordercolor=violet,  
    linktocpage=true
}
\usepackage{setspace}
\usepackage{showkeys}
\usetikzlibrary{arrows}
\usetikzlibrary{decorations.markings}
\usetikzlibrary{decorations}
\usetikzlibrary{patterns}
\usepackage{tikz-3dplot}
\newcounter {res}[section]
\numberwithin{res}{section}
\newtheorem{thm}[res]{Theorem}
\newtheorem*{thm0}{Theorem}
\newtheorem*{claim}{Claim}
\newtheorem{lem}[res]{Lemma}
\newtheorem{prop}[res]{Proposition}
\newtheorem{cor}[res]{Corollary}
\newtheorem*{coro0}{Corollary}

\theoremstyle{definition}
\newtheorem{notation}[res]{Notation}
\newtheorem{dfn}[res]{Definition}
\newtheorem{rmk}[res]{Remark}

\newcommand{\NN}{\ensuremath{\mathbb{N}}} 
\newcommand{\ZZ}{\ensuremath{\mathbb{Z}}} 
\newcommand{\CC}{\ensuremath{\mathbb{C}}} 
\newcommand{\QQ}{\ensuremath{\mathbb{Q}}}
\newcommand{\RR}{\ensuremath{\mathbb{R}}} 

\renewcommand{\SS}{\ensuremath{\mathbb{S}}} 

\newcommand{\F}{\mathcal{F}}
\let\oldpi\pi
\renewcommand{\pi}{\ensuremath{s}}
\newcommand{\Fz}{\ensuremath{\mathcal{F}_{\mathbf{z}}}}
\newcommand{\Tz}{\ensuremath{\mathcal{T}_{\mathbf{z}}}}
\let\oldmarginpar\marginpar
\renewcommand\marginpar[1]{\oldmarginpar{\color{blue}\fbox{\begin{minipage}{1.4cm} \tiny #1 \end{minipage}}}}

\newcommand\marginparM[1]{}
\newcommand\marginparL[1]{}

\newcommand\eqdef{\ensuremath{\stackrel{\textrm{def}}{=}}}
\newcommand\kups[1]{\left\langle #1 \right\rangle}

\newcommand\kup[1]{\ensuremath\left\langle #1 \right\rangle}
\newcommand\kupc[1]{\ensuremath{\left\langle #1 \right\rangle}}

\newcommand{\imagesfolder}{.}
\newcommand{\sll}{\ensuremath{\mathfrak{sl}}}

\newcommand{\Col}{\ensuremath{\mathbb{P}}}
\newcommand{\NB}[1]{\ensuremath{\vcenter{\hbox{#1}}}}

\newcommand{\FoamN}{\ensuremath{{\mathbf{Foam}_N}}}
\newcommand{\PolN}{\ensuremath{\mathbf{Pol}_N}}
\newcommand{\FQR}{\ensuremath{\mathcal{F}_{QR}}}
\def\co{\colon\thinspace}
\title{A closed formula for the evaluation of $\mathfrak{sl}_N$-foams}
\author{Louis-Hadrien Robert}
\address{Universit\"a{}t Hamburg, Bundesstra\ss{}e 55, 20146 Hamburg, Germany}
\email{louis-hadrien.robert@uni-hamburg.de}
\author{Emmanuel Wagner}
\address{Universit\'e{}~de Bourgogne Franche-Comté, IMB, UMR 5584, 21000 Dijon, France } 
\email{emmanuel.wagner@u-bourgogne.fr}
\tikzset{>=latex}
\tikzset{->-/.style={decoration={
  markings,
  mark=at position .5 with {\arrow{>}}},postaction={decorate}}}
\tikzset{-<-/.style={decoration={
  markings,
  mark=at position .5 with {\arrow{<}}},postaction={decorate}}}

\input{digon}

\input{assoc}
\input{square}

\input{\imagesfolder/cef_autosquare}

\input{\imagesfolder/cef_autosquarereverse}

\usepackage{geometry}

\begin{document}
\begin{abstract}
We give a purely combinatorial formula for evaluating closed decorated foams. Our evaluation gives an integral polynomial and is directly connected to an integral equivariant version of the $\mathfrak{sl}_N$ link homology categorifying the $\mathfrak{sl}_N$ link polynomial. We also provide connections to the equivariant cohomology rings of partial flag manifolds.
\end{abstract}

\maketitle

\tableofcontents
\allowdisplaybreaks
\renewcommand{\and}{\ensuremath{\quad \textrm{and} \quad}}

\section{Introduction}
\label{sec:introduction}
\input{introduction}

\section{Evaluation of closed $\mathfrak{sl}_N$-foams}
\label{sec:eval-clos-mathfr}

\input{evaluation}

\section{Categorification of the MOY calculus}
\label{sec:categ-moy-calc}
\input{categorification}
\section{Applications and conjectures}
\label{sec:applications}
\input{appandconj}

\appendix
\section{Two identities on Schur polynomials}
\label{sec:an-identiy-schur}
\input{schur}

\section{Details of computations}
\label{sec:tables}
\newcommand{\z}{\cellcolor{green!20}}
\newcommand{\q}{\cellcolor{yellow!20}}
\clearpage\label{tables}
\includepdf[pages={1}]{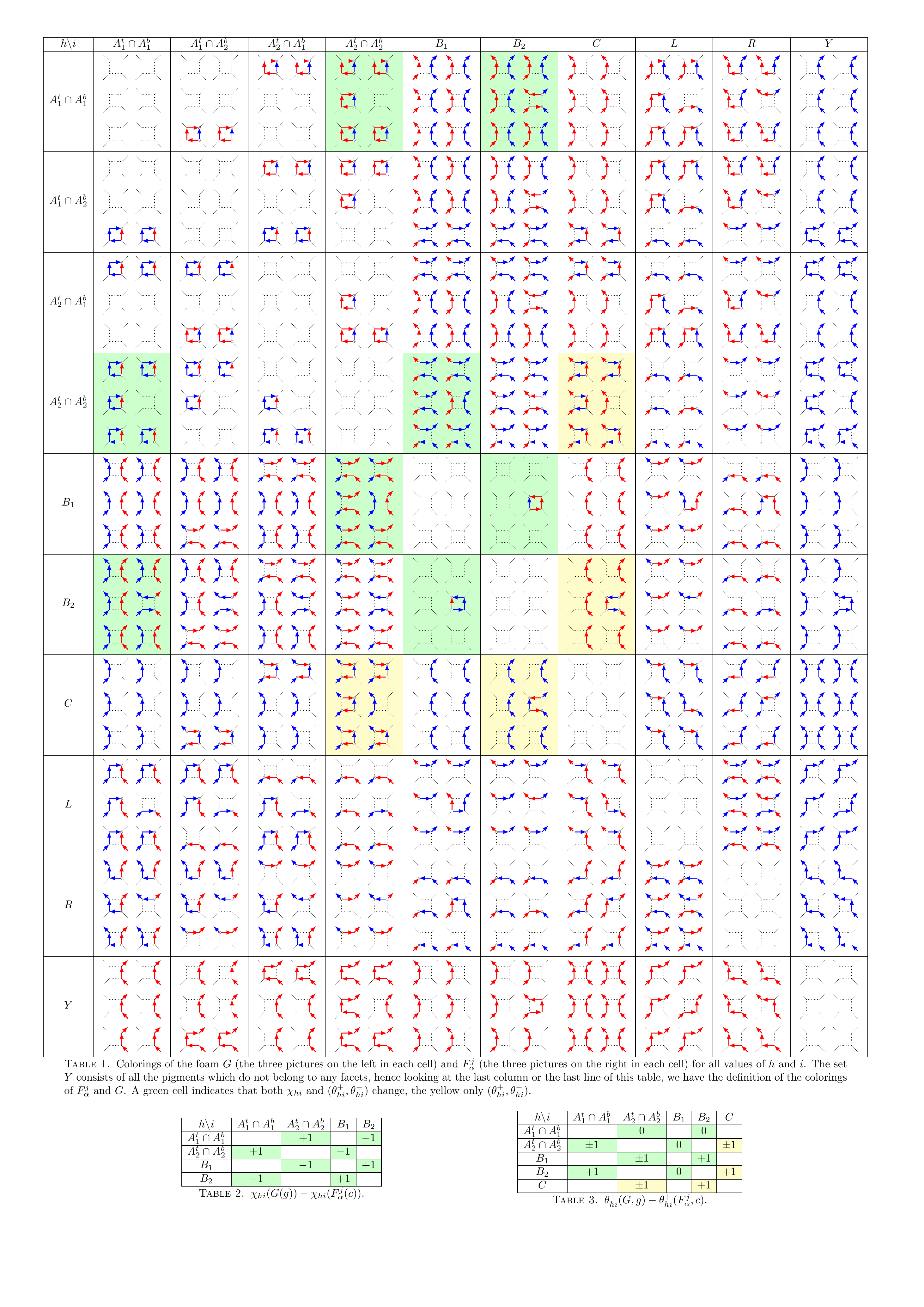}

\label{tables2}
\includepdf[pages={1}]{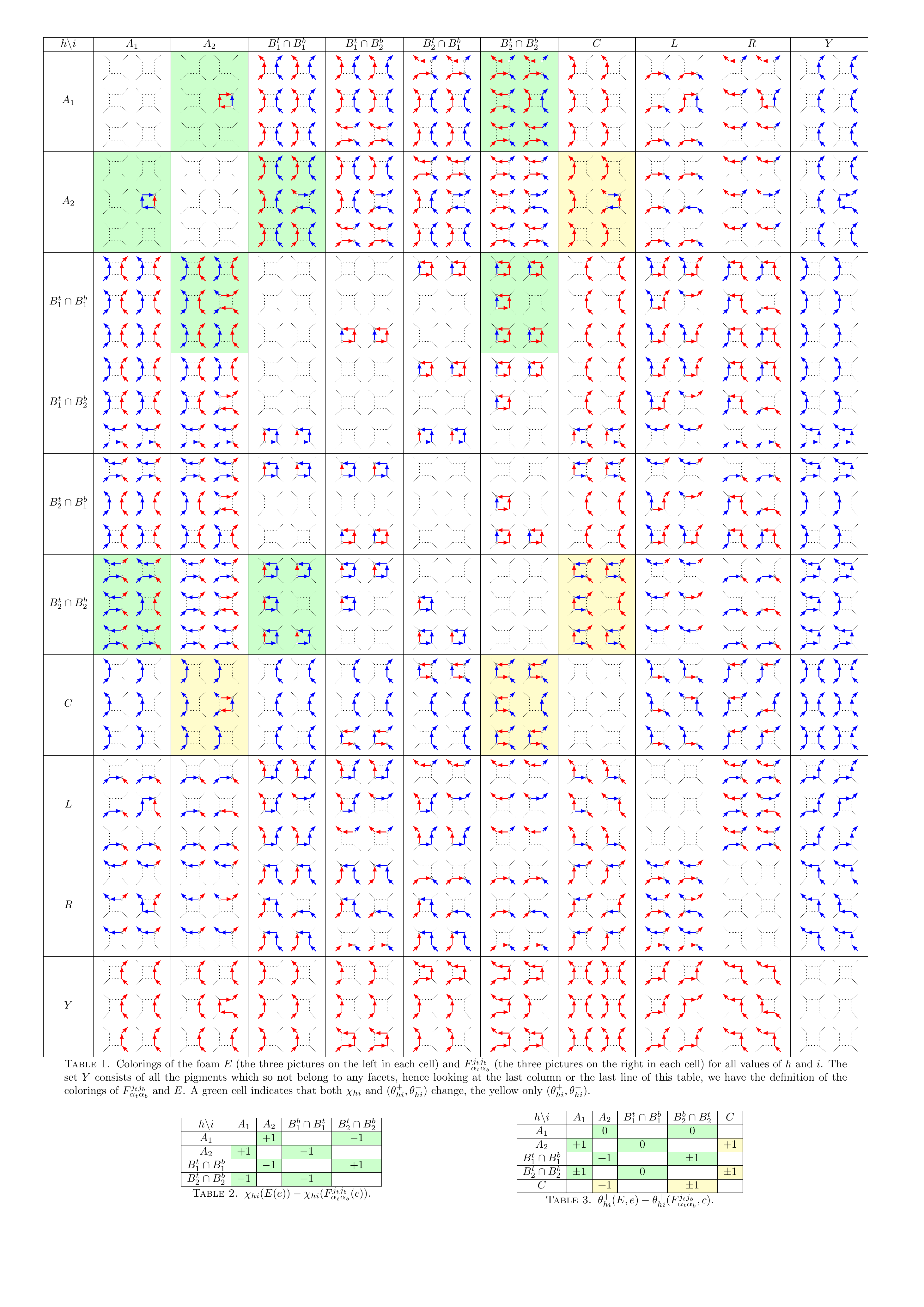}

\bibliographystyle{alphaurl}
\bibliography{bibliocef}



\end{document}

%% file: digon.tex
\newcommand{\digona}{\ensuremath{\vcenter{\hbox{\tikz[scale=0.4]{
\coordinate (B) at (0,0);
\coordinate (V1) at (0,0.5);
\coordinate (V2) at (0,2.5);
\coordinate (T) at (0,3);
\draw[white] (0, -0.5) -- (0, 3.5);
\draw[->] (B) -- (V1) node[midway, left] {\tiny{$m+n$}};
\draw[->] (V2) -- (T) node[midway, left] {\tiny{$m+n$}};
\draw[->] (V1) .. controls +(+0.5, +0.5) and +(+0.5, -0.5).. (V2) node[midway, right] {\tiny{$n$}};
\draw[->] (V1) .. controls +(-0.5, +0.5) and +(-0.5, -0.5).. (V2) node[midway, left] {\tiny{$m$}};
}}}}}

\newcommand{\verta}{\ensuremath{\vcenter{\hbox{\tikz[scale=0.4]{
\coordinate (B) at (0,0);
\coordinate (T) at (0,3);
\draw[white] (0, -0.5) -- (0, 3.5);
\draw[->] (B) -- (T) node[midway, right] {\tiny{$m+n$}};
}}}}}

\newcommand{\digonb}{\ensuremath{\vcenter{\hbox{\tikz[scale=0.4]{
\coordinate (B) at (0,0);
\coordinate (V1) at (0,0.5);
\coordinate (V2) at (0,2.5);
\coordinate (T) at (0,3);
\draw[white] (0, -0.5) -- (0, 3.5);
\draw[->] (B) -- (V1) node[midway, left] {\tiny{$m$}};
\draw[->] (V2) -- (T) node[midway, left] {\tiny{$m$}};
\draw[<-] (V1) .. controls +(+0.5, +0.5) and +(+0.5, -0.5).. (V2) node[midway, right] {\tiny{$n$}};
\draw[->] (V1) .. controls +(-0.5, +0.5) and +(-0.5, -0.5).. (V2) node[midway, left] {\tiny{$m+n$}};
}}}}}

\newcommand{\vertb}{\ensuremath{\vcenter{\hbox{\tikz[scale=0.4]{
\coordinate (B) at (0,0);
\coordinate (T) at (0,3);
\draw[white] (0, -0.5) -- (0, 3.5);
\draw[->] (B) -- (T) node[midway, right] {\tiny{$m$}};
}}}}}

%% file: assoc.tex
\newcommand{\stgamma}{\ensuremath{\vcenter{\hbox{\tikz[scale=0.3]{
\coordinate (B) at (0,0);
\coordinate (V1) at (0,1);
\coordinate (V2) at (1,2);
\coordinate (T1) at (-2,3);
\coordinate (T2) at (0,3);
\coordinate (T3) at (2,3);
\draw[->] (B) -- (V1) node [at start, below] {\tiny{$i+j+k$}};
\draw[->] (V1) -- (T1) node [at end, above] {\tiny{$i$}};
\draw[->] (V1)  -- (V2) node[midway, right] {\tiny{$j+k$}};
\draw[->] (V2) -- (T2) node[at end, above] {\tiny{$j$}};
\draw[->] (V2) -- (T3) node[at end, above] {\tiny{$k$}};
}}}}}

\newcommand{\stgammaprime}{\ensuremath{\vcenter{\hbox{\tikz[scale=0.3]{
\coordinate (B) at (0,0);
\coordinate (V1) at (0,1);
\coordinate (V2) at (-1,2);
\coordinate (T1) at (-2,3);
\coordinate (T2) at (0,3);
\coordinate (T3) at (2,3);
\draw[->] (B) -- (V1) node [at start, below] {\tiny{$i+j+k$}};
\draw[->] (V1) -- (T3) node [at end, above] {\tiny{$k$}};
\draw[->] (V1)  -- (V2) node[midway, left] {\tiny{$i+j$}};
\draw[->] (V2) -- (T1) node[at end, above] {\tiny{$i$}};
\draw[->] (V2) -- (T2) node[at end, above] {\tiny{$j$}};
}}}}}

%% file: square.tex
\newcommand{\squarea}{\ensuremath{\vcenter{\hbox{\tikz[scale=0.4]{
\coordinate (B1) at (-1,0);
\coordinate (B2) at (1,0);
\coordinate (C1) at (-1,1);
\coordinate (D1) at (-1,2);
\coordinate (C2) at (1,1);
\coordinate (D2) at (1,2);
\coordinate (T1) at (-1,3);
\coordinate (T2) at (1,3);
\draw[->] (B1) -- (C1) node[at start, below] {\tiny{$1$}};
\draw[->] (D1) -- (C1) node[midway, left   ] {\tiny{$m$}};
\draw[->] (D1) -- (T1) node[at end , above ] {\tiny{$1$}};
\draw[->] (C2) -- (B2) node[at end, below] {\tiny{$m$}};
\draw[->] (C2) -- (D2) node[midway, right] {\tiny{$1$}};
\draw[->] (T2) -- (D2) node[at start, above] {\tiny{$m$}};
\draw[->] (D2) -- (D1) node[midway, above] {\tiny{$m+1$}};
\draw[->] (C1) -- (C2) node[midway, below] {\tiny{$m+1$}};
}}}}}

\newcommand{\twoverta}{\ensuremath{\vcenter{\hbox{\tikz[scale=0.4]{
\coordinate (B1) at (-1,0);
\coordinate (T1) at (-1,3);
\coordinate (B2) at (1,0);
\coordinate (T2) at (1,3);
\draw[->] (B1) -- (T1) node[midway, left] {\tiny{$1$}};
\draw[->] (T2) -- (B2) node[midway, right] {\tiny{$m$}};
}}}}}

\newcommand{\doubleYa}{\ensuremath{\vcenter{\hbox{\tikz[scale=0.4]{
\coordinate (B1) at (-1,0);
\coordinate (T1) at (-1,3);
\coordinate (C) at (0,1);
\coordinate (D) at (0,2);
\coordinate (B2) at (1,0);
\coordinate (T2) at (1,3);
\draw[->] (B1) -- (C) node[at start, below] {\tiny{$1$}};
\draw[->] (C) -- (B2) node[at end, below] {\tiny{$m$}};
\draw[->] (D) -- (C) node[midway, left] {\tiny{$m-1$}};
\draw[->] (T2) -- (D) node[at start, above] {\tiny{$m$}};
\draw[->] (D) -- (T1) node[at end, above] {\tiny{$1$}};
}}}}}

\newcommand{\squareb}{\ensuremath{\vcenter{\hbox{\tikz[scale=0.55]{
\coordinate (B1) at (-1,0);
\coordinate (B2) at (1,0);
\coordinate (C1) at (-1,1);
\coordinate (D1) at (-1,2);
\coordinate (C2) at (1,1);
\coordinate (D2) at (1,2);
\coordinate (T1) at (-1,3);
\coordinate (T2) at (1,3);
\draw[->] (B1) -- (C1) node[at start, below] {\tiny{$1$}};
\draw[->] (C1) -- (D1) node[midway, left   ] {\tiny{$l+n$}};
\draw[->] (D1) -- (T1) node[at end , above ] {\tiny{$l$}};
\draw[->] (B2) -- (C2) node[at start, below] {\tiny{$m+l-1$}};
\draw[->] (C2) -- (D2) node[midway, right] {\tiny{$m-n$}};
\draw[->] (D2) -- (T2) node[at end, above] {\tiny{$m$}};
\draw[->] (D1) -- (D2) node[midway, above] {\tiny{$n$}};
\draw[->] (C2) -- (C1) node[midway, below] {\tiny{$l+n-1$}};
}}}}}

\newcommand{\squarec}{\ensuremath{\vcenter{\hbox{\tikz[xscale=0.65, yscale=0.55]{
\coordinate (B1) at (-1,0);
\coordinate (B2) at (1,0);
\coordinate (C1) at (-1,1);
\coordinate (D1) at (-1,2);
\coordinate (C2) at (1,1);
\coordinate (D2) at (1,2);
\coordinate (T1) at (-1,3);
\coordinate (T2) at (1,3);
\draw[->] (B1) -- (C1) node[at start, below] {\tiny{$n$}};
\draw[->] (C1) -- (D1) node[midway, left   ] {\tiny{$n+k $}};
\draw[->] (D1) -- (T1) node[at end , above ] {\tiny{$m$}};
\draw[->] (B2) -- (C2) node[at start, below] {\tiny{$m+l$}};
\draw[->] (C2) -- (D2) node[midway, right] {\tiny{$m+l-k$}};
\draw[->] (D2) -- (T2) node[at end, above] {\tiny{$n+l$}};
\draw[->] (D1) -- (D2) node[midway, above] {\tiny{$n+k-m$}};
\draw[->] (C2) -- (C1) node[midway, below] {\tiny{$k$}};
}}}}}

\newcommand{\squared}{\ensuremath{\vcenter{\hbox{\tikz[yscale=0.55, xscale=0.65]{
\coordinate (B1) at (-1,0);
\coordinate (B2) at (1,0);
\coordinate (C1) at (-1,1);
\coordinate (D1) at (-1,2);
\coordinate (C2) at (1,1);
\coordinate (D2) at (1,2);
\coordinate (T1) at (-1,3);
\coordinate (T2) at (1,3);
\draw[->] (B1) -- (C1) node[at start, below] {\tiny{$n$}};
\draw[->] (C1) -- (D1) node[midway, left   ] {\tiny{$m-j$}};
\draw[->] (D1) -- (T1) node[at end , above ] {\tiny{$m$}};
\draw[->] (B2) -- (C2) node[at start, below] {\tiny{$m+l$}};
\draw[->] (C2) -- (D2) node[midway, right] {\tiny{$n+l+j$}};
\draw[->] (D2) -- (T2) node[at end, above] {\tiny{$n+l$}};
\draw[->] (D2) -- (D1) node[midway, above] {\tiny{$j$}};
\draw[->] (C1) -- (C2) node[midway, below] {\tiny{$n+j-m$}};
}}}}}

\newcommand{\doubleYb}{\ensuremath{\vcenter{\hbox{\tikz[scale=0.4]{
\coordinate (B1) at (-1,0);
\coordinate (T1) at (-1,3);
\coordinate (C) at (0,1);
\coordinate (D) at (0,2);
\coordinate (B2) at (1,0);
\coordinate (T2) at (1,3);
\draw[->] (B1) -- (C) node[at start, below] {\tiny{$1$}};
\draw[<-] (C) -- (B2) node[at end, below] {\tiny{$m+l-1$}};
\draw[<-] (D) -- (C) node[midway, left] {\tiny{$l+m$}};
\draw[<-] (T2) -- (D) node[at start, above] {\tiny{$m$}};
\draw[->] (D) -- (T1) node[at end, above] {\tiny{$l$}};
}}}}}

\newcommand{\bigHb}{\ensuremath{\vcenter{\hbox{\tikz[scale=0.4]{
\coordinate (B1) at (-1,0);
\coordinate (T1) at (-1,3);
\coordinate (M1) at (-1,1.5);
\coordinate (M2) at (1,1.5);
\coordinate (B2) at (1,0);
\coordinate (T2) at (1,3);
\draw[->] (B1) -- (M1) node[at start, below] {\tiny{$1$}};
\draw[->] (B2) -- (M2) node[at start, below] {\tiny{$m+l-1$}};
\draw[->] (M2) -- (M1) node[midway, above] {\tiny{$l-1$}};
\draw[->] (M2) -- (T2) node[at end, above] {\tiny{$m$}};
\draw[->] (M1) -- (T1) node[at end, above] {\tiny{$l$}};
}}}}}

%% file: cef_autosquare.tex
\newcommand{\pointing}[1]{\filldraw[ultra thin, draw = black, fill=white] #1 circle (2mm);
\fill #1 circle (0.5mm);}
\newcommand{\cross}[1]{
\filldraw[ultra thin, draw = black, fill=white] #1 circle (2mm);
\draw[ultra thin] #1 -- +(45:2mm); 
\draw[ultra thin] #1 -- +(135:2mm);
\draw[ultra thin] #1 -- +(-45:2mm); 
\draw[ultra thin] #1 -- +(-135:2mm);}
\newcommand{\mysquare}[3]{\NB{
\begin{tikzpicture}
\begin{scope}[xscale = -0.2, yscale=0.2]
\draw[white, opacity=0] (0,-2.5) -- (0, 2.5);
\coordinate (tl) at (-1, 1);  
\coordinate (bl) at (-1,-1);  
\coordinate (tr) at ( 1, 1);  
\coordinate (br) at ( 1,-1);  
\coordinate (TL) at (-2, 2);  
\coordinate (BL) at (-2,-2);  
\coordinate (TR) at ( 2, 2);   
\coordinate (BR) at ( 2,-2);  
\coordinate (HH) at (BL);
\coordinate (DA) at (bl); 
\ifthenelse{\equal{#1}{c} \or \equal{#1}{a1} \or \equal{#1}{a2} \or \equal{#1}{r} \and \(\equal{#2}{l} \or \equal{#2}{b1} \or \equal{#2}{b2} \or \equal{#2}{x} \)}{\draw[->,thick, blue] (HH) -- (DA);}{\draw[densely dotted] (HH) -- (DA);}
\ifthenelse{\equal{#2}{c} \or \equal{#2}{a1} \or \equal{#2}{a2} \or \equal{#2}{r} \and \(\equal{#1}{l} \or \equal{#1}{b1} \or \equal{#1}{b2} \or \equal{#1}{x} \)}{\draw[->,thick, red ] (HH) -- (DA);}{}
\coordinate (HH) at (bl);
\coordinate (DA) at (tl);
\ifthenelse{\equal{#1}{c} \or \equal{#1}{a1} \and \(\equal{#2}{a2} \or \equal{#2}{r} \or \equal{#2}{l} \or \equal{#2}{b1} \or \equal{#2}{b2} \or \equal{#2}{x} \)}{\draw[->,thick, blue] (HH) -- (DA);}{\draw[densely dotted] (HH) -- (DA);}
\ifthenelse{\equal{#2}{c} \or \equal{#2}{a1} \and \(\equal{#1}{a2} \or \equal{#1}{r} \or \equal{#1}{l} \or \equal{#1}{b1} \or \equal{#1}{b2} \or \equal{#1}{x} \)}{\draw[->,thick, red ] (HH) -- (DA);}{}
\coordinate (HH) at (tl);
\coordinate (DA) at (TL);
\ifthenelse{\equal{#1}{c} \or \equal{#1}{a1} \or \equal{#1}{a2} \or \equal{#1}{l} \and \(\equal{#2}{r} \or \equal{#2}{b1} \or \equal{#2}{b2} \or \equal{#2}{x} \)}{\draw[->,thick, blue] (HH) -- (DA);}{\draw[densely dotted] (HH) -- (DA);}
\ifthenelse{\equal{#2}{c} \or \equal{#2}{a1} \or \equal{#2}{a2} \or \equal{#2}{l} \and \(\equal{#1}{r} \or \equal{#1}{b1} \or \equal{#1}{b2} \or \equal{#1}{x} \)}{\draw[->,thick, red ] (HH) -- (DA);}{}
\coordinate (HH) at (BR);
\coordinate (DA) at (br);
\ifthenelse{\equal{#1}{c} \or \equal{#1}{b1} \or \equal{#1}{b2} \or \equal{#1}{l} \and \(\equal{#2}{a2} \or \equal{#2}{r} \or \equal{#2}{a1} \or \equal{#2}{x} \)}{\draw[->,thick, blue] (HH) -- (DA);}{\draw[densely dotted] (HH) -- (DA);}
\ifthenelse{\equal{#2}{c} \or \equal{#2}{b1} \or \equal{#2}{b2} \or \equal{#2}{l} \and \(\equal{#1}{a2} \or \equal{#1}{r} \or \equal{#1}{a1} \or \equal{#1}{x} \)}{\draw[->,thick, red ] (HH) -- (DA);}{}
\coordinate (HH) at (br);
\coordinate (DA) at (tr);
\ifthenelse{\equal{#1}{c} \or \equal{#1}{b1} \or \equal{#1}{b2} \or \equal{#1}{l} \or \equal{#1}{a2} \or \equal{#1}{r} \and \( \equal{#2}{a1} \or \equal{#2}{x} \)}{\draw[->,thick, blue] (HH) -- (DA);}{\draw[densely dotted] (HH) -- (DA);}
\ifthenelse{\equal{#2}{c} \or \equal{#2}{b1} \or \equal{#2}{b2} \or \equal{#2}{l} \or \equal{#2}{a2} \or \equal{#2}{r} \and \( \equal{#1}{a1} \or \equal{#1}{x} \)}{\draw[->,thick, red ] (HH) -- (DA);}{}
\coordinate (HH) at (tr);
\coordinate (DA) at (TR);
\ifthenelse{\equal{#1}{c} \or \equal{#1}{b1} \or \equal{#1}{b2} \or \equal{#1}{r} \and \(\equal{#2}{l} \or \equal{#2}{a2} \or \equal{#2}{a1} \or \equal{#2}{x} \)}{\draw[->,thick, blue] (HH) -- (DA);}{\draw[densely dotted] (HH) -- (DA);}
\ifthenelse{\equal{#2}{c} \or \equal{#2}{b1} \or \equal{#2}{b2} \or \equal{#2}{r} \and \(\equal{#1}{l} \or \equal{#1}{a2} \or \equal{#1}{a1} \or \equal{#1}{x} \)}{\draw[->,thick, red ] (HH) -- (DA);}{}
\coordinate (HH) at (bl);
\coordinate (DA) at (br);
\ifthenelse{\equal{#1}{a2} \or \equal{#1}{r} \and \(\equal{#2}{l} \or \equal{#2}{a1} \or  \equal{#2}{b1} \or \equal{#2}{b2}\or  \equal{#2}{c} \or \equal{#2}{x} \)}{\draw[->,thick, blue] (HH) -- (DA);}{\draw[densely dotted] (HH) -- (DA);}
\ifthenelse{\equal{#2}{a2} \or \equal{#2}{r} \and \(\equal{#1}{l} \or \equal{#1}{b1} \or \equal{#1}{b2} \or  \equal{#1}{a1} \or \equal{#1}{c} \or \equal{#1}{x}  \)}{\draw[->,thick, red ] (HH) -- (DA);}{}
\coordinate (HH) at (tr);
\coordinate (DA) at (tl);
\ifthenelse{\equal{#1}{a2} \or \equal{#1}{l} \and \(\equal{#2}{r} \or \equal{#2}{a1} \or  \equal{#2}{b1} \or \equal{#2}{b2}\or  \equal{#2}{c} \or \equal{#2}{x} \)}{\draw[->,thick, blue] (HH) -- (DA);}{\draw[densely dotted] (HH) -- (DA);}
\ifthenelse{\equal{#2}{a2} \or \equal{#2}{l} \and \(\equal{#1}{r} \or \equal{#1}{b1} \or \equal{#1}{b2} \or  \equal{#1}{a1} \or \equal{#1}{c} \or \equal{#1}{x} \)}{\draw[->,thick, red ] (HH) -- (DA);}{}

\ifthenelse{\equal{#3}{F}}{
\pointing{(bl)}
\cross{(br)}
\pointing{(tr)}
\cross{(tl)}
}{}
\ifthenelse{\equal{#3}{R}}{
\pointing{(bl)}
\cross{(tl)}
}{}
\ifthenelse{\equal{#3}{L}}{
\cross{(br)}
\pointing{(tr)}
}{}
\ifthenelse{\equal{#3}{T}}{
\pointing{(tr)}
\cross{(tl)}
}{}
\ifthenelse{\equal{#3}{B}}{
\pointing{(bl)}
\cross{(br)}
}{}
\ifthenelse{\equal{#3}{BR}}{
\pointing{(bl)}
}{}
\ifthenelse{\equal{#3}{BL}}{
\cross{(br)}
}{}
\ifthenelse{\equal{#3}{TL}}{
\pointing{(tr)}
}{}
\ifthenelse{\equal{#3}{TR}}{
\cross{(tl)}
}{}

\end{scope}
\end{tikzpicture}
}}

%% file: cef_autosquarereverse.tex
\newcommand{\mysquareM}[3]{\NB{
\begin{tikzpicture}
\begin{scope}[xscale= 0.2, yscale = 0.2]
\draw[white, opacity=0] (0,-2.5) -- (0, 2.5);
\coordinate (tl) at (-1, 1);  
\coordinate (bl) at (-1,-1);  
\coordinate (tr) at ( 1, 1);  
\coordinate (br) at ( 1,-1);  
\coordinate (TL) at (-2, 2);  
\coordinate (BL) at (-2,-2);  
\coordinate (TR) at ( 2, 2);  
\coordinate (BR) at ( 2,-2); 
\coordinate (HH) at (BL);
\coordinate (DA) at (bl);
\ifthenelse{\equal{#1}{c} \or \equal{#1}{b1} \or \equal{#1}{b2} \or \equal{#1}{l} \and \(\equal{#2}{r} \or \equal{#2}{a2} \or \equal{#2}{a1} \or \equal{#2}{x} \)}{\draw[->,thick, blue] (HH) -- (DA);}{\draw[densely dotted] (HH) -- (DA);}
\ifthenelse{\equal{#2}{c} \or \equal{#2}{b1} \or \equal{#2}{b2} \or \equal{#2}{l} \and \(\equal{#1}{r} \or \equal{#1}{a2} \or \equal{#1}{a1} \or \equal{#1}{x} \)}{\draw[->,thick, red ] (HH) -- (DA);}{}
\coordinate (HH) at (bl);
\coordinate (DA) at (tl);
\ifthenelse{\equal{#1}{c} \or \equal{#1}{b1} \and \(\equal{#2}{b2} \or \equal{#2}{l} \or \equal{#2}{r} \or \equal{#2}{a2} \or \equal{#2}{a1} \or \equal{#2}{x} \)}{\draw[->,thick, blue] (HH) -- (DA);}{\draw[densely dotted] (HH) -- (DA);}
\ifthenelse{\equal{#2}{c} \or \equal{#2}{b1} \and \(\equal{#1}{b2} \or \equal{#1}{l} \or \equal{#1}{r} \or \equal{#1}{a2} \or \equal{#1}{a1} \or \equal{#1}{x} \)}{\draw[->,thick, red ] (HH) -- (DA);}{}
\coordinate (HH) at (tl);
\coordinate (DA) at (TL);
\ifthenelse{\equal{#1}{c} \or \equal{#1}{b1} \or \equal{#1}{b2} \or \equal{#1}{r} \and \(\equal{#2}{l} \or \equal{#2}{a2} \or \equal{#2}{a1} \or \equal{#2}{x} \)}{\draw[->,thick, blue] (HH) -- (DA);}{\draw[densely dotted] (HH) -- (DA);}
\ifthenelse{\equal{#2}{c} \or \equal{#2}{b1} \or \equal{#2}{b2} \or \equal{#2}{r} \and \(\equal{#1}{l} \or \equal{#1}{a2} \or \equal{#1}{a1} \or \equal{#1}{x} \)}{\draw[->,thick, red ] (HH) -- (DA);}{}
\coordinate (HH) at (BR);
\coordinate (DA) at (br);
\ifthenelse{\equal{#1}{c} \or \equal{#1}{a2} \or \equal{#1}{a1} \or \equal{#1}{r} \and \(\equal{#2}{b1} \or \equal{#2}{l} \or \equal{#2}{b2} \or \equal{#2}{x} \)}{\draw[->,thick, blue] (HH) -- (DA);}{\draw[densely dotted] (HH) -- (DA);}
\ifthenelse{\equal{#2}{c} \or \equal{#2}{a2} \or \equal{#2}{a1} \or \equal{#2}{r} \and \(\equal{#1}{b1} \or \equal{#1}{l} \or \equal{#1}{b2} \or \equal{#1}{x} \)}{\draw[->,thick, red ] (HH) -- (DA);}{}
\coordinate (HH) at (br);
\coordinate (DA) at (tr);
\ifthenelse{\equal{#1}{c} \or \equal{#1}{a2} \or \equal{#1}{a1} \or \equal{#1}{r} \or \equal{#1}{b2} \or \equal{#1}{l} \and \( \equal{#2}{b1} \or \equal{#2}{x} \)}{\draw[->,thick, blue] (HH) -- (DA);}{\draw[densely dotted] (HH) -- (DA);}
\ifthenelse{\equal{#2}{c} \or \equal{#2}{a2} \or \equal{#2}{a1} \or \equal{#2}{r} \or \equal{#2}{b2} \or \equal{#2}{l} \and \( \equal{#1}{b1} \or \equal{#1}{x} \)}{\draw[->,thick, red ] (HH) -- (DA);}{}
\coordinate (HH) at (tr);
\coordinate (DA) at (TR);
\ifthenelse{\equal{#1}{c} \or \equal{#1}{a2} \or \equal{#1}{a1} \or \equal{#1}{l} \and \(\equal{#2}{r} \or \equal{#2}{b1} \or \equal{#2}{b2} \or \equal{#2}{x} \)}{\draw[->,thick, blue] (HH) -- (DA);}{\draw[densely dotted] (HH) -- (DA);}
\ifthenelse{\equal{#2}{c} \or \equal{#2}{a2} \or \equal{#2}{a1} \or \equal{#2}{l} \and \(\equal{#1}{r} \or \equal{#1}{b1} \or \equal{#1}{b2} \or \equal{#1}{x} \)}{\draw[->,thick, red ] (HH) -- (DA);}{}
\coordinate (HH) at (bl);
\coordinate (DA) at (br);
\ifthenelse{\equal{#1}{b2} \or \equal{#1}{l} \and \(\equal{#2}{r} \or \equal{#2}{b1} \or \equal{#2}{a2} \or \equal{#2}{a1} \or \equal{#2}{c} \or \equal{#2}{x} \)}{\draw[->,thick, blue] (HH) -- (DA);}{\draw[densely dotted] (HH) -- (DA);}
\ifthenelse{\equal{#2}{b2} \or \equal{#2}{l} \and \(\equal{#1}{r} \or \equal{#1}{a2} \or \equal{#1}{a1} \or \equal{#1}{b1} \or \equal{#1}{c} \or \equal{#1}{x} \)}{\draw[->,thick, red ] (HH) -- (DA);}{}
\coordinate (HH) at (tr);
\coordinate (DA) at (tl);
\ifthenelse{\equal{#1}{b2} \or \equal{#1}{r} \and \(\equal{#2}{l} \or \equal{#2}{b1} \or \equal{#2}{a2} \or \equal{#2}{a1} \or \equal{#2}{c} \or \equal{#2}{x} \)}{\draw[->,thick, blue] (HH) -- (DA);}{\draw[densely dotted] (HH) -- (DA);}
\ifthenelse{\equal{#2}{b2} \or \equal{#2}{r} \and \(\equal{#1}{l} \or \equal{#1}{a2} \or \equal{#1}{a1} \or \equal{#1}{b1} \or \equal{#1}{c} \or \equal{#1}{x} \)}{\draw[->,thick, red ] (HH) -- (DA);}{}

\ifthenelse{\equal{#3}{F}}{
\pointing{(bl)}
\cross{(br)}
\pointing{(tr)}
\cross{(tl)}
}{}
\ifthenelse{\equal{#3}{L}}{
\pointing{(bl)}
\cross{(tl)}
}{}
\ifthenelse{\equal{#3}{R}}{
\cross{(br)}
\pointing{(tr)}
}{}
\ifthenelse{\equal{#3}{T}}{
\pointing{(tr)}
\cross{(tl)}
}{}
\ifthenelse{\equal{#3}{B}}{
\pointing{(bl)}
\cross{(br)}
}{}
\ifthenelse{\equal{#3}{BL}}{
\pointing{(bl)}
}{}
\ifthenelse{\equal{#3}{BR}}{
\cross{(br)}
}{}
\ifthenelse{\equal{#3}{TR}}{
\pointing{(tr)}
}{}
\ifthenelse{\equal{#3}{TL}}{
\cross{(tl)}
}{}
\end{scope}
\end{tikzpicture}
}}

%% file: introduction.tex
\subsection{Background}
\label{sec:exposition-context}

Categorification in knot theory developed rapidly over the last fifteen years. The Khovanov homology \cite{MR1740682} and the knot Heegaard--Floer homology \cite{MR2704683, MR2065507} are probably the most popular examples of categorifications of knot invariants. The first categorifies the Jones polynomial, the last categorifies the Alexander polynomial.


On the one hand, the strength of the Alexander polynomial and the knot Heegaard--Floer homology relies in the geometric and topological nature of their definition. This allows to extract information about the genus \cite{MR2023281}, the slice genus \cite{MR2026543} or the fiberedness~\cite{MR2450204, MR2357503} of knots.


On the other hand, the power of the Jones polynomial and the Khovanov homology comes from their very simple pictorial definitions (through the Kauffman bracket and the so-called hypercube of resolution, see~\cite{MR2174270}). Among the greatest achievements of the theory, let us mention Rasmussen's purely combinatorial proof the Milnor's conjecture on the slice genus of torus knots~\cite{MR2729272}.


The vitality of the subject comes from the interplay between those two sides: the Khovanov homology and the Heegaard--Floer theories are related through spectral sequences~\cite{MR2141852}. A very similar spectral sequence relates the Khovanov homology and the singular instanton homology of knots. This yields a proof that the Khovanov homology detects the unknot~\cite{MR2805599}.


The Jones polynomial can be seen as a prototype of the Reshetikhin--Turaev construction \cite{MR1036112} producing link invariants from representations of quantum groups. This procedure has been categorified in full generality by Webster~\cite{MR3084241}. From this perspective, the Jones polynomial corresponds to the the standard representation of quantum $\mathfrak{sl}_2$.


The Alexander polynomial also has a quantum flavor as explained by Kauffman and Saleur \cite{MR1133269} and Viro \cite{MR2255851}. It fits as well naturally in the framework of non-semisimple invariants as developed by Geer, Patureau--Mirand and their collaborators \cite{MR2480500, MR2640994}, but this has not been categorified yet.


The Jones polynomial and the Alexander polynomial can be seen as a specialization of the HOMFLYPT polynomial. Even if its quantum nature is not completely clear, the HOMFLYPT polynomial has been categorified by Khovanov and Rozansky \cite{MR2421131}. Their construction is called the triply graded homology and uses matrix factorizations. This homology theory is a focal point of attention since it is at the crossroads of different areas in mathematics: algebraic geometry~\cite{MR2922375}, representation theory~\cite{MR3273582}, mathematical physics~\cite{MR3294947}.


Even if Rasmussen and Wedrich \cite{MR3447099, 2016arXiv160202769W} proved that it can be related to the Khovanov homology through a spectral sequence, the triply graded homology fails to have a pictorial definition. One can define the triply graded homology using Soergel bimodules and Hochschild homology \cite{WW, MR2258045}. However the geometrical aspects of this definition are quite far from topological geometry and from symplectic geometry used to define the knot Floer homology. Actually constructing a spectral sequence  from the triply graded homology  to the knot Floer homology is a challenging open problem, see for example~\cite{MR3459959}. 



Some natural intermediate steps between the Khovanov homology and the triply graded homology are the so-called $\sll_N$-homologies. They categorify the $\sll_N$-polynomials, other specializations of the HOMFLYPT polynomial \cite{MR2391017}. These theories present many features of the triply graded homology since they have various interesting definitions through different areas of mathematics: algebraic geometry \cite{zbMATH05278251, zbMATH05343986}, Lie theory \cite{MR2567504, MR2710319}, symplectic geometry \cite{MR2313538, MR2254624} categorified quantum groups \cite{queffelec2014mathfrak}, mathematical physics \cite{MR2193547}\dots


Some of the previous constructions categorify as an intermediate step the intertwiners between the exterior powers of the natural representation of quantum $\mathfrak{sl}_N$. This feature can be thought of as a categorification of the so-called MOY-calculus \cite{MR1659228}. It allows to consider all the colored $\mathfrak{sl}_N$ link polynomials at once and gives a natural framework for their categorifications. This has been done by Wu~\cite{pre06302580} and Yonezawa~\cite{MR2863366} using matrix factorizations. Another approach is due to Queffelec and Rose~\cite{queffelec2014mathfrak} (see also \cite{MR3426687}). It uses a categorical version of the skew Howe duality given by Cautis, Kamnitzer and Morrison \cite{MR3263166} in order to describe the category of representations of quantum $\mathfrak{sl}_N$. 


None of the previous definitions has the simple, diagrammatic, computation-friendly features of the Khovanov homology as promoted by Bar-Natan \cite{MR2174270}. The case $N=3$ remains peculiar since the categorification given by Khovanov \cite{MR2100691} has these features. It is based on the graphical calculus for $\mathfrak{sl}_3$-webs introduced by Kuperberg \cite{MR1403861} and introduces cobordisms between webs, called foams. Lewark and Lobb \cite{MR3458146, MR3248745} derived from this construction striking results about the slice genus of knots.


Among all the previous construction of the $\sll_N$ homologies, some of them are already of combinatorial nature. The approach of Mackaay, Sto\v si\'c and Vaz \cite{MR2491657} is very close in spirit to the one of Khovanov for $\mathfrak{sl}_2$ and  $\sll_3$. However they need to use Kapustin--Li formula to evaluate closed foams \cite{MR2039036}. This formula relies heavily on  matrix factorizations and is barely useful in practice.


The aim of the paper is to give a purely combinatorial evaluation of closed $\sll_N$-foams\footnote{Let us mention here that, strictly speaking, we are working with $\mathfrak{gl}_N$-foams but we preferred to stick to the $\sll_n$ notation and formulation to match the one for the link polynomials and the link homologies. See \cite{2016arXiv160108010} for a more detailed discussion on this subject.}. This provides a replacement to the Kapustin--Li formula and yields the $\sll_N$ link homologies in a purely TQFT-way (and their colored and equivariant integral versions). This gives a categorification of the MOY calculus by using a universal construction à la \cite{MR1362791}. Its self-contained and combinatorial nature permits to perform reverse-engineering and provides applications in computations in the cohomology rings of partial flag manifolds.


\subsection{Summary of the results}
\label{sec:summary-results}

\subsubsection{Main result}

The two main topological characters of this papers will be webs and foams.\\

 A web is 
a planar trivalent oriented graph whose edges are labeled by positive integers (circles are also allowed). The orientation is such that none of the vertices is a source or a sink, and the labels around each vertex is such that the sum of the integers of the edges abutting the vertex is equal to the sum of the integers of the edges leaving the vertex, see all around the paper for examples. 
As mentioned earlier they can be thought of as describing the intertwiners between exterior powers of the natural representation of quantum $\sll_n$. 
Hence, we will consider formal linear combinations of graphs in a skein-theoretical fashion which will express relations between intertwiners. All the relations we use are pictured in Definition~\ref{dfn:MOYevaluation} and are called MOY relations. We also called the webs we are using MOY graphs (see Definition~\ref{dfn:MOYgraph}). It is now known by work of Cautis--Kamnitzer--Morrison \cite{MR3263166} that these relations describe completely the category of representations of quantum $\sll_n$ generated monoidally by the exteriors powers of the natural representation.\\

Foams are the natural two dimensional analogs of webs and as such a prototype is obtained by taking a web times an interval or a circle. 
This reveals two things that will be used thought the paper. First, foams can be thought of as cobordisms between webs and altogether they form a cobordism category. Second, closed foams play a particular role. We can already see that foams have facets labeled by integers, and one dimensional oriented singular locus, locally looking as the trivalent vertex times an interval. Moreover, there are conditions on the orientations and the labels of the facets around the singular locus which are inherited from the one around the trivalent vertex on webs. In addition we consider more general foams by allowing closed surfaces (and connected sums of facets with closed surfaces) and singular vertices looking like the cone on the 1-skeleton of a tetrahedron (see Section~\ref{sec:eval-clos-mathfr} for a more detailed definition). Our foams need not to be embedded in the three-dimensional space but we need a cyclic ordering of the facets surrounding the singular locus (if the foam is embedded the cyclic ordering is given by the left-hand rule).\\


The main characters from the algebraic and combinatorial point of views are Schur polynomials and we point to Appendix A for an account of what will be useful to us.\\

In addition to labels, the facets have two more enhancements: a coloring and (eventually) a decoration. For this fix a integer $N$ and a set of variables $X\eqdef (X_1, X_2, \dots , X_N)$. For a facet labeled $k\leq N$ choose a subset of $k$ elements in $X$. A coloring of a foam is such a choice for each facet such that, around each singular edge, one of the subset is the disjoint union of the two others. A decoration of a facet labeled $k$ is a choice of a symmetric polynomial in $k$ variables. We call such foam decorated.

Now from a colored foam $F(c)$, choose a variable $X_i$ appearing in the coloring and consider the union of the facet where the variable $X_i$ appears. It turns out to be a closed oriented surface and it is denoted $F_i(c)$ and is called the monochrome surface. Moreover given two distinct variables denote $F_{ij}(c)$ the symmetric difference of the surfaces $F_i(c)$ and $F_j(c)$. The surface $F_{ij}(c)$ is a well closed and oriented and is called a bichrome surface. In addition there are preferred oriented simple closed curves drawn on $F_{ij}(c)$ which separate regions where the variable $X_i$ appears  from the one where the variable $X_j$ appears.


We have now roughly defined all the ingredients of our formula and we can state our main theorem.


\begin{thm0}
\begin{enumerate}
\item There exists a closed combinatorial evaluation $F\mapsto \langle F \rangle_N \in \ZZ[X_1, \dots, X_N]^{\mathfrak{S}_N}$ which satisfies local relations which categorify the $MOY$ relations.
\item Moreover, given a closed decorated foam $F$ the evaluation $\langle F \rangle_N$ can be expressed as a state sum formula:
$$ \langle F \rangle_N=\sum_c \langle F, c \rangle_N,$$
where the value $\langle F, c \rangle_N$  only depends on the Euler characteristics of the monochrome surfaces and of the bichrome surfaces, on the decoration and on the orientations of the preferred simple closed curves in the bichrome surfaces.
\end{enumerate}
\end{thm0}

The variables $X_i$ can be thought of as a set of simple $\mathfrak{gl}_N$-roots. With this point of view, the importance of the bichrome surfaces comes from their correspondence with the positive $\sll_N$-roots.  \\

The local relations satisfied by the formula can be found for instance in Proposition~\ref{prop:additionalrelation} and Proposition~\ref{prop:complicatedfoam}. The Proposition~\ref{prop:complicatedfoam} is the heart of the above theorem and corresponds to a categorified version of relation~(\ref{eq:relsquare3}) of Definition~\ref{dfn:MOYevaluation}. It is its proof that requires the full strength of the computations in Appendix A.\\

Let us also mention that the previous state-sum formula on colorings naturally generalizes the one of Murakami--Ohtsuki--Yamada for the graph polynomial \cite{MR1659228}, as one can immediately see by looking to foams which are the product of a circle and a closed web.\\

This formula has a certain numbers of advantages that we outline. Maybe the first one is that it allows to give a clean definition of the relevant quotient of the foam cobordism category necessary for constructing link homology theories. This was the original motivation for Mackaay--Sto\v{s}i\'c-Vaz to consider the Kasputin--Li formula in the non-colored case.\\

\subsubsection{Cohomology of partial flag varieties and other applications}

First, mention that the above theorem answers completely a problem suggested by Wedrich and Rose \cite{2015arXiv150102567R} which was to give a complete foam version of the equivariant theory considered by Krasner  \cite{MR2580427} and generalized by Wu \cite{MR3392963}.\\

Second, its integral and equivariant nature allows to recover all the previously known versions of the categorification of the MOY calculus. Comparison with the other versions follows essentially by the work of Queffelec--Rose \cite{queffelec2014mathfrak} who gave a complete set of local relations that are enough to evaluate closed foams recursively and our formula satisfies all these relations, see Section~\ref{sec:applications} for a more detailed account.\\

Third, one can start from our formula to obtain an integral version of the colored equivariant $\sll_N$ link homology and construct the $\sll_N$-web algebra \cite{MR3198835}. All these constructions are straightforward due the TQFT nature of our construction and formula. For instance the invariance of the colored equivariant $\sll_N$ link homology will follow from the local relations that our formula satisfies. In a same vein one can directly use our formula to reconstruct the $2$-functor of Queffelec--Rose \cite{queffelec2014mathfrak} by a now standard procedure first described by Khovanov \cite{MR1928174} in the $\sll_2$ case. \\

Fourth, we hope our formula will allow concrete computations of $\sll_N$-concordance invariant in the same fashion that Lewark did for the $N=3$ case \cite{MR3248745}.\\

All the previous applications are very interesting but stay in the area of link homology and categorification. We finish by consequences of different nature that we would like to emphasize.\\

As we mentioned before, the closed formula together with the universal construction allows to construct a trivalent TQFT. This means that it associates with any web $\Gamma$ a graded vector space $V(\Gamma)$
and with any foam a linear map. It is well-known that if a web $\Gamma$ has a symmetry axis then the (graded) vector space $V(\Gamma)$ carries naturally a structure of a Frobenius algebra. As a byproduct we obtain infinitely many explicit Frobenius algebras for which we can explicitly compute the structure constants.


\begin{thm0}
For any web $\Gamma$ with a symmetry axis, the graded vector space $V(\Gamma)$ is a Frobenius algebra, whose structure constants can be computed explicitly.\\
If $\Gamma$ is a circle labeled $k$ then $V(\Gamma)$ is isomorphic to the cohomology of the Grassmannian of $k$ planes in $\mathbb{C}^N$ and if $\Gamma$ is a generalized theta graph (see Section~\ref{sec:equiv-cohom-part}) then $V(\Gamma)$ is isomorphic to the cohomology of a partial flag manifold.
\end{thm0}

Cohomology of Grassmannians and partial flag manifolds  have been guidelines in the history of the categorification of the $\sll_N$-knot invariant see  \cite{MR2391017}, \cite{MR2100691}, \cite{MR3190356} or \cite{queffelec2014mathfrak}. Here we make the correspondence completely explicit and use it to give for instance a closed formula for the Littlewood--Richardson coefficients (see Corollary~\ref{cor:LRgrassmanian}):

\begin{coro0}
Let $\alpha$, $\beta$ and $\lambda$ be three Young diagrams such that $|\alpha| + |\beta|=  |\lambda|$. Choose two non-negative integers  $a$ and $b$ such that $\alpha$, $\beta$ and $\lambda$ are in $T(b,a)$, then the Littlewood--Richardson coefficient $c_{{\alpha} {\beta}}^{\lambda}$ and can be computed via:
\begin{align*}
c_{\alpha {\beta}}^{\mathbf{\lambda}}= &(-1)^{|\widehat{\lambda}|+ N(N+1)/2 }\kup{\scriptstyle{\NB{\tikz[scale =0.8]{\input{\imagesfolder/cef_thetaLR}}}}}_N \\ &= (-1)^{N(N+1)/2 +|\widehat{\lambda}|} \sum_{\substack{A\sqcup B = \{X_1, \dots, X_N\} \\ |A|=a, |B| = b}}(-1)^{|B<A|} \frac{a_\alpha(A) a_\beta(A) a_{\widehat{\lambda}}(B)}{\Delta(X_1, \dots, X_N)}.
\end{align*}
where $N= a+b$.
\end{coro0}

All the necessary definitions of the polynomials are given in Appendix A. Note that this formula can derived from the Atiyah--Bott--Berline--Vergne integration formula (see for example \cite[Theorem 2.10 and Section 3.6]{MR2976939}).\\

Let us also mention that in \cite{MR3190356}, Lobb and Zentner explored the connections between the MOY calculus and the cohomology of certain moduli associated with MOY graphs. They proved that the Euler characteristic of their moduli is equal to the $\sll_N$ polynomial evaluated at $-1$. They noticed in addition that the cohomology of their moduli can not be equal to the vector space associated with this MOY graph by the $\sll_N$-homology. We plan 
 to investigate further similar questions in light of the results of this paper.\\

We finish by coming back to the physical origin of the Kasputin--Li formula \cite{MR2039036}. Kasputin and Li were using matrix factorizations to understand $D$-branes on Calabi--Yau models. It would be great to understand their work as well as the beautiful work of Vaintrob and Polishchuk \cite{MR2954619} through the prism of our formula.\\

\subsection{Outline of the paper}
\label{sec:outline-paper}

Aside the introduction, the paper is divided in three sections and two appendices.

In Section~\ref{sec:eval-clos-mathfr} we introduce  closed $\sll_N$-foams and define a $\ZZ[X_1, \dots, X_N]^{\mathfrak{S}_N}$-valued evaluation of these foams. For this purpose we introduce \emph{colorings} of $\sll_N$-foams and define a $\QQ(X_1, \dots, X_N)$-evaluation of colored foams. In Section~\ref{sec:foams-colors-}, we discuss the behavior of this colored evaluation under some semi-local moves of the colorings. In Section~\ref{sec:symmetric-evaluation}, we prove that the evaluation of an uncolored foam is a symmetric polynomial. Finally in Section~\ref{sec:2-rks}, we relate this evaluation to different normalization of the evaluation of $\sll_2$-foams and discuss how it can be extended to some generalized $\sll_N$-foams.

In Section~\ref{sec:categ-moy-calc}, we prove that after applying a universal construction, the foam evaluation extends to a trivalent TQFT functor $\mathcal{F}$ from the category of ($\sll_N$-MOY-graphs, $\sll_N$-foams) to the category of finitely generated projective $\ZZ[X_1, \dots, X_N]$\footnote{We use the ring $\ZZ[X_1, \dots, X_N]$ to compare with the $T$-equivariant cohomology rings in Section~\ref{sec:applications} but we could have also used the ring $\ZZ[X_1, \dots, X_N]^{\mathfrak{S}_N}$ and compared with $GL_n$-equivariant cohomology rings.}-modules. In Section~\ref{sec:MOYgraph}, we recall some results on the MOY-calculus, while in Section~\ref{sec:univ-constr-main}, we prove Theorem~\ref{thm:main} stating that $\mathcal{F}$ categorifies the MOY-calculus. The proof is basically based on some careful computations (some detailed tables are given in Appendix~\ref{sec:tables}) and on some identities on Schur polynomials given in Appendix~\ref{sec:an-identiy-schur}.


In Section~\ref{sec:applications}, we relate our work with some other results and discuss some applications. 
In Section~\ref{sec:relat-with-appr}, we investigate the specializations of the functor $\mathcal{F}$ when the variables $X_i$ take complex values. First, we state that when setting all variables to $0$, we recover the functor defined by Queffelec--Rose in \cite{queffelec2014mathfrak}. Then, we explain that other ``deformations'' decompose in a way detailed by Rose--Wedrich in \cite{2015arXiv150102567R}. Finally we emphasize that the computations of \cite{queffelec2014mathfrak} applied to the functor $\mathcal{F}$ yields an equivariant $\sll_N$-homology for links.
In Section~\ref{sec:equiv-cohom-part}, we prove that the image of any $\sll_N$-web with a symmetry axis is naturally endowed with a structure of a Frobenius algebra. We observe that the ring associated with some special $\sll_N$-MOY-graph is isomorphic to the cohomology rings of some flag varieties. We give a few results which enable handy computations in these rings. In Section~\ref{sec:furth-poss-devel}, we give a non-exhaustive list of questions or problems which can be either tackled or developed using the results presented in this paper.

Appendix~\ref{sec:an-identiy-schur}, is devoted to the proof of two identities on Schur polynomials. We intend to be more or less self-contained. 

\subsection{Acknowledgments}
\label{sec:acknowledgments}

L.-H.~R.{} thanks, Ehud Meir, Tobias Ohrmann and Ingo Runkel for interesting discussions and the Pony Bar for its nice working atmosphere. L.-H.~R.{} and E.~W.{} wishes to thank Mikhail Khovanov and Allen Knutson for their enlightening suggestions and Hoel Queffelec, Daniel Tubbenhauer and Paul Wedrich for very careful readings of a preliminary version.


%% file: cef_thetaLR.tex
\begin{scope}
  \fill[fill opacity =0.3, thick, fill = red]   (1,0) arc (0:-180:1cm) arc (180:0:1cm and 0.5cm);
  \fill[fill opacity =0.3, thick, fill = blue]  (0,0) ellipse (1cm and 0.5cm);
  \node at (0,0) {$\scriptstyle{b}$};
  \draw[thick, dotted] (1,0) arc (0:180:1cm and 0.5cm);
  \fill[fill opacity =0.3, thick, fill = green] (1,0) arc (0:180:1cm)  node[midway, below, opacity =1, black] {$\scriptstyle{a}$} arc (180:0:1cm and 0.5cm);
   \draw[red, thick, <-] (0.5,0) .. controls (0.5,0.3).. +(0.6, 0.3) node[right, red] {$\pi_{\alpha}\pi_{\beta}$};
  \fill[fill opacity =0.3, thick, fill = red]   (1,0) arc (0:-180:1cm) node[midway,above, opacity =1] {$\scriptstyle{N}$} arc (-180:0:1cm and 0.5cm);
  \fill[fill opacity =0.3, thick, fill = green] (1,0) arc (0:180:1cm)  arc (-180:0:1cm and 0.5cm);
  \draw[thick] (0,0) circle (1cm);
  \draw[thick, ->] (1,0) arc (0:-90:1cm and 0.5cm);
  \draw[thick, -]  (-1,0) arc (-180:-90:1cm and 0.5cm);

   \draw[red, thick, <-] (60:1) -- +(0.6, 0) node[right, red] {$\pi_{\widehat{\lambda}}$};
\end{scope}

%% file: evaluation.tex
\subsection{Of foams and colors}
\label{sec:foams-colors-}
In this section $N$ is a fixed positive integer. 

\begin{dfn}\label{dfn:foam}
  An \emph{$\sll_N$-foam} (or simply \emph{foam}) $F$ is a collection
  of \emph{facets} $\mathcal{F}(F)=(\Sigma_i)_{i\in I}$, that is a collection of
  oriented connected surfaces with boundary, together with the
  following data:
  \begin{itemize}
  \item A labeling $l\co (\Sigma_i)_{i\in I} \to \{0, \dots, N\}$,
  \item A ``gluing recipe'' of the facets along their boundaries such
    that when glued together using the recipe we have the three
    possible local models:

\[
      \begin{tikzpicture}
        \input{\imagesfolder/cef_3localmodels}

      \end{tikzpicture}
      \label{fig:FBSP} \]

    The letter appearing on a facet indicates the label of this facet.
    That is we have: \emph{facets}, \emph{bindings} (which are compact oriented
    $1$-manifolds) and \emph{singular points}. Each binding carries:
    \begin{itemize}
    \item an orientation which agrees with the orientations of the facets with
      labels $a$ and $b$ and disagrees with the orientation of the facet with label
      $a+b$.
    \item a cyclic ordering of the three facets around it. When a foam
      is embedded in $\RR^3$, we require this cyclic ordering to agree with 
      the left-hand rule\footnote{This agrees with Khovanov's convention \cite{MR2100691}.} with respect to its
      orientation (the dotted circle in the middle indicates that the
      orientation of the binding points to the reader, a crossed circle indicate the other orientation see figure~\ref{fig:signsofcircles}):
      \[
        \begin{tikzpicture}[xscale=1]
          \input{\imagesfolder/sw_lhrule}
        \end{tikzpicture}
      \]
    \end{itemize}
    The cyclic orderings of the different bindings adjacent to a
    singular point should be compatible. This means that a
    neighborhood of the singular point is embeddable in $\RR^3$ in a
    way that respects the left-hand rule for the four binding adjacent to this singular point.
  \end{itemize}
\end{dfn}

\begin{rmk}
  Les us explain shortly what is meant by ``gluing recipe''. The boundaries of the facets forms a collection of circles. We denote it by  $\mathcal{S}$. The gluing recipe consists of:
  \begin{itemize}
  \item For a subset $\mathcal{S}'$ of $\mathcal{S}$, a subdivision of each circle of $\mathcal{S'}$ into a finite number of closed intervals. This gives us a collection $\mathcal{I}$ of closed intervals.  
  \item Partitions of $\mathcal{I} \cup (\mathcal{S} \setminus \mathcal{S'})$ into subsets of three elements. For every subset $(Y_1, Y_2, Y_3)$ of this partition, three diffeomorphisms $\phi_1 : Y_2 \to Y_3$, $\phi_2 : Y_3 \to Y_1$, $\phi_3 : Y_1 \to Y_2$  such that $\phi_3 \circ \phi_2 \circ \phi_1 = \mathrm{id}_{Y_2}$.
  \end{itemize}
A foam is obtained by gluing the facets along the diffeomorphisms, provided that the conditions given in the previous definition are fulfilled.  
\end{rmk}

\begin{figure}[h]
  \centering
  \begin{tikzpicture}[scale = 0.7]
    \input{\imagesfolder/cef_foam6j}
  \end{tikzpicture}
  \caption{Example of a foam. If $N= 6$ it has degree $- 26$. The cyclic ordering on the central binding is $(5,2,3)$.}
  \label{fig:exslnfoam}
\end{figure}
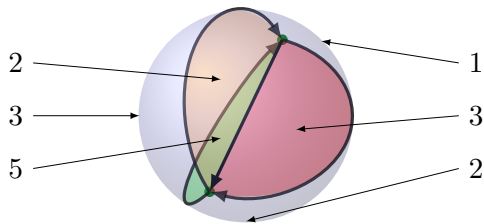

\begin{dfn}\label{dfn:degreefoam}
    We define the \emph{degree} $d_N$ of a foam $F$ as the sum of the following contributions:
    \begin{itemize}
    \item For each facets $f$ with label $a$, set $d(f)=a(N-a)\chi(f)$, where $\chi(f)$ stands for the Euler characteristic of $f$;
    \item For each interval binding $i$ (i. e. not circle-like binding) surrounded by three facets with labels $a$, $b$ and $a+b$, set $d(b)= ab + (a+b)(N-a -b)$;
    \item For each singular point $p$ surrounded with facets with labels $a$, $b$, $c$, $a+b$, $b+c$, $a+b+c$, set $d(p) = ab + bc+ cd + da + ac+ bd $ with $d =N -a -b -c $;
    \item Finally set \[
d_N(F) = -\sum_{f\textrm{ facet}}d(f) + \sum_{\substack{i \textrm{ interval} \\ \textrm{binding}}}d(i) -\sum_{\substack{p \textrm{ singular}\\ \textrm{point}}} d(p).
\]
    \end{itemize}
\end{dfn}
 The degree can be thought of as an analogue of the Euler characteristic.

\begin{dfn}\label{dfn:color}
  A \emph{pigment} is an element of $\Col= \{1,\dots,N\}$. If $A$ is a subset of $\Col$, $\overline{A}$ denotes $\Col\setminus A$. The set $\Col$ is endowed with the canonical order.

  A \emph{coloring} of a foam $F$ is a map $c\co \mathcal{F}(F) \to \mathcal{P}(\Col)$, such that:
  \begin{itemize}
  \item For each facet $f$, the number of elements $\# c(f)$ of $c(f)$ is equal to $l(f)$.
  \item For each binding joining a facet $f_1$ with label $a$, a facet $f_2$ with label $b$, and a facet $f_3$ with label $a+b$, we have $c(f_1) \cup c(f_2) = c(f_3)$. This condition is called the \emph{flow condition}.
  \end{itemize}
A \emph{colored foam} is a foam together with a coloring.
\end{dfn}

A facet labeled $0$ is colored by the empty set.
A careful inspection of how colorings are around bindings and singular points gives the following lemma:

\begin{lem}\label{lem:monobicercle}
  \begin{enumerate}
  \item If $(F,c)$ is a colored foam and $i$ is an element of $\Col$, the union (with the identification coming from the gluing procedure) of all the facets which contain the pigment $i$ in their colors is a surface. It is called the \emph{monochrome surface of $(F,c)$ associated with $i$} and it is denoted $F_i(c)$. The restriction we imposed on the orientations of facets ensures that $F_i(c)$ is oriented.
  \item  If $(F,c)$ is a colored foam and $i$ and $j$ are two distinct elements of $\Col$, the union (with the identification coming from the gluing procedure) of all the facets which contain $i$  or $j$ but not both in their colors is a surface. It is called the \emph{bichrome surface of $(F,c)$ associated with $i,j$}. It is the symmetric difference of $F_i(c)$ and $F_j(c)$ and it is denoted $F_{ij}(c)$. The restriction we imposed on the orientations of facets ensures that $F_{ij}(c)$ can be oriented by taking the orientation of the facets containing $i$ and the reverse orientation on the facets containing $j$.
  \item In the same situation, we may suppose $i<j$. We consider a binding joining the facets $f_1$, $f_2$ and $f_3$. Suppose that $i$ is in $c(f_1)$, $j$ is in $c(f_2)$ and $\{i,j\}$ is  in $c(f_3)$. We say that the binding is \emph{positive with respect to $(i,j)$} if the cyclic order on the binding is $(f_1, f_2, f_3)$ and \emph{negative with respect to $(i,j)$} otherwise. The set $F_{i}(c) \cap F_{j}(c) \cap F_{ij}(c)$ is a collection of disjoint circles. Each of these circles is a union of bindings, for every circle the bindings are either all positive or all negative with respect to $(i,j)$. 
\end{enumerate}
\end{lem}

Please note that the previous lemma contains the definition of \emph{monochrome} and \emph{bichrome surfaces}. It yields a definition of \emph{positive} and \emph{negative circles}:

\begin{dfn}
Let $(F,c)$ be a colored foam and $i<j$ be two pigments. A circle in $F_{i}(c) \cap F_{j}(c) \cap F_{ij}(c)$ is 
\emph{positive} (resp. \emph{negative}) \emph{with respect to $(i,j)$} if it consists of positive (resp. negative)
bindings. 
We denote by $\theta^+_{ij}(c)_F$ (resp. $\theta^-_{ij}(c)_F$) or simply $\theta^+_{ij}(c)$ (resp. $\theta^-_{ij}(c)$) the number of positive (resp. negative) circles with respect to $(i,j)$. We set  $\theta_{ij}(c)= \theta^+_{ij}(c) +\theta^{-}_{ij}(c)$. See Figure \ref{fig:signsofcircles} for a pictorial definition.
\end{dfn}

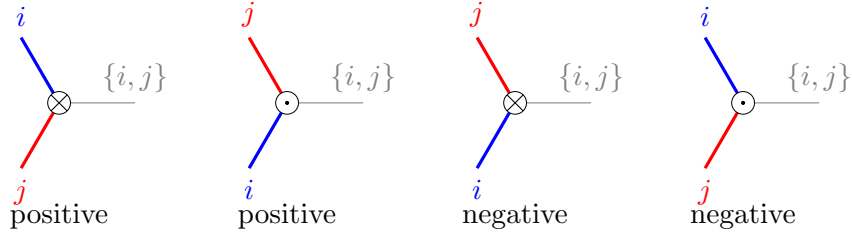
\begin{figure}[h]
  \centering
  \begin{tikzpicture}
    \input{\imagesfolder/cef_signsofcircles}
  \end{tikzpicture}
  \caption{A pictorial definition of the signs of the circle. Here we assume $i<j$.}
  \label{fig:signsofcircles}
\end{figure}

We now inspect the relationships between the different notions which have been introduced in Lemma~\ref{lem:monobicercle} and see how they behave when changing the coloring of a given foam. It will not be always possible to write proper identities, and we will use the symbol $\equiv$ to mean equality modulo 2.  

\begin{lem}\label{lem2.5}
  Let $(F,c)$ be a colored foam and $i$ and $j$ two distinct elements
  of $\Col$. We have:
  \begin{align*}
    \chi(F_{i j}(c)) &=   \chi(F_{i}(c))  + \chi(F_{j}(c)) - 2\chi(F_{i\cap j}(c)) \\
    \chi(F_{i\cap j}(c)) &\equiv \theta_{ij}(c)_F
  \end{align*}
  where $F_{i\cap j}(c)$ denotes the surface (with boundary)
  consisting of the union of all facets of $(F,c)$ containing both $i$
  and $j$ in their colors.
\end{lem}

\begin{proof}
  The first identity follows from the facts that $F_{i j}(c)$ is the
  symmetric difference of $F_i(c)$ and $F_{j}(c)$. The second follows
  from the fact that the boundary of $F_{i\cap j}(c)$ consists
  precisely of $\theta_{ij}(c)_F$ circles.
\end{proof}

\begin{rmk}
  We have an action of $\mathfrak{S}_N$, the symmetric group on $N$ letters, on the set of colorings of a
  given foam $F$ by permuting the pigments. We also have a more
  local move given by exchanging $i$ and $j$ along a closed sub surface $\Sigma$ of $F_{ij}(c)$. This is to be understood as a
  foamy analogue of the Kempe move for edge-colorings of graphs see for instance \cite{MR3408102}, therefore
  we denote such a move by \emph{Kempe move along $\Sigma$ relative
  to $i$ and $j$}. It naturally brings the question to know whether or not all colorings of a given foam can be related by Kempe moves.

\end{rmk}

\begin{lem}\label{lem2.7}
  Let $F$ be a foam and $c$ and $c'$ be two colorings of $F$ related
  by a Kempe move relative to $1$ and $2$. If $k\geq 3$ then we have
  \[
    \theta_{1k}^{+}(F,c) + \theta_{2k}^+(F,c) \equiv
    \theta_{1k}^{+}(F,c') + \theta_{2k}^+(F,c').
  \]
\end{lem}

\begin{proof}
  Let $\Sigma$ be the surface along which the Kempe move relates $c$
  and $c'$. We can distinguish three kinds of positive circles of type
  $(1,k)$ (resp. $(2,k)$) in $(F,c)$ (resp. resp. $(F,c')$):

  \begin{itemize}
  \item The ones disjoint from $\Sigma$, we denote by
    $\theta^{d+}_{1k}(F,c)$ (resp. $\theta^{d+}_{2k}(F,c)$,
    $\theta^{d+}_{1k}(F,c')$, $\theta^{d+}_{2k}(F,c')$) the number of such
    circles;
  \item The ones included in $\Sigma$, we denote by $\theta^{i+}_{1k}(F,c)$
    (resp. $\theta^{i+}_{2k}(F,c)$, $\theta^{i+}_{1k}(F,c')$,
    $\theta^{i+}_{2k}(F,c')$) the number of such circles;
  \item The other ones (which we call the \emph{mixed circles}), we
    denote by $\theta^{m+}_{1k}(F,c)$ (resp. $\theta^{m+}_{2k}(F,c)$,
    $\theta^{m+}_{1k}(F,c')$, $\theta^{m+}_{2k}(F,c')$) the number of such
    circles.
  \end{itemize}
  Since $1<k$ and $2<k$, we have:
  \begin{align*}
    \theta_{1k}^{d+}(F,c) &= \theta_{1k}^{d+}(F,c'), &&&  \theta_{2k}^{d+}(F,c) &= \theta_{2k}^{d+}(F,c'), \\
    \theta_{1k}^{i+}(F,c) &= \theta_{2k}^{i+}(F,c'), &\textrm{and}&&   \theta_{2k}^{i+}(F,c) &= \theta_{1k}^{i+}(F,c'). \\
  \end{align*}
  Hence we only need to prove
  \[
    \theta_{1k}^{m+}(F,c) + \theta_{2k}^{m+}(F,c) \equiv
    \theta_{1k}^{m+}(F,c') + \theta_{2k}^{m+}(F,c').
  \]
  
  Let us consider a mixed circle of type $(1,k)$ in $(F,c)$. It is a union of arcs in $\Sigma$ and of arcs disjoint from $\Sigma$. Let us consider an end of an arc of type $(1,k)$ in $\Sigma$. It has to be a singular point, moreover, if at this singular point the color $2$ is not involved, then all the bindings adjacent to this singular point are in $\Sigma$. 
  \begin{figure}[h]
    \centering
    \begin{tikzpicture}[scale= 0.9]
      \input{\imagesfolder/cef_singpointsigma}
    \end{tikzpicture}
    \caption{A singular point in $\Sigma$.
Up to symmetry, the two cases are: $1$ is in the facet labeled by $a$ and $2$ in the facet labeled by $b$ (this is depicted on the middle) and $1$ is in the facet labeled by $a$ and $2$ in the facet labeled by $c$ (depicted on the left). The surface $\Sigma$ is in red. An arc can end only if the local situation is depicted on the middle.}
    \label{fig:singpointsigma}
  \end{figure}
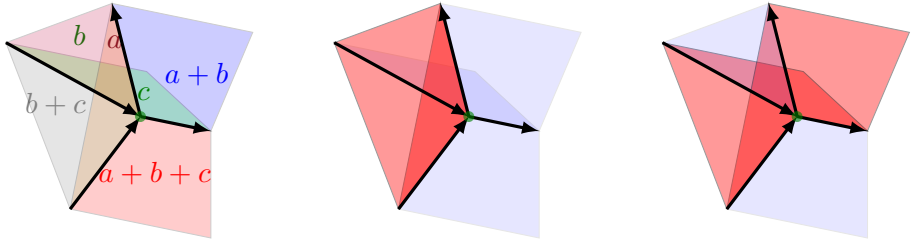

There are essentially two different singular points (see Figure~\ref{fig:singpointsigma}) on $\Sigma$ where both colors $1$ and $2$ are involved, but for only one of them, there is an adjacent binding which is not in $\Sigma$. A careful inspection of the situation shows that a singular point in $\Sigma$ is the end of an arc of type $(1,k)$ if and only if it is an end of an arc of type $(2,k)$, and the only binding going out of $\Sigma$ at this singular points is part of a positive circle of type $(1,k)$ and of a positive circle of type $(2,k)$ in $(F,c)$. Let us call such a singular point a \emph{critical singular point} for this proof. In a neighborhood of a critical singular point, the Kempe moves along $\Sigma$ performs the following change on the positive mixed circles of type  $(1,k)$ and $(2,k)$:

\[
  \begin{tikzpicture}
    \input{\imagesfolder/cef_kempemovethetaplus}
  \end{tikzpicture}
\]

If we now consider the abstract family $\mathcal{A}$ of oriented positive mixed circles of type $(1,k)$ and of type $(2,k)$ (i. e. we forgot about the colors). We see that at each the critical singular point the Kempe move along $\Sigma$ induces the following change on $\mathcal{A}$:

\[
\tikz{\input{\imagesfolder/cef_localchangetheta}}
\]

This changes the cardinality of $\mathcal{A}$ by $\pm1$. Since there are an even number of such singular points (twice as many as arcs positive arcs of type $(1,k)$ in $\Sigma$), this gives:

 \[\theta_{1k}^{m+}(F,c) + \theta_{2k}^{m+}(F,c) \equiv
    \theta_{1k}^{m+}(F,c') + \theta_{2k}^{m+}(F,c').  \]

\end{proof}

\begin{lem}\label{lem2.8}
  Let $F$ be a foam and $c$ and $c'$ be two colorings of $F$ related
  by a Kempe move relative to $1$ and $2$ along a surface $\Sigma$
  and $k\geq 3$. There exists an integer $\ell_\Sigma(c,k)$ such that:
  \[
    \chi(F_{1k}(c')) = \chi(F_{1k}(c)) + \ell_\Sigma(c,k) \quad \textrm{and} \quad
    \chi(F_{2k}(c')) = \chi(F_{2k}(c)) -\ell_\Sigma(c,k)
  \]
  Furthermore, this integer $\ell_\Sigma(c,k)$ depends only on how the
  surface $\Sigma$ is colored with respect to $1$, $2$ and $k$.
\end{lem}

\begin{proof}
  This follows from the locality of the Euler characteristic: on the
  surface $\Sigma$, the relevant facets, edges and vertices used to
  compute $\chi(F_{1k}(c))$ and $\chi(F_{2k}(c))$ are exchanged when
  computing $\chi(F_{1k}(c'))$ and $\chi(F_{2k}(c'))$.
\end{proof}

\subsection{A symmetric evaluation of foams}
\label{sec:symmetric-evaluation}
In what follows, $X\eqdef (X_1, X_2, \dots , X_N)$ is a set of variables, we fix their degree to be equal to $2$. 
If $A$ is a subset of $\Col$, and $P$ is a symmetric polynomial in $|A|$ variables, $P(A)$ denotes $P((X_a)_{a\in A})$.

\begin{dfn}\label{dfn:decoratedfoam}
  A \emph{decoration} of a foam $F$ is a map $f\mapsto P_f$ which associates with any facet $f$ of $F$ an homogeneous symmetric polynomial $P_f$ in $l(f)$ variables. A \emph{decorated foam} is a foam together with a decoration. A decorated foam has a \emph{degree} $d_N$ given by:
\[d_N(F) = d_N^{\textrm{undecorated}}(F) + \sum_{f} \deg(P_f).\]
\end{dfn}

  \begin{dfn} If $(F,c)$ is a colored decorated foam, we define:
    \begin{align*}
      s(F,c) &= \sum_{i=1}^N {i\chi(F_i(c))/2}  + \sum_{1\leq i < j \leq N} \theta^+_{ij}(F,c), \\
      P(F,c) &= \prod_{f \textrm{ facet of $F$}} P_f(c(f)),\\
      Q(F,c) &= \prod_{1\leq i < j \leq N} (X_i-X_j)^{\frac{\chi(F_{ij}(c))}{2}} \\
      \kup{F,c} &= (-1)^{s(F,c)} \frac{P(F,c)}{Q(F,c)}.
    \end{align*}
    If $F$ is a decorated foam, we define the \emph{evaluation} of the foam $F$ by:
    \[ \kup{F} := \sum_{c \textrm{ coloring of $F$}} \kup{F,c}. \]
  \end{dfn}

  \begin{rmk}\label{rmk:remove0faces}
    If $F$ is a foam and $F'$ the same foam where the facets with label $0$ have been removed, then there is a one-one correspondence between the colorings of $F$ and the colorings of $F'$. For every coloring $c$ of $F$ and its corresponding coloring $c'$ of $F'$, we have $\kupc{F,c} = \kupc{F',c'}$. This implies $\kup{F} = \kup{F'}$.  \end{rmk}

\begin{lem}\label{lem:degreefit}
   Let $F$ be an decorated $\sll_N$-foam, and $c$ be a coloring of $F$. Then the rational function  $\kup{F,c}$ is homogeneous and has degree $d_N(F)$.
\end{lem}

\begin{proof}
   The homogeneity is clear. The rest follows from the following facts: 
   \begin{itemize}
   \item a facet $f$ with label $a$ contributes with $\chi(f)$ to the Euler characteristic of exactly $a(N-a)$ bi-chromatic surfaces,
   \item a segment-like binding of type $(a,b,a+b)$ contributes with $-1$ to the Euler characteristic of exactly $ab + (a+b) (N-a-b)$
bi-chromatic surfaces,
   \item a point of type $(a,b,c, a+b, a+c, a+b+c)$ contributes with $+1$ to the Euler characteristic of exactly $ab + bc+ cd + da + ac+ bd $ (with $d = N- {a+b+c}$) bi-chromatic surfaces,
   \item the factor $\frac{1}{2}$ in the exponent of $(X_i - X_j)$ is balanced by the variables $X_\bullet$ having degree 2.
   \end{itemize}
\end{proof}

\begin{cor}\label{cor:degreefit}
  Let $F$ be an $\sll_N$-foam, then the rational function $\kup{F,c}$ is either $0$ or is homogeneous of degree $d_N(F)$.
\end{cor}

\begin{lem}\label{lem:transposition}
  Let $F$ be a foam and $c$ a coloring of $F$. Let us denote by $c'$ the coloring of $F$
  obtained by exchanging the colors $i$ and $i+1$ in $c$. We have:
  \[
    \kup{F, c'}(X_1, \dots X_N) = \kupc{F, c}(X_1, \dots, X_{i-1},
    X_{i+1}, X_{i}, X_{i+2}, \dots, X_N).
  \]
\end{lem}
\begin{proof}
  Let us denote by $X'$ the variables $X$ with $X_i$ and $X_{i+1}$
  exchanged. We want to prove
  \[\kupc{F, c'}(X) = \kupc{F, c}(X').\]
  We have:
  \[ P(F, c')(X) = P(F,c)(X') \and Q(F,c')(X) =
    (-1)^{\chi(F_{i(i+1)}(c))/2}Q(F,c)(X'). \] We now need to compute
  $\Delta(s):= s(F,c) - s(F,c')$:
  \begin{align*}
    \Delta(s) &=  \sum_{j=1}^N {j\chi(F_j(c))/2}  + \sum_{1\leq j< k \leq N} \theta^+_{jk}(F,c) -  \sum_{j=1}^N {j\chi(F_j(c'))/2}  - \sum_{1\leq j < k \leq N} \theta^+_{jk}(F,c')\\
    &= i \frac{\chi(F_i(c)) - \chi(F_i(c'))}2 + (i+1) \frac{\chi(F_{i+1}(c)) -  \chi(F_{i+1}(c'))}2 \\
    &\qquad + \sum_{j \neq i, i+1} \left(\theta^+_{ij}(F,c) - \theta^+_{(i+1)j}(F,c') \right) +  \sum_{j \neq i, i+1}\left( \theta^+_{(i+1)j}(F,c) - \theta^+_{ij}(F,c') \right) \\
    &\qquad + \theta^+_{i(i+1)}(F,c) - \theta^{+}_{i(i+1)}(F,c')\\
    &= \frac{\chi(F_{i+1}(c)) -  \chi(F_{i}(c))}2 + \theta^+_{i(i+1)}(F,c) - \theta^{-}_{i(i+1)}(F,c)\\
    &\equiv \frac{\chi(F_{i+1}(c)) +  \chi(F_{i}(c))}2 + \theta_{i(i+1)}(F,c) \\
    &\equiv \frac{\chi(F_{i (i+1)}(c))}2.
  \end{align*}
  The penultimate equality follows from Lemma \ref{lem2.5} and the fact that the Euler characteristic of closed surfaces is even.
  This gives:
  \[\kupc{F, c'}(X) = (-1)^{s(F,c')} \frac{P(F,c')(X)}{Q(F,c')(X)} =
    (-1)^{s(F,c)} \frac{P(F,c)(X')}{Q(F,c)(X')} = \kupc{F, c}(X') \]
\end{proof}

\begin{cor}
  The rational fraction $\kup{F}(X)$ is symmetric.
\end{cor}
\begin{proof}
  It is enough to prove that $\kup{F}(X) = \kup{F}(X')$ where $X'$ is
  like in the previous proof for an arbitrary $i$. We have:
  \begin{align*}
    \kup{F}(X') &= \sum_{c \textrm{ coloring of $F$}} \kup{F,c} (X') \\
    &= \sum_{c \textrm{ coloring of $F$}} \kup{F,c'} (X) \\
    &= \sum_{c' \textrm{ coloring of $F$}} \kup{F,c'} (X) \\
    &= \kup{F}(X).
  \end{align*}
\end{proof}

\begin{prop}\label{prop:sympol}
  Let $F$ be a foam, then $\kup{F}(X)$ is a (symmetric) polynomial.
\end{prop}

The rest of this section is devoted to the proof of 
Proposition~\ref{prop:sympol}. From its very definition $\kup{F}(X)$
is a rational function which can be written in the following form:

\[\kupc{F}(X) = \frac{F_1(X)}{\prod_{i<j}(X_i - X_j)^k}
\]
for some non-negative integer $k$ and some polynomial $F_1$. The
polynomial $F_1$ is either symmetric or anti-symmetric depending on
the parity of $k$. By symmetry it is enough to check that
$(X_1-X_2)^k$ divides $F_1(X)$.

\begin{lem}\label{lem:signKempe}
  Let $F$ be a foam and $c$ and $c'$ be two colorings of $F$ related
  by a Kempe move relative to $1$ and $2$ along a surface $\Sigma$,
  then
  \[
    s(F,c') \equiv s(F,c) + \chi(\Sigma)/2.
  \]
\end{lem}

\begin{proof}Note that this is somewhat a local version of the signs
  statement in the proof of Lemma~\ref{lem:transposition}.  We want to
  compute $\Delta(s) = s(F,c) - s(F,c')$:
  \begin{align*}
    \Delta(s) &=  \sum_{j=1}^N {j\chi(F_j(c))/2}  + \sum_{1\leq j< k \leq N} \theta^+_{jk}(F,c) -  \sum_{j=1}^N {j\chi(F_j(c'))/2}  - \sum_{1\leq j < k \leq N} \theta^+_{jk}(F,c')\\
    &\equiv \chi(F_1(c))/2 - \chi(F_1(c'))/2  + \theta^+_{12}(F,c) - \theta^+_{12}(F,c') \\
    &\equiv \frac{\chi(F_1(c) \cap \Sigma) - \chi(F_2(c) \cap \Sigma) }2 + \#\{\textrm{circles of type $(1,2)$ on $\Sigma$}\} \\
    & \equiv \frac{\chi(F_1(c) \cap \Sigma) - \chi(F_2(c) \cap \Sigma) }2 + \chi(F_2(c) \cap \Sigma) \\
    & \equiv \frac{\chi(F_1(c) \cap \Sigma) + \chi(F_2(c) \cap \Sigma) }2  \\
    & \equiv \frac{\chi(\Sigma)}{2}.  \\
  \end{align*}
  The second equality follows from Lemma \ref{lem2.7}

\end{proof}

\begin{proof}[Proof of Proposition~\ref{prop:sympol}] Let $F$ be a
  decorated foam.

  Let $\mathbf{\Sigma}=(\Sigma_1, \Sigma_2, \dots \Sigma_r)$ be a
  collection of disjoint connected closed surface in $F$. Let $c$ be a
  coloring of $F$ such that $F_{12}= \bigcup_{s=1}^r \Sigma_s$. We
  consider the set $C_c$ of colorings of $F$ obtained from $c$ by a sequence of Kempe moves relative to $1$ and $2$ along (part
  of) $\mathbf{\Sigma}$. There are precisely $2^r$ of them. We will
  show that
  \[\sum_{c'\in C_c} \kupc{F,c'} = \frac{A(X)}{B(X)}
  \]
  with $A(X)$ and $B(X)$ two polynomials and $B(X)$ is not divisible by
  $(X_1 - X_2)$. We need to set up a few notations to facilitate the computations.
  \begin{itemize}
  \item We denote by $X'$ the variables $X$ where $X_1$ and $X_2$ has
    been exchanged.
  \item We write
    \[P_{\Sigma_s,c}(X) = \prod_{f \textrm{ facet in $\Sigma_s$}}
      P({c(f)})\] and \[P_{F\setminus \mathbf{\Sigma},c}(X) = \prod_{f
      \textrm{ facet not in $\bigcup_s \Sigma_s$}} P(c(f)).\]
  \item We write
    \[\widetilde{Q}(X) = \frac{Q(F,c)\prod_{k>2,s} (X_1-X_k)^{\ell_{\Sigma_s}(c,k)/2} }{(X_1-X_2)^{\chi{F_{12}(c)}/2}
      },\] where $\ell_{\Sigma_s}(c,k)$ is defined in Lemma \ref{lem2.8}. Note that
    for all $c'$ in $C_c$, $\chi{F_{12}(c)} = \chi{F_{12}(c')}$ and
    that $(X_1 -X_2)$ is not a factor of $\widetilde{Q}(X)$.
  \item Finally we set
    \begin{align*}& T_s(X) =\left(P_{\Sigma_s,c}(X) \prod_{k=3}^N(X_1- X_k)^{\ell_{\Sigma_s}(c,k)/2} \right.\\ &\left.\qquad \qquad \quad+ (-1)^{\chi(\Sigma_s)/2} P_{\Sigma_s,c}(X') \prod_{k=3}^N(X_2- X_k)^{\ell_{\Sigma_s}(c,k)/2} \right).\end{align*}
  \end{itemize}
  We claim that the following identity holds:
  \begin{align*}
    \sum_{c'\in C_c} \kupc{F,c'} &= (-1)^{s(F,c)}\frac{P_{F\setminus
    \mathbf{\Sigma},c}(X)}{\widetilde{Q}(X)} \prod_{s=1}^r
    (X_1-X_2)^{-\chi(\Sigma_s)/2} T_s(X).
  \end{align*}
  Indeed developing the product gives exactly $2^r$ terms
  corresponding to the $2^r$ colorings of $C_c$. For every $s$ in $\{1, \dots r\}$, every coloring of $C_c$ either agrees with $c$ on $\Sigma_s$ or disagrees. These two possibilities correspond to the two terms in $T_s(X)$.
The signs are correct thanks to Lemma~\ref{lem:signKempe}.

  The only possibility to have a power of $(X_1-X_2)$ at the
  denominator arises when $\Sigma_s$ is a sphere. However in this
  case, $\chi(\Sigma_s)/2=1$ and
  \[T_s(X) =\left(P_{\Sigma_s,c}(X) (X_1- X_k)^{\ell_{\Sigma_s} (c,k)/2}
      - P_{\Sigma_s,c}(X') (X_2- X_k)^{\ell_{\Sigma_s}(c,k)/2} \right)
  \]
  is anti-symmetric\footnote{One may have to multiply $T_s$ with an
  appropriate power of $(X_1 -X_k)(X_2-X_k)$ to make it a
  polynomial. This does not affect the argumentation.} in $X_1$ and
  $X_2$. This implies that it is divisible by $(X_1- X_2)$. This
  proves that $\sum_{c'\in C_c} \kupc{F,c'} $ can be written as
  rational fraction without $(X_1-X_2)$ in the denominator. Since
  $\kup{F}$ is a sum contribution of the type
  $\sum_{c'\in C_c} \kupc{F,c'}$, this proves that $\kup{F}$ can be
  written as rational fraction without $(X_1-X_2)$ in the
  denominator. Since this rational fraction is symmetric in $X$ and
  has a denominator of the form $\prod_{i<j}(X_i - X_j)^k$, this
  proves that $\kup{F}$ is a polynomial.
\end{proof}

\subsection{Two remarks}
\label{sec:2-rks}

\subsubsection{Generalized foams}
\label{sec:generalized-foams}
As in \cite{2015arXiv151202370R} for MOY graphs, one could have here a more general version of foams. Instead of requiring the labels of the facets adjacent to a binding to be of the form $(a, b, a+b)$ with an additional condition on the orientation, one can ask that the labels (or minus the labels, depending on the orientation) add up to a multiple of $N$. The colorings of such foams are required to satisfy another condition along the bindings:
The union (as multi-sets) of the colors of the facets (or the complement of the color, depending on the orientation) are required to be a multiple of the set $\Col$. We believe this more flexible model can be handy for some applications.

The evaluation has to be changed a little bit. The monochrome components are no longer only closed surfaces, but can have boundary and singularities. Therefore $\chi(F_i)$ can be odd and $(-1)^{i\chi(F_i)/2}$ does not make sense anymore. We choose a square root $\omega$ of $-1$ and replace $(-1)^{i\chi(F_i)/2}$ with  $\omega^{i\chi(f_i)}$. 

If we look at the case $N=2$, there are essentially two theories which give satisfactory treatment of signs:
\begin{itemize}
\item  Blanchet's foams \cite{1195.57024}, this corresponds to our treatment.
\item  The \emph{disoriented} model from Clark--Morrison--Walker \cite{MR2496052} and the \emph{singular} model from Caprau \cite{MR2443094} . They correspond to this more general treatment, where all facets with labels $0$ are erased, but the bindings (which are only circles in this case) still exist. 
\end{itemize}
Hence the appearance of the fourth root of unity in \cite{MR2496052} and the fact that it does not in~\cite{1195.57024} becomes clear from the perspective of our formula.\\

\subsubsection{Sub-$\sll_{k}$-foams}
\label{sec:sub-sl_k-foams}

In our definition of the evaluation of colored foams, the term
\[
  \prod_{1\leq i < j \leq N} (X_i-X_j)^{\frac{\chi(F_{ij}(c))}{2}}(-1)^{\theta^+_{ij}(F,c)}.
\]
can be seen as the product of the contribution à la Blanchet \cite{1195.57024} of all colored sub-$\sll_2$-foams in the colored foam $(F,c)$. One can therefore try to express this part of 
the evaluation of colored $\sll_N$-foams as a product of sub-$\sll_{k}$-foams for arbitrary $k$. In principle one should then replace the previous product by:
\[
 (-1)^{k(k+1)/2}\prod_{1\leq i_1 < i_2 < \dots < i_k \leq N} \Delta((X_{i_r})_{1\leq r \leq k})^{\frac{d_k(F_{i_1\dots i_k}(c))}{2{k \choose 2} }}(-1)^{\frac{\sum_{1\leq r< s \leq k}\theta^+_{i_{r}i_s} (F,c)}{ {k \choose 2} }},
\]
where $\Delta$ stands for the Vandermonde determinant and $(F_{i_1\dots i_k}(c))$ denotes the sub-$\sll_k$-foam given by the pigments $i_1, \dots i_k$. However, it does not work since there is no reason for $d_k(F_{i_1\dots i_k}(c))$ to be divisible by $2{k \choose 2}$ except for $k=2$. 
One can still imagine modifying our formula to make sense of this idea but it will require more work.

%% file: cef_3localmodels.tex
\begin{scope}
\tdplotsetmaincoords{80}{140}
  \begin{scope}[tdplot_main_coords]
    \filldraw [very thin, fill=red, opacity = 0.2] (0,-1, -1) -- (0,1,-1) -- (0,1,1) -- (0,-1,1) -- (0,-1,-1);
    \node[sloped, red] at (0,0,0) {$a$};
  \end{scope}
\begin{scope}[xshift = 3cm, tdplot_main_coords]
  \begin{scope}
    \filldraw [very thin, fill =red, opacity = 0.2] (0,-1, -1) -- (0,1,-1) -- (0,1,0) -- (0,-1,0) -- (0,-1,-1);
        \node[sloped, red] at (0,0,-0.5) {$a+b$};
    \end{scope}
  \begin{scope}[rotate around y = 150]
    \filldraw [very thin, fill =blue, opacity = 0.2] (0,-1, -1) -- (0,1,-1) -- (0,1,0) -- (0,-1,0) -- (0,-1,-1);
        \node[sloped, blue] at (0,0,-0.5) {$a$};
  \end{scope}
  \begin{scope}[rotate around y =-130]
    \filldraw [very thin, fill =green, opacity = 0.2] (0,-1, -1) -- (0,1,-1) -- (0,1,0) -- (0,-1,0) -- (0,-1,-1);
        \node[sloped, green!50!black] at (0,0,-0.5) {$b$};
        \draw[very thick, ->] (0,1,0) -- (0,-1,0);
  \end{scope}
  \end{scope}
\begin{scope}[scale = 1.6, xshift = 4.5cm, tdplot_main_coords]
  \begin{scope}
    \filldraw [very thin, fill =red, opacity = 0.2] (0,-1, -1) -- (0,1,-1) -- (0,1,0) -- (0,0,0) -- (0,-1,-1);
    \coordinate (a) at (0, -1, -1);
        \node[sloped, red] at (0,0.2,-0.5) {$a+b+c$};
    \end{scope}
  \begin{scope}[rotate around y = 150]
    \filldraw [very thin, fill= blue, opacity = 0.2] (0,-1, -1) -- (0,1,-1) -- (0,1,0) -- (0,0,0) -- (0,-1,-1);
    \coordinate (b) at (0, -1, -1);
        \node[sloped, blue] at (0,0.5,-0.5) {$a+b$};
  \end{scope}
  \begin{scope}[rotate around y =-130]
    \filldraw [very thin, fill= green, opacity = 0.2] (0,-1, -1) -- (0,1,-1) -- (0,1,0) -- (0,0,0) -- (0,-1,-1);
    \coordinate (c) at (0, -1, -1);
        \node[sloped, green!50!black] at (0,0.5,-0.5) {$c$};
        \node[sloped, purple!50!black] at (0,1,-1.5) {$a$};
        \node[sloped, green!50!black] at (0,+0.5,-1.5) {$b$};
        \node[sloped, gray] at (0,-1.3,0.1) {$b+c$};
  \end{scope}
  \filldraw[very thin, fill= orange, opacity = 0.2] (a) -- (b) -- (0,0,0) -- (a);
  \filldraw[very thin, fill= gray, opacity = 0.2] (a) -- (c) -- (0,0,0) -- (a);
  \filldraw[very thin, fill= purple, opacity = 0.2] (b) -- (c) -- (0,0,0) -- (b);
  \draw[very thick, ->] (a) -- (0,0,0);
  \draw[very thick, <-] (b) -- (0,0,0);
  \draw[very thick, ->] (c) -- (0,0,0);
  \draw[very thick, <-] (0,1, 0) -- (0,0,0);
  \fill[green!50!black, opacity = 0.6] (0,0,0) circle (0.5mm);
\end{scope}
\end{scope}

%% file: sw_lhrule.tex
\begin{scope}
 \draw (0,0) -- +(0:1);
 \draw (0,0) -- +(120:1);
 \draw (0,0) -- +(240:1);
 \filldraw[fill= white, draw=black, very thin] (0,0) circle (0.15cm);
 \filldraw[fill = black] (0,0) circle (0.02cm);
 \draw[very thin,->] (-10:0.8cm) arc (-10:-110 :0.8); 
 \draw[very thin,->] (230:0.8cm) arc (230: 130:0.8); 
 \draw[very thin,->] (110:0.8cm) arc (110:10 :0.8); 
\end{scope}

%% file: cef_foam6j.tex
\begin{scope}
\tdplotsetmaincoords{40}{20}
\begin{scope}[tdplot_main_coords]
  \fill[green!50!black] (0,2) circle (1mm);
  \fill[green!50!black] (0,-2) circle (1mm);
  \begin{scope}[rotate around y=0]
    \coordinate (b) at (1,0);
    \fill [fill =red,fill opacity=0.4 ] (0,2) arc (90:-90:2) ;
    \draw [draw = black, very thick , ->, opacity=0.8 ] (0,2) arc (90:-90:2) ;
  \end{scope}
  \begin{scope}[rotate around y=120]
        \coordinate (a) at (1,0);
    \fill [green ,fill opacity=0.3 ] (0,2) arc (90:-90:2) ;
    \draw [draw = black, very thick , <-, opacity=0.8 ] (0,2) arc (90:-90:2) ;
  \end{scope}
  \begin{scope}[rotate around y=240]
    \coordinate (c) at (1,0);
    \fill [orange, fill opacity=0.3 ] (0,2) arc (90:-90:2) ;
    \draw [draw = black, very thick , <-,opacity=0.8 ] (0,2) arc (90:-90:2) ;
  \end{scope}
  \draw[very thick, -> ] (0,2) -- (0,-2);
  \shade [ball color=blue!40!white,opacity=0.30] (0,0,0) circle (2.02cm);
\end{scope}
  \coordinate (d) at (45:2);
  \coordinate (e) at (-90:2);
  \coordinate (f) at (180:2);
  \draw[very thin, -> ] (-4,0)  node [left] {$3$} -- (f);
  \draw[very thin, -> ] (-4,1)    node [left] {$2$} -- (c);
  \draw[very thin, -> ] (-4,-1)  node [left] {$5$} -- (a);
  \draw[very thin, -> ] (4,0)   node [right] {$3$} -- (b);
  \draw[very thin, -> ] (4,-1)   node [right] {$2$} -- (e);
  \draw[very thin, -> ] (4,1)    node [right] {$1$}-- (d);
\end{scope}

%% file: cef_signsofcircles.tex
\begin{scope}[xshift = 0cm]
 \draw[gray] (0,0) -- +(0:1)  node[above] {$\{i,j\}$};
 \draw[blue, very thick ] (0,0) -- +(120:1) node[above] {${i}$};;
 \draw[red, very thick] (0,0) -- +(240:1) node[below] {${j}$};;
 \filldraw[fill= white, draw=black, very thin] (0,0) circle (0.15cm);
 \filldraw[fill = black] (0,0) circle (0.02cm);
 \node at (0, -1.5) {negative};
\end{scope}

 \begin{scope}[xshift = -3cm]
  \draw[gray] (0,0) -- +(0:1) node[above] {$\{i,j\}$};
  \draw[red, very thick ] (0,0) -- +(120:1) node[above] {${j}$};;
  \draw[blue, very thick] (0,0) -- +(240:1) node[below] {${i}$};;
  \filldraw[fill= white, draw=black, very thin] (0,0) circle (0.15cm);
  \draw (45:0.15) --  (-135:0.15);
  \draw (-45:0.15) -- (135:0.15);
  \node at (0, -1.5) {negative};
 \end{scope}

\begin{scope}[xshift = -9cm]
 \draw[gray] (0,0) -- +(0:1) node[above] {$\{i,j\}$};
 \draw[blue, very thick ] (0,0) -- +(120:1) node[above] {${i}$};;
 \draw[red, very thick] (0,0) -- +(240:1) node[below] {${j}$};;
 \filldraw[fill= white, draw=black, very thin] (0,0) circle (0.15cm);
  \draw (45:0.15) --  (-135:0.15);
  \draw (-45:0.15) -- (135:0.15);
 \node at (0, -1.5) {positive};
\end{scope}

\begin{scope}[xshift = -6cm]
 \draw[gray] (0,0) -- +(0:1) node[above] {$\{i,j\}$};
 \draw[red, very thick ] (0,0) -- +(120:1) node[above] {${j}$};;
 \draw[blue, very thick] (0,0) -- +(240:1) node[below] {${i}$};;
 \filldraw[fill= white, draw=black, very thin] (0,0) circle (0.15cm);
 \filldraw[fill = black] (0,0) circle (0.02cm);
 \node at (0, -1.5) {positive};
\end{scope}

%% file: cef_singpointsigma.tex
\begin{scope}
\tdplotsetmaincoords{80}{140}
\begin{scope}[scale = 1.6, xshift = 0cm, tdplot_main_coords]
  \begin{scope}
    \filldraw [very thin, fill =red, opacity = 0.2] (0,-1, -1) -- (0,1,-1) -- (0,1,0) -- (0,0,0) -- (0,-1,-1);
    \coordinate (a) at (0, -1, -1);
        \node[sloped, red] at (0,0.2,-0.5) {$a+b+c$};
    \end{scope}
  \begin{scope}[rotate around y = 150]
    \filldraw [very thin, fill= blue, opacity = 0.2] (0,-1, -1) -- (0,1,-1) -- (0,1,0) -- (0,0,0) -- (0,-1,-1);
    \coordinate (b) at (0, -1, -1);
        \node[sloped, blue] at (0,0.5,-0.5) {$a+b$};
  \end{scope}
  \begin{scope}[rotate around y =-130]
    \filldraw [very thin, fill= green, opacity = 0.2] (0,-1, -1) -- (0,1,-1) -- (0,1,0) -- (0,0,0) -- (0,-1,-1);
    \coordinate (c) at (0, -1, -1);
        \node[sloped, green!50!black] at (0,0.5,-0.5) {$c$};
        \node[sloped, purple!50!black] at (0,1,-1.5) {$a$};
        \node[sloped, green!50!black] at (0,+0.5,-1.5) {$b$};
        \node[sloped, gray] at (0,-1.3,0.1) {$b+c$};
  \end{scope}
  \filldraw[very thin, fill= orange, opacity = 0.2] (a) -- (b) -- (0,0,0) -- (a);
  \filldraw[very thin, fill= gray, opacity = 0.2] (a) -- (c) -- (0,0,0) -- (a);
  \filldraw[very thin, fill= purple, opacity = 0.2] (b) -- (c) -- (0,0,0) -- (b);
  \draw[very thick, ->] (a) -- (0,0,0);
  \draw[very thick, <-] (b) -- (0,0,0);
  \draw[very thick, ->] (c) -- (0,0,0);
  \draw[very thick, <-] (0,1, 0) -- (0,0,0);
  \fill[green!50!black, opacity = 0.6] (0,0,0) circle (0.5mm);
\end{scope}

\begin{scope}[scale = 1.6, xshift = 3cm, tdplot_main_coords]
  \begin{scope}
    \filldraw [very thin, fill =blue, opacity = 0.1] (0,-1, -1) -- (0,1,-1) -- (0,1,0) -- (0,0,0) -- (0,-1,-1);
    \coordinate (a) at (0, -1, -1);
    \end{scope}
  \begin{scope}[rotate around y = 150]
    \filldraw [very thin, fill= blue, opacity = 0.1] (0,-1, -1) -- (0,1,-1) -- (0,1,0) -- (0,0,0) -- (0,-1,-1);
    \coordinate (b) at (0, -1, -1);
  \end{scope}
  \begin{scope}[rotate around y =-130]
    \filldraw [very thin, fill= blue, opacity = 0.1] (0,-1, -1) -- (0,1,-1) -- (0,1,0) -- (0,0,0) -- (0,-1,-1);
    \coordinate (c) at (0, -1, -1);
  \end{scope}
  \filldraw[very thin, fill= red, opacity = 0.4] (a) -- (b) -- (0,0,0) -- (a);
  \filldraw[very thin, fill= red, opacity = 0.4] (a) -- (c) -- (0,0,0) -- (a);
  \filldraw[very thin, fill= red, opacity = 0.4] (b) -- (c) -- (0,0,0) -- (b);
  \draw[very thick, ->] (a) -- (0,0,0);
  \draw[very thick, <-] (b) -- (0,0,0);
  \draw[very thick, ->] (c) -- (0,0,0);
  \draw[very thick, <-] (0,1, 0) -- (0,0,0);
  \fill[green!50!black, opacity = 0.6] (0,0,0) circle (0.5mm);
\end{scope}

\begin{scope}[scale = 1.6, xshift = 6cm, tdplot_main_coords]
  \begin{scope}
    \filldraw [very thin, fill =blue, opacity = 0.1] (0,-1, -1) -- (0,1,-1) -- (0,1,0) -- (0,0,0) -- (0,-1,-1);
    \coordinate (a) at (0, -1, -1);
    \end{scope}
  \begin{scope}[rotate around y = 150]
    \filldraw [very thin, fill= red, opacity = 0.4] (0,-1, -1) -- (0,1,-1) -- (0,1,0) -- (0,0,0) -- (0,-1,-1);
    \coordinate (b) at (0, -1, -1);
  \end{scope}
  \begin{scope}[rotate around y =-130]
    \filldraw [very thin, fill= red, opacity = 0.4] (0,-1, -1) -- (0,1,-1) -- (0,1,0) -- (0,0,0) -- (0,-1,-1);
    \coordinate (c) at (0, -1, -1);
  \end{scope}
  \filldraw[very thin, fill= red, opacity = 0.4] (a) -- (b) -- (0,0,0) -- (a);
  \filldraw[very thin, fill= red, opacity = 0.4] (a) -- (c) -- (0,0,0) -- (a);
  \filldraw[very thin, fill= blue, opacity = 0.1] (b) -- (c) -- (0,0,0) -- (b);
  \draw[very thick, ->] (a) -- (0,0,0);
  \draw[very thick, <-] (b) -- (0,0,0);
  \draw[very thick, ->] (c) -- (0,0,0);
  \draw[very thick, <-] (0,1, 0) -- (0,0,0);
  \fill[green!50!black, opacity = 0.6] (0,0,0) circle (0.5mm);
\end{scope}

\end{scope}

%% file: cef_kempemovethetaplus.tex
\begin{scope}
\tdplotsetmaincoords{80}{140}
\begin{scope}[scale = 1.6, xshift = 0cm, tdplot_main_coords]
  \begin{scope}
    \filldraw [very thin, fill =blue, opacity = 0.1] (0,-1, -1) -- (0,1,-1) -- (0,1,0) -- (0,0,0) -- (0,-1,-1);
    \coordinate (a) at (0, -1, -1);
    \coordinate (a1) at (0, -0.9, -1);
    \coordinate (ma) at (0,0, -0.05);
    \coordinate (aa) at (0,1,-0.05);
    \end{scope}
  \begin{scope}[rotate around y = 150]
    \filldraw [very thin, fill= blue, opacity = 0.1] (0,-1, -1) -- (0,1,-1) -- (0,1,0) -- (0,0,0) -- (0,-1,-1);
    \coordinate (b) at (0, -1, -1);
  \end{scope}
  \begin{scope}[rotate around y =-130]
    \filldraw [very thin, fill= blue, opacity = 0.1] (0,-1, -1) -- (0,1,-1) -- (0,1,0) -- (0,0,0) -- (0,-1,-1);
    \coordinate (c) at (0, -1, -1);
    \coordinate (c1) at (0, -0.9, -1);
    \coordinate (mc) at (0,0, -0.1);
    \coordinate (cc) at (0,1,-0.1);
  \end{scope}
  \filldraw[very thin, fill= red, opacity = 0.4] (a) -- (b) -- (0,0,0) -- (a);
  \filldraw[very thin, fill= red, opacity = 0.4] (a) -- (c) -- (0,0,0) -- (a);
  \filldraw[very thin, fill= red, opacity = 0.4] (b) -- (c) -- (0,0,0) -- (b);
  \draw [very thick, orange!50!black, ->] (a1) -- (ma) --(aa);
  \draw [very thick, green!50!black, ->] (c1) -- (mc) --(cc);

\end{scope}
\end{scope}
\node at (3.2, 0) {$\leftrightsquigarrow$};
\begin{scope}
\tdplotsetmaincoords{80}{140}
\begin{scope}[scale = 1.6, xshift = 4cm, tdplot_main_coords]
  \begin{scope}
    \filldraw [very thin, fill =blue, opacity = 0.1] (0,-1, -1) -- (0,1,-1) -- (0,1,0) -- (0,0,0) -- (0,-1,-1);
    \coordinate (a) at (0, -1, -1);
    \coordinate (a1) at (0, -0.9, -1);
    \coordinate (ma) at (0,0, -0.05);
    \coordinate (aa) at (0,1,-0.05);
    \end{scope}
  \begin{scope}[rotate around y = 150]
    \filldraw [very thin, fill= blue, opacity = 0.1] (0,-1, -1) -- (0,1,-1) -- (0,1,0) -- (0,0,0) -- (0,-1,-1);
    \coordinate (b) at (0, -1, -1);
  \end{scope}
  \begin{scope}[rotate around y =-130]
    \filldraw [very thin, fill= blue, opacity = 0.1] (0,-1, -1) -- (0,1,-1) -- (0,1,0) -- (0,0,0) -- (0,-1,-1);
    \coordinate (c) at (0, -1, -1);
    \coordinate (c1) at (0, -0.9, -1);
    \coordinate (mc) at (0,0.2, -0.1);
    \coordinate (cc) at (0,1,-0.1);
  \end{scope}
  \filldraw[very thin, fill= red, opacity = 0.4] (a) -- (b) -- (0,0,0) -- (a);
  \filldraw[very thin, fill= red, opacity = 0.4] (a) -- (c) -- (0,0,0) -- (a);
  \filldraw[very thin, fill= red, opacity = 0.4] (b) -- (c) -- (0,0,0) -- (b);
  \draw [very thick, orange!50!black, ->] (c1) -- (ma) --(aa);
  \draw [very thick, green!50!black, ->] (a1) -- (mc) --(cc);

\end{scope}
\end{scope}

%% file: cef_localchangetheta.tex
\begin{scope}
\draw[dotted] (0,0) circle (1cm); 
\draw[thick, ->, red!50!black] (45:1) .. controls (0,0).. (135:1);
\draw[thick, ->, red!50!black] (-135:1) .. controls (0,0) .. (-45:1);
\end{scope}
\node at (2, 0) {$\leftrightsquigarrow$};
\begin{scope}[xshift=4cm]
\draw[dotted] (0,0) circle (1cm); 
\draw[thick, ->, red!50!black] (45:1) .. controls (0,0) .. (-45:1);
\draw[thick, ->, red!50!black] (-135:1) .. controls (0,0) .. (135:1);
\end{scope}

%% file: categorification.tex
\subsection{MOY graphs and MOY relations}
\label{sec:MOYgraph}
\begin{dfn}\label{dfn:MOYgraph}
  A \emph{MOY graph} is a plane\footnote{By plane, we mean \emph{embedded} in the plane, not just embeddable.} oriented trivalent multi-graph with labels (non-negative integers). The labels and the orientations at every vertex are required to satisfy a flow condition: every vertex corresponds to one of these two models  
\[
    \tikz{
\draw[->] (0,0) -- (0,0.5) node [at end, above] {$a+b$};  
\draw[>-] (-0.5, -0.5) -- (0,0) node [at start, below] {$a$};  
\draw[>-] (+0.5, -0.5) -- (0,0) node [at start, below] {$b$};  
\begin{scope}[xshift = 5cm]
  \draw[-<] (0,0) -- (0,0.5) node [at end, above] {$a+b$};  
\draw[<-] (-0.5, -0.5) -- (0,0) node [at start, below] {$a$};  
\draw[<-] (+0.5, -0.5) -- (0,0) node [at start, below] {$b$};  
\end{scope}
}
\]
We also allow oriented (and labeled) circles. If all the labels of the MOY graph are in $\{0, \dots, N\}$, this graph can be considered as an \emph{$\sll_N$-MOY graph}. The set of all $\sll_N$-MOY graph is denoted $\mathcal{M}_N$.
\end{dfn}

\begin{dfn}(\cite{MR1659228}, see \cite[Theorem 2.4]{pre06302580} for a proof of the uniqueness)\label{dfn:MOYevaluation}
  Let $N$ be a positive integer, the $\sll_N$ evaluation of MOY graphs is the only map $\kup{\bullet}\co \mathcal{M}_N \to \ZZ[q, q^{-1}]$ which satisfies the following  local relations:

\begin{align}\label{eq:relcircle}
  \kups{\vcenter{\hbox{\tikz[scale= 0.5]{\draw[->] (0,0) arc(0:360:1cm) node[right] {\small{$\!k\!$}};}}}}=
\begin{bmatrix}
  N \\ k
\end{bmatrix}
\end{align}
\begin{align} \label{eq:relass}
   \kup{\stgamma} = \kup{\stgammaprime}
 \end{align}
\begin{align} \label{eq:relbin1} 
\kup{\digona} = \arraycolsep=2.5pt
  \begin{bmatrix}
    m+n \\ m
  \end{bmatrix}
\kup{\verta}
\end{align}
\begin{align} \label{eq:relbin2}
\arraycolsep=2.5pt
\kup{\digonb} = 
  \begin{bmatrix}
    N-m \\ n
  \end{bmatrix}
\kup{\vertb} 
\end{align}

\begin{align}
 \kup{\squarea} = \kup{\twoverta} + [N-m-1]_q\kup{\doubleYa} \label{eq:relsquare1}
\end{align}

\begin{align}
\kup{\squareb}=\!
  \begin{bmatrix}
    m-1 \\ n
  \end{bmatrix}
\kup{\bigHb}  +
\!\begin{bmatrix}
  m-1 \\n-1
\end{bmatrix}
\kup{\doubleYb} \label{eq:relsquare2}
\end{align}
\begin{align}
  \kup{\squarec}= \sum_{j=\max{(0, m-n)}}^m\begin{bmatrix}l \\ k-j \end{bmatrix}
 \kup{\squared}\label{eq:relsquare3}
\end{align}

In the previous formulas, $q$ is a formal variable, $[k]:= \frac{q^{+k}- q^{-k}}{q^{+1}- q^{-1}}$ and 
\[
  \begin{bmatrix}
    l \\ k
  \end{bmatrix}= 
  \begin{cases}
    0 & \textrm{ if $k<0$,} \\

\frac{[l][l-1] \cdots [l-k+1]}{[k][k-1] \cdots [1]}&\textrm{else,}
  \end{cases}
\]
\end{dfn}

\begin{rmk}
  Given any MOY-graph $\Gamma$, the $\sll_N$-evaluation of $\Gamma$ can be inductively computed using the relations given in the previous definition see the proof of \cite[Theorem 2.4]{pre06302580}.
\end{rmk}

The following two results can be easily deduced from the definition of the MOY evaluation. Alternative proofs can be found in \cite{2015arXiv151202370R}.

\begin{lem}\label{lem:remove0}
 Let $\Gamma$ be a $\sll_N$-MOY graph and $\widetilde{\Gamma}$ the same graph where all the edges with label $0$ have been removed. It is called a \emph{$0$-edge removal}. Then $\widetilde{\Gamma}$ is a $\sll_N$-MOY graph and $\kup{\Gamma}_N = \langle \widetilde{\Gamma}\rangle_N$.
\end{lem}

\begin{lem}\label{lem:turncycle}
 Let $\Gamma$ be a $\sll_N$-MOY graph, $\mathcal{C}$ an oriented cycle in $\Gamma$ and $\widetilde{\Gamma}$ the same graph as $\Gamma$ where for every edge of $\mathcal{C}$ we change the orientation and replace its label $a$ by $N-a$. Then $\widetilde{\Gamma}$ is a $\sll_N$-MOY graph and $\kup{\Gamma}_N = \langle\widetilde{\Gamma}\rangle_N$. In such a situation we say that $\Gamma$ and $\widetilde{\Gamma}$ are related by a \emph{cycle move} along $\mathcal{C}$.
\end{lem}

\begin{dfn}\label{dfn:facelike}
Let $\Gamma$ be a MOY graph, an oriented cycle $\mathcal{C}$ is \emph{face-like} if it bounds a disk in $\RR^2 \setminus \Gamma$.
\end{dfn}

\begin{lem}[{\cite[Lemma 4.2]{2015arXiv151202370R}}] \label{lem:uniquenessMOY}
  The $\sll_N$ evaluation of MOY graphs is the only map $\mathcal{M}_N \to \ZZ[q, q^{-1}]$  which fulfills the local Relations~(\ref{eq:relass}) and (\ref{eq:relsquare3}), is invariant under the operation described in the Lemmas~\ref{lem:remove0} and \ref{lem:turncycle} and satisfies $\kup{\emptyset} =1$. Moreover, the Relations~(\ref{eq:relass}) and (\ref{eq:relsquare3}), the cycle move and the $0$-labeled edge removal are enough to compute it.
\end{lem}

Actually we may restrict ourselves to the face-like cycles since we have the following lemma:

\begin{lem}
Suppose that $\Gamma$ and $\Gamma'$ are related by a cycle move, then there exists a finite sequence 
$\Gamma=\Gamma_0, \Gamma_1, \dots, \Gamma_{k}$ such that for all $i$ in $\{ 0, \dots, k-1\}$, $\Gamma_i$ and $\Gamma_{i+1}$ are related by a cycle move along a face-like cycle.
\end{lem}

\begin{proof}
  We proceed by induction on the number of faces surrounded by $\mathcal{C}$.
  If the cycle is not face-like, we can find (because of the flow condition on MOY graphs) an oriented path separating the faces surrounded by $\mathcal{C}$ into two non empty subsets. This gives one oriented cycle $\mathcal{C}_1$ we use the induction on $\mathcal{C}_1$, we get another cycle $\mathcal{C}_2$ and we use the induction $\mathcal{C}_2$.
\[
\tikz{\input{\imagesfolder/cef_cyclemove}}
\]

\end{proof}

We now turn back to $2$-dimensional objects and study some local relations satisfied by the formula for the evaluation of closed foams given in the previous section in Definition~\ref{dfn:decoratedfoam}. These local relations should be thought as $2$-dimensional analogue of the ones in 
Definition~\ref{dfn:MOYevaluation}.

\subsection{Local relation on foams}

The pictures in this subsection have to be understood as local pieces of a closed foam.

\begin{lem}
\label{lem-col-matveev-pierigliani}
  Let $F$ and $F'$ two foams which differ only in a small ball where they are described by the following picture:
\[
F = \, {\scriptstyle{\NB{\tikz[scale = 0.8]{\input{\imagesfolder/cef_matveevpierigliani}}}}} \quad \textrm{ and } \quad
F'= \,  {\scriptstyle{\NB{\tikz[scale = 0.8]{\input{\imagesfolder/cef_matveevpierigliani2}}}}}.
\]
Then the colorings of $F$ and $F'$ are in one-one correspondence. Moreover if $c$ is a coloring of $F$ and $c'$ is the corresponding coloring of $F'$, we have:
Then we have:
\[\kupc{F,c} =
 \kupc{F',c'}.
\]
\end{lem}
\begin{proof}
The statement about the colorings is obvious. 
For any two pigments $i<j$, the surfaces $F_i(c)$ and $F'_i(c')$ are diffeomorphic. The same holds for the surface 
$F_{ij}(c)$ and $F'_{ij} (c')$. This implies:
  \begin{align*}
    \chi(F_i(c)) = \chi(F'_i(c')) &&& \textrm{for all $i$ in $\Col$,}\\
    \chi(F_{ij}(c)) = \chi(F'_{ij}(c'))  &&& \textrm{for all $i$ and $j$ in $\Col$,}\\
\end{align*}
Moreover the positive\footnote{Actually, the negative as well, but we are not interested in them.} circles on $(F,c)$ and in $(F',c')$ are in one-one correspondence. Hence
\begin{align*}
    \theta_{ij}^+(F,c) = \theta_{ij}^+(F',c')  &&& \textrm{for all $i$ and $j$ in $\Col$.}\\
  \end{align*}
\end{proof}

\begin{cor} \label{cor:matveev-pierigliani}
  The following local relation holds for the foam evaluation.
\[\kup{ \scriptstyle{\NB{\tikz[scale = 0.8]{\input{\imagesfolder/cef_matveevpierigliani}}}} }
= \kup{ \scriptstyle{\NB{\tikz[scale = 0.8]{\input{\imagesfolder/cef_matveevpierigliani2}}}} }
\]
\end{cor}
\begin{proof}
  The colorings of the foam on the left-hand side are in 1-1 correspondence with the colorings of the foam on the right-hand side. Therefore, the result follows from Lemma~\ref{lem-col-matveev-pierigliani}.
\end{proof}
\begin{rmk}
  This is a categorified version of Relation~(\ref{eq:relass}) in Definition~\ref{dfn:MOYevaluation}.
\end{rmk}

The rest of the section is devoted to proving a categorified version of Relation~(\ref{eq:relsquare3}) in Definition~\ref{dfn:MOYevaluation} which is the content of Theorem~\ref{prop:complicatedfoam}. 

\begin{notation}
  \label{not:somenotations}
  \begin{itemize}
  \item If $\alpha$ is a Young diagram, $|\alpha|$ is the number of
    boxes of $\alpha$, $\pi_{\alpha}$ denotes the Schur polynomial
    associated with $\alpha$ and if $\alpha, \beta$ and $\gamma$ are
    three Young diagrams, $c_{\alpha\beta}^\gamma$ denotes the
    Littlewood-Richardson coefficient associated with $\alpha$,
    $\beta$ and $\gamma$.
  \item If $\alpha$ is a Young diagram, its transposed is denoted by
    $\alpha^t$.  If $a$ and $b$ are two non-negative integers,
    $T(a,b)$ denotes the set of all Young diagram having at most $a$
    columns and $b$ lines, $\rho(a,b)$ denotes the rectangular Young
    diagram with $a$ columns and $b$ lines.  If a Young diagram
    $\alpha$ is specified to be in a certain $T(a,b)$, then
    \begin{itemize}
    \item its complement is denote by $\alpha^c$: it is obtained by
      rotating by 180° the set of boxes of $\rho(a,b)$ which are not
      in $\alpha$,
    \item the \emph{dual} of $\alpha$ is defined as $(\alpha^t)^c$ or
      equivalently as $(\alpha^c)^t$ and is denoted by
      $\widehat{\alpha}$.  An illustration is given in
      Figure~\ref{fig:YDcomplement1}).
      \begin{figure}[!h]
        \centering
        \begin{tikzpicture}[scale = 0.8]
          \input{\imagesfolder/cef_YDhat}
        \end{tikzpicture}
        \caption{A Young diagram, its transposed, its complement and
        its dual.}
        \label{fig:YDcomplement1}
      \end{figure}
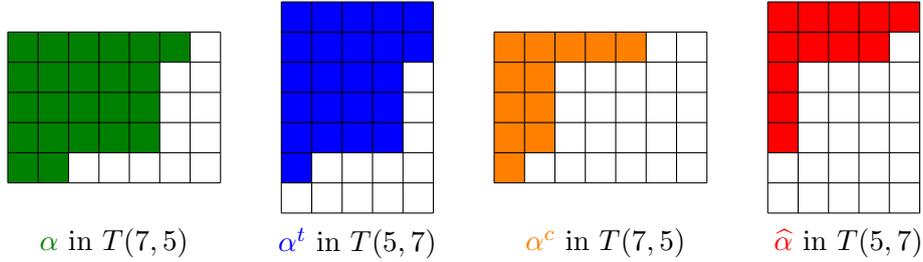
    \end{itemize}
  \item If $A=\{a_1, \dots, a_n\}$ is a set of variables (with the
    order given by indices) the \emph{Vandermonde determinant on $A$}
    is denoted by $\Delta(A)$.
  \item If $A$ and $B$ are two sets of variables, we set:
  \[
  \nabla (A,B) := \prod_{\substack{a\in A \\ b\in B}}(a -b) =
  \frac{\Delta(BA)}{\Delta(A)\Delta(B)}
    \]
  \item If $A$ and $B$ are two disjoint subsets of an ordered set $C$,
    $|A<B|$ denotes the number of pairs $(a,b)$ in $A\times B$ such
    that $a$ is lower than $b$. We have $|A<B|+ |B<A| = |A||B|$.
  \end{itemize}
\end{notation}

For details about Schur polynomials and Littlewood--Richardson constants, we refer to Appendix~\ref{sec:an-identiy-schur}. 
Until the end of this section, we fix $n$, $m$, $l$ and $k$ to be non-negative integers and define 

\[
F\eqdef \scriptstyle{\NB{\tikz[yscale = 0.5]{\input{\imagesfolder/cef_squarefoam}}}}
\]

For $j$ between $\max (0, m-n)$ and $m$ and $\alpha$ in $T(k-j, l-k+j)$, we set: 
\[
F_\alpha^j\eqdef \sum_{\substack{\beta_1, \beta_2\\ \gamma_1, \gamma_2}}   c^\alpha_{\beta_1\beta_2} c^{\widehat{\alpha}}_{\gamma_1 \gamma_2} 
\scriptstyle{\NB{\tikz[scale = 1]{\input{\imagesfolder/cef_squarefoam2}}}}
\]
In order to simplify the figures, the orientation of the facets are not depicted, they are meant to be consistent with the orientations of $F$.

\begin{prop}\label{prop:complicatedfoam}
The following local relation holds:
\[
\kup{F} = \sum_{j={\max(0, m-n)}}^m \sum_{\alpha \in T(k-j, l-k+j)} \kup{(-1)^{|\alpha| + (l-k+j)(m-j)} F^j_\alpha},
\]
where the evaluation of foams is linearly extended to formal linear combination of foams.
Furthermore on the right hand-side the terms are orthogonal idempotents. 
\end{prop}

We need to deal with the colorings of the foams appearing in Proposition~\ref{prop:complicatedfoam}. We introduce a few conventions and notations which will be used throughout the section:

The colorings of the foam $F$ are parametrized (and denoted by) $(A_1, A_2, B, C, L, R)$ where  $A_1$, $A_2$, $B$, $C$, $L$ and $R$ are disjoint subsets of $\Col$. A fixed pigment $i$ appearing in the coloring of $F$ (and in the local piece under consideration) belongs exactly to one of the subsets $(A_1, A_2, B, C, L, R)$. The correspondence with the coloring is given by Table~\ref{tab:coloringF}. 
\begin{table}[!h]
  \centering
  \begin{tabular}{|c|c|c|c|c|c|c|}
\hline 
$i\in $  & $A_1$ & $A_2$ & $B$ & $C$ & $L$ & $R$ \\ \hline
$F_i$    &\mysquare{a1}{x}{N} &\mysquare{a2}{x}{N}&\mysquare{b1}{x}{N}&\mysquare{c}{x}{N}&\mysquare{l}{x}{N}&\mysquare{r}{x}{N} \\ \hline
  \end{tabular}
  \caption{Correspondence between $(A_1, A_2, B, C, L, R)$ and the coloring of $F$. For readability, the foam $F$ has been projected onto the horizontal plane.}
  \label{tab:coloringF}
\end{table}
Note that the sets $A:=A_1\cup A_2$, $B$, $C$, $L$ and $R$ are completely determined by the coloring of the four ``winglets'' of the foam $F$.

The colorings of the foam
\[
F^j\eqdef\scriptstyle{\NB{\tikz[yscale=0.5]{\input{\imagesfolder/cef_squarefoamFj}}}}
\]
are parametrized and denoted by $(A, B_1, B_2, C, L, R)$ where  $A$, $B_1$, $B_2$, $C$, $L$ and $R$ are disjoint subsets of $\{1,\dots, N\}$. The correspondence with the coloring is given by Table~\ref{tab:coloringFj}. 
\begin{table}[!h]
  \centering
  \begin{tabular}{|c|c|c|c|c|c|c|}
\hline 
$i\in $  & $A$ & $B_1$ & $B_2$ & $C$ & $L$ & $R$ \\ \hline
$F^j_i$  &\mysquareM{a2}{x}{N}&\mysquareM{b1}{x}{N}&\mysquareM{b2}{x}{N}&\mysquareM{c}{x}{N}&\mysquareM{l}{x}{N}&\mysquareM{r}{x}{N}  \\ \hline
  \end{tabular}
  \caption{Correspondence between $(A, B_1, B_2, C, L, R)$ and the coloring of $F^j$. For readability, the foam $F^j$ has been projected onto the horizontal plane.}
  \label{tab:coloringFj}
\end{table}
Note that the sets $A$, $B:=B_1 \cup B_2$, $C$, $L$ and $R$ are completely determined by the coloring of the four ``winglets'' of the foam $F^j$.

It makes sense to speak of a coloring of $F^j_\alpha$ since it is a linear combination of the same foam with different decorations. Using the previous notations, we can parameterize a coloring of $F^j_\alpha$ by  $(A_1^t, A_2^t, A_1^b, A_2^b, B_1, B_2, C, L, R)$ with $A_1^t \cup A_2^t = A = A_1^b\cup A_2^b$ ($t$ and $b$ are for top and bottom).

\begin{lem}\label{lem:squareratioevaluation1}
  Let us consider a coloring $c \eqdef (A_1^t, A_2^t, A_1^b, A_2^b, B_1, B_2, C, L, R)$ of $F^j_\alpha$ and let $s \eqdef |A_1^t\cap A_2 ^b| = |A_2^t\cap A_1^b|$, we have
\begin{align*}\frac{\kup{F^j_\alpha,c}}{\kup{G,g}}&= (-1)^{|B_2<B_1| + |C|(|A_2^b\cap A_2^t| - |B_2|)} \\ &\frac{\nabla(B_1, A_2^b\cap A_2^t) \nabla(B_2, A_1^b\cap A_1^t) \Delta(B_1) \Delta(B_2) \pi_\alpha(A_1^t B_2  C L R) \pi_{\widehat{\alpha}}(A_2^b B_1  C L R )}{\Delta(B) \nabla(A_1^b\cap A_1^t, A_2^b \cap A_2^t) \nabla(A_1^b\cap A_2^t, A_2^b \cap A_1^t)}, \end{align*}
where 
\[
G \eqdef \scriptstyle{\NB{\tikz[scale=1]{\input{\imagesfolder/cef_squarefoamG}}}}
\]
and $g$ is the only coloring of  $G$ compatible with $(A_1^t, A_2^t, A_1^b, A_2^b, B, C, L, R)$. This means that on top and bottom, it is the same as the coloring $c$ of $F^j_\alpha$, and on the middle it is given by Table~\ref{tab:coloringofG}.
\begin{table}[h!]
  \centering
  \begin{tabular}[h!]{|c|c|c|c|c|c|c|c|c|}
    \hline 
$i \in$   &    $A_1^t\cap A_1^b$& $A_2^t\cap A_2^b$ &$A_1^t\cap A_2^b$& $A_2^t\cap A_1^b$& $B$  & $C$ & $L$ & $R$ \\ \hline
$G_i$&
\mysquare{a1}{x}{N}&\mysquare{a2}{x}{N}&\mysquare{a1}{x}{N}&\mysquare{a1}{x}{N}&\mysquare{b1}{x}{N}&\mysquare{c}{x}{N}&\mysquare{l}{x}{N}&\mysquare{r}{x}{N} \\
    \hline
  \end{tabular}
  \caption{Details of the coloring $g$ of the foam $G$: the picture reflects how the coloring looks on the middle of $G$.}
  \label{tab:coloringofG}
\end{table}
\end{lem}

\begin{rmk}
The foam $G$ and its coloring $g$ depend neither on $j$ nor on $\alpha$.
\end{rmk}
\begin{proof}
  The proof is a careful computation of 
$\chi((F^j_\alpha)_i) - \chi((G)_i)$, $\chi((F^j_\alpha)_{hi}) - \chi((G)_{hi})$, and $\theta^+_{hi}(c) - \theta^+_{hi}(g)$ for every pair of $(h,i)$ of distinct element of $\{1, \dots, N \}$. This is done with help of the tables of page~\pageref{tables} (where the set $Y$ stands for $\Col \setminus{A_1^t \cup A_2^t \cup B \cup C \cup L \cup R})$.
\end{proof}

Let us fix $A_1^t$, $A_2^t$, $A_1^b$, $A_2^b$, $B$, $C$, $L$ and $R$ and consider the set $\mathcal{C}$ of all colorings of $F^j_{\alpha}$ compatible with this data.

\begin{cor}\label{cor:squarecomplicated1}
 If $A_1^t\neq A_1^b$ (and therefore $A_2^t \neq A_1^b$), then:
\[
\sum_{j= \max(0,m-n)}^m \sum_{\alpha \in T(k-j, l-k+j)} \sum_{c\in \mathcal{C}}(-1)^{|\alpha| +(l-k+j)(m-j)|}\kup{F^j_\alpha,c} = 0.
\]
If $A_1^t = A_1^b$ ($\eqdef A_1$) (and therefore $A_2^t = A_2^b$ ($\eqdef A_2$)), then:
\[\sum_{j= \max(0,m-n)}^m \sum_{\alpha \in T(k-j, l-k+j)} \sum_{c\in \mathcal{C}}(-1)^{|\alpha| +(l-k+j)(m-j)}\kup{F^j_\alpha,c} = \kup{F,f},\]
where $f$ is the coloring of $F$ parametrized by $(A_1, A_2, B, C, L, R)$.
\end{cor}

\begin{proof}

Let us first notice that $|C| + |B_1| = m-j $ and $|A_2^b| - |B_2| = |A_2^t| -|B_2|= l-k+j$.
If we sum over all $j$ and $\alpha$, we get:
\begin{align*}
&\sum_{j= \max(0,m-n)}^m \sum_{\alpha \in T(k-j, l-k+j)} \sum_{c\in \mathcal{C}}(-1)^{|\alpha| + (l-k+j)(m-j)}\kup{F^j_\alpha,c} \\ 
&=\kup{G,g}\sum_{j= \max(0,m-n)}^m \sum_{\alpha \in T(k-j, l-k+j)} \sum_{c\in \mathcal{C}}(-1)^{|\alpha| + (l-k+j)(m-j)}\frac{\kup{F^j_\alpha,c}}{\kup{G,g}} \\ 
&=\sum_{B_1 \cup B_2 =B} \sum_{\alpha \in T(k-j, l-k+j)} (-1)^{|\alpha|+|C|(|A_2^t| - |A_2^t\cap A_2^b|)+|B_1|(|A_2^t| -|B_2|)}\\ & \qquad \frac{\nabla(A_1^t\cap A_1^b, B_2) \nabla(A_2^t\cap A_2^b, B_1) \pi_{\alpha}(A_1^t B_2 C  L  R) \pi_{\alpha}(A_2^b B_1 C L R)}{\nabla(A_1^t\cap A_1^b, A_2^t\cap A_2^b) \nabla(B_1, B_2)}.\end{align*}

Multiplying by $\nabla(A_1^t\cap A_2^b,B) = \nabla(A_1^t\cap A_2^b,B_1)\nabla(A_1^t\cap A_2^b,B_2) $ and applying Proposition~\ref{prop:schurformula_complicated}, we get:

\begin{align*}
&\nabla(A_1^t\cap A_2^b,B)\sum_{j= \max(0,m-n)}^m \sum_{\alpha \in T(k-j, l-k+j)} \sum_{c\in \mathcal{C}}\frac{\kup{F^j_\alpha,c}}{\kup{G,g}} \\
&\qquad \qquad \qquad \qquad= \frac{(-1)^{|C|(|A_2^t|-|A_2^t\cap A_2^b|)}\nabla(A_1^t,A_2^b) }{\nabla(A_1^b\cap A_1^t, A_2^b \cap A_2^t) \nabla(A_1^b\cap A_2^t, A_2^b \cap A_1^t)}\end{align*}

If $A_1^t\neq A_1^b$, we have $A_1^t\cap A_2^b\neq 0$, and the sum is equal to $0$. On the other hand, if $A_1^t= A_1^b = A_1$, we have $A_2^b=A_2^t= A_2$. Moreover, in this case the foam $F$ and $G$ are equal and the coloring $f$ and $g$ are the same. Finally, we have  $\nabla(A_1^t\cap A_2^b,B)=1$ and
\[
\nabla(A_1^b\cap A_1^t, A_2^b \cap A_2^t) \nabla(A_1^b\cap A_2^t, A_2^b \cap A_1^t) = \nabla(A_1,A_2).
\]
Hence, the sum is equal to $1$.
\end{proof}

\begin{cor}
  The following local relation holds:
\[
\kup{F} = \sum_{j = \max(0, m-n)}^m \sum_{\alpha \in T(k-j, l-k+j)} (-1)^{|\alpha| + (l-k+j)(m-j)}\kup{F^j_\alpha}.
\]
\end{cor}

This proves the first assertion of Proposition~\ref{prop:complicatedfoam}. 

In order to deal with the rest of the proposition we need to introduce another foam which is the top of $F^{j_b}_{\alpha_b}$ glued together with the bottom of $F^{j_t}_{\alpha_t}$:

\[
F^{j_t j_b}_{\alpha_t \alpha_b}\eqdef \sum_{\substack{\gamma_1,\gamma_2 \\ \beta_1 \beta_2}} \scriptstyle{\NB{ \tikz[scale =1]{\input{\imagesfolder/cef_squarefoamFjjaa}}} }
\]

Using the same conventions as before, we can parameterize the colorings of $F^{j_t j_b}_{\alpha_t \alpha_b}$ by 

$(A_1, A_2, B_1^t, B_2^t, B_1^b, B_2^b, C, L, R)$.  

\begin{lem}\label{lem:squareratio2}
   Let us consider a coloring $c \eqdef (A_1, A_2, B_1^t, B_2^t, B_1^b, B_2^b, C, L, R)$ of $F^{j_t j_b}_{\alpha_t \alpha_b}$ and let $s \eqdef |B_1^t\cap B_2 ^b|$ and $q \eqdef |B_2^t\cap B_1^b|$, we have:
\begin{align*}&\frac{\kup{F^{j_t j_b}_{\alpha_t \alpha_b},c}}{\kup{E,e}}= (-1)^{|C|(|A_2|-|B_2^b\cap B_2^t|)+|A_1<A_2| +|B_1^b\cap B_1^t||B_2^b\cap B_2^t|} \\ &\frac{\nabla(A_1, B_2^b\cap B_2^t) \nabla(A_2, B_1^b\cap B_1^t) \Delta(A_1) \Delta(A_2) \pi_{\alpha_b}(A_1 B_2^t C L R) \pi_{\alpha_t}(A_2 B_1^b C L R )}{\nabla(B_1^t\cap B_1^b, B_2^t\cap B_2^b) \Delta(A) }
 \end{align*}
Where 
\[
E \eqdef 
\NB{\tikz[scale=1]{\input{\imagesfolder/cef_squarefoamE}}}
\]
and $e$ is the only coloring of $E$ compatible with $(A, B_1^t, B_2^t, B_1^b, B_2^b, C, L, R)$. This means that on top and bottom, it is the same as the coloring $c$ of $F^{j_t j_b}_{\alpha_t \alpha_b}$, and on the middle it is given by Table~\ref{tab:coloringofE}.
\begin{table}[h!]
  \centering
  \begin{tabular}[h!]{|c|c|c|c|c|c|c|c|c|}
    \hline 
$i \in$   & $A$ &   $B_1^t\cap B_1^b$& $B_2^t\cap B_2^b$ &$B_1^t\cap B_2^b$& $B_2^t\cap B_1^b$  & $C$ & $L$ & $R$ \\ \hline 
$E_i$ &\mysquareM{a1}{x}{N}&\mysquareM{b1}{x}{N}&\mysquareM{b2}{x}{N}&\mysquareM{b1}{x}{N}&\mysquareM{b1}{x}{N}&\mysquareM{c}{x}{N}&\mysquareM{l}{x}{N}&\mysquareM{r}{x}{N}\\
    \hline
  \end{tabular}
  \caption{Detail of the coloring $e$ of the foam $E$: the picture reflects how the coloring looks on the middle of $E$.}
  \label{tab:coloringofE}
\end{table}

\end{lem}

\begin{proof}
  This is essentially the same proof as for Lemma~\ref{lem:squareratioevaluation1}: We use the tables of page~\pageref{tables2} to evaluate $\chi((F^{j_t j_b}_{\alpha_t \alpha_b})_i(c)) - \chi(E_i(e))$, $\chi((F^{j_t j_b}_{\alpha_t \alpha_b})_{hi}(c)) - \chi(E_{hi}(e))$, and 
$\theta^+_{hi}(c) - \theta^+_{hi}(e)$.
\end{proof}

\begin{cor}\label{cor3.16}
  Let us fix $A$, $B_1^t$, $B_2^t$, $B_1^b$, $B_2^b$, $C$, $L$ and  $R$ and consider the set $\mathcal{C}$ of colorings of $ F^{j_t j_b}_{\alpha_t \alpha_b}$ compatible with $A$, $B_1^t$, $B_2^t$, $B_1^b$, $B_2^b$, $C$, $L$.
Then if $B_1^t \neq B_1 ^b$ (or $B_2^t \neq B_2^b$ these two conditions are equivalent), we have:
\[
\sum_{c \in \mathcal{C}} \kup{F^{j_t j_b}_{\alpha_t \alpha_b},c} = 0.
\]
If $B_1^t = B_1 ^b \eqdef B_1$ (and therefore $B_2^t = B_2^b \eqdef B_2$ and $j_b =j_t$) 
then we have:
\[\sum_{c \in } \kup{F^{j_t j_b}_{\alpha_t \alpha_b},c} =
\begin{cases}
(-1)^{|\alpha_b| + (m-j)(l-k+j) }\kup{F^j, f^j} & \textrm{if $\alpha_b = \widehat{\alpha_t}$,} \\
0 & \textrm{else.}
\end{cases}
\]
Where $f^j$ is the coloring of $F^j$ parametrized by $(A, B_1, B_2, C, L, R)$.
\end{cor}

\begin{proof}
 We have:
 \begin{align*}
  & \sum_{c \in \mathcal{C}} \frac{\kup{F^{j_t j_b}_{\alpha_t \alpha_b},c}}{\kup{E,e}}    = \sum_{A_1 \sqcup A_2 = A} (-1)^{|C|(|A_2|-|B_2^b\cap B_2^t|)+|A_1<A_2| +|B_1^b\cap B_1^t||B_2^b\cap B_2^t|}  \\ &\qquad\frac{\nabla(A_1, B_2^b\cap B_2^t) \nabla(A_2, B_1^b\cap B_1^t) \Delta(A_1) \Delta(A_2) \pi_{\alpha_b}(A_1 B_2^t C L R) \pi_{\alpha_b}(A_2 B_1^b C L R )}{\nabla(B_1^t\cap B_1^b, B_2^t\cap B_2^b) \Delta(A) }.
\\ & \qquad = \sum_{A_1 \sqcup A_2 = A} (-1)^{|C|(|A_2|-|B_2^b\cap B_2^t|)+|A_1<A_2| +|B_1^b\cap B_1^t||B_2^b\cap B_2^t|}  \\ &\qquad  \frac{\nabla(A_1, B_2^b\cap B_2^t) \nabla(A_2, B_1^b\cap B_1^t) \Delta(A_1) \Delta(A_2) \pi_{\alpha_b}(A_1 B_2^t C L R) \pi_{\alpha_t}(A_2 B_1^b C L R )}{\nabla(B_1^t\cap B_1^b, B_2^t\cap B_2^b) \Delta(A) }.
 \end{align*}
The numerator is an anti-symmetric polynomial in $A$. Its degree in a variable $X$ in $A$ is lower than or equal to:
\[
\max{(|A|-1 +|B_2^b\cap B_2^t|-|B_2^t|, |A|-1 +|B_1^b\cap B_1^t|-|B_1^b|)}.
\] 
Hence if $B_2^t\not \subseteq B_2^b$ and $B_1^b\not \subseteq B_1^t$ (the two condition are actually equivalent), the degree in $X$ of this polynomial is strictly lower than $|A|-1$, hence it is equal to $0$. 

Suppose that $B_2^t \subseteq B_2^b$ and $B_1^t \subseteq B_1^b$. This imposes $j_b \leq j_t$. We  have:
\begin{align*}
 \sum_{c \in \mathcal{C}} \frac{\kup{F^{j_t j_b}_{\alpha_t \alpha_b},c}}{\kup{E,e}} 
&=  \sum_{A_1 \sqcup A_2 = A} (-1)^{|C|(|A_2|-|B_2^t|)+|A_1<A_2| +|B_1^b||B_2^t|}  \\ &\quad \frac{\nabla(A_1, B_2^t) \nabla(A_2, B_1^b) \Delta(A_1) \Delta(A_2) \pi_{\alpha_b}(A_1 B_2^t C L R) \pi_{\alpha_t}(A_2 B_1^b C L R )}{\nabla(B_1^b, B_2^t) \Delta(A) }.
\end{align*}
Thanks to Proposition~\ref{prop:orthog}, we have:
\begin{align*}
   \sum_{c \in \mathcal{C}} \frac{\kup{F^{j_t j_b}_{\alpha_t \alpha_b},c}}{\kup{E,e}}
= 
\begin{cases}
  (-1)^{|C|(|A_2|-|B_2^t|)+|A_2||A_1| +|B_1^b||B_2^t| + |B_2^t|(|A_1|-|B_1^b|)  + |{\alpha_b}| } &\textrm{ if $\alpha_t = \widehat{\alpha_b}$} \\
  0 &\textrm{else},
\end{cases}
\end{align*}
where $\widehat{\bullet}$ is understood in $T(k-j_b, l-k+j_t)$. But $\alpha_t$ is in  $T(k-j_t, l-k+j_t)$ and $\alpha_b$ in  $T(k-j_b, l-k+j_b)$. If $j_t \neq j_b$  $\alpha_t$ cannot  be equal to $\widehat{\alpha_b}$ in $T(k-j_t, l-k+j_b)$. If $j_b = j_t =j$, it means that $B_2^t = B_2^b = B_2$, $B_1^t = B_1^b =B_1$,  $E=F^j$ and $e = f^j$. In this case,  we have: 
\begin{align*}
\sum_{c \in \mathcal{C}} {\kup{F^{j j}_{\alpha_t \alpha_b},c}}=
\begin{cases}
  (-1)^{|C|(|A_2|-|B_2|)+|A_2||A_1| +|B_1||B_2| + |\alpha_t| }  {\kup{F^j,f^j}} &\textrm{ if $\alpha_t = \widehat{\alpha_b}$}, \\
  0 &\textrm{else},
\end{cases}
\end{align*}
This concludes, since
\[
|C|(|A_2|-|B_2|)+|A_2||A_1| +|B_1||B_2| + |\alpha_t| \equiv (m-j)(l-k+j)+ |\alpha_b|  \mod 2.
\]
\end{proof}
\begin{cor}\label{cor:idempotents}
  The foams $\left( (−1)^{|\alpha|+(l−k+j)(m−j)} F^j_\alpha\right)$ are pairwise orthogonal idempotents.
\end{cor}
\begin{proof}
  If we concatenate $F^{j_t}_{\alpha_t}$ and $F^{j_b}_{\alpha_b}$, we get:
\[
\sum_{\substack{\beta_1,\beta_2 \\ \gamma_1,\gamma_2 \\ \delta_1, \delta_2 \\ \epsilon_1, \epsilon_2}} c_{\beta_1 \beta_2}^{\alpha_t}c_{\gamma_1 \gamma_2}^{\widehat{\alpha_t}} c_{\delta_1 \delta_2}^{\alpha_b} c_{\epsilon_1 \epsilon_2}^{\widehat{\alpha_b}} \scriptstyle{\NB{\tikz{\input{\imagesfolder/cef_squarefoamFjjjjaaaa}}}}
\]
Thanks to Corollary~\ref{cor3.16}, we get that the previous (linear combination of) foam(s) is equal to 
\[
(-1)^{|\alpha_t|+ (l-k+j_t)(m-j_t)}\delta_{j_bj_t}\delta_{\alpha_b \alpha_t} F^{j_t}_{\alpha_t}
\]
\end{proof}

The two next lemmas, will be used in the proof of Theorem~\ref{thm:main} in order to get categorified versions of Lemmas~\ref{lem:remove0} and \ref{lem:turncycle}.
\begin{lem}\label{lem:turningface}
The following local identity holds:
\[
\kup{\NB{\tikz[xscale=0.9,yscale=0.9]{\input{\imagesfolder/cef_turnfacefoam}}}}= \kup{\NB{\tikz[xscale=0.9,yscale=0.9]{\input{\imagesfolder/cef_turnfacefoam2}}}}.
\]
where the hashed facets are meant to have labels $N$, and the equality holds not only for squares but for all face-like cycles (see Definition~\ref{dfn:facelike}).
\end{lem}

\begin{proof}
  Let us denote by $F$ the foam on the right-hand side and by $F'$ the foam on the left-hand side. 
  There is a one-one correspondence between the colorings of $F$ and the colorings of $F'$. Let us fix a coloring $c$ of $F$ and let us denote by $c'$ the corresponding coloring of $F'$. We will prove that $\kupc{F',c'} = \kupc{F,c}$.
We clearly have:    
\begin{align*}
  \chi(F_i(c)) = \chi(F'_i(c')) - 4 &&& \textrm{for all $i$ in $\Col$,}\\
  \chi(F_{ij}(c)) = \chi(F'_{ij}(c'))  &&& \textrm{for all $i$ and $j$ in $\Col$,}\\
\end{align*}
The only thing to realize is that the following holds:
\[ \theta^+_{ij}(F,c) \equiv \theta^+_{ij}(F',c')  \quad \textrm{for all $i$ and $j$ in $\Col$.} \]
Indeed, in comparison with the foam $(F,c)$, some new positive circles might be created. This happens only by pair, since the foam $(F',c')$ is the concatenation of two identical foams.  
On the other hand, the other circles of $(F',c')$ are in one-one correspondence with the ones of $(F, c)$.
\end{proof}
\begin{lem}\label{lem:remove0foam}
  Let $\Gamma$ be a MOY-graph and $\Gamma'$ be the same MOY-graph where all the $0$-edges has been removed. There exists a  $\Gamma$-foam-$\Gamma'$ foam $F$ such that if we remove the $0$-faces of $F$ we obtain $\Gamma'\times [0,1]$.
\end{lem}

\begin{proof}
  We will present the foam $F$ by a movie. We start with $\Gamma$ and for each $0$-labeled edge, we perform:
\[
  \begin{tikzpicture}[scale=0.5]
    \input{\imagesfolder/cef_movie0face1}
  \end{tikzpicture}
\]

For each vertex of $\Gamma$ which has a $0$-labeled adjacent edge and whose other adjacent edges are not $0$-labeled we perform\footnote{The orientations of the $0$-label edges adjacent to the non-$0$-labeled edges and of the non-$0$-labeled edges might be in the other direction, the movie should then be adapted by changing the orientation of these edges.

The careful reader, might find strange the orientations involved in these movies. This is because, on a binding of type $(0,0,0+0)$, all the facets have the biggest label. Hence one has to make a choice. However, this choice is part of the data because of the orientation conditions we impose in Definition~\ref{dfn:foam}.
}:
\[
  \begin{tikzpicture}[scale=0.5]
    \input{\imagesfolder/cef_movie0face2}
  \end{tikzpicture}
\]
For each vertex of $\Gamma$ which has only $0$-labeled adjacent edges, we perform:
\[
  \begin{tikzpicture}[scale=0.5]
    \input{\imagesfolder/cef_movie0face3}
  \end{tikzpicture}
\]
\end{proof}

\subsection{Universal construction and the main theorem}
\label{sec:univ-constr-main}

We explain how the now classical \emph{universal construction} à la \cite{MR1362791} allows us to derive a functor from the evaluation of closed foams we defined in Section~\ref{sec:eval-clos-mathfr}. We first need to define the categories involved in this functor we will define.

\begin{dfn}\label{dfn:foamwithboundary}
Let $\Gamma_0$ and $\Gamma_1$ be two $\sll_N$-MOY-graphs. An \emph{$\sll_N$-foam with boundary $(\Gamma_0, \Gamma_1)$ } or \emph{$\Gamma_1$-foam-$\Gamma_0$} is the intersection of a decorated foam $F$ embedded in $\RR^3$ with $\RR^2 \times [0,1]$ such that: 
\begin{itemize}
\item There exist an $\epsilon_0>0$  such that $F \cap \RR^2 \times [0,\epsilon_0] = (-\Gamma_0) \times [0, \epsilon]$, 
\item There exist an $\epsilon_1<1$  such that $F \cap \RR^2 \times [\epsilon_1, 1] = \Gamma_1 \times [\epsilon_1, 1]$.
\end{itemize}
By $-\Gamma$ we mean the graph $\Gamma$ with opposite orientation. 
The foam with boundary is considered up to ambient isotopy 
fixing the boundary. If $F_b$ is a $\Gamma_1$-foam-$\Gamma_0$ and $F_t$ is a $\Gamma_2$-foam-$\Gamma_1$, then $F_t\circ F_b$ is the  $\Gamma_2$-foam-$\Gamma_0$ obtained by stacking $F_t$ on the top of $F_b$ and re-scaling to fit in $\RR^2\times [0,1]$. We say that $F_t\circ F_b$ is \emph{the composition of $F_b$ and $F_t$}.  A $\Gamma_1$-foam-$\Gamma_0$ $F$ has a degree given in Definition~\ref{dfn:decoratedfoam}.
\end{dfn}

\begin{rmk}
  The degree is additive with respect to the composition of foam. This is the same degree as in \cite{queffelec2014mathfrak}, but since we are not in a $2$-categorical setting, the contributions of $\Gamma_0$ and $\Gamma_1$ to the degree are equal to $0$. 
\end{rmk}

\begin{dfn}\label{dfn:Foam}
The category $\FoamN$ has for objects $\sll_N$-MOY-graphs and for (graded) hom-sets\footnote{In this definition and the following one, we write $\mathrm{HOM}$ rather than $\mathrm{hom}$ to emphasize that we consider all morphism, not only the one with degree 0.}:
\[
\mathrm{HOM}_{\FoamN}(\Gamma_0, \Gamma_1)= \{ \textrm{$\Gamma_1$-foam-$\Gamma_0$}. \}
\]
The composition of morphisms is given by the composition of foams. It has a monoidal structure given by the disjoint union of $\sll_N$-MOY-graphs.  
\end{dfn}

\begin{dfn}\label{dfn:PolN}
The category $\widetilde{\PolN}$ has for objects projective graded $\ZZ[X_1, X_2, \dots, X_N]$-modules. And for hom-sets
\[
\mathrm{HOM}_{\PolN}(M, N)=  \{ \textrm{$\ZZ[X_1, X_2, \dots, X_N]$-linear map from $M$ to $N$}\}
\]
The category $\PolN$ is a full subcategory of  $\widetilde{\PolN}$ with object finitely generated projective graded $\ZZ[X_1, X_2, \dots, X_N]$-modules.
If $M$ is an object of $\widetilde{\PolN}$ and $n$ an integer we denote by $M\{q^n\}$ the module $M$ shifted by $n$. We have: $M\{q^n\}_k = M_{k-n}$ for all $k$ in $\ZZ$.
If $P= \sum_{i\in \ZZ} n_i q^i$ is a Laurent polynomial in $q$, $M\{P\}\eqdef \bigoplus_{i\in \ZZ} n_i M\{q^i\}$.
\end{dfn}

\begin{rmk}
  Thanks to Quillen--Suslin's theorem \cite{MR0427303, MR0344238}, finitely generated projective $\ZZ[X_1, X_2, \dots, X_N]$-modules are free.
\end{rmk}

We first define a naive functor $\widetilde{\F}\colon \FoamN \to \widetilde{\PolN}$ and then use the evaluation of foam to construct  ${\F}: \FoamN \to {\PolN}$.
The functor $\widetilde{\F}$ send an $\sll_N$-MOY-graph $\Gamma$ on the free module spanned by all $\emptyset$-foams-$\Gamma$ and a $\Gamma_1$-foam-$\Gamma_2$ on the corresponding linear map. Formally this reads
\[
  \begin{array}{crcl}
    \widetilde{\F}\colon & \FoamN & \longrightarrow     & \widetilde{\PolN} \\
                         & \Gamma & \longmapsto & \bigoplus_{G \in \mathrm{HOM}_{\FoamN}(\emptyset, \Gamma)} \ZZ[X_1, X_2, \dots, X_N]\{d_N(G)\} \\
                         & F\colon \Gamma_0 \to \Gamma_1      &\longmapsto  & \left(
                         \begin{array}{crcl}
                           \widetilde{\F}(F)\colon & \widetilde{\F}(\Gamma_0)  & \to & \widetilde{\F}(\Gamma_1) \\
                           & \mathrm{HOM}_{\FoamN}(\emptyset, \Gamma_0) \ni G &\mapsto &F\circ G  \in \mathrm{HOM}_{\FoamN}(\emptyset, \Gamma_1)
                         \end{array} \right).
  \end{array}
\]

If $\Gamma$ is an $\sll_N$-MOY-graph, for every element $G$ in $\mathrm{HOM}_{\FoamN}(\Gamma, \emptyset)$, we   define the $\ZZ[X_1, X_2, \dots, X_N]$-linear map:
\[
  \begin{array}{crcl}
\phi_G\colon&  \widetilde{\F} (\Gamma) &\to& \ZZ[X_1, X_2, \dots, X_N] \\
& {G \in \mathrm{HOM}_{\FoamN}(\emptyset, \Gamma)} \ni F & \mapsto & \kup{G\circ F}
  \end{array}
\]

We define \[\F(\Gamma) \eqdef \widetilde{\F}(\Gamma) / \bigcap_{G \in \mathrm{HOM}_{\FoamN}(\Gamma, \emptyset)} \ker \phi_G. \]

\begin{rmk}\label{rmk:evaluationwelldefined}
  For every $G$ in  $\mathrm{HOM}_{\FoamN}(\Gamma, \emptyset)$, the map $\phi_G$ induces a map from $\F(\Gamma)$ to $\ZZ[X_1, \dots, X_N]$ still denoted by $\phi_G$.
\end{rmk}
\begin{lem}
  Let $F$ be a $\Gamma_1$-foam-$\Gamma_0$ and $x$ in $\widetilde{\F}(\Gamma_0)$. If $x$ is in  $\bigcap_{G_0 \in \mathrm{HOM}_{\FoamN}(\Gamma_0, \emptyset)} \ker \phi_{G_0}$, then $\widetilde{\F}(F)(x)$ is in $\bigcap_{G_1 \in \mathrm{HOM}_{\FoamN}(\Gamma_1, \emptyset)} \ker \phi_{G_1}$.
\end{lem}
\begin{proof}
  This is straightforward since if $G_1$ is in $\mathrm{HOM}_{\FoamN}(\Gamma_1, \emptyset)$, then $G_1\circ F$ is in $\mathrm{HOM}_{\FoamN}(\Gamma_0, \emptyset)$.
\end{proof}

The previous lemma implies that if $F$ is a $\Gamma_1$-foam-$\Gamma_0$, then the $\ZZ[X_1, X_2, \dots, X_N]$-linear map
\[
\widetilde{\F}(\Gamma_0) \stackrel{\widetilde{\F}(F)}{\longrightarrow} \widetilde{\F}(\Gamma_1) \stackrel{\oldpi}{\longrightarrow} \F(\Gamma_1).
\]
induces a map $\F(\Gamma_0) \to \F(\Gamma_1)$, which we denote by $\F(F)$.

\begin{thm}\label{thm:main}
  The previous construction yields a functor $\mathcal{F}:\FoamN \to \PolN$ such that for any MOY-graph $\Gamma$, the graded rank of $\mathcal{F}(\Gamma)$ is equal to $\kup{\Gamma}$.
\end{thm}

\begin{rmk}\label{rmk:orthbase}
  Since for every $\sll_N$-MOY-graph $\Gamma$, $\F(\Gamma)$ is a finitely generated free $\ZZ[X_1, X_2, \dots, X_N]$-module, for every base $(b_i)_{i\in I}$ of $\F(\Gamma)$, there exists a dual base $(b_i^\star)_{i\in I}$ in $\mathrm{HOM}_{\PolN}(\F(\Gamma), \ZZ[X_1, \dots, X_n])$ i.~e.~ we have $b_j^{\star}(b_i) = \delta_{ij}$. Furthermore, these elements $(b_i^\star)_{i\in I}$ can be thought of as (linear combinations of) $\emptyset$-foams-$\Gamma$. 

The proof of the theorem shows how to construct inductively a base of $\F(\Gamma)$ using MOY calculus.

Note that, as in a classical TQFT, $\F(\Gamma\times \SS^1)=\kup{\Gamma\times \SS^1}$ is equal to the rank of module $\F(\Gamma)$. A coloring of $\Gamma\times \SS^1$ is essentially a coloring of $\Gamma$. Since all monochrome and bichrome surfaces are torus, the value $\kup{\Gamma\times \SS^1}$ is precisely the number of states used in the state-sum which defines the evaluation of $\sll_N$-MOY-graphs \cite{MR1659228}. 
\end{rmk}

\begin{proof}[Proof of Theorem~\ref{thm:main}.]
  \begin{claim}
    The space $\F(\emptyset)$ is isomorphic to $\simeq \ZZ[X_1, X_2, \dots, X_N]$.
  \end{claim}
The space $\widetilde{\F}(\emptyset)$ is spanned by all closed foams. Let $F$ be a closed foam and $G$ be an element of $\mathrm{HOM}_{\FoamN}(\emptyset,\emptyset)$ ($G$ is as well a closed foam), then $\kup{G\circ F} = \kup{G} \kup{F} = \kup{F}\kup{G\circ \emptyset}$. This means that $F - \kup{F}\emptyset$ is in $\bigcap_{G \in \mathrm{HOM}_{\FoamN}(\emptyset, \emptyset)} \ker \phi_G$. Therefore $\F(\emptyset)$ is spanned by (the equivalence class of) $\emptyset$. On the other hand $\phi_{\emptyset}(\emptyset)=1$, hence $\emptyset \neq 0$ in $\F(\emptyset)$. This proves that $\F(\emptyset) \simeq \ZZ[X_1, X_2, \dots, X_N]$.

\begin{claim}
  Let $\Gamma$ be an $\sll_N$-MOY graph and $\Gamma'$ be the same graph with the $0$ edges removed, then $\F(\Gamma) \simeq \F(\Gamma')$. An isomorphism (and its inverse) is given by (the image by $\F$ of) the foam described in the proof of Lemma~\ref{lem:remove0foam} (and its mirror image with respect to $\RR^2 \times \{\frac12\}$). Note that this is a categorified version of Lemma~\ref{lem:remove0}.
\end{claim}
This is a direct consequence of Remark~\ref{rmk:remove0faces}. 

\begin{claim}
  Let $\Gamma$ be an $\sll_N$-MOY graph and $\mathcal{C}$  a face-like cycle in $\Gamma$ and $\Gamma'$ be the $\sll_N$-MOY-graph obtained from $\Gamma$ by a cycle move along $\Gamma$. Then $\F(\Gamma) \simeq \F(\Gamma')$.
Note that this is a categorified version of Lemma~\ref{lem:turncycle}.
\end{claim}

An isomorphism and its inverse are given by:
\[
\F\left( \NB{\tikz[scale=0.8]{\input{\imagesfolder/cef_turnfacefoam3}}}\right) \quad \textrm{and} \quad  \F\left( \NB{\tikz[scale=0.8]{\input{\imagesfolder/cef_turnfacefoam4}}}\right),
\]
where the hashed faces are meant to have label $N$.
The claim follows directly from Lemma~\ref{lem:turningface}.

\begin{claim}
  Suppose that $\Gamma$ and $\Gamma'$ are two $\sll_N$-MOY-graphs which are identical except in a disk where they look like:
\[ 
\Gamma = \stgamma \quad \textrm{and} \quad \Gamma'= \stgammaprime,
\]
Then $\F(\Gamma) \simeq \F(\Gamma')$. Note that this is a categorified version of Relation~(\ref{eq:relass}) in Definition~\ref{dfn:MOYevaluation}.
\end{claim}
The isomorphisms are given by
\[
\F\left( {\scriptstyle{\NB{\tikz[scale=0.7]{\input{\imagesfolder/cef_matveevpierigliani3}}}}}\right) \quad \textrm{and} \quad  
\F\left( {\scriptstyle{\NB{\tikz[scale=0.7]{\input{\imagesfolder/cef_matveevpierigliani4}}}}}\right)
\]
and the claim follows directly from Corollary~\ref{cor:matveev-pierigliani}.
\begin{claim}
  The following local relation holds:
\begin{align*}
\F\left(\!\!\! \squarec \!\!\! \right)\! \simeq \!\!\! \bigoplus_{j= \max (m-n,0)}^m  \!\!\!\F
\left(\!\!\! \squared \!\!\! \right)\left\{ \begin{bmatrix}l \\ k-j \end{bmatrix}
 \right\}. \
\end{align*} 
Note that this is a categorified version of Relation~(\ref{eq:relsquare3}) in Definition~\ref{dfn:MOYevaluation}.
\end{claim}
A collection of couples of injection/projection is given by:
\begin{align*}
&\left(
(-1)^{|\alpha| + (l-k+j)(m-j)} \sum_{\beta_1, \beta_2} c_{\beta_1 \beta_2}^{\alpha}\F\left({\scriptstyle{\NB{\tikz[scale=0.7]{\input{\imagesfolder/cef_squarefoam8}}}}}\right), \right. \\ &\qquad \quad
\left. \sum_{\gamma_1, \gamma_2} c_{\gamma_1 \gamma_2}^{\widehat{\alpha}}\F\left({\scriptstyle{\NB{\tikz[scale=0.7]{\input{\imagesfolder/cef_squarefoam9}}}}}\right)
\right)_{\substack{j\in \{\max (0, m-n),\dots, m\}  \\ \alpha \in T(k-j, l -k+j)}}.
\end{align*}
This follows from Proposition~\ref{prop:complicatedfoam}.

These four claims and Lemma~\ref{lem:uniquenessMOY} finish the proof.
\end{proof}

We now include a couple of local relations satisfied by $\kup{\bullet}$ without proof. They either follow from Proposition~\ref{prop:complicatedfoam} or are quite easy (in comparison with Proposition~\ref{prop:complicatedfoam}) to prove. They categorify some of the relations of the MOY relations given in Definition~\ref{dfn:MOYevaluation}.

\begin{prop}\label{prop:additionalrelation}
  The following relations hold:
\begin{align}
 \label{eq:sphere-ev}
   \kup{\scriptstyle{\NB{\tikz[scale=0.7]{\input{\imagesfolder/cef_sphere}}}}\!\!} =
   \begin{cases}
   (-1)^{\frac{a(a+1)}2}  \pi_\beta (X_1, \dots X_N) & 
\begin{minipage}{4.1cm}
\textrm{if $\alpha$ has less than $a$ columns and $\alpha = \frac{\rho(a,N-a)}{\beta}$,}
\end{minipage} \\
     0 & \textrm{else,} \\
   \end{cases}
\end{align}
where $\frac{\rho(a,N-a)}{\beta}$ is the Young diagram obtained by stacking $\rho(a, N-a)$ on the top of $\beta$, 

If $\alpha \in T(a, b)$, $\beta \in T(b,a)$ and $\gamma \in T(a+b, N-a-b)$,
\begin{align} \label{eq:theta-ev}
   \kup{\scriptstyle{\NB{\tikz[scale=1]{\input{\imagesfolder/cef_theta}}}}} =
   \begin{cases}
     (-1)^{(a+b)(a+b+1)/2 + |\alpha|} &
\begin{minipage}{3cm}
\textrm{if $\widehat{\beta}=\alpha$ and $\gamma= \rho(a+b, N-a -b)$,} 
\end{minipage}\\ 
0 &\textrm{else;}
   \end{cases}
\end{align}
\begin{align} \label{eq:neckcutting-ev}
    \kup{\scriptstyle{\NB{\tikz[scale=0.7]{\input{\imagesfolder/cef_neckcuting}}}}} = \sum_{\alpha \in T(a,N-a)}(-1)^{|\widehat{\alpha}| + \frac{N(N+1)}2}
    \kup{\scriptstyle{\NB{\tikz[scale=0.7]{\input{\imagesfolder/cef_neckcuting2}}}}},
\end{align}
moreover the terms of the right-hand side are pairwise orthogonal idempotents;
\begin{align} \label{eq:dot-migration}
    \kup{\scriptstyle{\NB{\tikz[scale=1]{\input{\imagesfolder/cef_dotmigration}}}}} = \sum_{\alpha, \beta} c^{\gamma}_{\alpha\beta}
    \kup{\scriptstyle{\NB{\tikz[scale=1]{\input{\imagesfolder/cef_dotmigration2}}}}};
\end{align}
\begin{align}  \label{eq:digon}
    \kup{\scriptstyle{\NB{\tikz[scale=0.7]{\input{\imagesfolder/cef_digonfoam}}}}} = \sum_{\alpha \in T(a,b)}(-1)^{|\widehat{\alpha}|}
    \kup{\scriptstyle{\NB{\tikz[scale=0.7]{\input{\imagesfolder/cef_digonfoam2}}}}},
\end{align}
moreover the terms of the right-hand side are pairwise orthogonal idempotents;
\begin{align} \label{eq:digonDUR}
    \kup{\scriptstyle{\NB{\tikz[scale=0.7]{\input{\imagesfolder/cef_digonfoamDUR}}}}} = \!\!\!\!\!\sum_{\alpha \in T(b,N-a -b)} \!\!\!\!\!(-1)^{|\alpha|}
    \kup{\scriptstyle{\NB{\tikz[scale=0.7]{\input{\imagesfolder/cef_digonfoamDUR2}}}}},
\end{align}
where the two hashed disks are meant to have label $N$,  moreover the terms of the right-hand side are pairwise orthogonal idempotents;
\begin{align} \label{eq:joint}
    \kup{\scriptstyle{\NB{\tikz[scale=0.7]{\input{\imagesfolder/cef_jointrsfoam}}}}} = \sum_{\alpha \in T(r,s)}(-1)^{|\widehat{\alpha}|}
    \kup{\scriptstyle{\NB{\tikz[scale=0.7]{\input{\imagesfolder/cef_jointrsfoam2}}}}},
\end{align}
moreover the terms of the right-hand side are pairwise orthogonal idempotents, this last relation is to be compared with \cite[Relation 2.9]{MR3263166}.
\end{prop}

%% file: cef_cyclemove.tex
\begin{scope}[scale = 0.4]
\begin{scope}
  \node at (4, -1) {$\mathcal{C}$};
  \draw[    very thick] (3,2) -- (1,3) -- (-1,2); 
  \draw[->, very thick] (-1,2) -- (-3,2) -- (-4,0);
  \draw[    very thick] (-4,0) -- (-4, -2) -- (-2,-3) -- (-1, -2);
  \draw[    very thick] (-1,-2) -- (1, -3) -- (3, -2) -- (3,0);
  \draw[->, very thick] (3,0) -- (4,1) -- (3,2);
  \draw[->] (-1,-2) -- (0,0) ;
  \draw  (0,0) -- (-1,2);
\end{scope}

\begin{scope}[xshift = 11cm]
  \node[red] at (-2, 0) {$\mathcal{C}_1$};
  \draw[    very thick] (3,2) -- (1,3) -- (-1,2); 
  \draw[red, very thick] (-1,2) -- (-3,2) -- (-4,0);
  \draw[red,<-,   very thick] (-4,0) -- (-4, -2) -- (-2,-3) -- (-1, -2);
  \draw[    very thick] (-1,-2) -- (1, -3) -- (3, -2) -- (3,0);
  \draw[->, very thick] (3,0) -- (4,1) -- (3,2);
  \draw[red ] (-1,-2) -- (0,0) ;
  \draw[red,<-]  (0,0) -- (-1,2);
\end{scope}

\begin{scope}[xshift = 22cm]
  \node[green!50!black] at (1, 1) {$\mathcal{C}_2$};
  \draw[green!50!black, very thick] (3,2) -- (1,3) -- (-1,2); 
  \draw[ very thick] (-1,2) -- (-3,2) -- (-4,0);
  \draw[<-,   very thick] (-4,0) -- (-4, -2) -- (-2,-3) -- (-1, -2);
  \draw[ green!50!black,    very thick] (-1,-2) -- (1, -3) -- (3, -2) -- (3,0);
  \draw[<-,green!50!black , very thick] (3,0) -- (4,1) -- (3,2);
  \draw[green!50!black,-> ] (-1,-2) -- (0,0) ;
  \draw[green!50!black]  (0,0) -- (-1,2);
\end{scope}
\end{scope}

%% file: cef_matveevpierigliani.tex
\tdplotsetmaincoords{60}{110}
\begin{scope}[tdplot_main_coords]
  \coordinate (bT) at (-1, 0, 3);
  \coordinate (dT) at (1, 0, 3);
  \coordinate (AT) at (-2, -2, 3);
  \coordinate (BT) at (-2, 2, 3);
  \coordinate (CT) at (2, 2, 3);
  \coordinate (DT) at (2, -2, 3);
  \coordinate (MB) at (0,0, -1.5);
  \coordinate (bB) at (-1, 0, -3);
  \coordinate (dB) at (1, 0, -3);
  \coordinate (AB) at (-2, -2, -3);
  \coordinate (BB) at (-2, 2, -3);
  \coordinate (CB) at (2, 2, -3);
  \coordinate (DB) at (2, -2, -3);
  \draw[->, thick] (bT) -- (AT);
  \draw[->, thick] (bT) -- (BT);
  \draw[->, thick] (dT) -- (CT);
  \draw[->, thick] (DT) -- (dT);
  \draw[->, thick] (dT) -- (bT);
  \draw[->, thick] (bB) -- (AB);
  \draw[->, thick] (bB) -- (BB);
  \draw[->, thick] (dB) -- (CB);
  \draw[->, thick] (DB) -- (dB);
  \draw[->, thick] (dB) -- (bB);
  \draw[->, thick] (bB) -- (bT);
  \draw[->, thick] (dB) -- (dT);
  \filldraw[draw = black, rounded corners=1pt, thick, fill opacity = 0.3, fill = red]     (BT) -- (BB) -- (bB)  -- (bT)  -- (BT) node[midway, below, sloped, opacity=1] {$\scriptstyle{b}$};
  \filldraw[draw = black, rounded corners=1pt, thick, fill opacity = 0.3, fill = red]     (CT) -- (CB)  -- (dB) node[midway, above, sloped, opacity=1] { $\scriptstyle{c}$} -- (dT)-- cycle;
  \filldraw[draw = black, rounded corners=1pt, thick, fill opacity = 0.3, fill = red]     (AT) -- (AB) -- (bB) -- (bT) -- (AT) node[midway, below, sloped, opacity=1] {$\scriptstyle{a}$};
  \filldraw[draw = black, rounded corners=1pt, thick, fill opacity = 0.3, fill = red]     (DT) -- (DB) -- (dB) node[midway, above, sloped, opacity=1] { $\scriptstyle{a+b+c}$} -- (dT) -- cycle;
  \filldraw[draw = black, rounded corners=1pt, thick, fill opacity = 0.3, fill = blue]    (bT)  -- (dT) node[pos= 0.7, below, sloped, opacity=1] {$\scriptstyle{a+b}$}  -- (dB) -- (bB)-- cycle;
\end{scope}

%% file: cef_matveevpierigliani2.tex
\tdplotsetmaincoords{60}{115}
\begin{scope}[tdplot_main_coords]
  \coordinate (bT) at (-1, 0, 3);
  \coordinate (dT) at (1, 0, 3);
  \coordinate (AT) at (-2, -2, 3);
  \coordinate (BT) at (-2, 2, 3);
  \coordinate (CT) at (2, 2, 3);
  \coordinate (DT) at (2, -2, 3);
  \coordinate (MT) at (0,0, 1.5);
  \coordinate (aM) at (0, -1, 0);
  \coordinate (cM) at (0, 1, 0);
  \coordinate (AM) at (-2, -2, 0);
  \coordinate (BM) at (-2, 2, 0);
  \coordinate (CM) at (2, 2, 0);
  \coordinate (DM) at (2, -2, 0);
  \coordinate (MB) at (0,0, -1.5);
  \coordinate (bB) at (-1, 0, -3);
  \coordinate (dB) at (1, 0, -3);
  \coordinate (AB) at (-2, -2, -3);
  \coordinate (BB) at (-2, 2, -3);
  \coordinate (CB) at (2, 2, -3);
  \coordinate (DB) at (2, -2, -3);
  \draw[->, thick] (bT) -- (AT);
  \draw[->, thick] (bT) -- (BT);
  \draw[->, thick] (dT) -- (CT);
  \draw[->, thick] (DT) -- (dT);
  \draw[->, thick] (dT) -- (bT);
  \draw[->, thick] (bB) -- (AB);
  \draw[->, thick] (bB) -- (BB);
  \draw[->, thick] (dB) -- (CB);
  \draw[->, thick] (DB) -- (dB);
  \draw[->, thick] (dB) -- (bB);
  \draw[->, thick] (MT) -- (bT);
  \draw[->, thick] (MT) -- (dT);
  \draw[->, thick] (bB) -- (MB);
  \draw[->, thick] (dB) -- (MB);
  \draw[->, thick] (aM) -- (MT);
  \draw[->, thick] (cM) -- (MT);
  \filldraw[draw = black, rounded corners=1pt, thick, fill opacity = 0.3, fill = red]     (BT) -- (BB) -- (bB) -- (MB) -- (cM) -- (MT) -- (bT) -- (BT) node[midway, below, sloped, opacity=1] {$\scriptstyle{b}$};
  \filldraw[draw = black, rounded corners=1pt, thick, fill opacity = 0.3, fill = red]   (AT) -- (AB) -- (bB) -- (MB) -- (aM) -- (MT) -- (bT)-- (AT)  node[midway, below, sloped, opacity=1] { $\scriptstyle{a}$};;
  \filldraw[draw = black, rounded corners=1pt, thick, fill opacity = 0.3, fill = red]    (CT) -- (CB) -- (dB) node[midway, above, sloped, opacity=1] { $\scriptstyle{c}$}-- (MB) -- (cM) -- (MT) -- (dT)-- (CT);
  \filldraw[draw = black, rounded corners=1pt, thick, fill opacity = 0.3, fill = red]  (DT) -- (DB) -- (dB) node[midway, above, sloped, opacity=1] { $\scriptstyle{a+b+c}$} -- (MB) -- (aM) -- (MT) -- (dT)-- cycle;
  \filldraw[draw = black, rounded corners=1pt, thick, fill opacity = 0.3, fill = blue]  (bT) -- (dT) node[pos= 0.7, below, sloped, opacity=1] {$\scriptstyle{a+b}$} -- (MT) -- cycle;
  \filldraw[draw = black, rounded corners=1pt, thick, fill opacity = 0.3, fill = blue]     (bB)-- (dB) node[pos = 0.3, above, sloped, opacity=1] {$\scriptstyle{a+b}$}  -- (MB)-- cycle;
  \filldraw[draw = black, rounded corners=1pt, thick, fill opacity = 0.3, fill = green]     (MT)-- (aM) -- (MB)-- (cM)-- cycle;
  \draw[dotted] (AM) -- (aM) -- (cM) node[midway, above, sloped, opacity=1] {$\scriptstyle{b+c}$}  -- (CM);
  \draw[dotted] (BM) -- (cM);
  \draw[dotted] (DM) -- (aM);
\end{scope}

%% file: cef_YDhat.tex
\begin{scope}[scale =0.5]
\node at (-3.5, -2) {{\color{black!50!green}{$\alpha$}} in $T(7,5)$};
\begin{scope}[xscale= -1]
  \fill[black!50!green] (1,5) -- (7, 5) -- (7, 0) -- (5, 0) -- (5, 1) -- (2,1) -- (2, 4) -- (1, 4) -- (1, 5); 
  \foreach \x in {0,...,7} 
  { 
  \draw (\x, 0) -- (\x, 5);
  } 
  \foreach \y in {0,...,5} 
  {
  \draw (0,\y) -- (7, \y); 
  } 
\end{scope}

\begin{scope}[xshift = 9cm , yshift = 5cm,rotate= - 180, xscale =-1]
\node at (3.5, 7) {{ {\color{orange}$\alpha^c$}} in $T(7,5)$};
  \fill[orange] (1,5) -- (0,5) -- (0,0) --   (5, 0) -- (5, 1) -- (2,1) -- (2, 4) -- (1, 4) -- (1, 5); 
  \foreach \x in {0,...,7} 
  { 
  \draw (\x, 0) -- (\x, 5);
  } 
  \foreach \y in {0,...,5} 
  {
  \draw (0,\y) -- (7, \y); 
  } 
\end{scope}

\begin{scope}[yscale = -1, xscale = -1, xshift =-7cm, yshift = 1cm, rotate = -90 ]
\node at (-1, 2.5) {{\color{blue}{$\alpha^t$}} in $T(5,7)$};
  \fill[blue] (1,5) -- (7, 5) -- (7, 0) -- (5, 0) -- (5, 1) -- (2,1) -- (2, 4) -- (1, 4) -- (1, 5); 
  \foreach \x in {0,...,7} 
  { 
  \draw (\x, 0) -- (\x, 5);
  } 
  \foreach \y in {0,...,5} 
  {
  \draw (0,\y) -- (7, \y); 
  } 
\end{scope}

\begin{scope}[xshift = 18cm , yshift = 6cm,rotate= - 90]
\node at (8, 2.5) {{ {\color{red}$\widehat{\alpha}$}} in $T(5,7)$};
  \fill[red] (1,5) -- (0,5) -- (0,0) --   (5, 0) -- (5, 1) -- (2,1) -- (2, 4) -- (1, 4) -- (1, 5); 
  \foreach \x in {0,...,7} 
  { 
  \draw (\x, 0) -- (\x, 5);
  } 
  \foreach \y in {0,...,5} 
  {
  \draw (0,\y) -- (7, \y); 
  } 
\end{scope}
\end{scope}

%% file: cef_squarefoam.tex
\tdplotsetmaincoords{70}{100}
\begin{scope}[tdplot_main_coords]
  \coordinate (aT) at (-1, -1, 3);
  \coordinate (bT) at (-1, 1, 3);
  \coordinate (cT) at (1, 1, 3);
  \coordinate (dT) at (1, -1, 3);
  \coordinate (AT) at (-2, -2, 3);
  \coordinate (BT) at (-2, 2, 3);
  \coordinate (CT) at (2, 2, 3);
  \coordinate (DT) at (2, -2, 3);
  \coordinate (DM) at (2, -2, 0);
  \coordinate (aB) at (-1, -1, -3);
  \coordinate (bB) at (-1, 1, -3);
  \coordinate (cB) at (1, 1, -3);
  \coordinate (dB) at (1, -1, -3);
  \coordinate (AB) at (-2, -2, -3);
  \coordinate (BB) at (-2, 2, -3);
  \coordinate (CB) at (2, 2, -3);
  \coordinate (DB) at (2, -2, -3);
  \draw[thin, <-] ($(aT)!0.5!(dB)$) -- +(0,-1.5,0) node [left, sloped] {$\scriptstyle{n+k}$}; 
  \draw[thin, <-] ($(cB)!0.5!(bT)$) -- +(0,+1.5,0) node [right, sloped] {$\scriptstyle{m+l-k}$}; 
  \draw[->, thick] (aT) -- (AT);
  \draw[->, thick] (DT) -- (dT);
  \draw[->, thick] (bT) -- (BT);
  \draw[->, thick] (CT) -- (cT);
  \draw[->, thick] (dT) -- (aT);
  \draw[->, thick] (aT) -- (bT);
  \draw[->, thick] (cT) -- (dT);
  \draw[->, thick] (cT) -- (bT);
  \draw[->, thick] (dB) -- (aB);
  \draw[->, thick] (aB) -- (bB);
  \draw[->, thick] (cB) -- (dB);
  \draw[->, thick] (cB) -- (bB);
  \draw[->, thick] (aB) -- (AB);
  \draw[->, thick] (DB) -- (dB);
  \draw[->, thick] (bB) -- (BB);
  \draw[->, thick] (CB) -- (cB);
  \draw[->, thick] (aB) -- (aT);
  \draw[->, thick] (bT) -- (bB);
  \draw[->, thick] (cB) -- (cT);
  \draw[->, thick] (dT) -- (dB);
  \filldraw[draw = black, rounded corners=1pt, thick, fill opacity = 0.3, fill = red]  (aT) -- (aB) -- (AB) -- (AT) -- (aT) node[sloped, midway, below, opacity = 1] {$\scriptstyle{m}$};
  \filldraw[draw = black, rounded corners=1pt, thick, fill opacity = 0.3, fill = red]  (bT) -- (bB) -- (BB)  -- (BT) -- (bT) node[sloped, midway, below, opacity = 1] {$\scriptstyle{n+l}$};
  \filldraw[draw = black, rounded corners=1pt, thick, fill opacity = 0.3, fill = green]   (aT) --  (aB) -- (bB) -- (bT) --  (aT) node[sloped, midway, below, opacity = 1] {$\scriptstyle{n+k-m}$};
  \filldraw[draw = black, rounded corners=1pt, thick, fill opacity = 0.3, fill = red]  (dT) -- (dB) -- (DB) node[sloped, midway, above, opacity = 1] {$\scriptstyle{n}$}-- (DT) -- cycle;
  \filldraw[draw = black, rounded corners=1pt, thick, fill opacity = 0.3, fill = green]   (aT) --  (aB) -- (dB) -- (dT) --  cycle;
  \filldraw[draw = black, rounded corners=1pt, thick, fill opacity = 0.3, fill = green]   (cT) -- (cB) --  (bB) --  (bT) -- cycle;
  \filldraw[draw = black, rounded corners=1pt, thick, fill opacity = 0.3, fill = green]   (cT) -- (cB) --  (dB) node[sloped, midway, above, opacity = 1] {$\scriptstyle{k}$} --  (dT) -- cycle;
   \filldraw[draw = black, rounded corners=1pt, thick, fill opacity = 0.3, fill = red]  (cT) -- (cB) -- (CB) node[sloped, midway, above, opacity = 1] {$\scriptstyle{m+l}$} -- (CT) -- cycle; 
\end{scope}

%% file: cef_squarefoam2.tex
\tdplotsetmaincoords{70}{100}
\begin{scope}[tdplot_main_coords]
  \coordinate (aT) at (-1, -1, 3);
  \coordinate (bT) at (-1, 1, 3);
  \coordinate (cT) at (1, 1, 3);
  \coordinate (dT) at (1, -1, 3);
  \coordinate (AT) at (-2, -2, 3);
  \coordinate (BT) at (-2, 2, 3);
  \coordinate (CT) at (2, 2, 3);
  \coordinate (DT) at (2, -2, 3);
  \coordinate (aMt) at (-1, -1, 1);
  \coordinate (aMb) at (-1, -1, -1);
  \coordinate (bMt) at (-1, 1, 1);
  \coordinate (bMb) at (-1, 1, -1);
  \coordinate (cMt) at (1, 1, 1);
  \coordinate (cMb) at (1, 1, -1);
  \coordinate (dMt) at (1, -1, 1);
  \coordinate (dMb) at (1, -1, -1);
  \coordinate (AMt) at (-2, -2, 1);
  \coordinate (AMb) at (-2, -2, -1);
  \coordinate (BMt) at (-2, 2, 1);
  \coordinate (BMb) at (-2, 2, -1);
  \coordinate (CMt) at (2, 2, 1);
  \coordinate (CMb) at (2, 2, -1);
  \coordinate (DM) at (2, -2, 0);
  \coordinate (aB) at (-1, -1, -3);
  \coordinate (bB) at (-1, 1, -3);
  \coordinate (cB) at (1, 1, -3);
  \coordinate (dB) at (1, -1, -3);
  \coordinate (AB) at (-2, -2, -3);
  \coordinate (BB) at (-2, 2, -3);
  \coordinate (CB) at (2, 2, -3);
  \coordinate (DB) at (2, -2, -3);
  \draw[thin, <-] ($(dMb)!0.7!(aMt)$) --    +(0,-1.5,0) node [left, sloped] {$\scriptstyle{m-j}$}; 
  \draw[thin, <-] ($(aT)!0.3!(dMt)$) --     +(0,-1.5,0) node [left, sloped] {$\scriptstyle{n+k}$}; 
  \draw[red, thick, <-]($(aMb)!0.3!(dMt)$) --  +(0,-1.5,0) node [left, sloped] {$\scriptstyle{\pi_{\gamma_2}}$}; 
  \draw[red,thick, <-] ($(bT)!0.7!(cMt)$)-- +(0,+1.5,0) node [right, sloped] {$\scriptstyle{\pi_{\beta_1}}$}; 
  \draw[thin, <-] ($(dB)!0.5!(aMb)$) --     +(0,-1.5,0) node [left, sloped] {$\scriptstyle{n+k}$}; 
  \draw[thin, <-] ($(cMb)!0.7!(bMt)$) --    +(0,+1.5,0) node [right, sloped] {$\scriptstyle{n+l+j}$}; 
  \draw[thin, <-] ($(bT)!0.3!(cMt)$) --     +(0,+1.5,0) node [right, sloped] {$\scriptstyle{m+l-k}$}; 
  \draw[thin, <-] ($(cB)!0.7!(bMb)$) --     +(0,+1.5,0) node [right, sloped] {$\scriptstyle{m+l-k}$}; 
  \draw[thin, <-] ($(aMb)!0.3!(bMt)$) .. controls +(-3,0,0) and  +(0,0,-1) .. +(-3,0,+3) node [above, sloped] {$\scriptstyle{j}$}; 
  \draw[red, thick, <-] ($(aB)!0.5!(bMb)$) .. controls +(-3,0,0) and  +(0,+1,0) .. +(-3,-3.1,0) node [left, sloped] {$\scriptstyle{\pi_{\gamma_1}}$}; 
  \filldraw[draw = black, rounded corners=1pt, thick, fill opacity = 0.3, fill = red]  (aT) -- (aB) -- (AB) -- (AT) -- (aT) node[sloped, midway, below, opacity = 1] {$\scriptstyle{m}$};
  \filldraw[draw = black, rounded corners=1pt, thick, fill opacity = 0.3, fill = red]  (bT) -- (bB) -- (BB)  -- (BT) -- (bT) node[sloped, midway, below, opacity = 1] {$\scriptstyle{n+l}$};
  \filldraw[draw = black, rounded corners=1pt, thick, fill opacity = 0.3, fill = green]   (aT) --  (aMt) -- (bMt) -- (bT) --  (aT) node[sloped, midway, below, opacity = 1] {$\scriptstyle{n+k-m}$};
  \filldraw[draw = black, rounded corners=1pt, thick, fill opacity = 0.3, fill = green]   (aMb) -- (aB) --  (bB) node[sloped, midway, above, opacity = 1] {$\scriptstyle{n+k-m}$} --  (bMb) -- cycle;
  \filldraw[draw = black, rounded corners=1pt, thick, fill opacity = 0.3, fill = red]  (dT) -- (dB) -- (DB) node[sloped, midway, above, opacity = 1] {$\scriptstyle{n}$}-- (DT) -- cycle;
  \filldraw[draw = black, rounded corners=1pt, thick, fill opacity = 0.3, fill = yellow]  (aMt) -- (aMb) -- (dMb) -- (dMt) -- cycle;
  \filldraw[draw = black, rounded corners=1pt, thick, fill opacity = 0.3, fill = green]   (aT) --  (aMt) -- (dMt) -- (dT) --  cycle;
  \filldraw[draw = black, rounded corners=1pt, thick, fill opacity = 0.3, fill = green]   (aMb) -- (aB) --  (dB) --  (dMb) -- cycle;
  \filldraw[draw = black, rounded corners=1pt, thick, fill opacity = 0.3, fill = yellow]  (cMt) -- (cMb) -- (dMb)  -- (dMt) -- cycle;
  \filldraw[draw = black, rounded corners=1pt, thick, fill opacity = 0.3, fill = yellow]  (aMt) -- (aMb) -- (bMb) -- (bMt) -- cycle;
  \filldraw[draw = black, rounded corners=1pt, thick, fill opacity = 0.3, fill = yellow]  (cMt) -- (cMb) -- (bMb) -- (bMt) -- cycle;
  \filldraw[draw = black, rounded corners=1pt, thick, fill opacity = 0.3, fill = green]   (cT) --  (cMt) -- (bMt) -- (bT) --  (cT);
  \filldraw[draw = black, rounded corners=1pt, thick, fill opacity = 0.3, fill = green]   (cMb) -- (cB) --  (bB) --  (bMb) -- cycle;
  \filldraw[draw = black, rounded corners=1pt, thick, fill opacity = 0.3, fill = green]   (cT) --  (cMt) -- (dMt) -- (dT) --  (cT) node[sloped, midway, below, opacity = 1] {$\scriptstyle{k}$};
  \filldraw[draw = black, rounded corners=1pt, thick, fill opacity = 0.3, fill = green]   (cMb) -- (cB) --  (dB) node[sloped, midway, above, opacity = 1] {$\scriptstyle{k}$} --  (dMb) -- cycle;
  \filldraw[draw = black, rounded corners=1pt, thick, fill opacity = 0.3, fill = blue]    (aMt) -- (bMt) -- (cMt) -- (dMt) node[sloped, midway, above, opacity = 1] {$\scriptstyle{n+k-m+j}$} -- cycle;
  \filldraw[draw = black, rounded corners=1pt, thick, fill opacity = 0.3, fill = blue]    (aMb) -- (bMb) -- (cMb) -- (dMb) node[sloped, midway, above, opacity = 1] {$\scriptstyle{n+k-m+j}$} -- cycle;
  \filldraw[draw = black, rounded corners=1pt, thick, fill opacity = 0.3, fill = red]  (cT) -- (cB) -- (CB) node[sloped, midway, above, opacity = 1] {$\scriptstyle{m+l}$} -- (CT) -- cycle; 
  \draw[thin, <-] ($(dMb)!0.3!(cMt)$) .. controls +(+3,0,0) and  +(0,0,+1) .. (+3,0,-3) node [below, sloped] {$\scriptstyle{n+j-m}$}; 
  \draw[red, thick, <-]($(cMt)!0.3!(dMb)$) .. controls  +(+1,0,0) and +(0, -1, 0) .. +(1,3,0)  node [right, sloped] {$\scriptstyle{\pi_{\beta_2}}$}; 
\end{scope}

%% file: cef_squarefoamFj.tex
\tdplotsetmaincoords{70}{100}
\begin{scope}[tdplot_main_coords]
  \coordinate (aT) at (-1, -1, 3);
  \coordinate (bT) at (-1, 1, 3);
  \coordinate (cT) at (1, 1, 3);
  \coordinate (dT) at (1, -1, 3);
  \coordinate (AT) at (-2, -2, 3);
  \coordinate (BT) at (-2, 2, 3);
  \coordinate (CT) at (2, 2, 3);
  \coordinate (DT) at (2, -2, 3);
  \coordinate (DM) at (2, -2, 0);
  \coordinate (aB) at (-1, -1, -3);
  \coordinate (bB) at (-1, 1, -3);
  \coordinate (cB) at (1, 1, -3);
  \coordinate (dB) at (1, -1, -3);
  \coordinate (AB) at (-2, -2, -3);
  \coordinate (BB) at (-2, 2, -3);
  \coordinate (CB) at (2, 2, -3);
  \coordinate (DB) at (2, -2, -3);
  \draw[thin, <-] ($(aT)!0.5!(dB)$) -- +(0,-1.5,0) node [left, sloped] {$\scriptstyle{m-j}$}; 
  \draw[thin, <-] ($(cB)!0.5!(bT)$) -- +(0,+1.5,0) node [right, sloped] {$\scriptstyle{n+l-j}$}; 
  \filldraw[draw = black, rounded corners=1pt, thick, fill opacity = 0.3, fill = red]  (aT) -- (aB) -- (AB) -- (AT) -- (aT) node[sloped, midway, below, opacity = 1] {$\scriptstyle{m}$};
  \filldraw[draw = black, rounded corners=1pt, thick, fill opacity = 0.3, fill = red]  (bT) -- (bB) -- (BB)  -- (BT) -- (bT) node[sloped, midway, below, opacity = 1] {$\scriptstyle{n+l}$};
  \filldraw[draw = black, rounded corners=1pt, thick, fill opacity = 0.3, fill = yellow]   (aT) --  (aB) -- (bB) -- (bT) --  (aT) node[sloped, midway, below, opacity = 1] {$\scriptstyle{j}$};
  \filldraw[draw = black, rounded corners=1pt, thick, fill opacity = 0.3, fill = red]  (dT) -- (dB) -- (DB) node[sloped, midway, above, opacity = 1] {$\scriptstyle{n}$}-- (DT) -- cycle;
  \filldraw[draw = black, rounded corners=1pt, thick, fill opacity = 0.3, fill = yellow]   (aT) --  (aB) -- (dB) -- (dT) --  cycle;
  \filldraw[draw = black, rounded corners=1pt, thick, fill opacity = 0.3, fill = yellow]   (cT) -- (cB) --  (bB) --  (bT) -- cycle;
  \filldraw[draw = black, rounded corners=1pt, thick, fill opacity = 0.3, fill = yellow]   (cT) -- (cB) --  (dB) node[sloped, midway, above, opacity = 1] {$\scriptstyle{n+j-m}$} --  (dT) -- cycle;
   \filldraw[draw = black, rounded corners=1pt, thick, fill opacity = 0.3, fill = red]  (cT) -- (cB) -- (CB) node[sloped, midway, above, opacity = 1] {$\scriptstyle{m+l}$} -- (CT) -- cycle; 
\end{scope}

%% file: cef_squarefoamG.tex
\tdplotsetmaincoords{70}{100}
\begin{scope}[tdplot_main_coords]
  \coordinate (aT) at (-1, -1, 3);
  \coordinate (bT) at (-1, 1, 3);
  \coordinate (cT) at (1, 1, 3);
  \coordinate (dT) at (1, -1, 3);
  \coordinate (AT) at (-2, -2, 3);
  \coordinate (BT) at (-2, 2, 3);
  \coordinate (CT) at (2, 2, 3);
  \coordinate (DT) at (2, -2, 3);
  \coordinate (aMt) at (-1, -1, 1);
  \coordinate (aMb) at (-1, -1, -1);
  \coordinate (bMt) at (-1, 1, 1);
  \coordinate (bMb) at (-1, 1, -1);
  \coordinate (cMt) at (1, 1, 1);
  \coordinate (cMb) at (1, 1, -1);
  \coordinate (dMt) at (1, -1, 1);
  \coordinate (dMb) at (1, -1, -1);
  \coordinate (AMt) at (-2, -2, 1);
  \coordinate (AMb) at (-2, -2, -1);
  \coordinate (BMt) at (-2, 2, 1);
  \coordinate (BMb) at (-2, 2, -1);
  \coordinate (CMt) at (2, 2, 1);
  \coordinate (CMb) at (2, 2, -1);
  \coordinate (DM) at (2, -2, 0);
  \coordinate (aB) at (-1, -1, -3);
  \coordinate (bB) at (-1, 1, -3);
  \coordinate (cB) at (1, 1, -3);
  \coordinate (dB) at (1, -1, -3);
  \coordinate (AB) at (-2, -2, -3);
  \coordinate (BB) at (-2, 2, -3);
  \coordinate (CB) at (2, 2, -3);
  \coordinate (DB) at (2, -2, -3);
  \draw[thin, <-] ($(dMb)!0.7!(aMt)$) --    +(0,-1.5,0) node [left, sloped] {$\scriptstyle{n+k-s}$}; 
  \draw[thin, <-] ($(aT)!0.3!(dMt)$) --     +(0,-1.5,0) node [left, sloped] {$\scriptstyle{n+k}$}; 
  \draw[thin, <-] ($(dB)!0.5!(aMb)$) --     +(0,-1.5,0) node [left, sloped] {$\scriptstyle{n+k}$}; 
  \draw[thin, <-] ($(cMb)!0.7!(bMt)$) --    +(0,+1.5,0) node [right, sloped] {$\scriptstyle{m+l-k+s}$}; 
  \draw[thin, <-] ($(cT)!0.3!(bMt)$) --     +(0,+1.5,0) node [right, sloped] {$\scriptstyle{m+l-k}$}; 
  \draw[thin, <-] ($(cB)!0.7!(bMb)$) --     +(0,+1.5,0) node [right, sloped] {$\scriptstyle{m+l-k}$}; 
  \draw[thin, <-] ($(aMb)!0.3!(bMt)$) .. controls +(-3,0,0) and  +(0,0,-1) .. +(-3,0,+3) node [above, sloped] {$\scriptstyle{n+k-m-s}$}; 
  \filldraw[draw = black, rounded corners=1pt, thick, fill opacity = 0.3, fill = red]  (aT) -- (aB) -- (AB) -- (AT) -- (aT) node[sloped, midway, below, opacity = 1] {$\scriptstyle{m}$};
  \filldraw[draw = black, rounded corners=1pt, thick, fill opacity = 0.3, fill = red]  (bT) -- (bB) -- (BB)  -- (BT) -- (bT) node[sloped, midway, below, opacity = 1] {$\scriptstyle{n+l}$};
  \filldraw[draw = black, rounded corners=1pt, thick, fill opacity = 0.3, fill = green]   (aT) --  (aMt) -- (bMt) -- (bT) --  (aT) node[sloped, midway, below, opacity = 1] {$\scriptstyle{n+k-m}$};
  \filldraw[draw = black, rounded corners=1pt, thick, fill opacity = 0.3, fill = green]   (aMb) -- (aB) --  (bB) node[sloped, midway, above, opacity = 1] {$\scriptstyle{n+k-m}$} --  (bMb) -- cycle;
  \filldraw[draw = black, rounded corners=1pt, thick, fill opacity = 0.3, fill = red]  (dT) -- (dB) -- (DB) node[sloped, midway, above, opacity = 1] {$\scriptstyle{n}$}-- (DT) -- cycle;
  \filldraw[draw = black, rounded corners=1pt, thick, fill opacity = 0.3, fill = green!50!blue]  (aMt) -- (aMb) -- (dMb) -- (dMt) -- cycle;
  \filldraw[draw = black, rounded corners=1pt, thick, fill opacity = 0.3, fill = green]   (aT) --  (aMt) -- (dMt) -- (dT) --  cycle;
  \filldraw[draw = black, rounded corners=1pt, thick, fill opacity = 0.3, fill = green]   (aMb) -- (aB) --  (dB) --  (dMb) -- cycle;
  \filldraw[draw = black, rounded corners=1pt, thick, fill opacity = 0.3, fill = green!50!blue]  (cMt) -- (cMb) -- (dMb)  -- (dMt) -- cycle;
  \filldraw[draw = black, rounded corners=1pt, thick, fill opacity = 0.3, fill = green!50!blue]  (aMt) -- (aMb) -- (bMb) -- (bMt) -- cycle;
  \filldraw[draw = black, rounded corners=1pt, thick, fill opacity = 0.3, fill = green!50!blue]  (cMt) -- (cMb) -- (bMb) -- (bMt) -- cycle;
  \filldraw[draw = black, rounded corners=1pt, thick, fill opacity = 0.3, fill = green]   (cT) --  (cMt) -- (bMt) -- (bT) --  (cT);
  \filldraw[draw = black, rounded corners=1pt, thick, fill opacity = 0.3, fill = green]   (cMb) -- (cB) --  (bB) --  (bMb) -- cycle;
  \filldraw[draw = black, rounded corners=1pt, thick, fill opacity = 0.3, fill = green]   (cT) --  (cMt) -- (dMt) -- (dT) --  (cT) node[sloped, midway, below, opacity = 1] {$\scriptstyle{k}$};
  \filldraw[draw = black, rounded corners=1pt, thick, fill opacity = 0.3, fill = green]   (cMb) -- (cB) --  (dB) node[sloped, midway, above, opacity = 1] {$\scriptstyle{k}$} --  (dMb) -- cycle;
  \filldraw[draw = black, rounded corners=1pt, thick, fill opacity = 0.3, fill = orange]    (aMt) -- (bMt) -- (cMt) -- (dMt) node[sloped, midway, above, opacity = 1] {$\scriptstyle{s}$} -- cycle;
  \filldraw[draw = black, rounded corners=1pt, thick, fill opacity = 0.3, fill = orange]    (aMb) -- (bMb) -- (cMb) -- (dMb) node[sloped, midway, above, opacity = 1] {$\scriptstyle{s}$} -- cycle;
  \filldraw[draw = black, rounded corners=1pt, thick, fill opacity = 0.3, fill = red]  (cT) -- (cB) -- (CB) node[sloped, midway, above, opacity = 1] {$\scriptstyle{m+l}$} -- (CT) -- cycle; 
  \draw[thin, <-] ($(dMb)!0.5!(cMt)$) .. controls +(+3,0,0) and  +(0,0,+1) .. (+3,0,-3) node [below, sloped] {$\scriptstyle{k-s}$}; 

\end{scope}

%% file: cef_squarefoamFjjaa.tex
\tdplotsetmaincoords{70}{100}
\begin{scope}[tdplot_main_coords]
  \coordinate (aT) at (-1, -1, 3);
  \coordinate (bT) at (-1, 1, 3);
  \coordinate (cT) at (1, 1, 3);
  \coordinate (dT) at (1, -1, 3);
  \coordinate (AT) at (-2, -2, 3);
  \coordinate (BT) at (-2, 2, 3);
  \coordinate (CT) at (2, 2, 3);
  \coordinate (DT) at (2, -2, 3);
  \coordinate (aMt) at (-1, -1, 1);
  \coordinate (aMb) at (-1, -1, -1);
  \coordinate (bMt) at (-1, 1, 1);
  \coordinate (bMb) at (-1, 1, -1);
  \coordinate (cMt) at (1, 1, 1);
  \coordinate (cMb) at (1, 1, -1);
  \coordinate (dMt) at (1, -1, 1);
  \coordinate (dMb) at (1, -1, -1);
  \coordinate (AMt) at (-2, -2, 1);
  \coordinate (AMb) at (-2, -2, -1);
  \coordinate (BMt) at (-2, 2, 1);
  \coordinate (BMb) at (-2, 2, -1);
  \coordinate (CMt) at (2, 2, 1);
  \coordinate (CMb) at (2, 2, -1);
  \coordinate (DM) at (2, -2, 0);
  \coordinate (aB) at (-1, -1, -3);
  \coordinate (bB) at (-1, 1, -3);
  \coordinate (cB) at (1, 1, -3);
  \coordinate (dB) at (1, -1, -3);
  \coordinate (AB) at (-2, -2, -3);
  \coordinate (BB) at (-2, 2, -3);
  \coordinate (CB) at (2, 2, -3);
  \coordinate (DB) at (2, -2, -3);
  \draw[thin, <-] ($(dMb)!0.5!(aMt)$) --    +(0,-1.5,0) node [left, sloped] {$\scriptstyle{n+k}$}; 
  \draw[thin, <-] ($(aT)!0.5!(dMt)$) --     +(0,-1.5,0) node [left, sloped] {$\scriptstyle{m-j_t}$}; 
 \draw[red, thick, <-]($(aT)!0.3!(dMt)$) --  +(0,-1.5,0) node [left, sloped] {$\scriptstyle{\pi_{\gamma_2}}$}; 
  \draw[red,thick, <-] ($(bMb)!0.7!(cB)$)-- +(0,+1.5,0) node [right, sloped] {$\scriptstyle{\pi_{\beta_1}}$}; 
  \draw[thin, <-] ($(dB)!0.5!(aMb)$) --     +(0,-1.5,0) node [left, sloped] {$\scriptstyle{m-j_b}$}; 
  \draw[thin, <-] ($(bMt)!0.5!(cMb)$) --    +(0,+1.5,0) node [right, sloped] {$\scriptstyle{m+l-k}$}; 
  \draw[thin, <-] ($(cT)!0.5!(bMt)$) --     +(0,+1.5,0) node [right, sloped] {$\scriptstyle{n+l+j_t}$}; 
  \draw[thin, <-] ($(cB)!0.5!(bMb)$) --     +(0,+1.5,0) node [right, sloped] {$\scriptstyle{n+l+j_b}$}; 
 \draw[red, thick, <-] ($(aMt)!0.5!(bMb)$) .. controls +(-3,0,0) and  +(0,+1,0) .. +(-3,-3.1,0) node [left, sloped] {$\scriptstyle{\pi_{\gamma_1}}$}; 
  \draw[thin, <-] ($(aMb)!0.5!(bMt)$) .. controls +(-3,0,0) and  +(0,0,-1) .. (-3,0,+3) node [above, sloped] {$\scriptstyle{n+k-m}$}; 
  \filldraw[draw = black, rounded corners=1pt, thick, fill opacity = 0.3, fill = red]  (aT) -- (aB) -- (AB) -- (AT) -- (aT) node[sloped, midway, below, opacity = 1] {$\scriptstyle{m}$};
  \filldraw[draw = black, rounded corners=1pt, thick, fill opacity = 0.3, fill = red]  (bT) -- (bB) -- (BB)  -- (BT) -- (bT) node[sloped, midway, below, opacity = 1] {$\scriptstyle{n+l}$};
  \filldraw[draw = black, rounded corners=1pt, thick, fill opacity = 0.3, fill = yellow]   (aT) --  (aMt) -- (bMt) -- (bT) --  (aT) node[sloped, midway, below, opacity = 1] {$\scriptstyle{j_t}$};
  \filldraw[draw = black, rounded corners=1pt, thick, fill opacity = 0.3, fill = yellow]   (aMb) -- (aB) --  (bB) node[sloped, midway, above, opacity = 1] {$\scriptstyle{j_b}$} --  (bMb) -- cycle;
  \filldraw[draw = black, rounded corners=1pt, thick, fill opacity = 0.3, fill = red]  (dT) -- (dB) -- (DB) node[sloped, midway, above, opacity = 1] {$\scriptstyle{n}$}-- (DT) -- cycle;
  \filldraw[draw = black, rounded corners=1pt, thick, fill opacity = 0.3, fill = green]  (aMt) -- (aMb) -- (dMb) -- (dMt) -- cycle;
  \filldraw[draw = black, rounded corners=1pt, thick, fill opacity = 0.3, fill = yellow]   (aT) --  (aMt) -- (dMt) -- (dT) --  cycle;
  \filldraw[draw = black, rounded corners=1pt, thick, fill opacity = 0.3, fill = yellow]   (aMb) -- (aB) --  (dB) --  (dMb) -- cycle;
  \filldraw[draw = black, rounded corners=1pt, thick, fill opacity = 0.3, fill = green]  (cMt) -- (cMb) -- (dMb)  -- (dMt) -- cycle;
  \filldraw[draw = black, rounded corners=1pt, thick, fill opacity = 0.3, fill = green]  (aMt) -- (aMb) -- (bMb) -- (bMt) -- cycle;
  \filldraw[draw = black, rounded corners=1pt, thick, fill opacity = 0.3, fill = green]  (cMt) -- (cMb) -- (bMb) -- (bMt) -- cycle;
  \filldraw[draw = black, rounded corners=1pt, thick, fill opacity = 0.3, fill = yellow]   (cT) --  (cMt) -- (bMt) -- (bT) --  (cT);
  \filldraw[draw = black, rounded corners=1pt, thick, fill opacity = 0.3, fill = yellow]   (cMb) -- (cB) --  (bB) --  (bMb) -- cycle;
  \filldraw[draw = black, rounded corners=1pt, thick, fill opacity = 0.3, fill = yellow]   (cT) --  (cMt) -- (dMt) -- (dT) --  (cT) node[sloped, midway, below, opacity = 1] {$\scriptstyle{n+j_t-m}$};
  \filldraw[draw = black, rounded corners=1pt, thick, fill opacity = 0.3, fill = yellow]   (cMb) -- (cB) --  (dB) node[sloped, midway, above, opacity = 1] {$\scriptstyle{n+j_b-m}$} --  (dMb) -- cycle;
  \filldraw[draw = black, rounded corners=1pt, thick, fill opacity = 0.3, fill = blue]    (aMt) -- (bMt) -- (cMt) -- (dMt) node[sloped, midway, above, opacity = 1] {$\scriptstyle{n+k-m+j_t}$} -- cycle;
  \filldraw[draw = black, rounded corners=1pt, thick, fill opacity = 0.3, fill = blue]    (aMb) -- (bMb) -- (cMb) -- (dMb) node[sloped, midway, above, opacity = 1] {$\scriptstyle{n+k-m+j_b}$} -- cycle;
  \filldraw[draw = black, rounded corners=1pt, thick, fill opacity = 0.3, fill = red]  (cT) -- (cB) -- (CB) node[sloped, midway, above, opacity = 1] {$\scriptstyle{m+l}$} -- (CT) -- cycle; 
  \draw[red, thick, <-]($(cMb)!0.3!(dMt)$) .. controls  +(+1,0,0) and +(0, -1, 0) .. +(1,3,0)  node [right, sloped] {$\scriptstyle{\pi_{\beta_2}}$}; 
  \draw[thin, <-] ($(dMb)!0.5!(cMt)$) .. controls +(+3,0,0) and  +(0,0,+1) .. (+3,0,-3) node [below, sloped] {$\scriptstyle{k}$}; 

\end{scope}

%% file: cef_squarefoamE.tex
\tdplotsetmaincoords{70}{100}
\begin{scope}[tdplot_main_coords]
  \coordinate (aT) at (-1, -1, 3);
  \coordinate (bT) at (-1, 1, 3);
  \coordinate (cT) at (1, 1, 3);
  \coordinate (dT) at (1, -1, 3);
  \coordinate (AT) at (-2, -2, 3);
  \coordinate (BT) at (-2, 2, 3);
  \coordinate (CT) at (2, 2, 3);
  \coordinate (DT) at (2, -2, 3);
  \coordinate (aMt) at (-1, -1, 1);
  \coordinate (aMb) at (-1, -1, -1);
  \coordinate (bMt) at (-1, 1, 1);
  \coordinate (bMb) at (-1, 1, -1);
  \coordinate (cMt) at (1, 1, 1);
  \coordinate (cMb) at (1, 1, -1);
  \coordinate (dMt) at (1, -1, 1);
  \coordinate (dMb) at (1, -1, -1);
  \coordinate (AMt) at (-2, -2, 1);
  \coordinate (AMb) at (-2, -2, -1);
  \coordinate (BMt) at (-2, 2, 1);
  \coordinate (BMb) at (-2, 2, -1);
  \coordinate (CMt) at (2, 2, 1);
  \coordinate (CMb) at (2, 2, -1);
  \coordinate (DM) at (2, -2, 0);
  \coordinate (aB) at (-1, -1, -3);
  \coordinate (bB) at (-1, 1, -3);
  \coordinate (cB) at (1, 1, -3);
  \coordinate (dB) at (1, -1, -3);
  \coordinate (AB) at (-2, -2, -3);
  \coordinate (BB) at (-2, 2, -3);
  \coordinate (CB) at (2, 2, -3);
  \coordinate (DB) at (2, -2, -3);
  \draw[thin, <-] ($(dMb)!0.5!(aMt)$) --    +(0,-1.5,0) node [left, sloped] {$\scriptstyle{m-j_t+s}$}; 
  \draw[thin, <-] ($(aT)!0.5!(dMt)$) --     +(0,-1.5,0) node [left, sloped] {$\scriptstyle{m-j_t}$}; 
  \draw[thin, <-] ($(dB)!0.5!(aMb)$) --     +(0,-1.5,0) node [left, sloped] {$\scriptstyle{m-j_b}$}; 
  \draw[thin, <-] ($(bMt)!0.5!(cMb)$) --    +(0,+1.5,0) node [right, sloped] {$\scriptstyle{m+l+j_t-s}$}; 
  \draw[thin, <-] ($(cT)!0.5!(bMt)$) --     +(0,+1.5,0) node [right, sloped] {$\scriptstyle{n+l+j_t}$}; 
  \draw[thin, <-] ($(cB)!0.5!(bMb)$) --     +(0,+1.5,0) node [right, sloped] {$\scriptstyle{n+l+j_b}$}; 
  \draw[thin, <-] ($(dMb)!0.5!(cMt)$) .. controls +(+3,0,0) and  +(0,0,+1) .. (+3,0,-3) node [below, sloped] {$\scriptstyle{n+j_b-m-s}$}; 
  \draw[thin, <-] ($(aMb)!0.5!(bMt)$) .. controls +(-3,0,0) and  +(0,0,-1) .. (-3,0,+3) node [above, sloped] {$\scriptstyle{j_b-s}$}; 
  \filldraw[draw = black, rounded corners=1pt, thick, fill opacity = 0.3, fill = red]  (aT) -- (aB) -- (AB) -- (AT) -- (aT) node[sloped, midway, below, opacity = 1] {$\scriptstyle{m}$};
  \filldraw[draw = black, rounded corners=1pt, thick, fill opacity = 0.3, fill = red]  (bT) -- (bB) -- (BB)  -- (BT) -- (bT) node[sloped, midway, below, opacity = 1] {$\scriptstyle{n+l}$};
  \filldraw[draw = black, rounded corners=1pt, thick, fill opacity = 0.3, fill = yellow]   (aT) --  (aMt) -- (bMt) -- (bT) --  (aT) node[sloped, midway, below, opacity = 1] {$\scriptstyle{j_t}$};
  \filldraw[draw = black, rounded corners=1pt, thick, fill opacity = 0.3, fill = yellow]   (aMb) -- (aB) --  (bB) node[sloped, midway, above, opacity = 1] {$\scriptstyle{j_b}$} --  (bMb) -- cycle;
  \filldraw[draw = black, rounded corners=1pt, thick, fill opacity = 0.3, fill = red]  (dT) -- (dB) -- (DB) node[sloped, midway, above, opacity = 1] {$\scriptstyle{n}$}-- (DT) -- cycle;
  \filldraw[draw = black, rounded corners=1pt, thick, fill opacity = 0.3, fill = orange]  (aMt) -- (aMb) -- (dMb) -- (dMt) -- cycle;
  \filldraw[draw = black, rounded corners=1pt, thick, fill opacity = 0.3, fill = yellow]   (aT) --  (aMt) -- (dMt) -- (dT) --  cycle;
  \filldraw[draw = black, rounded corners=1pt, thick, fill opacity = 0.3, fill = yellow]   (aMb) -- (aB) --  (dB) --  (dMb) -- cycle;
  \filldraw[draw = black, rounded corners=1pt, thick, fill opacity = 0.3, fill = orange]  (cMt) -- (cMb) -- (dMb)  -- (dMt) -- cycle;
  \filldraw[draw = black, rounded corners=1pt, thick, fill opacity = 0.3, fill = orange]  (aMt) -- (aMb) -- (bMb) -- (bMt) -- cycle;
  \filldraw[draw = black, rounded corners=1pt, thick, fill opacity = 0.3, fill = orange]  (cMt) -- (cMb) -- (bMb) -- (bMt) -- cycle;
  \filldraw[draw = black, rounded corners=1pt, thick, fill opacity = 0.3, fill = yellow]   (cT) --  (cMt) -- (bMt) -- (bT) --  (cT);
  \filldraw[draw = black, rounded corners=1pt, thick, fill opacity = 0.3, fill = yellow]   (cMb) -- (cB) --  (bB) --  (bMb) -- cycle;
  \filldraw[draw = black, rounded corners=1pt, thick, fill opacity = 0.3, fill = yellow]   (cT) --  (cMt) -- (dMt) -- (dT) --  (cT) node[sloped, midway, below, opacity = 1] {$\scriptstyle{n+j_t-m}$};
  \filldraw[draw = black, rounded corners=1pt, thick, fill opacity = 0.3, fill = yellow]   (cMb) -- (cB) --  (dB) node[sloped, midway, above, opacity = 1] {$\scriptstyle{n+j_b-m}$} --  (dMb) -- cycle;
  \filldraw[draw = black, rounded corners=1pt, thick, fill opacity = 0.3, fill = gray]    (aMt) -- (bMt) -- (cMt) -- (dMt) node[sloped, midway, above, opacity = 1] {$\scriptstyle{q}$} -- cycle;
  \filldraw[draw = black, rounded corners=1pt, thick, fill opacity = 0.3, fill = gray]    (aMb) -- (bMb) -- (cMb) -- (dMb) node[sloped, midway, above, opacity = 1] {$\scriptstyle{s}$} -- cycle;
  \filldraw[draw = black, rounded corners=1pt, thick, fill opacity = 0.3, fill = red]  (cT) -- (cB) -- (CB) node[sloped, midway, above, opacity = 1] {$\scriptstyle{m+l}$} -- (CT) -- cycle; 
\end{scope}

%% file: cef_squarefoamFjjjjaaaa.tex
\tdplotsetmaincoords{70}{100}
\begin{scope}[tdplot_main_coords,yscale =1]
  \coordinate (aT) at (-1, -1, 4);
  \coordinate (bT) at (-1, 1,  4);
  \coordinate (cT) at (1, 1,   4);
  \coordinate (dT) at (1, -1,  4);
  \coordinate (AT) at (-2, -2, 4);
  \coordinate (BT) at (-2, 2,  4);
  \coordinate (CT) at (2, 2,   4);
  \coordinate (DT) at (2, -2,  4);
  \coordinate (aTt) at (-1, -1,  3);
  \coordinate (aTb) at (-1, -1,  1);
  \coordinate (aBt) at (-1, -1, -1);
  \coordinate (aBb) at (-1, -1, -3);
  \coordinate (bTt) at (-1, 1,   3);
  \coordinate (bTb) at (-1, 1,   1);
  \coordinate (bBt) at (-1, 1,  -1);
  \coordinate (bBb) at (-1, 1,  -3);
  \coordinate (cTt) at (1, 1,    3);   
  \coordinate (cTb) at (1, 1,    1);   
  \coordinate (cBt) at (1, 1,   -1);   
  \coordinate (cBb) at (1, 1,   -3);   
  \coordinate (dTt) at (1, -1,   3);   
  \coordinate (dTb) at (1, -1,   1);   
  \coordinate (dBt) at (1, -1,  -1);
  \coordinate (dBb) at (1, -1,  -3);
  \coordinate (aB) at (-1, -1, -4);
  \coordinate (bB) at (-1, 1,  -4);
  \coordinate (cB) at (1, 1,   -4);
  \coordinate (dB) at (1, -1,  -4);
  \coordinate (AB) at (-2, -2, -4);
  \coordinate (BB) at (-2, 2,  -4);
  \coordinate (CB) at (2, 2,   -4);
  \coordinate (DB) at (2, -2,  -4);

  \draw[red, thick, <-]($(aTb)!0.3!(dTt)$) --  +(0,-1.5,0) node [left, sloped] {$\scriptstyle{\pi_{\gamma_2}}$}; 
  \draw[red,thick, <-] ($(bT)!0.3!(cTt)$)-- +(0,+1.5,0) node [right, sloped] {$\scriptstyle{\pi_{\beta_1}}$}; 
  \draw[red, thick, <-]($(aBb)!0.3!(dBt)$) --  +(0,-1.5,0) node [left, sloped] {$\scriptstyle{\pi_{\epsilon_2}}$}; 
  \draw[red,thick, <-] ($(bTb)!0.3!(cBt)$)-- +(0,+1.5,0) node [right, sloped] {$\scriptstyle{\pi_{\delta_1}}$}; 
  \draw[red, thick, <-] ($(aBt)!0.5!(bTb)$) .. controls +(-3,0,0) and  +(0,+1,0) .. +(-3,-3.1,0) node [left, sloped] {$\scriptstyle{\pi_{\gamma_1}}$}; 
  \draw[red, thick, <-] ($(aB)!0.5!(bBb)$) .. controls +(-3,0,0) and  +(0,+1,0) .. +(-3,-3.1,0) node [left, sloped] {$\scriptstyle{\pi_{\epsilon_1}}$};

  \filldraw[draw = black, rounded corners=1pt, thick, fill opacity = 0.3, fill = red]  (aT) -- (aB) -- (AB) -- (AT) -- (aT) node[sloped, midway, below, opacity = 1] {$\scriptstyle{m}$};
  \filldraw[draw = black, rounded corners=1pt, thick, fill opacity = 0.3, fill = red]  (bT) -- (bB) -- (BB)  -- (BT) -- (bT) node[sloped, midway, below, opacity = 1] {$\scriptstyle{n+l}$};
  \filldraw[draw = black, rounded corners=1pt, thick, fill opacity = 0.3, fill = green]  (aT) -- (aTt) -- (bTt) -- (bT) -- cycle;
  \filldraw[draw = black, rounded corners=1pt, thick, fill opacity = 0.3, fill = yellow]  (aTt) -- (aTb) -- (bTb) -- (bTt) -- cycle;
  \filldraw[draw = black, rounded corners=1pt, thick, fill opacity = 0.3, fill = green]  (aTb) -- (aBt) -- (bBt) -- (bTb) -- cycle;
  \filldraw[draw = black, rounded corners=1pt, thick, fill opacity = 0.3, fill = yellow]  (aBt) -- (aBb) -- (bBb) -- (bBt) -- cycle;
  \filldraw[draw = black, rounded corners=1pt, thick, fill opacity = 0.3, fill = green]  (aB) -- (aBb) -- (bBb) -- (bB) -- cycle;
  \filldraw[draw = black, rounded corners=1pt, thick, fill opacity = 0.3, fill = green]  (aT) -- (aTt) -- (dTt) -- (dT) -- cycle;
  \filldraw[draw = black, rounded corners=1pt, thick, fill opacity = 0.3, fill = yellow]  (aTt) -- (aTb) -- (dTb) -- (dTt) -- cycle;
  \filldraw[draw = black, rounded corners=1pt, thick, fill opacity = 0.3, fill = green]  (aTb) -- (aBt) -- (dBt) -- (dTb) -- cycle;
  \filldraw[draw = black, rounded corners=1pt, thick, fill opacity = 0.3, fill = yellow]  (aBt) -- (aBb) -- (dBb) -- (dBt) -- cycle;
  \filldraw[draw = black, rounded corners=1pt, thick, fill opacity = 0.3, fill = green]  (aB) -- (aBb) -- (dBb) -- (dB) -- cycle;
  \filldraw[draw = black, rounded corners=1pt, thick, fill opacity = 0.3, fill = blue]    (aTt) -- (bTt) -- (cTt) -- (dTt) node[sloped, midway, above, opacity = 1] {$\scriptstyle{n+k-m+j_t}$} -- cycle;
  \filldraw[draw = black, rounded corners=1pt, thick, fill opacity = 0.3, fill = blue]    (aTb) -- (bTb) -- (cTb) -- (dTb) node[sloped, midway, above, opacity = 1] {$\scriptstyle{n+k-m+j_t}$} -- cycle;
  \filldraw[draw = black, rounded corners=1pt, thick, fill opacity = 0.3, fill = blue]    (aBt) -- (bBt) -- (cBt) -- (dBt) node[sloped, midway, above, opacity = 1] {$\scriptstyle{n+k-m+j_b}$} -- cycle;
  \filldraw[draw = black, rounded corners=1pt, thick, fill opacity = 0.3, fill = blue]    (aBb) -- (bBb) -- (cBb) -- (dBb) node[sloped, midway, above, opacity = 1] {$\scriptstyle{n+k-m+j_b}$} -- cycle;
  \filldraw[draw = black, rounded corners=1pt, thick, fill opacity = 0.3, fill = green]  (cT) -- (cTt) -- (bTt) -- (bT) -- cycle;
  \filldraw[draw = black, rounded corners=1pt, thick, fill opacity = 0.3, fill = yellow]  (cTt) -- (cTb) -- (bTb) -- (bTt) -- cycle;
  \filldraw[draw = black, rounded corners=1pt, thick, fill opacity = 0.3, fill = green]  (cTb) -- (cBt) -- (bBt) -- (bTb) -- cycle;
  \filldraw[draw = black, rounded corners=1pt, thick, fill opacity = 0.3, fill = yellow]  (cBt) -- (cBb) -- (bBb) -- (bBt) -- cycle;
  \filldraw[draw = black, rounded corners=1pt, thick, fill opacity = 0.3, fill = green]  (cB) -- (cBb) -- (bBb) -- (bB) -- cycle;
  \filldraw[draw = black, rounded corners=1pt, thick, fill opacity = 0.3, fill = red]  (cT) -- (cB) -- (CB) node[sloped, midway, above, opacity = 1] {$\scriptstyle{m+l}$} -- (CT) -- cycle; 
  \filldraw[draw = black, rounded corners=1pt, thick, fill opacity = 0.3, fill = red]  (dT) -- (dB) -- (DB) node[sloped, midway, above, opacity = 1] {$\scriptstyle{n}$}-- (DT) -- cycle;
  \filldraw[draw = black, rounded corners=1pt, thick, fill opacity = 0.3, fill = green]  (cT) -- (cTt) -- (dTt) -- (dT) -- cycle;
  \filldraw[draw = black, rounded corners=1pt, thick, fill opacity = 0.3, fill = yellow]  (cTt) -- (cTb) -- (dTb) -- (dTt) -- cycle;
  \filldraw[draw = black, rounded corners=1pt, thick, fill opacity = 0.3, fill = green]  (cTb) -- (cBt) -- (dBt) -- (dTb) -- cycle;
  \filldraw[draw = black, rounded corners=1pt, thick, fill opacity = 0.3, fill = yellow]  (cBt) -- (cBb) -- (dBb) -- (dBt) -- cycle;
  \filldraw[draw = black, rounded corners=1pt, thick, fill opacity = 0.3, fill = green]  (cB) -- (cBb) -- (dBb) -- (dB) -- cycle;
 \draw[red, thick, <-]($(cTt)!0.3!(dTb)$) .. controls  +(+1,0,0) and +(0, -1, 0) .. +(1,3,0)  node [right, sloped] {$\scriptstyle{\pi_{\beta_2}}$}; 
 \draw[red, thick, <-]($(cBt)!0.3!(dBb)$) .. controls  +(+1,0,0) and +(0, -1, 0) .. +(1,3,0)  node [right, sloped] {$\scriptstyle{\pi_{\delta_2}}$};

\end{scope}

%% file: cef_turnfacefoam.tex
\tdplotsetmaincoords{75}{100}
\begin{scope}[tdplot_main_coords]
  \coordinate (aT) at (-1, -1, 3);
  \coordinate (bT) at (-1, 1, 3);
  \coordinate (cT) at (1, 1, 3);
  \coordinate (dT) at (1, -1, 3);

  \coordinate (amT) at (-1, -1, 1.35);
  \coordinate (bmT) at (-1, 1, 1.35);
  \coordinate (cmT) at (1, 1, 1.35); 
  \coordinate (dmT) at (1, -1, 1.35);

  \coordinate (amB) at (-1, -1, -1.35);
  \coordinate (bmB) at (-1, 1, -1.35);
  \coordinate (cmB) at (1, 1, -1.35); 
  \coordinate (dmB) at (1, -1, -1.35);

  \coordinate (at) at (-1, -1, 2);
  \coordinate (ab) at (-1, -1, -2);
  \coordinate (bt) at (-1, 1, 2);
  \coordinate (bb) at (-1, 1, -2);
  \coordinate (ct) at (1, 1, 2);
  \coordinate (cb) at (1, 1, -2);
  \coordinate (dt) at (1, -1, 2);
  \coordinate (db) at (1, -1, -2);
  \coordinate (AT) at (-2, -2, 3);
  \coordinate (BT) at (-2, 2, 3);
  \coordinate (CT) at (2, 2, 3);
  \coordinate (DT) at (2, -2, 3);
  \coordinate (aMt) at (-1, -1, 0.7);
  \coordinate (aMb) at (-1, -1, -0.7);
  \coordinate (bMt) at (-1, 1, 0.7);
  \coordinate (bMb) at (-1, 1, -0.7);
  \coordinate (cMt) at (1, 1, 0.7);
  \coordinate (cMb) at (1, 1, -0.7);
  \coordinate (dMt) at (1, -1, 0.7);
  \coordinate (dMb) at (1, -1, -0.7);
  \coordinate (AMt) at (-2, -2, 0.7);
  \coordinate (AMb) at (-2, -2, -0.7);
  \coordinate (BMt) at (-2, 2, 0.7);
  \coordinate (BMb) at (-2, 2, -0.7);
  \coordinate (CMt) at (2, 2, 0.7);
  \coordinate (CMb) at (2, 2, -0.7);
  \coordinate (DM) at (2, -2, 0);
  \coordinate (aB) at (-1, -1, -3);
  \coordinate (bB) at (-1, 1, -3);
  \coordinate (cB) at (1, 1, -3);
  \coordinate (dB) at (1, -1, -3);
  \coordinate (AB) at (-2, -2, -3);
  \coordinate (BB) at (-2, 2, -3);
  \coordinate (CB) at (2, 2, -3);
  \coordinate (DB) at (2, -2, -3);
  \coordinate (aM) at (-1, -1, 0);
  \coordinate (bM) at (-1, 1, 0);
  \coordinate (cM) at (1, 1, 0);
  \coordinate (dM) at (1, -1, 0);

  \filldraw[draw = black, rounded corners=1pt, thick, fill opacity = 0.3, fill = red]  (aT) -- (aB) -- (AB) -- (AT) -- (aT);
  \filldraw[draw = black, rounded corners=1pt, thick, fill opacity = 0.3, fill = red]  (bT) -- (bB) -- (BB)  -- (BT) -- (bT);
  \filldraw[draw = black, rounded corners=1pt, thick, fill opacity = 0.3, fill = green]    (aT) --  (at) -- (bt) -- (bT) --  (aT);
  \filldraw[draw = black, rounded corners=1pt, thick, fill opacity = 0.3, fill = yellow]   (aMt) --  (at) -- (bt) -- (bMt) --  (aMt);
  \filldraw[draw = black, rounded corners=1pt, thick, fill opacity = 0.3, fill = green]    (aMb) --  (aMt) -- (bMt) -- (bMb) --  (aMb);
  \filldraw[draw = black, rounded corners=1pt, thick, fill opacity = 0.3, fill = yellow]   (aMb) --  (ab) -- (bb) -- (bMb) --  (aMb);
  \filldraw[draw = black, rounded corners=1pt, thick, fill opacity = 0.3, fill = green]    (ab) -- (aB) --  (bB) --  (bb) -- cycle;
  \filldraw[draw = black, rounded corners=1pt, thick, fill opacity = 0.3, fill = red]  (dT) -- (dB) -- (DB) -- (DT) -- cycle;

  \filldraw[draw = black, rounded corners=1pt, thick, fill opacity = 0.3, fill = green]   (aT) --  (at) -- (dt) -- (dT) --  cycle;
  \filldraw[draw = black, rounded corners=1pt, thick, fill opacity = 0.3, fill = yellow]  (aMt) -- (at) -- (dt) -- (dMt) -- cycle;
  \filldraw[draw = black, rounded corners=1pt, thick, fill opacity = 0.3, fill = green]   (aMb) --  (aMt) -- (dMt) -- (dMb) --  cycle;
  \filldraw[draw = black, rounded corners=1pt, thick, fill opacity = 0.3, fill = yellow]  (aMb) -- (ab) -- (db) -- (dMb) -- cycle;
  \filldraw[draw = black, rounded corners=1pt, thick, fill opacity = 0.3, fill = green]   (ab) -- (aB) --  (dB) --  (db) -- cycle;

  \filldraw[draw = black, rounded corners=1pt, thick, fill opacity = 0.3, fill = green]   (cT) --  (ct) -- (dt) -- (dT) --  (cT);
  \filldraw[draw = black, rounded corners=1pt, thick, fill opacity = 0.3, fill = yellow]  (ct) -- (cMt) -- (dMt)  -- (dt) -- cycle;
  \filldraw[draw = black, rounded corners=1pt, thick, fill opacity = 0.3, fill = green]   (cMb) -- (cMt) --  (dMt)  --  (dMb) -- cycle;
  \filldraw[draw = black, rounded corners=1pt, thick, fill opacity = 0.3, fill = yellow]  (cMb) -- (cb) -- (db)  -- (dMb) -- cycle;
  \filldraw[draw = black, rounded corners=1pt, thick, fill opacity = 0.3, fill = green]   (cb) -- (cB) --  (dB)  --  (db) -- cycle;

  \filldraw[draw = black, rounded corners=1pt, thick, fill opacity = 0.3, fill = green]   (cT) --  (ct) -- (bt) -- (bT) --  (cT);
  \filldraw[draw = black, rounded corners=1pt, thick, fill opacity = 0.3, fill = yellow]  (cMt) -- (ct) -- (bt) -- (bMt) -- cycle;
  \filldraw[draw = black, rounded corners=1pt, thick, fill opacity = 0.3, fill = green]   (cMt) -- (cMb) --  (bMb) --  (bMt) -- cycle;
  \filldraw[draw = black, rounded corners=1pt, thick, fill opacity = 0.3, fill = yellow]  (cb) -- (cMb) -- (bMb) -- (bb) -- cycle;
  \filldraw[draw = black, rounded corners=1pt, thick, fill opacity = 0.3, fill = green]   (cb) -- (cB) --  (bB) --  (bb) -- cycle;

  \filldraw[draw = black, rounded corners=1pt, thick, fill opacity = 0.3, fill = blue, pattern=north west lines, pattern color=black]    (aMt) -- (bMt) -- (cMt) -- (dMt) -- cycle;
  \filldraw[draw = black, rounded corners=1pt, thick, fill opacity = 0.3, fill = blue, pattern=north west lines, pattern color=black]    (aMb) -- (bMb) -- (cMb) -- (dMb) -- cycle;
  \filldraw[draw = black, rounded corners=1pt, thick, fill opacity = 0.3, fill = blue, pattern=north west lines, pattern color=black]    (at) -- (bt) -- (ct) -- (dt) -- cycle;
  \filldraw[draw = black, rounded corners=1pt, thick, fill opacity = 0.3, fill = blue, pattern=north west lines, pattern color=black]    (ab) -- (bb) -- (cb) -- (db) -- cycle;
  \filldraw[draw = black, rounded corners=1pt, thick, fill opacity = 0.3, fill = blue]    (aMt) -- (bMt) -- (cMt) -- (dMt) -- cycle;
  \filldraw[draw = black, rounded corners=1pt, thick, fill opacity = 0.3, fill = blue]    (aMb) -- (bMb) -- (cMb) -- (dMb) -- cycle;
  \filldraw[draw = black, rounded corners=1pt, thick, fill opacity = 0.3, fill = blue]    (at) -- (bt) -- (ct) -- (dt) -- cycle;
  \filldraw[draw = black, rounded corners=1pt, thick, fill opacity = 0.3, fill = blue]    (ab) -- (bb) -- (cb) -- (db) -- cycle;

  \filldraw[draw = black, rounded corners=1pt, thick, fill opacity = 0.3, fill = red]  (cT) -- (cB) -- (CB) -- (CT) -- cycle; 
 \draw[very thick, ->] (aT) -- (bT);
  \draw[very thick, ->] (bT) -- (cT);
  \draw[very thick, ->] (cT) -- (dT);
  \draw[very thick, ->] (dT) -- (aT);
  \draw[very thick, ->] (aB) -- (bB);
  \draw[very thick, ->] (bB) -- (cB);
  \draw[very thick, ->] (cB) -- (dB);
  \draw[very thick, ->] (dB) -- (aB);
  \draw[very thick, ->, dashed] (aM) -- (bM);
  \draw[very thick, ->, dashed] (bM) -- (cM);
  \draw[very thick, ->, dashed] (cM) -- (dM);
  \draw[very thick, ->, dashed] (dM) -- (aM);

  \draw[very thick, <-, dashed] (amT) -- (bmT);
  \draw[very thick, <-, dashed] (bmT) -- (cmT);
  \draw[very thick, <-, dashed] (cmT) -- (dmT);
  \draw[very thick, <-, dashed] (dmT) -- (amT);

  \draw[very thick, <-, dashed] (amB) -- (bmB);
  \draw[very thick, <-, dashed] (bmB) -- (cmB);
  \draw[very thick, <-, dashed] (cmB) -- (dmB);
  \draw[very thick, <-, dashed] (dmB) -- (amB);

\end{scope}

%% file: cef_turnfacefoam2.tex
\tdplotsetmaincoords{75}{100}
\begin{scope}[tdplot_main_coords]
  \coordinate (aT) at (-1, -1, 3);
  \coordinate (bT) at (-1, 1, 3);
  \coordinate (cT) at (1, 1, 3);
  \coordinate (dT) at (1, -1, 3);
  \coordinate (AT) at (-2, -2, 3);
  \coordinate (BT) at (-2, 2, 3);
  \coordinate (CT) at (2, 2, 3);
  \coordinate (DT) at (2, -2, 3);
  \coordinate (DM) at (2, -2, 0);
  \coordinate (aB) at (-1, -1, -3);
  \coordinate (bB) at (-1, 1, -3);
  \coordinate (cB) at (1, 1, -3);
  \coordinate (dB) at (1, -1, -3);
  \coordinate (AB) at (-2, -2, -3);
  \coordinate (BB) at (-2, 2, -3);
  \coordinate (CB) at (2, 2, -3);
  \coordinate (DB) at (2, -2, -3);
 \coordinate (aM) at (-1, -1, 0);
  \coordinate (bM) at (-1, 1, 0);
  \coordinate (cM) at (1, 1, 0);
  \coordinate (dM) at (1, -1, 0);
  \filldraw[draw = black, rounded corners=1pt, thick, fill opacity = 0.3, fill = red]  (aT) -- (aB) -- (AB) -- (AT) -- (aT);
  \filldraw[draw = black, rounded corners=1pt, thick, fill opacity = 0.3, fill = red]  (bT) -- (bB) -- (BB)  -- (BT) -- (bT);
  \filldraw[draw = black, rounded corners=1pt, thick, fill opacity = 0.3, fill = green]   (aT) --  (aB) -- (bB) -- (bT) --  (aT);
  \filldraw[draw = black, rounded corners=1pt, thick, fill opacity = 0.3, fill = red]  (dT) -- (dB) -- (DB) -- (DT) -- cycle;
  \filldraw[draw = black, rounded corners=1pt, thick, fill opacity = 0.3, fill = green]   (aT) --  (aB) -- (dB) -- (dT) --  cycle;
  \filldraw[draw = black, rounded corners=1pt, thick, fill opacity = 0.3, fill = green]   (cT) -- (cB) --  (bB) --  (bT) -- cycle;
  \filldraw[draw = black, rounded corners=1pt, thick, fill opacity = 0.3, fill = green]   (cT) -- (cB) --  (dB) --  (dT) -- cycle;
   \filldraw[draw = black, rounded corners=1pt, thick, fill opacity = 0.3, fill = red]  (cT) -- (cB) -- (CB) -- (CT) -- cycle; 
 \draw[very thick, ->] (aT) -- (bT);
  \draw[very thick, ->] (bT) -- (cT);
  \draw[very thick, ->] (cT) -- (dT);
  \draw[very thick, ->] (dT) -- (aT);
  \draw[very thick, ->] (aB) -- (bB);
  \draw[very thick, ->] (bB) -- (cB);
  \draw[very thick, ->] (cB) -- (dB);
  \draw[very thick, ->] (dB) -- (aB);
\end{scope}

%% file: cef_movie0face1.tex
\begin{scope}
\draw[draw=  gray, line width = 2mm] (-0.03, 2) --  (15.03, 2); 
\draw[draw=  gray, line width = 2mm] (-0.03,-2) --  (15.03,-2); 
\draw[draw=  gray, very thick] (0, -2) -- +(0,4);
\draw[draw=  gray, very thick] (5, -2) -- +(0,4);
\draw[draw=  gray, very thick] (10,-2) -- +(0,4);
\draw[draw=  gray, very thick] (15,-2) -- +(0,4);
\draw[draw= white, dotted, line width =1.2mm] (0,2) --  (15,2); 
\draw[draw= white, dotted, line width = 1.2mm] (0,-2) --  (15,-2); 
\begin{scope}[xshift=2.5cm]  
  \draw [->] (-2,0) -- (0,0);
  \draw  (0,0) -- (2,0);
\end{scope}
\begin{scope}[xshift=7.5cm]
  \draw [->] (-2,0) -- (-1,0);
  \draw [->] (1,0) -- (2,0);
  \draw [->] (-1, 0) .. controls (0, -1) .. ( 1,0);
  \draw [->] ( 1, 0) .. controls (0,  1) .. (-1,0);
\end{scope}
\begin{scope}[xshift=12.5cm]
  \draw [->] (-2,0) -- (-1,0);
  \draw [->] (1,0) -- (2,0);
  \draw [->] (-1, 0) .. controls (-0, -1.5) and (-0,  1.5) .. (-1,0);
  \draw [->] ( 1, 0) .. controls ( 0,  1.5) and ( 0, -1.5) .. ( 1,0);  
\end{scope}
\end{scope}

%% file: cef_movie0face2.tex
\begin{scope}
  \draw[draw=  gray, line width = 2mm] (-0.03, 2) --  (15.03, 2); 
  \draw[draw=  gray, line width = 2mm] (-0.03,-2) --  (15.03,-2); 
  \draw[draw=  gray, very thick] (0, -2) -- +(0,4);
  \draw[draw=  gray, very thick] (5, -2) -- +(0,4);
  \draw[draw=  gray, very thick] (10,-2) -- +(0,4);
  \draw[draw=  gray, very thick] (15,-2) -- +(0,4);
  \draw[draw= white, dotted, line width =1.2mm] (0,2) --  (15,2); 
  \draw[draw= white, dotted, line width = 1.2mm] (0,-2) --  (15,-2); 
\begin{scope}[xshift=1.5cm]  
  \draw [thick, red!50!black, ->] (0,-1.5) -- (0,0) -- (0,1.5);
  \draw [->] (0,0) -- (1,0); 
  \draw [<-] ( 1, 0) .. controls ( 2,  1.5) and ( 2, -1.5) .. ( 1,0);  
\end{scope}
\begin{scope}[xshift=7.5cm]
  \draw [thick, red!50!black, ->] (0,-1.5) -- (0,0) -- (0,1.5);
  \draw [->] (0,-0.5) .. controls ( 1,-0.5) and ( 1, 0.5)  .. (0,0.5); 
\end{scope}
\begin{scope}[xshift=12.5cm]
  \draw [thick, red!50!black, ->] (0,-1.5) -- (0,0) -- (0,1.5);
\end{scope}
\end{scope}

%% file: cef_movie0face3.tex
\begin{scope}
  \draw[draw=  gray, line width = 2mm] (-0.03, 2) --  (25.03, 2); 
  \draw[draw=  gray, line width = 2mm] (-0.03,-2) --  (25.03,-2); 
  \draw[draw=  gray, very thick] (0, -2) -- +(0,4);
  \draw[draw=  gray, very thick] (5, -2) -- +(0,4);
  \draw[draw=  gray, very thick] (10,-2) -- +(0,4);
  \draw[draw=  gray, very thick] (15,-2) -- +(0,4);
  \draw[draw=  gray, very thick] (20,-2) -- +(0,4);
  \draw[draw=  gray, very thick] (25,-2) -- +(0,4);
  \draw[draw= white, dotted, line width =1.2mm] (0,2) --  (25,2); 
  \draw[draw= white, dotted, line width = 1.2mm] (0,-2) --  (25,-2); 
\begin{scope}[xshift=2cm]  
  \draw [->] (0,0) -- (1,0); 
  \draw [<-] ( 1, 0) .. controls ( 2,  1.5) and ( 2, -1.5) .. ( 1,0);  
  \begin{scope}[rotate=-120]
  \draw [<-] (0,0) -- (1,0); 
  \draw [<-] ( 1, 0) .. controls ( 2,  1.5) and ( 2, -1.5) .. ( 1,0);  
  \end{scope}
  \begin{scope}[rotate=120]
  \draw [->] (0,0) -- (1,0); 
  \draw [<-] ( 1, 0) .. controls ( 2,  1.5) and ( 2, -1.5) .. ( 1,0);  
  \end{scope}
\end{scope}
\begin{scope}[xshift=8cm]  
  \draw [<-] (120:0.5) .. controls ( 1,  1.5) and ( 1, -1.5) .. (-120:0.5);  
  \draw [<-, rounded corners] (120:1) -- (0,0) --  (-120:1);  
  \begin{scope}[rotate=-120]
  \draw [<-] ( 1, 0) .. controls ( 2,  1.5) and ( 2, -1.5) .. ( 1,0);  
  \end{scope}
  \begin{scope}[rotate=120]
  \draw [<-] ( 1, 0) .. controls ( 2,  1.5) and ( 2, -1.5) .. ( 1,0);  
  \end{scope}
\end{scope}
\begin{scope}[xshift=12.5cm]  
  \draw [<-, rounded corners] (90:0.8) -- (0,0) --  (-90:0.8);  
  \begin{scope}[rotate=-90]
  \draw [<-] ( 0.8, 0) .. controls ( 1.8,  1.5) and ( 1.8, -1.5) .. ( 0.8,0);  
  \end{scope}
  \begin{scope}[rotate=90]
  \draw [<-] ( 0.8, 0) .. controls ( 1.8,  1.5) and ( 1.8, -1.5) .. ( 0.8,0);  
  \end{scope}
\end{scope}
\begin{scope}[xshift=17.5cm]  
  \draw[->] (0.5,0) arc (0:360:0.5 and 1.5);
  \end{scope}
\end{scope}

%% file: cef_turnfacefoam3.tex
\tdplotsetmaincoords{75}{100}
\begin{scope}[tdplot_main_coords]
  \coordinate (aT) at (-1, -1, 3);
  \coordinate (bT) at (-1, 1, 3);
  \coordinate (cT) at (1, 1, 3);
  \coordinate (dT) at (1, -1, 3);
  \coordinate (amT) at (-1, -1, 1.35);
  \coordinate (bmT) at (-1, 1, 1.35);
  \coordinate (cmT) at (1, 1, 1.35); 
  \coordinate (dmT) at (1, -1, 1.35);
  \coordinate (amB) at (-1, -1, -1.35);
  \coordinate (bmB) at (-1, 1, -1.35);
  \coordinate (cmB) at (1, 1, -1.35); 
  \coordinate (dmB) at (1, -1, -1.35);
  \coordinate (at) at (-1, -1, 2);
  \coordinate (ab) at (-1, -1, -2);
  \coordinate (bt) at (-1, 1, 2);
  \coordinate (bb) at (-1, 1, -2);
  \coordinate (ct) at (1, 1, 2);
  \coordinate (cb) at (1, 1, -2);
  \coordinate (dt) at (1, -1, 2);
  \coordinate (db) at (1, -1, -2);
  \coordinate (At) at (-2, -2, 2);
  \coordinate (Bt) at (-2, 2, 2);
  \coordinate (Ct) at (2, 2, 2);
  \coordinate (Dt) at (2, -2, 2);
  \coordinate (aMt) at (-1, -1, 0.7);
  \coordinate (aMb) at (-1, -1, -0.7);
  \coordinate (bMt) at (-1, 1, 0.7);
  \coordinate (bMb) at (-1, 1, -0.7);
  \coordinate (cMt) at (1, 1, 0.7);
  \coordinate (cMb) at (1, 1, -0.7);
  \coordinate (dMt) at (1, -1, 0.7);
  \coordinate (dMb) at (1, -1, -0.7);
  \coordinate (AMt) at (-2, -2, 0.7);
  \coordinate (AMb) at (-2, -2, -0.7);
  \coordinate (BMt) at (-2, 2, 0.7);
  \coordinate (BMb) at (-2, 2, -0.7);
  \coordinate (CMt) at (2, 2, 0.7);
  \coordinate (CMb) at (2, 2, -0.7);
  \coordinate (DM) at (2, -2, 0);
  \coordinate (aB) at (-1, -1, -3);
  \coordinate (bB) at (-1, 1, -3);
  \coordinate (cB) at (1, 1, -3);
  \coordinate (dB) at (1, -1, -3);
  \coordinate (AB) at (-2, -2, -3);
  \coordinate (BB) at (-2, 2, -3);
  \coordinate (CB) at (2, 2, -3);
  \coordinate (DB) at (2, -2, -3);
  \coordinate (aM) at (-1, -1, 0);
  \coordinate (bM) at (-1, 1, 0);
  \coordinate (cM) at (1, 1, 0);
  \coordinate (dM) at (1, -1, 0);

  \filldraw[draw = black, rounded corners=1pt, thick, fill opacity = 0.3, fill = red]  (at) -- (aB) -- (AB) -- (At) -- (at);
  \filldraw[draw = black, rounded corners=1pt, thick, fill opacity = 0.3, fill = red]  (bt) -- (bB) -- (BB)  -- (Bt) -- (bt);

  \filldraw[draw = black, rounded corners=1pt, thick, fill opacity = 0.3, fill = yellow]   (aMt) --  (at) -- (bt) -- (bMt) --  (aMt);
  \filldraw[draw = black, rounded corners=1pt, thick, fill opacity = 0.3, fill = green]    (aMb) --  (aMt) -- (bMt) -- (bMb) --  (aMb);
  \filldraw[draw = black, rounded corners=1pt, thick, fill opacity = 0.3, fill = yellow]   (aMb) --  (ab) -- (bb) -- (bMb) --  (aMb);
  \filldraw[draw = black, rounded corners=1pt, thick, fill opacity = 0.3, fill = green]    (ab) -- (aB) --  (bB) --  (bb) -- cycle;

  \filldraw[draw = black, rounded corners=1pt, thick, fill opacity = 0.3, fill = red]  (dt) -- (dB) -- (DB) -- (Dt) -- cycle;

  \filldraw[draw = black, rounded corners=1pt, thick, fill opacity = 0.3, fill = yellow]  (aMt) -- (at) -- (dt) -- (dMt) -- cycle;
  \filldraw[draw = black, rounded corners=1pt, thick, fill opacity = 0.3, fill = green]   (aMb) --  (aMt) -- (dMt) -- (dMb) --  cycle;
  \filldraw[draw = black, rounded corners=1pt, thick, fill opacity = 0.3, fill = yellow]  (aMb) -- (ab) -- (db) -- (dMb) -- cycle;
  \filldraw[draw = black, rounded corners=1pt, thick, fill opacity = 0.3, fill = green]   (ab) -- (aB) --  (dB) --  (db) -- cycle;

  \filldraw[draw = black, rounded corners=1pt, thick, fill opacity = 0.3, fill = yellow]  (ct) -- (cMt) -- (dMt)  -- (dt) -- cycle;
  \filldraw[draw = black, rounded corners=1pt, thick, fill opacity = 0.3, fill = green]   (cMb) -- (cMt) --  (dMt)  --  (dMb) -- cycle;
  \filldraw[draw = black, rounded corners=1pt, thick, fill opacity = 0.3, fill = yellow]  (cMb) -- (cb) -- (db)  -- (dMb) -- cycle;
  \filldraw[draw = black, rounded corners=1pt, thick, fill opacity = 0.3, fill = green]   (cb) -- (cB) --  (dB)  --  (db) -- cycle;

  \filldraw[draw = black, rounded corners=1pt, thick, fill opacity = 0.3, fill = yellow]  (cMt) -- (ct) -- (bt) -- (bMt) -- cycle;
  \filldraw[draw = black, rounded corners=1pt, thick, fill opacity = 0.3, fill = green]   (cMt) -- (cMb) --  (bMb) --  (bMt) -- cycle;
  \filldraw[draw = black, rounded corners=1pt, thick, fill opacity = 0.3, fill = yellow]  (cb) -- (cMb) -- (bMb) -- (bb) -- cycle;
  \filldraw[draw = black, rounded corners=1pt, thick, fill opacity = 0.3, fill = green]   (cb) -- (cB) --  (bB) --  (bb) -- cycle;

  \filldraw[draw = black, rounded corners=1pt, thick, fill opacity = 0.3, fill = blue, pattern=north west lines, pattern color=black]    (aMt) -- (bMt) -- (cMt) -- (dMt) -- cycle;
  \filldraw[draw = black, rounded corners=1pt, thick, fill opacity = 0.3, fill = blue, pattern=north west lines, pattern color=black]    (aMb) -- (bMb) -- (cMb) -- (dMb) -- cycle;
  \filldraw[draw = black, rounded corners=1pt, thick, fill opacity = 0.3, fill = blue, pattern=north west lines, pattern color=black]    (ab) -- (bb) -- (cb) -- (db) -- cycle;
  \filldraw[draw = black, rounded corners=1pt, thick, fill opacity = 0.3, fill = blue]    (aMt) -- (bMt) -- (cMt) -- (dMt) -- cycle;
  \filldraw[draw = black, rounded corners=1pt, thick, fill opacity = 0.3, fill = blue]    (aMb) -- (bMb) -- (cMb) -- (dMb) -- cycle;

  \filldraw[draw = black, rounded corners=1pt, thick, fill opacity = 0.3, fill = blue]    (ab) -- (bb) -- (cb) -- (db) -- cycle;

  \filldraw[draw = black, rounded corners=1pt, thick, fill opacity = 0.3, fill = red]  (ct) -- (cB) -- (CB) -- (Ct) -- cycle; 
  \draw[very thick, <-] (at) -- (bt);
  \draw[very thick, <-] (bt) -- (ct);
  \draw[very thick, <-] (ct) -- (dt);
  \draw[very thick, <-] (dt) -- (at);
  \draw[very thick, ->] (aB) -- (bB);
  \draw[very thick, ->] (bB) -- (cB);
  \draw[very thick, ->] (cB) -- (dB);
  \draw[very thick, ->] (dB) -- (aB);
  \draw[very thick, ->, dashed] (aM) -- (bM);
  \draw[very thick, ->, dashed] (bM) -- (cM);
  \draw[very thick, ->, dashed] (cM) -- (dM);
  \draw[very thick, ->, dashed] (dM) -- (aM);

  \draw[very thick, <-, dashed] (amB) -- (bmB);
  \draw[very thick, <-, dashed] (bmB) -- (cmB);
  \draw[very thick, <-, dashed] (cmB) -- (dmB);
  \draw[very thick, <-, dashed] (dmB) -- (amB);

\end{scope}

%% file: cef_turnfacefoam4.tex
\tdplotsetmaincoords{75}{100}
\begin{scope}[tdplot_main_coords]
  \coordinate (aT) at (-1, -1, 3);
  \coordinate (bT) at (-1, 1, 3);
  \coordinate (cT) at (1, 1, 3);
  \coordinate (dT) at (1, -1, 3);
  \coordinate (amT) at (-1, -1, 1.35);
  \coordinate (bmT) at (-1, 1, 1.35);
  \coordinate (cmT) at (1, 1, 1.35); 
  \coordinate (dmT) at (1, -1, 1.35);
  \coordinate (amB) at (-1, -1, -1.35);
  \coordinate (bmB) at (-1, 1, -1.35);
  \coordinate (cmB) at (1, 1, -1.35); 
  \coordinate (dmB) at (1, -1, -1.35);
  \coordinate (at) at (-1, -1, 2);
  \coordinate (ab) at (-1, -1, -2);
  \coordinate (bt) at (-1, 1, 2);
  \coordinate (bb) at (-1, 1, -2);
  \coordinate (ct) at (1, 1, 2);
  \coordinate (cb) at (1, 1, -2);
  \coordinate (dt) at (1, -1, 2);
  \coordinate (db) at (1, -1, -2);
  \coordinate (AT) at (-2, -2, 3);
  \coordinate (BT) at (-2, 2, 3);
  \coordinate (CT) at (2, 2, 3);
  \coordinate (DT) at (2, -2, 3);
  \coordinate (aMt) at (-1, -1, 0.7);
  \coordinate (aMb) at (-1, -1, -0.7);
  \coordinate (bMt) at (-1, 1, 0.7);
  \coordinate (bMb) at (-1, 1, -0.7);
  \coordinate (cMt) at (1, 1, 0.7);
  \coordinate (cMb) at (1, 1, -0.7);
  \coordinate (dMt) at (1, -1, 0.7);
  \coordinate (dMb) at (1, -1, -0.7);
  \coordinate (AMt) at (-2, -2, 0.7);
  \coordinate (AMb) at (-2, -2, -0.7);
  \coordinate (BMt) at (-2, 2, 0.7);
  \coordinate (BMb) at (-2, 2, -0.7);
  \coordinate (CMt) at (2, 2, 0.7);
  \coordinate (CMb) at (2, 2, -0.7);
  \coordinate (DM) at (2, -2, 0);
  \coordinate (aB) at (-1, -1, -3);
  \coordinate (bB) at (-1, 1, -3);
  \coordinate (cB) at (1, 1, -3);
  \coordinate (dB) at (1, -1, -3);
  \coordinate (AMt) at (-2, -2, 0.7);
  \coordinate (BMt) at (-2, 2, 0.7);
  \coordinate (CMt) at (2, 2, 0.7);
  \coordinate (DMt) at (2, -2, 0.7);
  \coordinate (aM) at (-1, -1, 0);
  \coordinate (bM) at (-1, 1, 0);
  \coordinate (cM) at (1, 1, 0);
  \coordinate (dM) at (1, -1, 0);

  \filldraw[draw = black, rounded corners=1pt, thick, fill opacity = 0.3, fill = red]  (aT) -- (aMt) -- (AMt) -- (AT) -- (aT);
  \filldraw[draw = black, rounded corners=1pt, thick, fill opacity = 0.3, fill = red]  (bT) -- (bMt) -- (BMt)  -- (BT) -- (bT);

  \filldraw[draw = black, rounded corners=1pt, thick, fill opacity = 0.3, fill = green]    (aT) --  (at) -- (bt) -- (bT) --  (aT);
  \filldraw[draw = black, rounded corners=1pt, thick, fill opacity = 0.3, fill = yellow]   (aMt) --  (at) -- (bt) -- (bMt) --  (aMt);

  \filldraw[draw = black, rounded corners=1pt, thick, fill opacity = 0.3, fill = red]  (dT) -- (dMt) -- (DMt) -- (DT) -- cycle;

  \filldraw[draw = black, rounded corners=1pt, thick, fill opacity = 0.3, fill = green]   (aT) --  (at) -- (dt) -- (dT) --  cycle;
  \filldraw[draw = black, rounded corners=1pt, thick, fill opacity = 0.3, fill = yellow]  (aMt) -- (at) -- (dt) -- (dMt) -- cycle;

  \filldraw[draw = black, rounded corners=1pt, thick, fill opacity = 0.3, fill = green]   (cT) --  (ct) -- (dt) -- (dT) --  (cT);
  \filldraw[draw = black, rounded corners=1pt, thick, fill opacity = 0.3, fill = yellow]  (ct) -- (cMt) -- (dMt)  -- (dt) -- cycle;

  \filldraw[draw = black, rounded corners=1pt, thick, fill opacity = 0.3, fill = green]   (cT) --  (ct) -- (bt) -- (bT) --  (cT);
  \filldraw[draw = black, rounded corners=1pt, thick, fill opacity = 0.3, fill = yellow]  (cMt) -- (ct) -- (bt) -- (bMt) -- cycle;
  \filldraw[draw = black, rounded corners=1pt, thick, fill opacity = 0.3, fill = blue, pattern=north west lines, pattern color=black]    (at) -- (bt) -- (ct) -- (dt) -- cycle;
  \filldraw[draw = black, rounded corners=1pt, thick, fill opacity = 0.3, fill = blue]    (at) -- (bt) -- (ct) -- (dt) -- cycle;

  \filldraw[draw = black, rounded corners=1pt, thick, fill opacity = 0.3, fill = red]  (cT) -- (cMt) -- (CMt) -- (CT) -- cycle; 
  \draw[very thick, <-] (aMt) -- (bMt);
  \draw[very thick, <-] (bMt) -- (cMt);
  \draw[very thick, <-] (cMt) -- (dMt);
  \draw[very thick, <-] (dMt) -- (aMt);
  \draw[very thick, ->] (aT) -- (bT);
  \draw[very thick, ->] (bT) -- (cT);
  \draw[very thick, ->] (cT) -- (dT);
  \draw[very thick, ->] (dT) -- (aT);

\end{scope}

%% file: cef_matveevpierigliani3.tex
\tdplotsetmaincoords{70}{110}
\begin{scope}[tdplot_main_coords]
  \coordinate (bT) at (-1, 0, 3);
  \coordinate (dT) at (1, 0, 3);
  \coordinate (AT) at (-2, -2, 3);
  \coordinate (BT) at (-2, 2, 3);
  \coordinate (CT) at (2, 2, 3);
  \coordinate (DT) at (2, -2, 3);
  \coordinate (MT) at (0,0, 1.5);
  \coordinate (aM) at (0, -1, 0);
  \coordinate (cM) at (0, 1, 0);
  \coordinate (AM) at (-2, -2, 0);
  \coordinate (BM) at (-2, 2, 0);
  \coordinate (CM) at (2, 2, 0);
  \coordinate (DM) at (2, -2, 0);
  \coordinate (MB) at (0,0, -1.5);
  \coordinate (bB) at (-1, 0, -3);
  \coordinate (dB) at (1, 0, -3);
  \coordinate (AB) at (-2, -2, -3);
  \coordinate (BB) at (-2, 2, -3);
  \coordinate (CB) at (2, 2, -3);
  \coordinate (DB) at (2, -2, -3);
  \filldraw[draw = black, rounded corners=1pt, thick, fill opacity = 0.3, fill = red]     (BT) -- (BM) -- (cM) -- (MT) -- (bT) -- (BT) node[midway, below, sloped, opacity=1] {$\scriptstyle{j}$};
  \filldraw[draw = black, rounded corners=1pt, thick, fill opacity = 0.3, fill = red] (AT) -- (AM) -- (aM) -- (MT)  -- (bT)  -- (AT) node[midway, below, opacity=1, black] { $\scriptstyle{i}$};
  \filldraw[draw = black, rounded corners=1pt, thick, fill opacity = 0.3, fill = red]    (CT) -- (CM)  --  (cM) -- (MT) -- (dT)-- (CT) node[near end, below, opacity=1] { $\scriptstyle{k}$};
  \filldraw[draw = black, rounded corners=1pt, thick, fill opacity = 0.3, fill = red]  (DT) -- (DM) node[midway, above, sloped, opacity=1] { $\scriptstyle{i+j+k}$} -- (aM) -- (MT) -- (dT)-- cycle;
  \filldraw[draw = black, rounded corners=1pt, thick, fill opacity = 0.3, fill = blue]  (bT) -- (dT) node[pos= 0.7, below, sloped, opacity=1] {$\scriptstyle{i+j}$} -- (MT) -- cycle;

  \filldraw[draw = black, rounded corners=1pt, thick, fill opacity = 0.3, fill = green]     (MT)-- (aM) -- (cM)-- cycle;
\draw (cM) node[midway, above, sloped, opacity=1] {$\scriptstyle{j+k}$}  -- (CM);

\end{scope}

%% file: cef_matveevpierigliani4.tex
\tdplotsetmaincoords{70}{110}
\begin{scope}[tdplot_main_coords]
  \coordinate (bT) at (-1, 0, 3);
  \coordinate (dT) at (1, 0, 3);
  \coordinate (AT) at (-2, -2, 3);
  \coordinate (BT) at (-2, 2, 3);
  \coordinate (CT) at (2, 2, 3);
  \coordinate (DT) at (2, -2, 3);
  \coordinate (MT) at (0,0, 1.5);
  \coordinate (aM) at (0, -1, 0);
  \coordinate (cM) at (0, 1, 0);
  \coordinate (AM) at (-2, -2, 0);
  \coordinate (BM) at (-2, 2, 0);
  \coordinate (CM) at (2, 2, 0);
  \coordinate (DM) at (2, -2, 0);
  \coordinate (MB) at (0,0, -1.5);
  \coordinate (bB) at (-1, 0, -3);
  \coordinate (dB) at (1, 0, -3);
  \coordinate (AB) at (-2, -2, -3);
  \coordinate (BB) at (-2, 2, -3);
  \coordinate (CB) at (2, 2, -3);
  \coordinate (DB) at (2, -2, -3);
  \filldraw[draw = black, rounded corners=1pt, thick, fill opacity = 0.3, fill = red]     (BB) -- (BM) -- (cM) -- (MB) -- (bB) -- (BB) node[midway, above, sloped, opacity=1] {$\scriptstyle{j}$};
  \filldraw[draw = black, rounded corners=1pt, thick, fill opacity = 0.3, fill = red] (AB) -- (AM) -- (aM) -- (MB)  -- (bB)  -- (AB) node[midway, above, opacity=1, black] { $\scriptstyle{i}$};
  \filldraw[draw = black, rounded corners=1pt, thick, fill opacity = 0.3, fill = red]    (CB) -- (CM)  --  (cM) -- (MB) -- (dB) -- (CB) node[near end, above, opacity=1] { $\scriptstyle{k}$};
  \filldraw[draw = black, rounded corners=1pt, thick, fill opacity = 0.3, fill = red]  (DB) -- (DM) node[midway, below, sloped, opacity=1] { $\scriptstyle{i+j+k}$} -- (aM) -- (MB) -- (dB)-- cycle;
  \filldraw[draw = black, rounded corners=1pt, thick, fill opacity = 0.3, fill = blue]  (bB) -- (dB) node[pos= 0.3, above, sloped, opacity=1] {$\scriptstyle{i+j}$} -- (MB) -- cycle;

  \filldraw[draw = black, rounded corners=1pt, thick, fill opacity = 0.3, fill = green]     (MB)-- (aM) -- (cM)-- cycle;
\draw (cM) node[midway, below, sloped, opacity=1] {$\scriptstyle{j+k}$}  -- (CM);

\end{scope}

%% file: cef_squarefoam8.tex
\tdplotsetmaincoords{70}{100}
\begin{scope}[tdplot_main_coords]
  \coordinate (aT) at (-1, -1, 3);
  \coordinate (bT) at (-1, 1, 3);
  \coordinate (cT) at (1, 1, 3);
  \coordinate (dT) at (1, -1, 3);
  \coordinate (AT) at (-2, -2, 3);
  \coordinate (BT) at (-2, 2, 3);
  \coordinate (CT) at (2, 2, 3);
  \coordinate (DT) at (2, -2, 3);
  \coordinate (aMt) at (-1, -1, 1);
  \coordinate (aMb) at (-1, -1, -1);
  \coordinate (bMt) at (-1, 1, 1);
  \coordinate (bMb) at (-1, 1, -1);
  \coordinate (cMt) at (1, 1, 1);
  \coordinate (cMb) at (1, 1, -1);
  \coordinate (dMt) at (1, -1, 1);
  \coordinate (dMb) at (1, -1, -1);
  \coordinate (AMt) at (-2, -2, 1);
  \coordinate (AMb) at (-2, -2, -1);
  \coordinate (BMt) at (-2, 2, 1);
  \coordinate (BMb) at (-2, 2, -1);
  \coordinate (CMt) at (2, 2, 1);
  \coordinate (CMb) at (2, 2, -1);
  \coordinate (DM) at (2, -2, 0);
  \coordinate (DMt) at (2, -2, 1);
  \coordinate (DMb) at (2, -2, -1);
  \coordinate (aB) at (-1, -1, -3);
  \coordinate (bB) at (-1, 1, -3);
  \coordinate (cB) at (1, 1, -3);
  \coordinate (dB) at (1, -1, -3);
  \coordinate (AB) at (-2, -2, -3);
  \coordinate (BB) at (-2, 2, -3);
  \coordinate (CB) at (2, 2, -3);
  \coordinate (DB) at (2, -2, -3);
  \draw[thin, <-] ($(dMb)!0.5!(aMt)$) --    +(0,-1.5,0) node [left, sloped] {${\scriptstyle{m-j}}$}; 
  \draw[thin, <-] ($(aT)!0.3!(dMt)$) --     +(0,-1.5,0) node [left, sloped] {$\scriptstyle{n+k}$}; 
  \draw[red,thick, <-] ($(bT)!0.7!(cMt)$)-- +(0,+1.5,0) node [right, sloped] {$\scriptstyle{\pi_{\beta_1}}$}; 
  \draw[thin, <-] ($(cMb)!0.7!(bMt)$) --    +(0,+1.5,0) node [right, sloped] {$\scriptstyle{n+l+j}$}; 
  \draw[thin, <-] ($(bT)!0.3!(cMt)$) --     +(0,+1.5,0) node [right, sloped] {$\scriptstyle{m+l-k}$};

  \filldraw[draw = black, rounded corners=1pt, thick, fill opacity = 0.3, fill = red]  (aT) -- (aMb) -- (AMb) -- (AT) -- (aT) node[sloped, midway, below, opacity = 1] {$\scriptstyle{m}$};
  \filldraw[draw = black, rounded corners=1pt, thick, fill opacity = 0.3, fill = red]  (bT) -- (bMb) -- (BMb)  -- (BT) -- (bT) node[sloped, midway, below, opacity = 1] {$\scriptstyle{n+l}$};

  \filldraw[draw = black, rounded corners=1pt, thick, fill opacity = 0.3, fill = green]   (aT) --  (aMt) -- (bMt) -- (bT) --  (aT) node[sloped, midway, below, opacity = 1] {$\scriptstyle{n+k-m}$};

  \filldraw[draw = black, rounded corners=1pt, thick, fill opacity = 0.3, fill = red]  (dT) -- (dMb) -- (DMb) node[sloped, midway, above, opacity = 1] {$\scriptstyle{n}$}-- (DT) -- cycle;

  \filldraw[draw = black, rounded corners=1pt, thick, fill opacity = 0.3, fill = yellow]  (aMt) -- (aMb) -- (dMb) -- (dMt) -- cycle;
  \filldraw[draw = black, rounded corners=1pt, thick, fill opacity = 0.3, fill = green]   (aT) --  (aMt) -- (dMt) -- (dT) --  cycle;

  \filldraw[draw = black, rounded corners=1pt, thick, fill opacity = 0.3, fill = yellow]  (cMt) -- (cMb) -- (dMb)  node [below, sloped, midway, opacity=1, above] {$\scriptstyle{n+j-l}$}  -- (dMt) -- cycle;
  \filldraw[draw = black, rounded corners=1pt, thick, fill opacity = 0.3, fill = yellow]  (aMt) -- (aMb) -- (bMb)  node [above, sloped, midway, opacity=1, above] {$\scriptstyle{j}$} -- (bMt) -- cycle;
  \filldraw[draw = black, rounded corners=1pt, thick, fill opacity = 0.3, fill = yellow]  (cMt) -- (cMb) -- (bMb) -- (bMt) -- cycle;
  \filldraw[draw = black, rounded corners=1pt, thick, fill opacity = 0.3, fill = green]   (cT) --  (cMt) -- (bMt) -- (bT) --  (cT);

  \filldraw[draw = black, rounded corners=1pt, thick, fill opacity = 0.3, fill = green]   (cT) --  (cMt) -- (dMt) -- (dT) --  (cT) node[sloped, midway, below, opacity = 1] {$\scriptstyle{k}$};

  \filldraw[draw = black, rounded corners=1pt, thick, fill opacity = 0.3, fill = blue]    (aMt) -- (bMt) -- (cMt) -- (dMt) node[sloped, midway, above, opacity = 1] {$\scriptstyle{n+k-m+j}$} -- cycle;

  \filldraw[draw = black, rounded corners=1pt, thick, fill opacity = 0.3, fill = red]  (cT) -- (cMb) -- (CMb) node[sloped, midway, above, opacity = 1] {$\scriptstyle{m+l}$} -- (CT) -- cycle; 
  \draw[red, thick, <-]($(cMt)!0.3!(dMb)$) .. controls  +(+1,0,0) and +(0, -1, 0) .. +(1,3,0)  node [right, sloped] {$\scriptstyle{\pi_{\beta_2}}$}; 
\end{scope}

%% file: cef_squarefoam9.tex
\tdplotsetmaincoords{70}{100}
\begin{scope}[tdplot_main_coords]
  \coordinate (aT) at (-1, -1, 3);
  \coordinate (bT) at (-1, 1, 3);
  \coordinate (cT) at (1, 1, 3);
  \coordinate (dT) at (1, -1, 3);
  \coordinate (AT) at (-2, -2, 3);
  \coordinate (BT) at (-2, 2, 3);
  \coordinate (CT) at (2, 2, 3);
  \coordinate (DT) at (2, -2, 3);
  \coordinate (aMt) at (-1, -1, 1);
  \coordinate (aMb) at (-1, -1, -1);
  \coordinate (bMt) at (-1, 1, 1);
  \coordinate (bMb) at (-1, 1, -1);
  \coordinate (cMt) at (1, 1, 1);
  \coordinate (cMb) at (1, 1, -1);
  \coordinate (dMt) at (1, -1, 1);
  \coordinate (dMb) at (1, -1, -1);
  \coordinate (AMt) at (-2, -2, 1);
  \coordinate (AMb) at (-2, -2, -1);
  \coordinate (BMt) at (-2, 2, 1);
  \coordinate (BMb) at (-2, 2, -1);
  \coordinate (CMt) at (2, 2, 1);
  \coordinate (CMb) at (2, 2, -1);
  \coordinate (DM) at (2, -2, 0);
  \coordinate (DMt) at (2, -2, 1);
  \coordinate (DMb) at (2, -2, -1);
  \coordinate (aB) at (-1, -1, -3);
  \coordinate (bB) at (-1, 1, -3);
  \coordinate (cB) at (1, 1, -3);
  \coordinate (dB) at (1, -1, -3);
  \coordinate (AB) at (-2, -2, -3);
  \coordinate (BB) at (-2, 2, -3);
  \coordinate (CB) at (2, 2, -3);
  \coordinate (DB) at (2, -2, -3);
  \draw[thin, <-] ($(dMb)!0.7!(aMt)$) --    +(0,-1.5,0) node [left, sloped] {$\scriptstyle{m-j}$}; 
  \draw[red, thick, <-]($(aMb)!0.3!(dMt)$) --  +(0,-1.5,0) node [left, sloped] {$\scriptstyle{\pi_{\gamma_2}}$}; 
  \draw[thin, <-] ($(dB)!0.5!(aMb)$) --     +(0,-1.5,0) node [left, sloped] {$\scriptstyle{n+k}$}; 
  \draw[thin, <-] ($(cMb)!0.7!(bMt)$) --    +(0,+1.5,0) node [right, sloped] {$\scriptstyle{n+l+j}$}; 
  \draw[thin, <-] ($(cB)!0.7!(bMb)$) --     +(0,+1.5,0) node [right, sloped] {$\scriptstyle{m+l-k}$}; 
  \draw[red, thick, <-] ($(aB)!0.5!(bMb)$) .. controls +(-3,0,0) and  +(0,+1,0) .. +(-3,-3.1,0) node [left, sloped] {$\scriptstyle{\pi_{\gamma_1}}$}; 
  \filldraw[draw = black, rounded corners=1pt, thick, fill opacity = 0.3, fill = red]  (aMt) -- (aB) -- (AB) -- (AMt) -- (aMt) node[sloped, midway, below, opacity = 1] {$\scriptstyle{m}$};
  \filldraw[draw = black, rounded corners=1pt, thick, fill opacity = 0.3, fill = red]  (bMt) -- (bB) -- (BB)  -- (BMt) -- (bMt) node[sloped, midway, below, opacity = 1] {$\scriptstyle{n+l}$};

  \filldraw[draw = black, rounded corners=1pt, thick, fill opacity = 0.3, fill = green]   (aMb) -- (aB) --  (bB) node[sloped, midway, above, opacity = 1] {$\scriptstyle{n+k-m}$} --  (bMb) -- cycle;

  \filldraw[draw = black, rounded corners=1pt, thick, fill opacity = 0.3, fill = red]  (dMt) -- (dB) -- (DB) node[sloped, midway, above, opacity = 1] {$\scriptstyle{n}$}-- (DMt) -- cycle;

  \filldraw[draw = black, rounded corners=1pt, thick, fill opacity = 0.3, fill = yellow]  (aMt) -- (aMb) -- (dMb) -- (dMt) --  (aMt);
  \filldraw[draw = black, rounded corners=1pt, thick, fill opacity = 0.3, fill = green]   (aMb) -- (aB) --  (dB) --  (dMb) -- cycle;
  \filldraw[draw = black, rounded corners=1pt, thick, fill opacity = 0.3, fill = yellow]  (cMt) -- (cMb) -- (dMb)  -- (dMt) -- (cMt) node[midway, below, sloped , opacity =1] {$\scriptstyle{n+j-l}$};
  \filldraw[draw = black, rounded corners=1pt, thick, fill opacity = 0.3, fill = yellow]  (aMt) -- (aMb) -- (bMb) -- (bMt)-- cycle node[midway, below, opacity =1] {$\scriptstyle{j}$} ;
  \filldraw[draw = black, rounded corners=1pt, thick, fill opacity = 0.3, fill = yellow]  (cMt) -- (cMb) -- (bMb) -- (bMt) -- cycle;
  \filldraw[draw = black, rounded corners=1pt, thick, fill opacity = 0.3, fill = green]   (cMb) -- (cB) --  (bB) --  (bMb) -- cycle;
  \filldraw[draw = black, rounded corners=1pt, thick, fill opacity = 0.3, fill = green]   (cMb) -- (cB) --  (dB) node[sloped, midway, above, opacity = 1] {$\scriptstyle{k}$} --  (dMb) -- cycle;

  \filldraw[draw = black, rounded corners=1pt, thick, fill opacity = 0.3, fill = blue]    (aMb) -- (bMb) -- (cMb) -- (dMb) node[sloped, midway, above, opacity = 1] {$\scriptstyle{n+k-m+j}$} -- cycle;

  \filldraw[draw = black, rounded corners=1pt, thick, fill opacity = 0.3, fill = red]  (cMt) -- (cB) -- (CB) node[sloped, midway, above, opacity = 1] {$\scriptstyle{m+l}$} -- (CMt) -- cycle; 

\end{scope}

%% file: cef_sphere.tex
\begin{scope}
  \draw[fill opacity =0.3, thick, fill = red] (0,0) circle (1cm);
  \draw[thick, dotted] (1,0) arc (0:180:1cm and 0.5cm) node[midway, above] {$\scriptstyle{a}$};
  \filldraw[draw = black, fill opacity =0.3, thick, fill = red] (0,0) circle (1cm);
  \draw[thick] (1,0) arc (0:-180:1cm and 0.5cm);
  \draw[red, thick, <-] (60:1) -- +(0.5, 0) node[right, red] {$\pi_{\alpha}$};
\end{scope}

%% file: cef_theta.tex
\begin{scope}
  \fill[fill opacity =0.3, thick, fill = red]   (1,0) arc (0:-180:1cm) arc (180:0:1cm and 0.5cm);
  \fill[fill opacity =0.3, thick, fill = blue]  (0,0) ellipse (1cm and 0.5cm);
  \node at (0,0) {$\scriptstyle{b}$};
  \draw[thick, dotted] (1,0) arc (0:180:1cm and 0.5cm);
  \fill[fill opacity =0.3, thick, fill = green] (1,0) arc (0:180:1cm)  node[midway, below, opacity =1, black] {$\scriptstyle{a}$} arc (180:0:1cm and 0.5cm);
   \draw[red, thick, <-] (0.5,0) .. controls (0.5,0.3).. +(0.6, 0.3) node[right, red] {$\pi_{\beta}$};
  \fill[fill opacity =0.3, thick, fill = red]   (1,0) arc (0:-180:1cm) node[midway,above, opacity =1] {$\scriptstyle{a+b}$} arc (-180:0:1cm and 0.5cm);
  \fill[fill opacity =0.3, thick, fill = green] (1,0) arc (0:180:1cm)  arc (-180:0:1cm and 0.5cm);
  \draw[thick] (0,0) circle (1cm);
  \draw[thick, ->] (1,0) arc (0:-90:1cm and 0.5cm);
  \draw[thick, -]  (-1,0) arc (-180:-90:1cm and 0.5cm);

   \draw[red, thick, <-] (60:1) -- +(0.6, 0) node[right, red] {$\pi_{\alpha}$};
   \draw[red, thick, <-] (-60:1) -- +(0.6, 0) node[right, red] {$\pi_{\gamma}$};
\end{scope}

%% file: cef_neckcuting.tex
\tdplotsetmaincoords{90}{0}
\begin{scope}[tdplot_main_coords]
  \coordinate (aT) at (1, 0, 3);
  \coordinate (bT) at (-1, 0, 3);
  \coordinate (aB) at (1, 0, -1);
  \coordinate (bB) at (-1, 0, -1);
  \filldraw[draw = black, rounded corners=1pt, thick, fill opacity = 0.3, dotted, fill = red]  (aT) arc (0:-180:1cm and 0.5cm) -- (bB) arc (180:0:1cm and 0.5cm) -- (aT);
  \filldraw[draw = black, rounded corners=1pt, thick, fill opacity = 0.3, fill = red]    (aT) arc (0:180:1cm and 0.5cm) -- (bB) arc (-180:0:1cm and 0.5cm) -- (aT);
  \draw[thick, black, ->-]  (aT) arc (0:-180:1cm and 0.5cm) node[midway, below] {$\scriptstyle{a}$};
  \draw[thick, black, ->-]  (aB) arc (0:-180:1cm and 0.5cm);
\end{scope}

%% file: cef_neckcuting2.tex
\tdplotsetmaincoords{90}{0}
\begin{scope}[tdplot_main_coords]
  \coordinate (aT) at (1, 0, 3);
  \coordinate (bT) at (-1, 0, 3);
  \coordinate (aB) at (1, 0, -1);
  \coordinate (bB) at (-1, 0, -1);
  \coordinate (a2) at (1, 0, 0.4);
  \coordinate (b2) at (-1, 0, 0.4);
  \coordinate (a1) at (1, 0, -0.2);
  \coordinate (am) at (1, 0, 0.3);
  \coordinate (b1) at (-1, 0, -0.2);
  \coordinate (M2) at (0, 0, 0.4);
  \coordinate (M1) at (0, 0, -0.2);
  \fill[rounded corners=1pt, thick, fill opacity = 0.3, fill = red]   (aT) arc (0:-180:1cm and 0.5cm) arc (-180:0:1cm);
  \draw[thick]  (aT) arc (0:-180:1cm and 0.5cm) node[midway, below] {$\scriptstyle{a}$};
  \fill[draw = black, rounded corners=1pt, thick, fill opacity = 0.3, fill = red] (aT) arc (0:180:1cm and 0.5cm) arc (-180:0:1cm);
  \filldraw[draw= black, dotted,  rounded corners=1pt, thick, fill opacity = 0.3, fill = red] (aB) arc (0:180: 1cm and 0.5cm) -- (b1) arc (180:0: 1cm and 0.5cm) -- (aB);
  \filldraw[draw= black, rounded corners=1pt, thick, fill opacity = 0.3, fill = yellow] (a1) arc (0:180: 1cm and 0.5cm) -- (b2) arc (180:0: 1cm) -- (a1);
  \fill[blue, opacity = 0.3]  (M1) ellipse (1cm and 0.5cm) node[opacity=1, black] {$\scriptstyle{N}$};
  \filldraw[pattern=north west lines, thick] (M1) ellipse (1cm and 0.5cm);
  \filldraw[draw= black, rounded corners=1pt, thick, fill opacity = 0.3, fill = red] (aB) arc (0:-180: 1cm and 0.5cm) node[midway, above, opacity=1] {$\scriptstyle{a}$} -- (b1) arc (-180:0: 1cm and 0.5cm) -- (aB);
  \filldraw[draw= black, rounded corners=1pt, thick, fill opacity = 0.3, fill = yellow] (a1) arc (0:-180: 1cm and 0.5cm) -- (b2) arc (180:0: 1cm)node[opacity = 1, black,midway, below] {$\scriptstyle{N-a}$} -- (a1);
  \draw[thick, ->-] (a1) arc (0: -180: 1cm and 0.5cm);
  \draw[thick, ->-] (aT) arc (0: -180: 1cm and 0.5cm);
  \draw[thick, ->-] (aB) arc (0: -180: 1cm and 0.5cm);
\end{scope}
\draw[thick, red, <-] (am) -- +(0.5,0) node[right] {$\pi_{\widehat{\alpha}}$};
\draw[thick, red, <-] (aT) -- +(0.5,0) node[right] {$\pi_{{\alpha}}$};

%% file: cef_dotmigration.tex
\tdplotsetmaincoords{50}{10}
\begin{scope}[tdplot_main_coords]
  \coordinate (aT) at (-1, 1, 1);
  \coordinate (bT) at (-1,-1, 1);
  \coordinate (cT) at ( 1, 0, 1);
  \coordinate (oT) at ( 0, 0, 1);
  \coordinate (aB) at (-1, 1,-1);
  \coordinate (bB) at (-1,-1,-1);
  \coordinate (cB) at ( 1, 0,-1);
  \coordinate (oB) at ( 0, 0,-1);
  \filldraw[thick, draw= black, fill = green, fill opacity =0.3] (aT)-- (oT) node[pos=0.2, below, opacity =1] {$\scriptstyle{a}$} -- (oB) -- (aB) -- (aT);
  \filldraw[thick, draw= black, fill = blue,  fill opacity =0.3] (bT)-- (oT) -- (oB) -- (bB) node[pos = 0.8, above, opacity =1] {$\scriptstyle{b}$} -- (bT);
  \filldraw[thick, draw= black, fill = red,   fill opacity =0.3] (cT)-- (oT) node[midway, below, opacity =1, rotate =-5] {$\scriptstyle{a+b}$} -- (oB) -- (cB) -- (cT);
\end{scope}
  \draw[red, thick, <-] ($(oB)!0.5!(cB)$) -- +(0, -0.55) node[below, red] {$\pi_{\gamma}$};

%% file: cef_dotmigration2.tex
\tdplotsetmaincoords{50}{10}
\begin{scope}[tdplot_main_coords]
  \coordinate (aT) at (-1, 1, 1);
  \coordinate (bT) at (-1,-1, 1);
  \coordinate (cT) at ( 1, 0, 1);
  \coordinate (oT) at ( 0, 0, 1);
  \coordinate (aB) at (-1, 1,-1);
  \coordinate (bB) at (-1,-1,-1);
  \coordinate (cB) at ( 1, 0,-1);
  \coordinate (oB) at ( 0, 0,-1);
  \filldraw[thick, draw= black, fill = green, fill opacity =0.3] (aT)-- (oT) node[pos=0.2, below, opacity =1] {$\scriptstyle{a}$} -- (oB) -- (aB) -- (aT);
  \filldraw[thick, draw= black, fill = blue,  fill opacity =0.3] (bT)-- (oT) -- (oB) -- (bB) node[pos = 0.8, above, opacity =1] {$\scriptstyle{b}$} -- (bT);
  \filldraw[thick, draw= black, fill = red,   fill opacity =0.3] (cT)-- (oT) node[midway, below, opacity =1, rotate =-5] {$\scriptstyle{a+b}$} -- (oB) -- (cB) -- (cT);
\end{scope}
\draw[red, thick, <-] ($(oT)!0.5!(aT)$) -- +(0.7, 0) node[right, red] {$\pi_{\alpha}$};
\draw[red, thick, <-] ($(oB)!0.5!(bB)$) -- +(0, -0.3) node[below, red] {$\pi_{\beta}$};

%% file: cef_digonfoam.tex
\tdplotsetmaincoords{60}{140}
\begin{scope}[tdplot_main_coords]
  \coordinate (aT) at (-1, 0, 2);
  \coordinate (bT) at (+1, 0, 2);
  \coordinate (AT) at (-2.5, 0, 2);
  \coordinate (BT) at (+2.5, 0, 2);
  \coordinate (cT) at (0, +1, 2);
  \coordinate (dT) at (0, -1, 2);
  \coordinate (aB) at (-1, 0, -2);
  \coordinate (bB) at (+1, 0, -2);
  \coordinate (AB) at (-2.5, 0, -2);
  \coordinate (BB) at (+2.5, 0, -2);
  \coordinate (cB) at (0, +1, -2);
  \coordinate (dB) at (0, -1, -2);
  \coordinate (aM) at (-1, 0, 0);
  \coordinate (bM) at (+1, 0, 0);
\draw[thick, -<-] (AT) -- (aT);
\draw[thick, -<-] (bT) -- (BT);
\draw[thick, -<-] (aT) .. controls (dT) .. (bT);
\draw[thick, -<-] (aT) .. controls (cT) .. (bT);
\draw[thick, -<-] (AB) -- (aB);
\draw[thick, -<-] (bB) -- (BB);
\draw[thick, -<-] (aB) .. controls (dB) .. (bB);
\draw[thick, -<-] (aB) .. controls (cB) .. (bB);
\draw[thick, -<-] (aB) -- (aT);
\draw[thick, -<-] (bT) -- (bB);
  \filldraw[draw = black, rounded corners=1pt, thick, fill opacity = 0.3, fill = red]    (AT) -- (aT) node[midway, opacity =1, below, sloped] {$\scriptstyle{a+b}$} -- (aB) -- (AB) -- (AT);
  \filldraw[draw = black, rounded corners=1pt, thick, fill opacity = 0.3, fill = red]    (bT) -- (BT) node[midway, opacity =1, below, sloped] {$\scriptstyle{a+b}$} -- (BB) -- (bB) -- (bT);
  \filldraw[draw = black, rounded corners=1pt, thick, fill opacity = 0.3, fill = blue]   (bT) .. controls (dT) .. (aT) node[near end, below, opacity=1] {$\scriptstyle{a}$} -- (aB) .. controls (dB) .. (bB)-- (bT);
  \filldraw[draw = black, rounded corners=1pt, thick, fill opacity = 0.3, fill = green]   (bT) .. controls (cT) .. (aT) -- (aB)  .. controls (cB) .. (bB) node[near end, above, opacity=1] {$\scriptstyle{b}$} -- (bT);

\end{scope}

%% file: cef_digonfoam2.tex
\tdplotsetmaincoords{60}{140}
\begin{scope}[tdplot_main_coords]
  \coordinate (aT) at (-1, 0, 2);
  \coordinate (bT) at (+1, 0, 2);
  \coordinate (AT) at (-2.5, 0, 2);
  \coordinate (BT) at (+2.5, 0, 2);
  \coordinate (cT) at (0, +1, 2);
  \coordinate (dT) at (0, -1, 2);
  \coordinate (aB) at (-1, 0, -2);
  \coordinate (bB) at (+1, 0, -2);
  \coordinate (AB) at (-2.5, 0, -2);
  \coordinate (BB) at (+2.5, 0, -2);
  \coordinate (cB) at (0, +1, -2);
  \coordinate (dB) at (0, -1, -2);
  \coordinate (aM) at (-1, 0, 0);
  \coordinate (bM) at (+1, 0, 0);
\draw[thick, -<-] (AT) -- (aT);
\draw[thick, -<-] (bT) -- (BT);
\draw[thick, -<-] (aT) .. controls (dT) .. (bT);
\draw[thick, -<-] (aT) .. controls (cT) .. (bT);
\draw[thick, -<-] (AB) -- (aB);
\draw[thick, -<-] (bB) -- (BB);
\draw[thick, -<-] (aB) .. controls (dB) .. (bB);
\draw[thick, -<-] (aB) .. controls (cB) .. (bB);
\draw[thick, -<-] (bT) .. controls (bM) and (aM) .. (aT);
\draw[thick, -<-] (aB) .. controls (aM) and (bM) .. (bB);
  \filldraw[draw = black, rounded corners=1pt, thick, fill opacity = 0.3, fill = red]    (AT) -- (aT)  .. controls (aM) and (bM) .. (bT) -- (BT) -- (BB) node[midway, opacity =1, right, sloped, rotate= 112 ] {$\scriptstyle{a+b}$} -- (bB) .. controls (bM) and (aM) .. (aB) -- (AB) -- (AT);
  \filldraw[draw = black, rounded corners=1pt, thick, fill opacity = 0.3, fill = blue]    (aT) .. controls (aM) and (bM) .. (bT) .. controls (dT) .. (aT) node [opacity=1, near end, below] {$\scriptstyle{a}$} coordinate[pos=0.3] (alpha); 
  \filldraw[draw = black, rounded corners=1pt, thick, fill opacity = 0.3, fill = green]   (aT) .. controls (aM) and (bM) .. (bT) .. controls (cT) .. (aT) node [opacity=1, near start, below] {$\scriptstyle{b}$} ;
  \filldraw[draw = black, rounded corners=1pt, thick, fill opacity = 0.3, fill = blue]    (aB) .. controls (aM) and (bM) .. (bB) .. controls (dB) .. (aB) node [opacity=1, near end, above] {$\scriptstyle{a}$};
  \filldraw[draw = black, rounded corners=1pt, thick, fill opacity = 0.3, fill = green]   (aB) .. controls (aM) and (bM) .. (bB) .. controls (cB) .. (aB) node [opacity=1, near start, above]  {$\scriptstyle{b}$} coordinate[pos=0.7] (alphah);
\end{scope}
  \draw[thick, red, <-] (alphah) -- +(0.5,0) node[right, red] {$\pi_{\widehat{\alpha}}$};
  \draw[thick, red, <-] (alpha) -- +(-0.5,0) node[left, red] {$\pi_\alpha$};

%% file: cef_digonfoamDUR.tex
\tdplotsetmaincoords{60}{140}
\begin{scope}[tdplot_main_coords]
  \coordinate (aT) at (-1, 0, 2);
  \coordinate (bT) at (+1, 0, 2);
  \coordinate (AT) at (-2.5, 0, 2);
  \coordinate (BT) at (+2.5, 0, 2);
  \coordinate (cT) at (0, +1, 2);
  \coordinate (dT) at (0, -1, 2);
  \coordinate (aB) at (-1, 0, -2);
  \coordinate (bB) at (+1, 0, -2);
  \coordinate (AB) at (-2.5, 0, -2);
  \coordinate (BB) at (+2.5, 0, -2);
  \coordinate (cB) at (0, +1, -2);
  \coordinate (dB) at (0, -1, -2);
  \coordinate (aM) at (-1, 0, 0);
  \coordinate (bM) at (+1, 0, 0);
\draw[thick, -<-] (AT) -- (aT);
\draw[thick, -<-] (bT) -- (BT);
\draw[thick, ->-] (aT) .. controls (dT) .. (bT);
\draw[thick, -<-] (aT) .. controls (cT) .. (bT);
\draw[thick, -<-] (AB) -- (aB);
\draw[thick, -<-] (bB) -- (BB);
\draw[thick, ->-] (aB) .. controls (dB) .. (bB);
\draw[thick, -<-] (aB) .. controls (cB) .. (bB);
\draw[thick, -<-] (aT) -- (aB);
\draw[thick, -<-] (bB) -- (bT);

  \filldraw[draw = black, rounded corners=1pt, thick, fill opacity = 0.3, fill = red]    (AT) -- (aT) node[midway, opacity =1, below, sloped] {$\scriptstyle{a}$} -- (aB) -- (AB) -- (AT);
  \filldraw[draw = black, rounded corners=1pt, thick, fill opacity = 0.3, fill = red]    (bT) -- (BT) node[midway, opacity =1, below, sloped] {$\scriptstyle{a}$} -- (BB) -- (bB) -- (bT);
  \filldraw[draw = black, rounded corners=1pt, thick, fill opacity = 0.3, fill = blue]   (bT) .. controls (dT) .. (aT) node[near end, below, opacity=1] {$\scriptstyle{b}$} -- (aB) .. controls (dB) .. (bB)-- (bT);
  \filldraw[draw = black, rounded corners=1pt, thick, fill opacity = 0.3, fill = green]   (bT) .. controls (cT) .. (aT) -- (aB)  .. controls (cB) .. (bB) node[near end, above, opacity=1] {$\scriptstyle{a+b}$} -- (bT);
\end{scope}

%% file: cef_digonfoamDUR2.tex
\tdplotsetmaincoords{60}{140}
\begin{scope}[tdplot_main_coords]
  \coordinate (aT) at (-1, 0, 2);
  \coordinate (bT) at (+1, 0, 2);
  \coordinate (AT) at (-2.5, 0, 2);
  \coordinate (BT) at (+2.5, 0, 2);
  \coordinate (cT) at (0, +1, 2);
  \coordinate (dT) at (0, -1, 2);
  \coordinate (aB) at (-1, 0, -2);
  \coordinate (bB) at (+1, 0, -2);
  \coordinate (AB) at (-2.5, 0, -2);
  \coordinate (BB) at (+2.5, 0, -2);
  \coordinate (cB) at (0, +1, -2);
  \coordinate (dB) at (0, -1, -2);
  \coordinate (aM) at (-1, 0, 0);
  \coordinate (bM) at (+1, 0, 0);
\draw[thick, -<-] (AT) -- (aT);
\draw[thick, -<-] (bT) -- (BT);
\draw[thick, ->-] (aT) .. controls (dT) .. (bT);
\draw[thick, -<-] (aT) .. controls (cT) .. (bT);
\draw[thick, -<-] (AB) -- (aB);
\draw[thick, -<-] (bB) -- (BB);
\draw[thick, ->-] (aB) .. controls (dB) .. (bB);
\draw[thick, -<-] (aB) .. controls (cB) .. (bB);
\draw[thick, -<-] (bB) .. controls (bM) and (aM) .. (aB);
\draw[thick, -<-] (aT) .. controls (aM) and (bM) .. (bT);
  \filldraw[draw = black, rounded corners=1pt, thick, fill opacity = 0.3, fill = red]    (AT) -- (aT)  .. controls (aM) and (bM) .. (bT) -- (BT) -- (BB) node[midway, opacity =1, right, sloped, rotate= 112 ] {$\scriptstyle{a}$} -- (bB) .. controls (bM) and (aM) .. (aB) -- (AB) -- (AT);
  \filldraw[draw = black, rounded corners=1pt, thick, fill opacity = 0.3, fill = blue]    (aT) .. controls (aM) and (bM) .. (bT) .. controls (dT) .. (aT) node [opacity=1, near end, below] {$\scriptstyle{b}$} coordinate[pos=0.3] (alpha); 
  \filldraw[draw = black, rounded corners=1pt, thick, fill opacity = 0.3, fill = green]   (aT) .. controls (aM) and (bM) .. (bT) .. controls (cT) .. (aT) node [opacity=1, near start, below] {$\scriptstyle{a+b}$} ;
  \filldraw[draw = black, rounded corners=1pt, thick, fill opacity = 0.3, fill = blue]    (aB) .. controls (aM) and (bM) .. (bB) .. controls (dB) .. (aB) node [opacity=1, near end, above] {$\scriptstyle{b}$};
  \filldraw[draw = black, rounded corners=1pt, thick, fill opacity = 0.3, fill = green]   (aB) .. controls (aM) and (bM) .. (bB) .. controls (cB) .. (aB) node [opacity=1, near start, above]  {$\scriptstyle{a+b}$} coordinate[pos=0.7] (alphah);
\end{scope}
\filldraw[thick, black, pattern = north west lines ]  ($(alphah) + (-0.2,0.5)$) ellipse (0.1 and 0.3);
\draw[thick, black, <- ]  ($(alphah) + (-0.3,0.5)$)  arc (180:-180: 0.1 and 0.3);
\filldraw[thick, draw= black, fill = yellow, fill opacity =0.3] ($(alphah) + (-0.2, 0.8)$) -- +(1.5,0) node[midway, below, opacity = 1, black] {$\scriptstyle{N-a-b}$}  arc (90: -90:0.3) -- +(-1.5, 0) coordinate[midway] (aaa) {} arc (-90:90:0.1 and 0.3); 
  \filldraw[thick, draw= black, fill = yellow, fill opacity =0.3] ($(alphah) + (-0.2, 0.8)$) -- +(1.5,0) node[midway, below, opacity = 1, black] {$\scriptstyle{N-a-b}$}  arc (90: -90:0.3) -- +(-1.5, 0) coordinate[midway] (aaa) {} arc (270:90:0.1 and 0.3); 

  \draw[thick, red, <-] (aaa) -- +(0,-0.5) node[below, red] {$\pi_{\widehat{\alpha}}$};

  \draw[thick, red, <-] (alpha) -- +(-0.5,0) node[left, red] {$\pi_{\alpha}$};

%% file: cef_jointrsfoam.tex
\tdplotsetmaincoords{75}{60}
\begin{scope}[tdplot_main_coords]
  \coordinate (aT) at (-1, -1, 3);
  \coordinate (bT) at (-1, 1, 3);
  \coordinate (cT) at (1, 1, 3);
  \coordinate (dT) at (1, -1, 3);
  \coordinate (AT) at (-1, -2, 3);
  \coordinate (BT) at (-1, 2, 3);
  \coordinate (CT) at (1, 2, 3);
  \coordinate (DT) at (1, -2, 3);
  \coordinate (abMt) at (-1, 0, 1);
  \coordinate (abMb) at (-1, 0, -1);
  \coordinate (cdMt) at (1, 0, 1);
  \coordinate (cdMb) at (1, 0, -1);
  \coordinate (O) at (0,0,0);
  \coordinate (aB) at (-1, -1, -3);
  \coordinate (bB) at (-1, 1, -3);
  \coordinate (cB) at (1, 1, -3);
  \coordinate (dB) at (1, -1, -3);
  \coordinate (AB) at (-1, -2, -3);
  \coordinate (BB) at (-1, 2, -3);
  \coordinate (CB) at (1, 2, -3);
  \coordinate (DB) at (1, -2, -3);
\draw[thick, ->-] (AT) -- (aT);
\draw[thick, ->-] (DT) -- (dT);
\draw[thick, ->-] (cT) -- (CT);
\draw[thick, ->-] (bT) -- (BT);
\draw[thick, ->-] (aT) -- (bT);
\draw[thick, ->-] (dT) -- (cT);
\draw[thick, ->-] (aT) -- (dT);
\draw[thick, ->-] (bT) -- (cT);
\draw[thick, ->-] (AB) -- (aB);
\draw[thick, ->-] (DB) -- (dB);
\draw[thick, ->-] (cB) -- (CB);
\draw[thick, ->-] (bB) -- (BB);
\draw[thick, ->-] (aB) -- (bB);
\draw[thick, ->-] (dB) -- (cB);
\draw[thick, ->-] (aB) -- (dB);
\draw[thick, ->-] (bB) -- (cB);

\draw[thick, ->-] (aB) -- (aT);
\draw[thick, ->-] (cT) -- (cB);
\draw[thick, ->-] (dT) -- (dB);
\draw[thick, ->-] (bB) -- (bT);

  \filldraw[draw = black, rounded corners=1pt, thick, fill opacity = 0.3, fill = red]     (AT) -- (aT) -- (aB) -- (AB) -- (AT) node[pos = 0.5, right, opacity =1, rotate =7] {$\scriptstyle{k+s}$} ;
  \filldraw[draw = black, rounded corners=1pt, thick, fill opacity = 0.3, fill = orange]  (BT) -- (bT) -- (bB) -- (BB) -- (BT)node[pos = 0.05, left, opacity =1, rotate =7] {$\scriptstyle{k-r}$} ;
  \filldraw[draw = black, rounded corners=1pt, thick, fill opacity = 0.3, fill = yellow]  (aT) -- (bT) node[midway, below, opacity =1] {$\scriptstyle{k}$} -- (bB) -- (aB) -- (aT);
  \filldraw[draw = black, rounded corners=1pt, thick, fill opacity = 0.3, fill = green]  (cT) -- (bT) node[midway, below, opacity =1] {$\scriptstyle{r}$} -- (bB) -- (cB) -- (cT);
  \filldraw[draw = black, rounded corners=1pt, thick, fill opacity = 0.3, fill = blue]   (dT) -- (aT) node[midway, below, opacity =1] {$\scriptstyle{s}$} -- (aB) -- (dB) -- (dT);
  \filldraw[draw = black, rounded corners=1pt, thick, fill opacity = 0.3, fill = yellow]  (cT) -- (dT) -- (dB) -- (cB) node[midway, above, opacity =1] {$\scriptstyle{l}$} --  (cT); 
  \filldraw[draw = black, rounded corners=1pt, thick, fill opacity = 0.3, fill = orange]  (DT) -- (dT) -- (dB) -- (DB) -- (DT) node[pos = 0.95, right, opacity =1, rotate =7] {$\scriptstyle{l-s}$} ;
  \filldraw[draw = black, rounded corners=1pt, thick, fill opacity = 0.3, fill = red]     (CT) -- (cT) -- (cB) -- (CB) -- (CT) node[pos = 0.5, left, opacity =1, rotate =7] {$\scriptstyle{l+r}$} ; 
\end{scope}

%% file: cef_jointrsfoam2.tex
\tdplotsetmaincoords{75}{60}
\begin{scope}[tdplot_main_coords]
  \coordinate (aT) at (-1, -1, 3);
  \coordinate (bT) at (-1, 1, 3);
  \coordinate (cT) at (1, 1, 3);
  \coordinate (dT) at (1, -1, 3);
  \coordinate (AT) at (-1, -2, 3);
  \coordinate (BT) at (-1, 2, 3);
  \coordinate (CT) at (1, 2, 3);
  \coordinate (DT) at (1, -2, 3);
  \coordinate (abMt) at (-1, 0, 1);
  \coordinate (abMb) at (-1, 0, -1);
  \coordinate (cdMt) at (1, 0, 1);
  \coordinate (cdMb) at (1, 0, -1);
  \coordinate (O) at (0,0,0);
  \coordinate (aB) at (-1, -1, -3);
  \coordinate (bB) at (-1, 1, -3);
  \coordinate (cB) at (1, 1, -3);
  \coordinate (dB) at (1, -1, -3);
  \coordinate (AB) at (-1, -2, -3);
  \coordinate (BB) at (-1, 2, -3);
  \coordinate (CB) at (1, 2, -3);
  \coordinate (DB) at (1, -2, -3);
\draw[thick, ->-] (AT) -- (aT);
\draw[thick, ->-] (DT) -- (dT);
\draw[thick, ->-] (cT) -- (CT);
\draw[thick, ->-] (bT) -- (BT);
\draw[thick, ->-] (aT) -- (bT);
\draw[thick, ->-] (dT) -- (cT);
\draw[thick, ->-] (aT) -- (dT);
\draw[thick, ->-] (bT) -- (cT);
\draw[thick, ->-] (AB) -- (aB);
\draw[thick, ->-] (DB) -- (dB);
\draw[thick, ->-] (cB) -- (CB);
\draw[thick, ->-] (bB) -- (BB);
\draw[thick, ->-] (aB) -- (bB);
\draw[thick, ->-] (dB) -- (cB);
\draw[thick, ->-] (aB) -- (dB);
\draw[thick, ->-] (bB) -- (cB);

\draw[thick, ->-] (abMb) -- (cdMb);
\draw[thick, ->-] (cdMt) -- (abMt);

\draw[thick, ->-] (abMb) -- (abMt);
\draw[thick, ->-] (cdMt) -- (cdMb);

\draw[thick, ->-] (aB) -- (abMb);
\draw[thick, ->-] (cT) -- (cdMt);
\draw[thick, ->-] (dT) -- (cdMt);
\draw[thick, ->-] (bB) -- (abMb);

\draw[thick, ->-] (abMt) -- (aT);
\draw[thick, ->-] (cdMb) -- (cB);
\draw[thick, ->-] (cdMb) -- (dB);
\draw[thick, ->-] (abMt) -- (bT);

  \filldraw[draw = black, rounded corners=1pt, thick, fill opacity = 0.3, fill = red]     (AT) -- (aT) -- (abMt) -- (abMb) -- (aB) -- (AB) -- (AT) node[pos = 0.5, right, opacity =1, rotate =7] {$\scriptstyle{k+s}$} ;
  \filldraw[draw = black, rounded corners=1pt, thick, fill opacity = 0.3, fill = orange]  (BT) -- (bT) -- (abMt) -- (abMb) -- (bB) -- (BB) -- (BT)node[pos = 0.05, left, opacity =1, rotate =7] {$\scriptstyle{k-r}$} ;
  \filldraw[draw = black, rounded corners=1pt, thick, fill opacity = 0.3, fill = yellow]  (aB) -- (bB) node[midway, above, opacity =1] {$\scriptstyle{k}$} -- (abMb) -- (aB);
  \filldraw[draw = black, rounded corners=1pt, thick, fill opacity = 0.3, fill = yellow]  (aT) -- (bT) node[midway, below, opacity =1] {$\scriptstyle{k}$} -- (abMt) -- (aT); 
  \draw [red, thick, <-] ($(bT)!0.40!(cdMt)$) -- +(0,0,1.5) node[above] {$\pi_{\alpha}$};
  \draw [red, thick, <-] ($(dB)!0.40!(abMb)$) -- +(0,0,-2) node[below] {$\pi_{\widehat{\alpha}}$};
  \filldraw[draw = black, rounded corners=1pt, thick, fill opacity = 0.3, fill = green]  (bT) -- (abMt) -- (cdMt) -- (cT) -- (bT) node[midway, below, opacity =1] {$\scriptstyle{r}$} ;
  \filldraw[draw = black, rounded corners=1pt, thick, fill opacity = 0.3, fill = green]  (bB) -- (abMb) -- (cdMb) -- (cB) -- (bB) node[midway, above, opacity =1] {$\scriptstyle{r}$} ;
  \filldraw[draw = black, rounded corners=1pt, thick, fill opacity = 0.3, fill = blue]   (aT) -- (abMt) -- (cdMt) -- (dT) -- (aT) node[midway, below, opacity =1] {$\scriptstyle{s}$} ;
  \filldraw[draw = black, rounded corners=1pt, thick, fill opacity = 0.3, fill = blue]   (aB) -- (abMb) -- (cdMb) -- (dB) -- (aB) node[midway, above, opacity =1] {$\scriptstyle{s}$} ;
  \filldraw[draw = black, rounded corners=1pt, thick, fill opacity = 0.3, fill = gray]   (abMb) -- (cdMb) -- (cdMt) -- (abMt)-- (abMb);
  \node[rotate = -15] at (O) {$\scriptstyle{r+s}$};
  \filldraw[draw = black, rounded corners=1pt, thick, fill opacity = 0.3, fill = yellow]  (cB) -- (dB) node[midway, above, opacity =1] {$\scriptstyle{l}$}-- (cdMb) -- (cB); 
  \filldraw[draw = black, rounded corners=1pt, thick, fill opacity = 0.3, fill = yellow]  (cT) -- (dT) node[midway, below, opacity =1] {$\scriptstyle{l}$}-- (cdMt) -- (cT); 
  \filldraw[draw = black, rounded corners=1pt, thick, fill opacity = 0.3, fill = orange]  (DT) -- (dT) -- (cdMt) -- (cdMb) -- (dB) -- (DB) -- (DT) node[pos = 0.95, right, opacity =1, rotate =7] {$\scriptstyle{l-s}$} ;
  \filldraw[draw = black, rounded corners=1pt, thick, fill opacity = 0.3, fill = red]     (CT) -- (cT) -- (cdMt) -- (cdMb) -- (cB) -- (CB) -- (CT) node[pos = 0.5, left, opacity =1, rotate =7] {$\scriptstyle{l+r}$} ; 
\end{scope}

%% file: appandconj.tex
\subsection{Evaluation of the variables}
\label{sec:relat-with-appr}

\input{relRW}

\subsection{Equivariant cohomology of (partial) flag varieties}
\label{sec:equiv-cohom-part}

\input{equivariantFM}

\subsection{Further possible developments}
\label{sec:furth-poss-devel}

\begin{itemize}
\item The foam evaluation relies on the evaluation of colored foams and as explained in Section~\ref{sec:sub-sl_k-foams}, it relates on how $\sll_2$-sub-foams are embedded in a colored foam. Since the classification of simple Lie relies basically on how different copies of $\sll_2$ interact with one another, it seems possible to adapt our strategy to categorify the graphical calculus for other simple Lie algebras and eventually to have a categorification of other quantum link invariants. In the introduction, we suggested that  pigments can be thought of as a simple $\mathfrak{gl}_N$-roots, and that bichrome surfaces may correspond to positive $\sll_N$-roots. This idea might be taken as a guideline in order to define the evaluation of (colored) foams for other Lie algebra types.
\item 
The 2-functor extension discussed in Subsection~\ref{sec:relat-with-appr}
would yield a new definition of the $\sll_N$-arc algebras \cite{MR3198835}. Arc algebras are of great interest since they are closely related to the dual canonical bases of some spaces. They have been investigated in various cases by Ehrig, Tubbenhauer and their coauthors \cite{2016arXiv161107444, 2016arXiv160108010, 2015arXiv151004884}. This new definition could be enlightening, since it could allow some combinatorial and geometric constructions similar to the ones in \cite{LHR2} and might eventually give a way to compute these dual-canonical bases.
\item Since we do not require $\sll_N$-MOY-graph to be ladders, the graphical presentation of the $\sll_N$-homology we have is simpler than the one given \cite{queffelec2014mathfrak}. One can therefore hope for a definition of a ``fully colored\footnote{That is an homology theory which would categorify the $\sll_N$-invariant of links colored by arbitrary finite dimensional representations, not only the minuscule ones.}'' $\sll_N$-homology with help of some resolutions of simple $\sll_N$-representations in terms of tensor powers of minuscules modules. See \cite{MR2124557} for the $\sll_2$-case and \cite{2015arXiv150308451R} for the $\sll_3$-case.
\end{itemize}

%% file: relRW.tex
\emph{Generalities on the evaluation.}
In this subsection, we inspect the behavior of the theory if one evaluates the variables $X_1,\dots, X_N$ on complex numbers $z_1, \dots, z_N:= \mathbf{z}$ ($\mathbf{z}$ is thought of as a multi-set). This is closely related to the work by Rose and Wedrich \cite{2015arXiv150102567R}.

The first thing to observe is that if the variables $z_1, \dots, z_N$ are not 2 by 2 distinct, the evaluation of colored foams make no sense. However, we can still evaluate $\kup{\bullet}$. 
Another big difference is that unless $z_1 = \dots =z_N =0$, the theory is not graded anymore but rather filtered. 

Using this new evaluation and the universal construction one can construct a functor $\Fz\colon \FoamN \to \CC$-$\mathsf{vect}_{\mathrm{filt}}$. The categories $\PolN$ and $\CC$-$\mathsf{vect}_{\mathrm{filt}}$ are related by the functor 
\[
  \begin{array}{crcl}
    \Tz \colon & \PolN& \to &\CC\textrm{-}\mathsf{vect}_{\mathrm{filt}} \\
&    M\in \mathrm{ob}(\PolN)& \mapsto & \CC \otimes_{\ZZ[X_1, X_2, \dots, X_N]} M \\ 
&    f\in \mathrm{HOM}(M,N) &\mapsto& f_{X_i \to z_i}, \\
  \end{array}
\]
where $\CC$  is seen as $\ZZ[X_1, X_2, \dots, X_N]$-module by evaluating $X_i$ on $z_i$.

\begin{prop}\label{prop:categoricalnonsense}
  The functors $\Fz$ and $\Tz \circ \F$ are isomorphic as monoidal functors.
\end{prop}

\begin{proof}
Since $\F$ and $\Fz$ are defined by universal construction, for every $\sll_N$-MOY-graph $\Gamma$, the elements of $\Fz(\Gamma)$ and $\F(\Gamma)$ can be thought of as (linear combinations of) foams. Furthermore, if a linear combination of foams in $\F(\Gamma)$ is equal to $0$ then the same linear combination (with the variable $X_\bullet$ evaluated in $\mathbf{z}$) is equal to $0$ in $\Fz(\Gamma)$. This correspondence defines a natural transformation $\tau_1\colon:   \Tz\circ \F \to \Fz$. To define $\tau_2$ the analogue (and inverse) natural transformation from $\Fz$ to $\Tz\circ \F$, we need to know that if a $\CC$-linear combination of foam $f$ is equal to $0$ in $\Fz(\Gamma)$, then the same linear combination is equal to $0$ in $\CC\otimes _{X_i \to z_i} \F(\Gamma)$. In $\F(\Gamma)$, we can find a basis $(f_i)_{i \in I}$ and a dual basis $(f_i^{\star})_{i\in I}$  (see Remark~\ref{rmk:orthbase}), we write $f = \sum P_i(X_1, \dots, X_N)f_i$ for some polynomial $P_\bullet$ in $\CC[X_1, \dots X_N]$. Evaluating $f_i^*$ on $f$, we deduce that $P_i(z_1, \dots, z_N)=0$ for any $i$ in $I$. This implies that $f$ is equal to $0$ in $\CC\otimes _{X_i \to z_i} \F(\Gamma)$. 
\end{proof}

\emph{The graded case.}
In this paragraph we are interested in the functor $\F_{\mathbf{0}}$. This still makes sense over $\ZZ$, hence we consider $\F_{\mathbf{0}}$ to be a functor from $\FoamN$ to $\ZZ$-$\mathsf{mod}_{\mathrm{gr}}$.
In this paper we were only concerned with closed $\sll_N$-MOY-graphs. The functor $\F_{\mathbf{0}}$ (just like the functor $\F$) could be easily promoted to a 2-functor from the 2-category of $\sll_N$-foams with corners to the $2$-category of graded $\ZZ$-algebras  (see \cite{MR1928174} and \cite{MR2174270} for the $\sll_2$-case and \cite{2012arXiv1206.2118M} and \cite{LHRThese} for the $\sll_3$-case and \cite{MR3198835} for the $\sll_N$-case). In \cite{queffelec2014mathfrak}, Queffelec and Rose construct a 2-functor $\FQR$ from the 2-category $N$-$\mathsf{Ladder}$ of $N$-ladders (which is a non-full sub-$2$-category of the $2$-category of $\sll_N$-foams with corners) to the $2$-category of graded $\ZZ$-algebras. 

\begin{prop}\label{prop:Z0QR}
  The functors $\FQR$ and the functor $\F_{\mathbf{0}}$ restricted to $N$-$\mathsf{Ladder}$ are isomorphic.
\end{prop}

\begin{proof}[Sketch of the proof]
The only things to prove for this proposition is that our evaluation of foams fulfills the local relations (3.8)-(3.20) in \cite{queffelec2014mathfrak}:
  \begin{itemize}
  \item Relation (3.8) follows from Lemma~\ref{lem-col-matveev-pierigliani},
  \item Relation (3.9) is relation (\ref{eq:dot-migration}) in Proposition~\ref{prop:additionalrelation},
  \item Relation (3.10) follows from relations~(\ref{eq:theta-ev}) and (\ref{eq:neckcutting-ev}) in Proposition~\ref{prop:additionalrelation},
  \item Relation (3.11) follows from relation~(\ref{eq:digon}) in Proposition~\ref{prop:additionalrelation},
  \item Relation (3.12)  is obvious,
  \item Relations (3.15) and (3.16) are special cases from Proposition~\ref{prop:complicatedfoam},
  \item Relations (3.17) -- (3.20) are obvious.
  \end{itemize}
Hence, only relations (3.13) and (3.14) are to be checked. This is an easy\footnote{This is \emph{really} easy since colorings on both sides are in one-one correspondence.} computation.
\end{proof}

\begin{cor}\label{cor:slnhomology}  
  The functor $\F$ together with the classical \emph{hypercube of resolution} technology developed by Bar-Natan in \cite{MR1917056} (see as well \cite[Section 4.2]{queffelec2014mathfrak} for the $\sll_N$ case) gives rise to an homology theory which categorifies the $\sll_n$-link invariant.
\end{cor}

\emph{The other cases.}
We now consider the general case and, just as in \cite{2015arXiv150102567R}, we ignore the filtrations. We will need to work with $\F$ for different values of $N$, so that we write $\F^N$ instead of $\F$.

\begin{lem}\label{lem:evallsamevalue}
  If $\mathbf{z} = {z_1^N}$ and we forget about the filtration (or graduation), the functors $\F^N_{\mathbf{0}}$ and $\F^N_{\mathbf{z}}$ are isomorphic.
\end{lem}
\begin{proof}
  This simply comes from a change of variables: $X_i \mapsto X_i -z_1$ in the evaluation of foams. Note that it only affects the $P$-part of the evaluation and that it does not respect the graduation.
\end{proof}

\begin{dfn}\label{dfn:symlagrangeidempotent}
  Let $y_1, \dots y_k$ be $k$ 2 by 2 distinct complex numbers, for $i=1,\dots, k$ the \emph{Lagrange polynomial $L^{y_i}_{y_1,\dots, y_k}$} is the one-variable polynomial given by:
\[
L^{y_i}_{y_1,\dots, y_k}(X) \eqdef \prod_{\substack{j=1\\ j \neq i}}^k \frac{X-y_j}{y_i - y_j}.
\]
It satisfies  $L^{y_i}_{y_1,\dots, y_k}(y_j) = \delta_{ij}$ for every $j = 1, \dots, k$. Furthermore, we have:
\[\sum _{i=1}^k L^{y_i}_{y_1,\dots, y_k} =1.
\]
Let $\underline{x}= (x_1, \dots, x_a)$ be an $a$-tuple of elements of $\{y_1,\dots, y_k\}$. The \emph{Lagrange polynomial $L^{\underline{x}}_{y_1,\dots, y_k}$} is the $a$-variables polynomial given by:
 \[L^{\underline{x}}_{y_1, \dots, y_k}(X_1, \dots, X_a)\eqdef 
\prod_{j=1}^a L^{x_j}_{y_1, \dots, y_k}(X_j).
\]
It satisfies $L^{\underline{x}}_{y_1, \dots, y_k} (\underline{t}) = \delta_{\underline{x}\underline{t}}$ for every $a$-tuple $\underline{t}$ of elements of $\mathbf{y}$. Furthermore, we have:
\[\sum _{\underline{x}\in {{y_1, \dots y_k}^a}} L^{\underline{x}}_{y_1, \dots, y_k} =1.
\]
Finally, let $\mathbf{y} =\{y_1^{N_1}, \dots, y_k^{N_k}\}$ be a multi-set (the exponent gives the multiplicities) and $\mathbf{x}$ be a multi-sub-set of elements of $\mathbf{y}$ with $a$ elements (counted with multiplicity) and $\underline{x}$ an $a$-tuple of elements of $\mathbf{y}$ given by fixing an order on $\mathbf{x}$. The \emph{symmetric Lagrange polynomial $s_aL^{\mathbf{x}}_{\mathbf{y}}$} is the symmetric $a$-variables polynomial given by:
\[
s_aL^{\mathbf{x}}_{\mathbf{y}} (X_1, \dots X_a) \eqdef \frac{\prod_{j=1}^k N_{\mathbf{x}, y_j}!}{a!} \sum_{\sigma \in \mathfrak{S}_a} L^{\underline{x}}_{y_1, \dots, y_k}(X_{\sigma(1)}, \dots, X_{\sigma(a)}),
\]
where $N_{\mathbf{x}, y_j}$ is the multiplicity of $y_j$ in $\mathbf{x}$. This polynomial does not depend on the choice of $\underline{x}$ and satisfies:
$s_aL^{\mathbf{x}}_{\mathbf{y}}(\mathbf{t}) = \delta_{\mathbf{xt}}$ for every multi-sub-set $\mathbf{t}$ of elements of $\mathbf{y}$ with $a$ elements. Note that since, these are symmetric polynomial, it makes sense to evaluate them on multi-sets. Furthermore we have:
\[
  \sum _{\substack{\mathbf{x}\subseteq \mathbf{y} \\ |\mathbf{x}|=a}} s_aL^{\mathbf{x}}_{\mathbf{y}} =1.
\]
\end{dfn}

We now suppose that $z_1, \dots z_k$ are 2 by 2 distinct and appear in $\mathbf{z}$ with multiplicity $N_1, \dots, N_k$. 

\begin{dfn}
  Let $\Gamma$ be an $\sll_N$-MOY graph. An \emph{$(N_1 \dots, N_k)$ decomposition of $\Gamma$} is a $k$-tuple $(\Gamma_1, \dots, \Gamma_k)$ of MOY graphs such that:
  \begin{itemize}
  \item For every $i$ in $\{1, \dots, k\}$, the underlying oriented graph of $\Gamma_i$  is the same as the underlying oriented graph of $\Gamma$,
  \item For every $i$ in  $\{1, \dots, k\}$, $\Gamma_i$ is an $\sll_{N_i}$-MOY-graph,
  \item For every $e$ in $E(\Gamma)$, $l(e) = \sum_{i=1}^k l_i(e)$, where $l: E(\Gamma) \to \NN$  (resp. $l_i: E(\Gamma_i)=E(\Gamma) \to \NN$) is the labeling of $\Gamma$ (resp. $\Gamma_i$).
  \end{itemize}
\end{dfn}

The following proposition relates our work with \cite[Theorem 1]{2015arXiv150102567R}. 

\begin{prop}\label{prop:decomposition}
For every $\sll_N$-MOY-graph $\Gamma$, we have:
  \[ 
    \Fz^N(\Gamma) \simeq \bigoplus_{(\Gamma_1, \dots, \Gamma_k)\textrm{ $(N_1 \dots, N_k)$-decomposition of $\Gamma$}} \bigotimes \F_{\mathbf{0}}^{N_i}(\Gamma_i).
\]
\end{prop}

\begin{proof}[Sketch of the proof]
  The functor $\Fz^N$ is defined by setting $X_1 =\dots =X_{N_1}= z_1$,  $X_{N_1+1} = \dots = X_{N_1 + N_2}= z_2$, \dots{} and $X_{N-N_k+1}= \dots = X_N = z_k$ in the evaluation of foams. As stated before,  we cannot set $X_\bullet = z_\bullet$ in the evaluation of colored foam. 
However, given a coloring $c$ of a foam $F$ and $i$ and $j$ two distinct  pigments of $\Col$ the proof of Proposition~\ref{prop:sympol} shows that if we consider the set $\mathcal{C}$ of all colorings of $F$ which are Kempe-equivalent relative to $i$ and $j$, then 
\[
\sum_{c'\in \mathcal{C}} \kupc{F,c'}
\]
does not have a factor $(X_i - X_j)$ at the denominator (which is a product of factor of the form $(X_\bullet - X_\bullet)$). We now consider the equivalence relation $\sim_{\mathbf{z}}$ on the colorings of $F$ generated by Kempe moves relatively to all colors $i$ and $j$ such that $X_i$ and  $X_j$ are sent on the same complex number $z_\bullet$. Thanks to the previous discussion, for all coloring $c$ of $F$, the quantity
\[
\sum_{c'\sim_{\mathbf{z}}c} \kupc{F,c'}
\]
is well-defined when setting $X_1 =\dots = X_{N_1}= z_1$,  $X_{N_1+1} = \dots = X_{N_1 + N_2}= z_2$, \dots{} and $X_{N-N_k+1}= \dots = X_N = z_k$. Suppose now that $F$ is a $\Gamma$-foam, using the symmetric Lagrange polynomials on the facets which intersect $\Gamma$ and using the ``Karoubi envelop technology'' as described by Bar-Natan and Morrison in \cite{MR2253455} (see as well \cite[Section 4.2]{2015arXiv150102567R}), we  can decompose $\Fz^N(\Gamma)$ into a direct sum of ``$\Fz^N(\Gamma,d)$'' where $d$ is an extra-labeling of $\Gamma$ which indicates on each edge of $\Gamma$ how many variables will be sent on $z_i$ for every $i$. Then using a splitting functor analogue to the one describe in \cite[Section 4.3]{2015arXiv150102567R}),  we can express $\Fz^N(\Gamma,d)$ as a tensor product of the form $\bigotimes \F_{\mathbf{0}}^{N_i}(\Gamma_i)$ for $(N_1 \dots, N_k)$-decomposition of $\Gamma$ which is given by $d$. This gives:
\[\Fz^N(\Gamma) \simeq \bigoplus_{(\Gamma_1, \dots, \Gamma_k)\textrm{ $(N_1 \dots, N_k)$-decomposition of $\Gamma$}} \bigotimes \F_{\mathbf{0}}^{N_i}(\Gamma_i).
\]
\end{proof}

%% file: equivariantFM.tex
Let us first recall a well-known fact about $\sll_N$-MOY-graphs with symmetries:

\begin{prop}\label{prop:Frobeniusalgebra}
  Suppose that a $\sll_N$ MOY graph $\Gamma$ has a symmetry axis which is transverse to every edge and that the symmetry reverses the orientations and preserves the labels (see Figure~\ref{fig:symweb} for an example). Then $\F(\Gamma)$ has a natural structure of Frobenius algebra. Note that different symmetry axes gives \emph{a priori} non-isomorphic Frobenius algebra structures.
\end{prop}

\begin{figure}[H]
  \centering
  \begin{tikzpicture}
  {\tiny  \input{\imagesfolder/cef_symweb}}
  \end{tikzpicture}
  \caption{A MOY-graph with a symmetry axis.}
  \label{fig:symweb}
\end{figure}
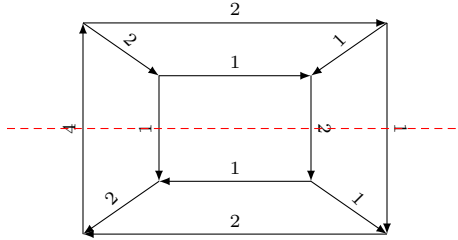

\begin{proof}[Sketch of the proof]
  We need to tell where the multiplication and the trace structures come from.
For a more systematic and category-oriented approach we refer to \cite{2012arXiv1206.2118M} or \cite{LHR2} for the $\sll_3$-case or to \cite{MR3198835} for the $\sll_N$ case. We consider an $\sll_N$-MOY-graph $\Gamma$ with a symmetry axis.
Since the functor $\F$ is defined by a universal construction, one can think of elements of $\F(\Gamma)$ as $\ZZ[X_1, \dots X_N]$-linear combination of foams. Hence it is enough to explain how the definition of the multiplication and the trace on foams. Let us denote by $\Sigma$ one half of the graph $\Gamma$, by \reflectbox{$\Sigma$} the other half and by $\epsilon$ the intersection of the symmetry axis with $\Gamma$ (hence we have $\Gamma = \Sigma \cup_\epsilon$\reflectbox{$\Sigma$}). The definitions of the multiplication and the trace  are given in Figures~\ref{fig:Frobprod} and \ref{fig:Frobtrace}.

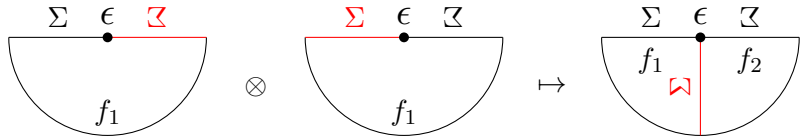
\begin{figure}[H]
  \centering
  \begin{tikzpicture}[scale =1.3]
    \input{\imagesfolder/cef_prod}
  \end{tikzpicture}
  \caption{Multiplication.}
  \label{fig:Frobprod}
\end{figure}

\begin{figure}[H]
  \centering
\[\tau\left( \NB{\tikz{\input{\imagesfolder/cef_trace1}}} \right) = \kup{\NB{\tikz{\input{\imagesfolder/cef_trace2}}}} \]
  \caption{Trace.}
  \label{fig:Frobtrace}
\end{figure}

The unit is given by a foam which is isotopic to $\Sigma\times I$. The associativity is obvious. The fact that the trace and the multiplication induces a non-degenerate pairing comes directly from the universal construction: indeed taking the trace of the product of $f_1$ and $f_2$ is the same as closing up $f_1$ with $f_2$. The co-multiplication is a little more complicated to write down, however we do not strictly need it and it follows from the functor $\F$ being monoidal.
\end{proof}

The cohomology of Grassmannian and flag varieties has been a guide line in the history of categorification of the $\sll_N$-knot invariants (see for example \cite{MR2391017}, \cite{MR2100691}, \cite{MR3190356} or \cite{queffelec2014mathfrak}). Here we prove that the spaces associated with the so-called generalized $\theta$-graphs are indeed isomorphic as rings to the cohomology rings of partial flag varieties.

\begin{dfn}\label{dfn:gentheta}
  Let $a_1, \dots, a_k$ be positive integers such that $\sum_{i=1}^k a_i=N$. The MOY-graph
\[
  \begin{tikzpicture}
  {\tiny  \input{\imagesfolder/cef_gentheta} }
  \end{tikzpicture}
\]
is called the \emph{generalized $\theta$-graph of type $(a_1, \dots a_k)$} and is denoted by $\theta(a_1, \dots a_k)$. Note that is comes naturally with a symmetry axis.
\end{dfn}

\begin{notation}\label{not:partiaflag}
  Let $a_1, \dots, a_k$ be positive integers such that $\sum_{i=1}^k a_i=N$. We denote by $Fl_N(a_1, \dots, a_k)$ the flag variety corresponding to the following inclusion:
\[
\CC^{a_1} \subseteq \CC^{a_1+a_2} \subseteq \CC^{a_1+a_2 + a_3} \subseteq  \dots \subseteq \CC^{N-a_k}
\]
in a fixed space $\CC^N$. It is isomorphic to $GL(N,\CC)/P(a_1,\dots, a_k)$ with $P(a_1, \dots, a_k)$ being the parabolic subgroup consisting of matrix of the form:
\newcommand{\partialbox}[2]{\ensuremath{\NB{\tikz[scale=0.2]{\draw (#2,-#2) -- (-#2,-#2) --(-#2,#2) -- (#2,#2) --cycle; \node at (0,0) {$#1$};}}} }
\[
  \begin{pmatrix}
    \partialbox{a_1}{2} & \star               & \star  & \cdots              & \star  \\
 0                      & \partialbox{a_2}{1} & \star  & \cdots               & \star  \\
    \vdots              &                     & \ddots &                     & \vdots \\
  0& \cdots              & 0      & \partialbox{a_{k-1}}{3} & \star  \\
 0      &0                & \cdots              & 0      & \partialbox{a_{k}}{2}   \\
  \end{pmatrix}
\]
where the size of the blocks are given by the integers $a_\bullet$. We denote by $T$ the maximal torus, that is the sub-group of invertible diagonal matrices.
We denote by $H_T^\bullet (Fl_N(a_1, \dots,a_N))$ the $T$-equivariant cohomology ring of $Fl_N(a_1, \dots, a_k)$ with coefficients in $\ZZ$. It is an algebra over $H_T^\bullet(\mathrm{point}) \simeq \ZZ[X_1, \dots X_N]$. For definition of the equivariant cohomology we refer to \cite{MR3263556}, and references therein.
\end{notation}

\begin{thm}[\cite{MR3263556}] \label{thm:pres_equivariant_coh}
  The space $H_T^\bullet (Fl_N(1, \dots,1))$ is isomorphic to:
\[
\frac{  \ZZ[T_1, \dots, T_N, X_1, \dots X_N] }{ \left\{P(X_1, \dots, X_N) - P(T_1, \dots T_N) | P \in \ZZ[Z_1, \dots ,Z_N]^{\mathfrak{S}_N}\right\}}
\] as a $\ZZ[X_1, \dots X_N]$-algebra.
Let $a_1, \dots, a_k$ be positive integers such that $\sum_{i=1}^k a_i =N$ and $\mathfrak{S}_{a_1,\dots, a_k} := \mathfrak{S}_{a_1} \times \dots \times \mathfrak{S}_{a_k} \subset \mathfrak{S}_N$. This group acts on $H_T^\bullet (Fl_N(1, \dots,1))$ by permuting the variables $T_\bullet$. 
 The space $H_T^\bullet (Fl_N(a_1, \dots,a_k))$ is isomorphic as  a $\ZZ[X_1, \dots X_N]$-algebra to
\begin{align*}
& H_T^\bullet (Fl_N(1, \dots,1))^{\mathfrak{S}_{a_1,\dots, a_k}} \\ & \quad \simeq 
\frac{\ZZ[X_1, \dots X_{a_1}]^{\mathfrak{S}_{a_1}} \otimes \dots \otimes\ZZ[X_{N-a_{k}+1}, \dots, X_{N}]^{\mathfrak{S}_{a_k}}\otimes \ZZ[T_1, \dots T_N]}{\left\{P(X_1, \dots, X_N) - P(T_1, \dots T_N) | P \in \ZZ[Z_1, \dots, Z_N]^{\mathfrak{S}_N}\right\}}.
\end{align*}
\end{thm}
\begin{prop}\label{prop:equiv}
  The $\ZZ[X_1, \dots X_N]$-algebra $\F(\theta(a_1, \dots, a_k))$ is isomorphic to the ring $H_T^\bullet(Fl_N(a_1,\dots a_N))$.
\end{prop}
\begin{proof}
  We use the presentation of $H_T^\bullet(Fl_N(a_1,\dots a_N))$ of Theorem~\ref{thm:pres_equivariant_coh} and define a $\ZZ[X_1, \dots X_N]$-linear map by: \[ 
   \begin{array}{crcl}
      \phi\colon & H_T^\bullet(Fl_N(a_1,\dots a_N)) &\to&  \F(\theta(a_1, \dots, a_k)) \\ 
      & \prod_{i=1}^k \pi_{\lambda_i}(X_{a_i+1}, \dots, X_{a_{i+1}}) & \mapsto &\hspace{-0.2cm}\NB{\tiny \tikz[scale=0.7]{\input{\imagesfolder/cef_genthetafoamschur}}}
    \end{array}.
\]
This is obviously an algebra map. Furthermore it is surjective since the foams on the right-hand-side when the $\lambda_\bullet$  vary obviously span $\F(\theta(a_1, \dots, a_k))$. This is a direct consequences of the neck-cutting (Relation~(\ref{eq:neckcutting-ev})) and dot-migrations (Relation~(\ref{eq:dot-migration})). Furthermore both are free of rank $\frac{N!}{\prod_{i=1}^k a_i!}$. 
\end{proof}

From the digon formula
in Proposition~\ref{prop:additionalrelation}, we immediately get the following proposition.

\begin{prop}\label{prop:basemultithetafoam}
  A base of $\F(\theta(a_1, \dots, a_k))$ as a  $\ZZ[X_1, \dots X_N]$-module, is given by the images by $\phi$ of $1\otimes \bigotimes_{i=2}^k \pi_{\lambda_i}$ with $\lambda_i$ in $T(a_i, \sum_{j=1}^{i-1} a_j)$, for $i$ in $\{2,\dots, k\}$. A dual basis (with respect to the trace introduced in Proposition~\ref{prop:Frobeniusalgebra}) is given by:
\[\left( (-1)^{\sum_{j=2}^{k} |\widehat{\lambda_i}| + N(N+1)/2}\NB{\tiny \tikz[scale=0.7]{\input{\imagesfolder/cef_genthetafoamschur2}}} \right)_{\substack{i=2,\dots, k \\ \lambda_i \in T(a_i, \sum_{j=1}^{i-1} a_j) }}.
 \]
\end{prop}

We can use foams to compute structure constants of the cohomology rings of flag varieties associated with the bases given in Proposition~\ref{prop:basemultithetafoam}:

\begin{prop}\label{prop:structuralconstant}
  Let $a_1, \dots, a_k$ be positive integers which sum to $N$ and $\mathbf{\alpha} = (\alpha_2, \dots, \alpha_k)$ and $\mathbf{\beta} = (\beta_2, \dots, \beta_k)$ two collections of Young diagrams such that for every $i$ in $\{2, \dots, k\}$, $\alpha_i$ and $\beta_i$ are in $T(a_i, \sum_{j=1}^{i-1} a_j)$. Finally let us denote by $h_{\mathbf{\alpha}}$ and $h_{\mathbf{\beta}}$ the elements of $H_T^\bullet(Fl_N(a_1,\dots a_N))$ corresponding to $\mathbf{\alpha}$ and $\mathbf{\beta}$, then we have:
\[
h_{\mathbf{\alpha}}h_{\mathbf{\beta}}= \sum_{\mathbf{\lambda}} \mathbf{c}_{\mathbf{\alpha}\mathbf{\beta}}^{\mathbf{\lambda}} h_{\mathbf{\lambda}},
\]
where the $\mathbf{\lambda}$ runs on all collections $(\lambda_2, \dots, \lambda_k)$ of Young diagrams with $\lambda_i$ are in $T(a_i, \sum_{j=1}^{i-1} a_j)$ and where 
\[
\mathbf{c}_{\mathbf{\alpha}\mathbf{\beta}}^{\mathbf{\lambda}} \eqdef (-1)^{\sum_{j=2}^k |\widehat{\lambda_j}| + N(N+1)/2} \kup{\!\!\!\!\scriptstyle{\NB{\tikz[scale =0.7]{\input{\imagesfolder/cef_genthetaclosed2}}}}\!\!\!\!}.
\]
  \end{prop}

\begin{rmk}\label{rmk:structuredegree0}
  \begin{enumerate}
  \item The last formula can be seen as an implementation of the Atiyah--Bott--Berline--Vergne integration formula for the computation of the structure constants of the partial flag varieties see \cite[Section 2.6]{MR2976939}.
  \item If we restrict to the case where $\sum_{j=2}^k |\lambda_j| = \sum_{j=2}^k |\alpha_j| + \sum_{j=2}^k |\beta_j|$, the coefficient $\mathbf{c}_{\mathbf{\alpha} \mathbf{\beta}}^{\lambda}$ is an integer (it is a symmetric polynomial of degree $0$). 
  \end{enumerate}
\end{rmk}

\begin{cor}\label{cor:LRgrassmanian}
   Let $\alpha$, $\beta$ and $\lambda$ be three Young diagrams such that $|\alpha| + |\beta|=  |\lambda|$. Choose two non-negative integers $a$ and $b$ such that $\alpha$, $\beta$ and $\lambda$ are in $T(b,a)$, then $\mathbf{c}_{{\alpha} {\beta}}^{\lambda}$ is precisely the Littlewood--Richardson coefficient $c_{{\alpha} {\beta}}^{\lambda}$ and can be computed via:
\begin{align*}
c_{\alpha {\beta}}^{\mathbf{\lambda}}= &(-1)^{|\widehat{\lambda}|+ N(N+1)/2 }\kup{\scriptstyle{\NB{\tikz[scale =0.8]{\input{\imagesfolder/cef_thetaLR}}}}}_N \\ &= (-1)^{N(N+1)/2 +|\widehat{\lambda}|} \sum_{\substack{A\sqcup B = \{X_1, \dots, X_N\} \\ |A|=a, |B| = b}}(-1)^{|B<A|} \frac{a_\alpha(A) a_\beta(A) a_{\widehat{\lambda}}(B)}{\Delta(X_1, \dots, X_N)}.
\end{align*}
where $N= a+b$. Moreover to evaluate this foam we may use any evaluation $X_i \mapsto z_i$ (see Section~\ref{sec:relat-with-appr}).
\end{cor}

%% file: cef_symweb.tex
\begin{scope}[yscale= 0.7]
  \coordinate (a) at ( 1, 1);
  \coordinate (b) at ( 1,-1);
  \coordinate (c) at (-1,-1);
  \coordinate (d) at (-1, 1);
  \coordinate (A) at ( 2, 2);
  \coordinate (B) at ( 2,-2);
  \coordinate (C) at (-2,-2);
  \coordinate (D) at (-2, 2);
  \draw[->] (a) -- (b) node [midway, sloped, above] {$2$};
  \draw[->] (b) -- (c) node [midway, sloped, above] {$1$};
  \draw[<-] (c) -- (d) node [midway, sloped, above] {$1$};
  \draw[->] (d) -- (a) node [midway, sloped, above] {$1$};
  \draw[->] (A) -- (B) node [midway, sloped, above] {$1$};
  \draw[->] (B) -- (C) node [midway, sloped, above] {$2$};
  \draw[->] (C) -- (D) node [midway, sloped, above] {$4$};
  \draw[->] (D) -- (A) node [midway, sloped, above] {$2$};
  \draw[<-] (a) -- (A) node [midway, sloped, above] {$1$};
  \draw[->] (b) -- (B) node [midway, sloped, above] {$1$};
  \draw[->] (c) -- (C) node [midway, sloped, above] {$2$};
  \draw[<-] (d) -- (D) node [midway, sloped, above] {$2$};
\draw[red, densely dashed] (-3,0) -- (3, 0);
\end{scope}

%% file: cef_prod.tex
\begin{scope}
\begin{scope}
\draw (1,0) arc (0:-180:1)  node[above, midway] {$f_1$};
\draw[red] (1,0) -- (0,0) node[above, midway] {\reflectbox{$\Sigma$}};
\draw (-1,0) -- (0,0) node[above, midway]  {$\Sigma$};
\fill (0,0) circle (0.5mm) node[above, midway] {$\epsilon$};
\node at (1.5,- 0.5) {$\otimes$};
\end{scope}
\begin{scope}[xshift = 3cm]
\draw (1,0) arc (0:-180:1)  node[above, midway] {$f_1$};
\draw (1,0) -- (0,0) node[above, midway] {\reflectbox{$\Sigma$}};
\draw[red] (-1,0) -- (0,0) node[above, midway]  {$\Sigma$};
\fill (0,0) circle (0.5mm) node[above, midway] {$\epsilon$};
\node at (1.5,- 0.5) {$\mapsto$};
\end{scope}
\begin{scope}[xshift = 6cm]
\draw (1,0) arc (0:-180:1);
\draw (1,0) -- (0,0) node[above, midway] {\reflectbox{$\Sigma$}} node[midway, below] {$f_2$};
\draw[red] (0,-1) -- (0,0) node[above, midway, sloped] {${\Sigma}$};
\draw (-1,0) -- (0,0) node[above, midway] {{$\Sigma$}} node[midway, below] {$f_1$};
\fill (0,0) circle (0.5mm) node[above, midway] {$\epsilon$};
\end{scope}
\end{scope}

%% file: cef_trace1.tex
\begin{scope}[scale= 1]
\draw (1,0) arc (0:-180:1)   node[above, midway] {$f$};
\draw (1,0) -- (0,0) node[above, midway] {\reflectbox{$\Sigma$}};
\draw (-1,0) -- (0,0) node[above, midway]  {$\Sigma$};
\fill (0,0) circle (0.5mm) node[above, midway] {$\epsilon$};
\end{scope}

%% file: cef_trace2.tex
\begin{scope}[scale= 1]
\draw (1,0) arc (0:-180:1)  node[above, midway] {$f$};
\draw (1,0) -- (0,0) node[above, midway] {\reflectbox{$\Sigma$}};
\draw (-1,0) -- (0,0) node[above, midway]  {$\Sigma$};
\fill (0,0) circle (0.5mm) node[above, midway]  {$\epsilon$};
\draw (1,0) arc (0:180:1) node[below, midway, sloped] {$\Sigma \times I$} ;
\end{scope}

%% file: cef_gentheta.tex
\begin{scope}
  \coordinate (T0)   at (0,1);
  \coordinate (T1)   at (1,1);
  \coordinate (T2)   at (2,1);
  \coordinate (T3)   at (3,1);

  \coordinate (TNm1) at (4,1);
  \coordinate (TN)   at (5,1);
  \coordinate (B0)   at (0,-1);
  \coordinate (B1)   at (1,-1);
  \coordinate (B2)   at (2,-1);
  \coordinate (B3)   at (3,-1);

  \coordinate (BNm1) at (4,-1);
  \coordinate (BN)   at (5,-1);
  \draw[->] (T1) -- (T0) node [above, midway, rotate = 45, anchor =west] {$a_1+ a_2$};
  \draw[<-] (B1) -- (B0) node [below, midway, rotate = -45, anchor =west] {$a_1+ a_2$};
  \draw[->] (T2) -- (T1) node [above, midway, rotate = 45, anchor =west] {$a_1+ a_2+ a_3$};
  \draw[<-] (B2) -- (B1) node [below, midway, rotate = -45, anchor =west] {$a_1+ a_2 + a_3$};
  \draw[->] (TN) -- (TNm1) node [above, midway, rotate = 45, anchor =west] {$N-a_k$};
  \draw[<-] (BN) -- (BNm1) node [below, midway, rotate = -45, anchor =west] {$N-a_k$};
  \draw[dotted] (T2) -- (TNm1);
  \draw[dotted] (B2) -- (BNm1);
  \draw[->] (T0) .. controls +(-1,0) and  +(-1,0) .. (B0) node[near start, left ] {$a_1$};
  \draw[->] (T0) -- (B0) node[near start, left] {$a_2$};
  \draw[->] (T1) -- (B1) node[near start, left] {$a_3$};
  \draw[->] (T2) -- (B2) node[near start, left] {$a_4$};
  \draw[->] (TNm1) -- (BNm1) node[near start, right] {$a_{k-1}$};
  \draw[<-] (TN) .. controls +(1,0) and  +(1,0) .. (BN) node[near start, right] {$N$};
  \draw[->] (TN) -- (BN) node[near start, right] {$a_k$};
  \draw [densely dashed, red] (-1,0) -- (6,0);
\end{scope}

%% file: cef_genthetafoamschur.tex
\tdplotsetmaincoords{70}{30}
\begin{scope}[tdplot_main_coords, scale = 1.5]
  \coordinate (T0)   at (0,1,0);
  \coordinate (T1)   at (1,1,0);
  \coordinate (T2)   at (2,1,0);
  \coordinate (T3)   at (2.5,1,0);
  \coordinate (T4)   at (3.5,1,0);

  \coordinate (TNm1) at (4,1,0);
  \coordinate (TN)   at (5,1,0);
  \coordinate (C0)   at (0,0,-2);
  \coordinate (C1)   at (1,0,-2);
  \coordinate (C2)   at (2,0,-2);
  \coordinate (C3)   at (2.5,0,-2);
  \coordinate (C4)   at (3.5,0,-2);

  \coordinate (CNm1) at (4,0,-2);
  \coordinate (CN)   at (5,0,-2);
  \coordinate (B0)   at (0,-1,0);
  \coordinate (B1)   at (1,-1,0);
  \coordinate (B2)   at (2,-1,0);
  \coordinate (B3)   at (2.5,-1,0);
  \coordinate (B4)   at (3.5,-1,0);

  \coordinate (BNm1) at (4,-1,0);
  \coordinate (BN)   at (5,-1,0);
  \draw[->] (T1) -- (T0);
  \draw[<-] (B1) -- (B0);
  \draw[->] (T2) -- (T1);
  \draw[<-] (B2) -- (B1);
  \draw[->] (TN) -- (TNm1);
  \draw[<-] (BN) -- (BNm1);
  \draw[dotted] (T2) -- (TNm1);
  \draw[dotted] (B2) -- (BNm1);
  \draw[->] (T0) arc (90:270:1) node[pos=0.3, below] {$a_1$} node [pos=1, above, black] {${\pi_{\lambda_1}}$};
 \draw[opacity = 0] (T4) -- (B4) .. controls +(0,0,-2)  and +(0,0,0)..(C4) coordinate[pos=0.77] (C42) .. controls +(0,0,0) and +(0,0,-2) .. (T4); 
 \draw[opacity = 0] (T3) -- (B3) .. controls +(0,0,-2)  and +(0,0,0)..(C3) coordinate[pos=0.77] (C32) .. controls +(0,0,0) and +(0,0,-2) .. (T3); 
  \filldraw[fill opacity = 0, fill =red] (T0) -- (B0) .. controls +(0,0,-2) and +(0,0,0) ..(C0) coordinate[pos=0.77] (C02).. controls +(0,0,0) and +(0,0,-2) .. (T0);
  \filldraw[fill opacity = 0, fill =red] (TN) -- (BN) .. controls +(0,0,-2)  and +(0,0,0)..(CN) coordinate[pos=0.77] (CN2) .. controls +(0,0,0) and +(0,0,-2) .. (TN); 
 \fill[opacity = 0.3, blue] (TNm1) -- (TN) arc (90:30:1) .. controls +(0,0, -1) and +(0.5, 0, 0) .. (CN2) .. controls +(0,0,0) and (C42) .. (C42) .. controls +(0,0.3,0) and +(0,0,-1) .. (T4); 
 \fill[opacity = 0.3, blue] (T3) -- (T0) arc (90:210:1) .. controls +(0,0, -1) and +(-0.5, 0, 0) .. (C02) .. controls +(0,0,0) and (C02) .. (C32) .. controls +(0,0.3,0) and +(0,0,-1) .. (T3); 
  \filldraw[fill opacity = 0.3, fill =red] (T0) -- (B0) .. controls +(0,0,-2) and +(0,0,0) ..(C0) .. controls +(0,0,0) and +(0,0,-2) .. (T0);
  \filldraw[fill opacity = 0.3, fill =red] (TN) -- (BN) .. controls +(0,0,-2)  and +(0,0,0)..(CN) .. controls +(0,0,0) and +(0,0,-2) .. (TN); 
  \filldraw[fill opacity = 0.3, fill =red] (T1) -- (B1) .. controls +(0,0,-2)  and +(0,0,0)..(C1).. controls +(0,0,0) and +(0,0,-2) .. (T1);
  \filldraw[fill opacity = 0.3, fill =red] (T2) -- (B2) .. controls +(0,0,-2)  and +(0,0,0)..(C2).. controls +(0,0,0) and +(0,0,-2) .. (T2);
  \filldraw[fill opacity = 0.3, fill =red] (TNm1) -- (BNm1) .. controls +(0,0,-2)  and +(0,0,0)..(CNm1).. controls +(0,0,0) and +(0,0,-2) .. (TNm1);
  \node[above, black] at ($(C0)+ (0,0,1)$)  {${\pi_{\lambda_2}}$};
  \node[above, black] at ($(C1)+ (0,0,1)$)  {${\pi_{\lambda_3}}$};
  \node[above, black] at ($(C2)+ (0,0,1)$)  {${\pi_{\lambda_4}}$};
  \node[above, black] at ($(CNm1)+ (0,0,1)$)  {${\pi_{\lambda_{k-1}}}$};
  \node[above, black] at ($(CN)+ (0,0,1)$)  {${\pi_{\lambda_k}}$};
  \draw[->] (T0) -- (B0) node[near start, below] {$a_2$};
  \draw[->] (T1) -- (B1) node[near start, below] {$a_3$};
  \draw[->] (T2) -- (B2) node[near start, below] {$a_4$};
  \draw[->] (TNm1) -- (BNm1) node[near start, below] {$a_{k-1}$};
  \draw[<-] (TN) arc (90:-90:1) node [pos= 0.1, below] {$N$};
  \draw[->] (TN) -- (BN) node[near start, below] {$a_k$};
  \fill[opacity = 0.1, blue] (BNm1) -- (BN) arc (-90:30:1) .. controls +(0,0, -1) and +(0.5, 0, 0) .. (CN2) .. controls +(0,0,0) and (C42) .. (C42) .. controls +(-0.3,0) and +(0,0,-1) .. (B4); 
  \fill[opacity = 0.1, blue] (B3) -- (B0) arc (270:210:1) .. controls +(0,0, -1) and +(-0.5, 0, 0) .. (C02) .. controls +(0,0,0) and (C32) .. (C32) .. controls +(-0.3,0) and +(0,0,-1) .. (B3); 
\end{scope}

%% file: cef_genthetafoamschur2.tex
\tdplotsetmaincoords{70}{30}
\begin{scope}[tdplot_main_coords, scale = 1.5]
  \coordinate (T0)   at (0,1,0);
  \coordinate (T1)   at (1,1,0);
  \coordinate (T2)   at (2,1,0);
  \coordinate (T3)   at (2.5,1,0);
  \coordinate (T4)   at (3.5,1,0);

  \coordinate (TNm1) at (4,1,0);
  \coordinate (TN)   at (5,1,0);
  \coordinate (C0)   at (0,0,-2);
  \coordinate (C1)   at (1,0,-2);
  \coordinate (C2)   at (2,0,-2);
  \coordinate (C3)   at (2.5,0,-2);
  \coordinate (C4)   at (3.5,0,-2);

  \coordinate (CNm1) at (4,0,-2);
  \coordinate (CN)   at (5,0,-2);
  \coordinate (B0)   at (0,-1,0);
  \coordinate (B1)   at (1,-1,0);
  \coordinate (B2)   at (2,-1,0);
  \coordinate (B3)   at (2.5,-1,0);
  \coordinate (B4)   at (3.5,-1,0);

  \coordinate (BNm1) at (4,-1,0);
  \coordinate (BN)   at (5,-1,0);
  \draw[->] (T1) -- (T0);
  \draw[<-] (B1) -- (B0);
  \draw[->] (T2) -- (T1);
  \draw[<-] (B2) -- (B1);
  \draw[->] (TN) -- (TNm1);
  \draw[<-] (BN) -- (BNm1);
  \draw[dotted] (T2) -- (TNm1);
  \draw[dotted] (B2) -- (BNm1);
  \draw[->] (T0) arc (90:270:1) node[pos=0.3, below] {$a_1$} node [pos=1, above, black] {${\pi_{\widehat{\lambda_2}}}$};
 \draw[opacity = 0] (T4) -- (B4) .. controls +(0,0,-2)  and +(0,0,0)..(C4) coordinate[pos=0.77] (C42) .. controls +(0,0,0) and +(0,0,-2) .. (T4); 
 \draw[opacity = 0] (T3) -- (B3) .. controls +(0,0,-2)  and +(0,0,0)..(C3) coordinate[pos=0.77] (C32) .. controls +(0,0,0) and +(0,0,-2) .. (T3); 
  \filldraw[fill opacity = 0, fill =red] (T0) -- (B0) .. controls +(0,0,-2) and +(0,0,0) ..(C0) coordinate[pos=0.77] (C02).. controls +(0,0,0) and +(0,0,-2) .. (T0);
  \filldraw[fill opacity = 0, fill =red] (TN) -- (BN) .. controls +(0,0,-2)  and +(0,0,0)..(CN) coordinate[pos=0.77] (CN2) .. controls +(0,0,0) and +(0,0,-2) .. (TN); 
 \fill[opacity = 0.3, blue] (TNm1) -- (TN) arc (90:30:1) .. controls +(0,0, -1) and +(0.5, 0, 0) .. (CN2) .. controls +(0,0,0) and (C42) .. (C42) .. controls +(0,0.3,0) and +(0,0,-1) .. (T4); 
 \fill[opacity = 0.3, blue] (T3) -- (T0) arc (90:210:1) .. controls +(0,0, -1) and +(-0.5, 0, 0) .. (C02) .. controls +(0,0,0) and (C02) .. (C32) .. controls +(0,0.3,0) and +(0,0,-1) .. (T3); 
  \filldraw[fill opacity = 0.3, fill =red] (T0) -- (B0) .. controls +(0,0,-2) and +(0,0,0) ..(C0) .. controls +(0,0,0) and +(0,0,-2) .. (T0);
  \filldraw[fill opacity = 0.3, fill =red] (TN) -- (BN) .. controls +(0,0,-2)  and +(0,0,0)..(CN) .. controls +(0,0,0) and +(0,0,-2) .. (TN); 
  \filldraw[fill opacity = 0.3, fill =red] (T1) -- (B1) .. controls +(0,0,-2)  and +(0,0,0)..(C1).. controls +(0,0,0) and +(0,0,-2) .. (T1);
  \filldraw[fill opacity = 0.3, fill =red] (T2) -- (B2) .. controls +(0,0,-2)  and +(0,0,0)..(C2).. controls +(0,0,0) and +(0,0,-2) .. (T2);
  \filldraw[fill opacity = 0.3, fill =red] (TNm1) -- (BNm1) .. controls +(0,0,-2)  and +(0,0,0)..(CNm1).. controls +(0,0,0) and +(0,0,-2) .. (TNm1);

  \node[above, black] at ($(T0)+ (0.3,0,-0.4)$)  {${\pi_{\widehat{\lambda_3}}}$};
  \node[above, black] at ($(T1)+ (0.3,0,-0.4)$)  {${\pi_{\widehat{\lambda_4}}}$};
  \node[above, black] at ($(T2)+ (0.3,0,-0.4)$)  {${\pi_{\widehat{\lambda_5}}}$};
  \node[above, black] at ($(TNm1)+ (0.3,0,-0.4)$)  {${\pi_{\widehat{\lambda_k}}}$};
  \draw[->] (T0) -- (B0) node[near start, below] {$a_2$};
  \draw[->] (T1) -- (B1) node[near start, below] {$a_3$};
  \draw[->] (T2) -- (B2) node[near start, below] {$a_4$};
  \draw[->] (TNm1) -- (BNm1) node[near start, below] {$a_{k-1}$};
  \draw[<-] (TN) arc (90:-90:1) node [pos= 0.1, below] {$N$};
  \draw[->] (TN) -- (BN) node[near start, below] {$a_k$};
  \fill[opacity = 0.1, blue] (BNm1) -- (BN) arc (-90:30:1) .. controls +(0,0, -1) and +(0.5, 0, 0) .. (CN2) .. controls +(0,0,0) and (C42) .. (C42) .. controls +(-0.3,0) and +(0,0,-1) .. (B4); 
  \fill[opacity = 0.1, blue] (B3) -- (B0) arc (270:210:1) .. controls +(0,0, -1) and +(-0.5, 0, 0) .. (C02) .. controls +(0,0,0) and (C32) .. (C32) .. controls +(-0.3,0) and +(0,0,-1) .. (B3); 
\end{scope}

%% file: cef_genthetaclosed2.tex
\tdplotsetmaincoords{70}{30}
\begin{scope}[tdplot_main_coords, scale = 1.5]
  \coordinate (T0)   at (0,1,0);
  \coordinate (T1)   at (1,1,0);
  \coordinate (T2)   at (2,1,0);
  \coordinate (T3)   at (2.5,1,0);
  \coordinate (T4)   at (3.5,1,0);
  \coordinate (TNm1) at (4,1,0);
  \coordinate (TN)   at (5,1,0);
  \coordinate (C0)   at (0,0,0);
  \coordinate (C1)   at (1,0,0);
  \coordinate (C2)   at (2,0,0);
  \coordinate (C3)   at (2.5,0,0);
  \coordinate (C4)   at (3.5,0,0);
  \coordinate (CNm1) at (4,0,0);
  \coordinate (CN)   at (5,0,0);
  \coordinate (B0)   at (0,-1,0);
  \coordinate (B1)   at (1,-1,0);
  \coordinate (B2)   at (2,-1,0);
  \coordinate (B3)   at (2.5,-1,0);
  \coordinate (B4)   at (3.5,-1,0);
  \coordinate (BNm1) at (4,-1,0);
  \coordinate (BN)   at (5,-1,0);

 \draw[opacity = 0] (C3) arc (0:360:0.5cm and 1cm) coordinate[pos=0.25] (C3T);
 \draw[opacity = 0] (C0) arc (0:360:0.5cm and 1cm) coordinate[pos=0.75] (C0B);
 \draw[opacity = 0] (C4) arc (0:360:0.5cm and 1cm) coordinate[pos=0.25] (C4T);
 \draw[opacity = 0] (CN) arc (0:360:0.5cm and 1cm) coordinate[pos=0.75] (CNB);
 \fill[fill opacity = 0.3, fill =blue] (C3T) -- +(-2.5, 0,0) .. controls +(-2,0,0) and +(-2,0,0) .. (C0B) -- +(2.5,0,0).. controls +(0.7,0,0) and +(0.7,0,0)  .. (C3T);
 \fill[fill opacity = 0.3, fill =blue] (C4T) -- +(1.5, 0,0) .. controls +(2,0,0) and +(2,0,0) .. (CNB) -- +(-1.5,0,0).. controls +(0.7,0,0) and +(0.7,0,0)  .. (C4T);
 \filldraw[->-, thick, fill opacity = 0.3, fill =red, draw opacity=1, ->-] (C0)   arc (0:360:0.5cm and 1cm);
 \filldraw[->-, thick, fill opacity = 0.3, fill =red, ->-, draw opacity=1] (C1)   arc (0:360:0.5cm and 1cm);
 \filldraw[->-, thick, fill opacity = 0.3, fill =red, ->-, draw opacity=1] (C2)   arc (0:360:0.5cm and 1cm);
 \filldraw[->-, thick, fill opacity = 0.3, fill =red, ->-, draw opacity=1] (CN)   arc (0:360:0.5cm and 1cm);
 \filldraw[->-, thick, fill opacity = 0.3, fill =red, ->-, draw opacity=1] (CNm1) arc (0:360:0.5cm and 1cm);
 \draw[ <-]  ($(C0)   + (-0.5, 0,-0.5)$) -- +(0,0,-1) node[below] {$\scriptstyle{a_2}$};
 \draw[ <-]  ($(C0B)   + (0.5, 0,0)$) -- +(0,0,-0.8) node[below] {$\scriptstyle{a_1+a_2}$};
 \draw[ <-]  ($(C0B)   + (1.5, 0,0)$) -- +(0,0,-1) node[below] {$\scriptstyle{a_1+a_2+a_3}$};
 \draw[ <-]  ($(C0B)   + (2.5, 0,0)$) -- +(0,0,-1.2) node[below] {$\scriptstyle{a_1+a_2+a_3+ a_4}$};
 \draw[ <-]  ($(C0B)   + (3.7, 0,0)$) -- +(0,0,-0.6) node[below] {$\scriptstyle{N -a_k - a_{k-1}}$};
\draw[ <-]  ($(C0B)   + (4.7, 0,0)$) -- +(0,0,-0.8) node[below] {$\scriptstyle{N -a_k}$};
 \draw[ <-]  ($(C1)   + (-0.5, 0,-0.5)$) -- +(0,0,-1) node[below] {$\scriptstyle{a_3}$};
 \draw[ <-]  ($(C2)   + (-0.5, 0,-0.5)$) -- +(0,0,-1) node[below] {$\scriptstyle{a_4}$};
 \draw[ <-]  ($(CNm1) + (-0.5, 0,-0.5)$) -- +(0,0,-1) node[below] {$\scriptstyle{a_{k-1}}$};
 \draw[ <-]  ($(CN)   + (-0.5, 0,-0.5)$) -- +(0,0,-1) node[below] {$\scriptstyle{a_k}$};
 \draw[thick, red, <-]  ($(C0)   + (-0.5, 0,0.5)$) -- +(0,0,+1.5) node[above] {$\scriptstyle{\pi_{\alpha_2} \pi_{\beta_2}}$};
 \draw[thick, red, <-]  ($(C1)   + (-0.5, 0,0.5)$) -- +(0,0,+1.5) node[above] {$\scriptstyle{\pi_{\alpha_3} \pi_{\beta_3}}$};
 \draw[thick, red, <-]  ($(C2)   + (-0.5, 0,0.5)$) -- +(0,0,+1.5) node[above] {$\scriptstyle{\pi_{\alpha_4} \pi_{\beta_4}}$};
 \draw[thick, red, <-]  ($(CNm1) + (-0.5, 0,0.5)$) -- +(0,0,+1.5) node[above] {$\scriptstyle{\pi_{\alpha_{k-1}} \pi_{\beta_{k-1}}}$};
 \draw[thick, red, <-]  ($(CN)   + (-0.5, 0,0.5)$) -- +(0,0,+1.5) node[above] {$\scriptstyle{\pi_{\alpha_k} \pi_{\beta_k}}$};
 \fill[fill opacity = 0.1, fill =blue] (C4T) -- +(1.5, 0,0) .. controls +(2,0,0) and +(2,0,0) .. (CNB) -- +(-1.5,0,0).. controls +(-0.7,0,0) and +(-0.7,0,0)  .. (C4T);
 \fill[fill opacity = 0.1, fill =blue] (C3T) -- +(-2.5, 0,0) .. controls +(-2,0,0) and +(-2,0,0) .. (C0B) -- +(2.5,0,0).. controls +(-0.7,0,0) and +(-0.7,0,0)  .. (C3T);
 \draw[thick, red, <-]  ($(C0)   + (-1.5, 0,0.5)$) -- +(0,0,+0.8) node[above] {$\scriptstyle{\pi_{\widehat{\lambda_2}}}$};
 \draw[thick, red, <-]  ($(C4T)   + (-1.0,0,0)$) -- +(0,0,+0.5) node[above] {$\scriptstyle{\pi_{\widehat{\lambda_5}}}$};
 \draw[thick, red, <-]  ($(C4T)   + (-2.0,0,0)$) -- +(0,0,+0.5) node[above] {$\scriptstyle{\pi_{\widehat{\lambda_4}}}$};
 \draw[thick, red, <-]  ($(C4T)   + (-3.0,0,0)$) -- +(0,0,+0.5) node[above] {$\scriptstyle{\pi_{\widehat{\lambda_3}}}$};
 \draw[thick, red, <-]  ($(C4T)   + (0.0,0,0)$) -- +(0,0,+0.5) node[above] {$\scriptstyle{\pi_{\widehat{\lambda_{k-1}}}}$};
 \draw[thick, red, <-]  ($(C4T)   + (1.0,0,0)$) -- +(0,0,+0.5) node[above] {$\scriptstyle{\pi_{\widehat{\lambda_{k}}}}$};
 \draw[<-]  ($(C0)   + (-1.5, 0, -0.5)$) -- +(0,0,-1) node[below] {$\scriptstyle{a_1}$};
 \draw[<-]  ($(CN)   + (0.5, 0, -0.5)$) -- +(0,0,-1) node[below] {$\scriptstyle{N}$};

\end{scope}

%% file: schur.tex
The aim of this section is to prove Propositions~\ref{prop:schurformula_complicated} and~\ref{prop:orthog} which are the key points in the proof of Theorem~\ref{thm:main}. For the sake of being self-contained and in order to introduce notations, we recall a few thing about symmetric polynomials. We use English notations for Young diagrams. For a kind introduction to symmetric polynomials we refer to \cite{MR3443860}.

Let us first recall some notations:
\begin{notation}\label{not:YD}
\begin{itemize} 
\item
If $\alpha$ is a Young diagram, $|\alpha|$ is the number of
    boxes of $\alpha$ and $\alpha^t$ is the transposed Young diagram.
If $a$ and $b$ are two non-negative integers,  $T(a,b)$ denotes the set of all Young diagram having at most $a$ columns and $b$ lines, $\rho(a,b)$ denotes the rectangular Young diagram with $a$ columns and $b$ lines. 
If a Young diagram $\alpha$ is specified to be in a certain $T(a,b)$, then 
\begin{itemize}
\item its complement is denote by $\alpha^c$: it is obtained by rotating by 180° the set of boxes of $\rho(a,b)$ which are not in $\alpha$,
\item the \emph{dual} of $\alpha$ is defined as $(\alpha^t)^c$ or equivalently as $(\alpha^c)^t$ and is denoted by $\widehat{\alpha}$.
\end{itemize}
An illustration is given in Figure~\ref{fig:YDcomplement}). Young diagrams will always be denoted by Greek letters. 
\begin{figure}[!h]
  \centering
  \begin{tikzpicture}[scale = 0.8]
    \input{\imagesfolder/cef_YDhat}
  \end{tikzpicture} 
  \caption{A Young diagram, its transposed, its complement and its dual.}
  \label{fig:YDcomplement}
\end{figure}
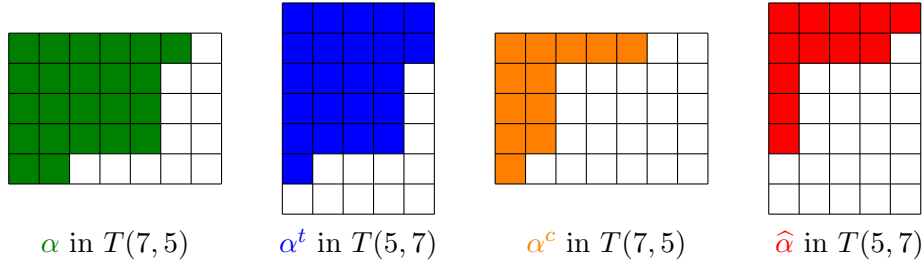
\item The capital roman letters will denote sets of variables. Since we will be dealing with anti-symmetric polynomial, it is very important that these sets are ordered. If $A=\{a_1, \dots, a_n\}$ is a set of variables (with the order given by indices) the \emph{Vandermonde determinant on $A$} is denoted by $\Delta(A)$ and is equal to
\[
\Delta(A) := \det
\begin{pmatrix}
  1 &a_1 & a_1^2 &\cdots & a_1^{n-1}  \\
  1 & a_2 & a_2^2 &\dots & a_2^{n-1} \\
  \vdots&\vdots & \vdots &  &\vdots \\
  1 &a_n & a_n^2 & \cdots & a_n ^{n-1}
\end{pmatrix}
 = \det \left( a_i^{j-1} \right)_{\substack{{1\leq i\leq n} \\ {1\leq j\leq n} }}
= \prod_{i<j} (a_j - a_i).\]
\item If  $A=\{a_1, \dots, a_n\}$ is a set of variables and $\lambda= (\lambda_1 \geq \lambda_2\geq \cdots \geq \lambda_{n}\geq 0) $ is a Young diagram with at most $n$ non-empty columns, the \emph{Schur polynomial in $A$ associated with $\alpha$} is denoted by $\pi_\alpha$ and is equal to:
\[
\pi_{\lambda}(A) :=\frac{ \det
\begin{pmatrix}
 a_1^{\lambda_n + 0} &a_1^{\lambda_{n-1} +1}  
&\cdots  &a_1^{\lambda_1 + n-1}  \\
 a_2^{\lambda_n + 0} &a_2^{\lambda_{n-1} +1} 
&\cdots  &a_2^{\lambda_1 + n-1}  \\
  \vdots&\vdots  
&  &\vdots \\
 a_n^{\lambda_n + 0} &a_n^{\lambda_{n-1} +1}  
&\cdots  &a_n^{\lambda_1 + n-1}  \\
\end{pmatrix} }{\Delta(A)} = \frac{
 \det \left( a_i^{\lambda_{n-j+1} + j-1} \right)_{\substack{{1\leq i\leq n} \\ {1\leq j\leq n} }}
}{\Delta(A)}.
\]
If $\lambda$ has more than $n$ non-empty columns, we set: $\pi_\lambda(A) =0$. The Schur polynomials in $A$ are symmetric in $A$ and they do not depend on the order we chose on $A$. It will be convenient to set $a_\lambda(A)= \pi_\lambda(A)\Delta(A)$. These are anti-symmetric polynomials and they depend on the order on $A$.
\item If $A$ and $B$ are two (ordered) set of variables, the symbol $AB$ denotes the disjoints union of these two sets with the conditions that every element in $A$ is lower than the every element of $B$. The symbols $A\sqcup B$ requires that $A$ and $B$ are both subsets of a bigger set $C$ and that they are disjoint, it denotes the union of $A$ and $B$ with the order induced by the order of $C$. 
\item If $A$ and $B$ are two disjoint subsets of an ordered set $C$, $|A<B|$ denotes the number of pairs $(a,b)$ in $A\times B$ such that $a$ is lower than $b$. We have $|A<B|+ |B<A| = |A||B|$.
\item If $A$ is a set of variables and $a_1$ and $a_2$ are two non-negative integers such that $a_1+a_2= |A|$, $A(a_1, a_2)$ denotes the set of partitions of $|A|$ into two subsets $A_1$ and $A_2$, such that $|A_1|=a_1$ and $|A_2|= a_2$.
\item If $A$ and $B$ are two sets of variables, we set:
  \[
    \nabla (A,B) := \prod_{\substack{a\in A \\ b\in B}}(a -b) = \frac{\Delta(BA)}{\Delta(A)\Delta(B)}.
\]
\item The \emph{Richardson-Littlewood coefficients} are denoted $c_{\alpha \beta}^{\gamma}$ for $\alpha$, $\beta$ and $\gamma$ three Young diagrams. It is zero unless $|\alpha|+ |\beta|=|\gamma|$. We will use the following facts:
  \begin{enumerate}
  \item Let $A$ be a set of variable and  $\alpha$ and $\beta$ two  Young diagrams. We have
    \[
\pi_{\alpha}(A)\pi_{\beta}(A) = \sum_\lambda c_{\alpha \beta}^{\gamma}\pi_\lambda(A).
\]
\item Let $A$ and $B$ be two disjoint sets of variables and $\gamma$ a Young diagram. We have:
\[
\pi_\lambda(AB) = \sum_{\alpha \beta} c_{\alpha \beta}^{\gamma}\pi_{\alpha}(A)\pi_{\beta}(B).
\]
\item If $a$ and $b$ are two non-negative integers, we have:
\[
c^{\rho(a,b)}_{\alpha \beta} =
\begin{cases}
  1 & \textrm{if $\alpha$ is in $T(a,b)$ and $\beta =\widehat{\alpha}^t$,} \\
0 &\textrm{else.}
\end{cases}
\]
  \end{enumerate}
\end{itemize}
\end{notation}

\begin{lem} \label{lem:VDM2antisym_fixYD}

Let $A$ be a set of variables, $a_1$ and $a_2$ two non-negative integers such that $|A|= a_1 + a_2$ and $\alpha$ a Young diagram in $T(a_1, a_2)$. Then we have:
\begin{align*}
\Delta(A) &= (-1)^{|\alpha|}\sum_{(A_1,A_2) \in A(a_1, a_2)} (-1)^{|A_2<A_1|} a_\alpha(A_1) a_{\widehat{\alpha}}(A_2) \\
&= (-1)^{|\widehat{\alpha}|}\sum_{(A_1,A_2) \in A(a_1, a_2)} (-1)^{|A_1<A_2|} a_\alpha(A_1) a_{\widehat{\alpha}}(A_2).
\end{align*} 
\end{lem}
\begin{proof}[Sketch of proof] This comes from developing the Vandermonde determinant along $a_1$ lines. 
If  $\alpha$ is the empty diagram: we develop it along the first $a_1$ lines and the sign is correct.  Moreover, thanks to the anti-symmetry of the determinant, the quantity
\[
\sum_{(A_1,A_2) \in A(a_1, a_2)} (-1)^{|A_1<A_2|} \pi_\alpha(A_1) \pi_{\widehat{\alpha}}(A_2)  
\] 
get multiplied by $-1$ whenever the number of boxes of $\alpha$  increases by one. 
\end{proof}
\begin{cor}\label{cor:nothat2null}
  Let $a_1$ and $a_2$ be two non-negative integers, $\alpha$ a Young diagram in $T(a_1, a_2)$,  $\beta$ a Young diagram in $T(a_2, a_1)$ and $A$ a set of variables such that $|A|=a_1 + a_2$. The following identity holds:
\[
\sum_{(A_1, A_2) \in A(a_1, a_2)} (-1)^{|A_2<A_1|}a_{\alpha}(A_1)a_{\beta}(A_2) = (-1)^{|\alpha|}\delta_{\alpha\widehat{\beta}}\Delta(A).
\]
\end{cor}
\begin{proof}
  The case $\alpha=\widehat{\beta}$ is done in the previous lemma. Whenever $\alpha\neq \widehat{\beta}$, the expression
\[
\sum_{(A_1, A_2) \in A(a_1, a_2)} (-1)^{|A_2<A_1|}a_{\alpha}(A_1)a_{\beta}(A_2) 
\]
computes the determinant by a multi-line development of a matrix which contains (at least) two columns which are equal.
\end{proof}

Similarly we have:
\begin{lem}  \label{lem:rectangle2schur}
Let $A$ be a set of variables, $a_1$ and $a_2$ two non-negative integers such that $|A|= a_1 + a_2$,$\alpha$ a Young diagram in $T(a_1, a_2)$ and $\beta$ a Young diagram in $T(a_2,a_1)$. Then we have:
\begin{align*}
&\sum_{(A_1, A_2) \in A(a_1, a_2)} (-1)^{|A_2<A_1|} a_{\frac{\rho(a_1,c)}{\alpha}}(A_1) a_{\frac{\rho(a_2,c)}{\widehat{\beta}}}(A_2) =  (-1)^{|\alpha|} \delta_{\widehat{\alpha}\beta}a_{\rho(a, c)}(A) \quad \textrm{and} \\
&\sum_{(A_1, A_2) \in A(a_1, a_2)} (-1)^{|A_1<A_2|} a_{\frac{\rho(a_1,c)}{\alpha}}(A_1) a_{\frac{\rho(a_2,c)}{\widehat{\beta}}}(A_2) =  (-1)^{|\widehat{\alpha}|}\delta_{\widehat{\alpha}\beta}a_{\rho(a, c)}(A),
\end{align*} 
where $\frac{\rho(a_1,c)}{\alpha}$ (resp. $\frac{\rho(a_2, c)}{\widehat{\alpha}}$) is the vertical concatenation of the rectangle $\rho(a_1,c)$ (resp. $\rho(a_2,c)$) with $\alpha$ (resp. $\widehat{\alpha}$). 
\end{lem}

\begin{lem}  \label{lem:VDM2antisym_fix_part}
Let $A$ be a set of variables, and $A= A_1 \sqcup A_2$ a decomposition of A. Then we have:
\begin{align*}
\Delta(A) &= (-1)^{|A_2<A_1|} \sum_{\alpha \in T(|A_1|, |A_2|)} (-1)^{|\alpha|}  a_\alpha(A_1) a_{\widehat{\alpha}}(A_2) \\
 &= (-1)^{|A_1<A_2|} \sum_{\alpha \in T(|A_1|, |A_2|)} (-1)^{|\widehat{\alpha}|}  a_\alpha(A_1) a_{\widehat{\alpha}}(A_2). 
\end{align*}
\end{lem}
\begin{proof}[Sketch of proof]
  This is very analogues to Lemma~\ref{lem:VDM2antisym_fixYD}. But instead of developing along lines, we develop along columns.   
\end{proof}

\begin{lem} \label{lem:nabla2schur}
Let $A$ and $B$ be two sets of variables.
We have:
\[
\nabla (A,B) = \sum_{\alpha  \in T(|A|, |B|)} (-1)^{|\widehat{\alpha}|}\pi_{\alpha}(A) \pi_{\widehat{\alpha}}(B).
\]
\end{lem}

\begin{proof}
We may suppose that the variables of $B$ are lower than the variables of $A$.
  We have:
  \begin{align*}
    \nabla(A,B) &= \frac{\Delta(BA)}{\Delta(A) \Delta(B)} \\
&= \frac{(-1)^{|A<B|}}{\Delta(A) \Delta(B)}\sum_{\alpha \in  T(|A|,|B|)} (-1)^{|\widehat{\alpha}|}a_\alpha(A) a_{\widehat{\alpha}}(B) \\
&= \sum_{\alpha \in  T(|A|,|B|)} (-1)^{|\widehat{\alpha}|}\pi_\alpha(A) \pi_{\widehat{\alpha}}(B).
  \end{align*}
\end{proof}

\begin{lem}\label{lem:comult_common_variable}
  Let $A, B$ and $C$ be 3 sets of variables, and $\gamma$ a Young diagram. Then the following identity holds:
\[
\sum_{\alpha, \beta } c_{\alpha \beta}^{\gamma} (-1)^{|\alpha|} \pi_\alpha(A\cup C) \pi_{\beta^t}(B\cup C) =
\sum_{\alpha, \beta } c_{\alpha \beta}^{\gamma} (-1)^{|\alpha|} \pi_\alpha(A) \pi_{\beta^t}(B).
\]
\end{lem}
\begin{proof}[Sketch of the proof]
Since the coproduct and the product of symmetric polynomials are compatible, it is enough to show the statement for $\gamma = l_n$ a line diagram (in this proof $l_i$ is a line tableau of length $i$ and $c_j$ is a column diagram of height $j$). 
We want to show:
\[
\sum_{i+j =n} (-1)^i \pi_{l_i}(A\cup C) \pi_{c_j}(B\cup C) =
\sum_{i+j =n} (-1)^i \pi_{l_i}(A) \pi_{c_j}(B). 
\]
It is enough to consider the case where $C$ has one element (denoted by $C$ as well). We have:
\begin{align*}
&\sum_{i+j =n} (-1)^i \pi_{l_i}(A\cup C) \pi_{c_j}(B\cup C) \\
&= \sum_{i+j =n} \sum_{i_1=0,1}  \sum _{j_1=0}^j(-1)^i C^{i_1}\pi_{l_{i-i_1}}(A)  C^{j_1}\pi_{c_{j-j_1}}(B) \\
&= \sum_{i+j =n} \sum _{j_1=0}^{j}(-1)^i C^{j_1} \pi_{l_{i}}(A)  \pi_{c_{j-j_1}}(B) \\
&\quad +  \sum_{\substack{i+j =n \\ i\geq 1}} \sum _{j_1=0}^{j}(-1)^i C^{j_1+1} \pi_{l_{i-1}}(A)  \pi_{c_{j-j_1}}(B) \\
&=\sum_{i+j =n} (-1)^i \pi_{l_i}(A) \pi_{c_j}(B). 
\end{align*}
\end{proof}
\begin{cor}\label{cor:Schur_common_variable}
   Let $A, B$ and $C$ be 3 sets of variables, and let $a$ and $b$ be two integers. The following equations holds:
 \[
 \sum_{\alpha \in T(a,b)} (-1)^{|\alpha|} \pi_\alpha(A\cup C) \pi_{\widehat{\alpha}}(B\cup C) = \sum_{\alpha \in T(a,b)} (-1)^{|\alpha|} \pi_\alpha(A) \pi_{\widehat{\alpha}}(B).
 \]
\end{cor}

\begin{proof}
  This is the previous lemma applied to $\lambda = \rho(a,b)$.
\end{proof}

We can translate the previous equality at the level of Littlewood-Richardson coefficients: 

\begin{cor}\label{cor:LRorthog}
  Let $a$ and $b$ be two non-negative integers and $\alpha$ in $T(a,b)$, $\beta$ in $T(b,a)$ and $\nu$ an arbitrary Young diagram. The following relation holds:
\[
\sum_{\substack{\tau, \alpha_1 \in T(a,b) \\ \alpha_2 \in T(b,a)}}(-1)^{|\tau|}c^{\nu}_{\alpha_1 \alpha_2}c^{\tau}_{\alpha_1 \alpha} c^{\widehat{\tau}}_{\alpha_2 \beta} =
\begin{cases}
  (-1)^{|\alpha|}& \textrm{if $\nu=\emptyset$ and $\beta = \widehat{\alpha}$},\\
  0 & \textrm{else}.
\end{cases}
\]
\end{cor}
\begin{proof}
  We have
\begin{align*}
&\sum_{\tau \in T(a,b)} (-1)^{|\tau|} \pi_\tau(A\cup C) \pi_{\widehat{\tau}}(B\cup C)  \\
&\qquad =
\sum_{\substack{\tau, \alpha_1 \in T(a,b) \\ \alpha_2 \in T(b,a)}}(-1)^{|\tau|}c^{\nu}_{\alpha_1 \alpha_2}c^{\tau}_{\alpha_1 \alpha} c^{\widehat{\tau}}_{\alpha_2 \beta}(-1)^{\alpha} \pi_\nu(C) \pi_{\alpha}(A) \pi_{\beta}(B) .
\end{align*}
But this is equal to 
\[
\sum_{\tau \in T(a,b)} (-1)^{|\alpha|} \pi_\tau(A) \pi_{\widehat{\tau}}(B).
\]
Identifying the terms gives the result.
\end{proof}

\begin{lem}\label{lem:hat_schur}
  Suppose $\alpha$, $\beta$ and $\gamma$ are three Young diagrams in $T(a,c), T(b,c)$ and $T(a+b,c)$. The following equality holds:
\[c^{\gamma}_{{\beta\alpha}} =  c^{\widehat{\gamma}}_{\widehat{\alpha} \widehat{\beta}}.\]
Note that the $\widehat{\bullet}$ has for each Young diagram a specific meaning. 
\end{lem}

\begin{proof}
  Let $A, B$ and $C$ be three disjoint sets of variables with cardinalities $a$, $b$ and $c$. We have:
\begin{align*}
\nabla(A\cup B, C) &= \sum_{{\gamma\in T(a+b,c)} } (-1)^{|\widehat{\gamma}|} \pi_{\gamma}(A\cup B) \pi_{\widehat{\gamma}}(C)
\\&= \sum_{\substack{\gamma\in T(a+b,c) \\ \alpha \in T(a,c) \\ \beta \in T(b,c)}} (-1)^{|\widehat{\gamma}|} c_{\alpha \beta}^{\gamma} \pi_{\alpha}(A) \pi_{\beta}(B) \pi_{\widehat{\gamma}}(C).
\end{align*}

On the other hand we have:
\begin{align*}
\nabla(A\cup B, C) &= \nabla(A,C) \nabla(B,C) \\
&=  \left(  \sum_{\alpha\in T(a,c)} (-1)^{|\widehat{\alpha}|} \pi_{\alpha}(A) \pi_{\widehat{\alpha}}(C) \right)
  \left(  \sum_{\beta\in T(b,c)} (-1)^{|\widehat{\beta}|} \pi_{\beta}(A) \pi_{\widehat{\beta}}(C) \right)
\\ &= \sum_{\substack{\gamma\in T(a+b,c) \\ \alpha \in T(a,c) \\ \beta \in T(b,c)}} (-1)^{|\widehat{\alpha}| + |\widehat{\beta}|} c_{\widehat{\alpha} \widehat{\beta}}^{\widehat{\gamma}} \pi_{\alpha}(A) \pi_{\beta}(B) \pi_{\widehat{\gamma}}(C).
\end{align*}
Identifying the terms we get what we wanted.
\end{proof}

The proofs of Lemma~\ref{lem:schurformula_simple} and of Propositions~\ref{prop:schurformula_complicated} and~\ref{prop:orthog} are relatively technical. We strongly advise the reader to first ignore the signs issues, and to deal with them separately on a second reading. For typographical reason, in some statements and some proofs, we may divide by the sign and not multiply by it, as it is usually done. We hope that, this will not disrupt the reading.

\begin{lem}\label{lem:schurformula_simple}
  Let $A_1$, $A_2$ and $B$ three disjoint sets of variables, such that $|B| =  |A_1\sqcup A_2|$ we have the following identity:
\[
\sum_{(B_1,B_2) \in B(|A_1|, |A_2|)} (-1)^{|B_2 < B_1|} \nabla(A_1,B_2) \nabla(A_2, B_1) \Delta(B_1) \Delta(B_2) =  \nabla(A_1, A_2) \Delta(B).
\]
\end{lem}
\begin{proof}
\begin{align*}
&  \sum_{(B_1,B_2) \in B(|A_1|, |A_2|)} 
(-1)^{|B_2 < B_1|} 
\nabla(A_1,B_2) \nabla(A_2, B_1) \Delta(B_1) \Delta(B_2) \\ 
&= \sum_{\substack{(B_1,B_2) \in B(|A_1|, |A_2|) \\ \alpha \in T(|A_1|, |A_2|) \\ \beta \in T(|A_2|, |A_1|)}} 
(-1)^{|B_2 < B_1| + |\widehat{\alpha}| + |\widehat{\beta}|} 
\pi_{\alpha}(A_1) \pi_{\widehat{\alpha}}(B_2) \pi_{\beta}(A_2) \pi_{\widehat{\beta}}(B_1) \Delta(B_1) \Delta(B_2) \\ 
&= \sum_{\substack{(B_1,B_2) \in B(|A_1|, |A_2|) \\ \alpha \in T(|A_1|, |A_2|) \\ \beta \in T(|A_2|, |A_1|)}} 
(-1)^{|B_2 < B_1| + |\widehat{\beta}| + |\widehat{\alpha}|} 
\pi_{\alpha}(A_1) \pi_{\beta}(A_2) a_{\widehat{\alpha}}(B_2) a_{\widehat{\beta}}(B_1).
\end{align*}
We use Corollary~\ref{cor:nothat2null} to continue:
\begin{align*}
& \qquad= \sum_{\substack{ \alpha \in T(|A_1|, |A_2|) \\ \beta \in T(|A_2|, |A_1|)}} 
(-1)^{ |\widehat{\alpha}|} \pi_{\alpha}(A_1) 
\pi_{\beta}(A_2) \delta_{\widehat{\alpha} \beta} \Delta(B) \\ 
& \qquad = \sum_{\substack{ \alpha \in T(|A_1|, |A_2|)}} 
(-1)^{ |\widehat{\alpha}|} 
\pi_{\alpha}(A_1) \pi_{\widehat{\alpha}}(A_2) \Delta(B) \\ 
&\qquad = \nabla (A_1, A_2) \Delta(B).
\end{align*}
\end{proof}

Exchanging the roles of $A\eqdef A_1\sqcup A_2$ and $B$, we get:
\begin{cor}\label{cor:schurformula_simple}
  Let $A$, $A_1$ and $B_2$ three disjoint sets of variables, such that $|A| =  |B_1\sqcup B_2|$. We have the following identity:
\[
\sum_{(B_1,B_2) \in B(|A_1|, |A_2|)} (-1)^{|A_2 < A_1|} \nabla(A_1,B_2) \nabla(A_2, B_1) \Delta(A_1) \Delta(A_2) =  \nabla(B_1, B_2) \Delta(A).
\]
\end{cor}

\begin{prop}\label{prop:schurformula_complicated}
  Let $A_1$, $A_2$ and $B$ three disjoint sets of variables, such that $|B| \leq  |A_1\sqcup A_2|$ we have the following identity:
\begin{align*}
&\sum_{\substack{B = B_1 \sqcup B_2 \\ \epsilon\in T(|A_1| -|B_1|, |A_2|-|B_2|)}}  \frac{
\nabla(A_1,B_2) \nabla(A_2, B_1) \Delta(B_1) \Delta(B_2) \pi_\epsilon(A_1 B_2) \pi_{\widehat{\epsilon}}(A_2 B_1)}{(-1)^{|B_2 < B_1| + |\epsilon| +|B_1|(|A_2| - |B_2|) )}} =  \nabla(A_1, A_2) \Delta(B).
\end{align*}
\end{prop}

\begin{proof}
Before starting the proof it is worth it to note that the partitions $B = B_1 \sqcup B_2$ which contribute to the sum satisfy $|B_1| \leq |A_1|$ and $|B_2| \leq |A_2|$, that Lemma~\ref{lem:schurformula_simple} is exactly this one in the case $|B| = |A_1| + |A_2|$ and that Lemma~\ref{lem:nabla2schur} is this one with $B=\emptyset$.

Let us also explain the strategy. We introduce a set $X$ of variables such that $|X| = |A_1| + |A_2| -|B|$ and we consider $G= \nabla(A_1, A_2) \Delta(BX)$. The previous lemma gives us an alternative expression for $G$. Then we look at $G$ as anti-symmetric polynomial in $X$ and identify the coefficient of  $a_{\rho(|X|, |B|)}(X)$ on both sides. We suppose that every variable in $X$ is greater than every variable in $B$. 
\begin{align*}
\\ G &=  \sum_{\substack{B = B_1 \sqcup B_2 \\  X = X_1 \sqcup X_2 \\  |X_1| + |B_1| = |A_1| \\ |X_2| + |B_2| =|A_2|}}
\frac{
\nabla(A_1, B_2X_2) \nabla(A_2, B_1X_1) \Delta(B_1 X_1)  \Delta(B_2 X_2)}
{ (-1)^{|B_2X_2< B_1X_1| }}
\\ &=  \sum_{\substack{B = B_1 \sqcup B_2 \\  X = X_1 \sqcup X_2 \\  |X_1| + |B_1| = |A_1| \\ |X_2| + |B_2| =|A_2| \\ \alpha \in T(|B_1|, |X_1|) \\ \beta \in T(|B_2|, |X_2|)}}
\frac{\nabla(A_1, B_2) \nabla(A_2, B_1) \nabla(A_1, X_2)\nabla(A_2,X_1)
 a_{\alpha}(B_1)  a_{\widehat{\alpha}}(X_1)  a_{{\beta}}(B_2)  a_{\widehat{\beta}}(X_2) }
{(-1)^{|B_2 X_2< B_1 X_1| + |\alpha| + |\beta|}}.
\end{align*}
We used Lemma~\ref{lem:VDM2antisym_fixYD}, and now we expand the $\nabla$'s thanks to Lemma~\ref{lem:nabla2schur}.
\begin{align*} 
G&=  \sum_{\substack{B = B_1 \sqcup B_2 \\  X = X_1 \sqcup X_2 \\  |X_1| + |B_1| = |A_1| \\ |X_2| + |B_2| =|A_2| \\ \alpha \in T(|B_1|, |X_1|) \\ \beta \in T(|B_2|, |X_2|)\\ \gamma \in T(|A_1|, |X_2|) \\ \delta \in T(|A_2|, |X_1|) }}
\frac{
\nabla(A_1, B_2) \nabla(A_2, B_1)
\pi_{\gamma}(A_1)  \pi_{\widehat{\gamma}}(X_2)  \pi_{{\delta}}(A_2)  \pi_{\widehat{\delta}}(X_1) 
a_{\alpha}(B_2)  a_{\widehat{\alpha}}(X_2)  a_{{\beta}}(B_1)  a_{\widehat{\beta}}(X_1) 
}{
(-1)^{|B_2X_2< B_1 X_1| +|\alpha| + |\beta| + |\widehat{\gamma}| +|\widehat{\delta}|} 
}
\\ &=  \sum_{\substack{B = B_1 \sqcup B_2 \\  X = X_1 \sqcup X_2 \\  |X_1| + |B_1| = |A_1| \\ |X_2| + |B_2| =|A_2| \\ \alpha \in T(|B_1|, |X_1|) \\ \beta \in T(|B_2|, |X_2|)\\ \gamma \in T(|A_1|, |X_2|) \\ \delta \in T(|A_2|, |X_1|) \\ \lambda \in T(|A_2| + |B_1|,|X_1|) \\ \mu \in T(|A_1| + |B_2|,|X_2|) }}
\frac{
 \nabla(A_1, B_2) \nabla(A_2, B_1)
 c^{\widehat{\lambda}}_{\widehat{\delta} \widehat{\beta}} c^{\widehat{\mu}}_{\widehat{\gamma} \widehat{\alpha}}
 \pi_{\gamma}(A_1) \pi_{{\delta}}(A_2)  a_{\widehat{\mu}}(X_2)    a_{\widehat{\lambda}}(X_1) 
 a_{\alpha}(B_2)   a_{{\beta}}(B_1)
}{
 (-1)^{|B_2X_2< B_1 X_1| +|\alpha| + |\beta| + |\widehat{\gamma}| +|\widehat{\delta}|}
}
\\ &=  \sum_{\substack{B = B_1 \sqcup B_2 \\  X = X_1 \sqcup X_2 \\  |X_1| + |B_1| = |A_1| \\ |X_2| + |B_2| =|A_2| \\ \alpha \in T(|B_1|, |X_1|) \\ \beta \in T(|B_2|, |X_2|)\\ \gamma \in T(|A_1|, |X_2|) \\ \delta \in T(|A_2|, |X_1|) \\ \lambda \in T(|A_2| + |B_1|,|X_1|) \\ \mu \in T(|A_1| + |B_2|,|X_2|) }}
\frac{
 \nabla(A_1, B_2) \nabla(A_2, B_1)
 c^{\lambda}_{\delta \beta} c^{\mu}_{\gamma \alpha}
 \pi_{\gamma}(A_1) \pi_{{\delta}}(A_2)  a_{\widehat{\mu}}(X_2)    a_{\widehat{\lambda}}(X_1) 
 a_{\alpha}(B_2)   a_{{\beta}}(B_1)
}{
(-1)^{|B_2< B_1|+ |X_2<X_1| +|B_2||X_1| +|\alpha| + |\beta| + |\widehat{\gamma}| +|\widehat{\delta}|}
}.
\end{align*}
We may sum over the decompositions $X= X_1\sqcup X_2$ with $|X_1|$ and $|X_2|$ fixed. This yields an anti-symmetric polynomial in $X$ with coefficients in $A_1 \sqcup A_2 \sqcup B$. We look at the coefficient of $a_{\rho(|X|, |B|)}(X)$. In the sum this implies that $\widehat{\lambda}$ contains $\rho(|B_1|,|X_2|)$ as a sub-Young diagram and $\widehat{\mu}$ contains $\rho(|B_2|,|X_1|)$ as a sub-Young diagram and that $\lambda = \widehat{\mu}$ in $T(|X_2|, |X_1|)$. 

We can use Lemma~\ref{lem:rectangle2schur} to compute the coefficient  of  $a_{\rho(|X|, |B|)}(X)$ in $G$ which we denote by $H$:

\begin{align*}
H&= \sum_{\substack{B = B_1 \sqcup B_2 \\ |X_1| + |B_1| = |A_1| \\ |X_2| + |B_2| =|A_2| \\ \alpha \in T(|B_1|, |X_1|) \\ \beta \in T(|B_2|, |X_2|)\\ \gamma \in T(|A_1|, |X_2|) \\ \delta \in T(|A_2|, |X_1|) \\ \tau \in T(|X_2|, |X_1|)}}
\frac{
\nabla(A_1, B_2) \nabla(A_2, B_1)
 c^{\tau}_{\delta \beta} c^{\widehat{\tau}}_{\gamma \alpha}
 \pi_{\gamma}(A_1) \pi_{{\delta}}(A_2) 
 a_{\alpha}(B_2)   a_{{\beta}}(B_1) 
}{
 (-1)^{|B_2< B_1| +|B_2||X_1| +|\alpha| + |\beta| + |\widehat{\gamma}| +|\widehat{\delta}| + |\widehat{\tau}|}
}.
\end{align*}
We renamed $\lambda$ into $\tau$ since these diagrams are not in the same rectangle. When
 $c^{\tau}_{\delta \beta} c^{\widehat{\tau}}_{\gamma \alpha} \neq 0$, we have:
 \begin{itemize}
 \item $ |\beta| + |\widehat{\delta}| = |\beta| - |\delta| + |A_2||X_1| \equiv |\tau| + |A_2||X_1|$,
 \item $|\widehat{\gamma}| + |\gamma| = |A_2||X_1|$,
 \item $|\alpha| + |\widehat{\gamma}| = |\alpha| - |\gamma| + |A_1||X_2|\equiv |\widehat{\tau}| + |A_1||X_2|.$
 \end{itemize}
Using the comultiplication of symmetric polynomials, we get:
\begin{align*}
H&= \sum_{\substack{B = B_1 \sqcup B_2 \\   |X_1| + |B_1| = |A_1| \\ |X_2| + |B_2| =|A_2|\\ \tau \in T(|X_2|, |X_1|) \\}}
\frac{
\nabla(A_1, B_2) \nabla(A_2, B_1)
\Delta(B_1) \Delta(B_2) \pi_\tau(A_2 B_1) \pi_{\widehat{\tau}} (A_1 B_2)
}{
  (-1)^{|B_2< B_1| +|B_2||X_1| + |A_2||X_1|+ |A_1| |X_2| +|\tau|}
}.
\end{align*}
We have:
\[
|B_2||X_1| + |A_2||X_1|+ |A_1| |X_2| \equiv |X_1|(|B_2|-|A_2|)+ |A_1| |X_2| \equiv |B_1|(|A_2| -|B_2|)
\]
and
\begin{align*}
H&= \sum_{\substack{B = B_1 \sqcup B_2 \\  |X_1| + |B_1| = |A_1| \\ |X_2| + |B_2| =|A_2|\\ \tau \in T(|X_2|, |X_1|) \\}}
\frac{
\nabla(A_1, B_2) \nabla(A_2, B_1)
\Delta(B_1) \Delta(B_2) \pi_\tau(A_2 B_1) \pi_{\widehat{\tau}} (A_1 B_2) 
}{
 (-1)^{|B_2< B_1 | + |\tau| + |B_1|(|A_2| - |B_2|)} 
}.
\end{align*}
On the other hand, we have:
\[
G=\nabla(A_1, A_2) \Delta(BX) = \nabla(A_1, A_2) \Delta(B) \sum_{\eta \in T(|B|, |X|)} (-1)^{|{\eta}|}\pi_\eta(B)a_{\widehat{\eta}}(X).
\]
This gives:
\[
H= \nabla(A_1, A_2) \Delta(B).
\]
\end{proof}

\begin{cor}\label{cor:schurformula_complicated}
  Let $A_1$, $A_2$, $B$ and $C$ four disjoint sets of variables, such that $|B| \leq  |A_1\sqcup A_2|$ we have the following identity:
\begin{align*}
&\sum_{\substack{B = B_1 \sqcup B_2 \\ \epsilon\in T(|A_1| -|B_1|, |A_2|-|B_2|)}}  
\frac{
\nabla(A_1,B_2) \nabla(A_2, B_1) \Delta(B_1) \Delta(B_2) \pi_\epsilon(A_1 B_2 C) \pi_{\widehat{\epsilon}}(A_2 B_1 C)  
}{
(-1)^{|B_2 < B_1| + |\epsilon| +|B_1|(|A_2| - |B_2|) )}
}
=\nabla(A_1, A_2) \Delta(B).
\end{align*}
\end{cor}

\begin{proof}
  It follows directly from Corollary~\ref{cor:Schur_common_variable} and Proposition~\ref{prop:schurformula_complicated}.
\end{proof}

\begin{lem}
  \label{lem:LRpoincare}
  If $a$ and $b$ are two non-negative integer and $\alpha$, $\beta$ and $\gamma$ are three Young diagram in $T(a,b)$, we have:
\[
c_{\alpha \beta}^{\gamma^c} = c_{\gamma \beta}^{\alpha^c}.
\]
\end{lem}

\begin{proof}[Sketch of the proof]
  Thanks to the Littlewood-Richardson rule, $c_{\alpha \beta}^{\gamma^c}$ (resp. $c_{\gamma \beta}^{\alpha^c}$ ) is given by the number of skew semi-standard tableaux of shape $\gamma^c/\alpha$ (resp. $\alpha^c/\gamma$) with weight $\beta$. There is an obvious one-one correspondence between these two families of skew semi-standard tableaux.
\end{proof}

\begin{prop}\label{prop:orthog}
Let $B_1$, $B_2$, $A$ and $C$ be four disjoint sets of variables such that $|B_1|+ |B_2| \leq |A|$ and $a_1$ and $a_2$, two nonnegative integers such that $a_1\geq |B_1|$ and $a_2 \geq |B_2|$. We consider a Young diagram $\alpha$ in $T(a_1 -|B_1|, a_2-|B_2|)$  and $\beta$ a Young diagram in $T(a_2 -|B_2|, a_1-|B_1|)$. 
  The following relation holds:
\begin{align*}
&\sum_{(A_1,A_2) \in A(a_1, a_2)} (-1)^{|A_2 <A_1|}{\nabla(A_1, B_2) \nabla(A_2, B_1)\pi_{\alpha}(A_1B_2C) \pi_{\beta} (A_2B_1C)}\Delta(A_1)\Delta(A_2) \\ &=
\begin{cases}
(-1)^  {|\beta|+ |B_2|(a_1- |B_1|)}\Delta(A)\nabla(B_1,B_2)  & \textrm{if $\beta=\widehat{\alpha}$,} \\
0 &\textrm{else}.
\end{cases}
\end{align*}
\end{prop}

\begin{proof}
  We use Corollary~\ref{cor:schurformula_simple}. Let $X_1$ and $X_2$ be two disjoint sets of variables such that $|X_1| = a_1 -|B_1|$ and $|X_2| = a_2 -|B_2|$. We suppose that the variables in $X_1$ and $X_2$ are greater than any variable in $A$ and $B_1$ and $B_2$ We have:
\[
\sum_{(A_1, A_2) \in A(a_1, a_2)} (-1)^{|A_2 < A_1|} \nabla(A_1,B_2X_2) \nabla(A_2, B_1X_1) \Delta(A_1) \Delta(A_2) = \nabla(B_1X_1, B_2X_2) \Delta(A).
\]
Developing on both sides we get:
\begin{align*}
 & \sum_{(A_1, A_2) \in A(a_1, a_2)} 
(-1)^{|A_2 < A_1|} \nabla(A_1,B_2)\nabla(A_1,X_2) \nabla(A_2, B_1)\nabla(A_2,X_1) \Delta(A_1) \Delta(A_2) \\ 
&\qquad \qquad \qquad= \nabla(B_1,B_2)\nabla(X_1, B_2) \nabla(B_1,X_2)\nabla(X_1,X_2) \Delta(A).
\end{align*}
Let us now inspect the left-hand side:
\begin{align*}
 & \sum_{(A_1, A_2) \in A(a_1, a_2)} (-1)^{|A_2 < A_1|} \nabla(A_1,B_2)\nabla(A_1,X_2) \nabla(A_2, B_1)\nabla(A_2,X_1) \Delta(A_1) \Delta(A_2) \\ &= 
 \sum_{\substack{(A_1, A_2) \in A(a_1,a_2) \\ \alpha_2 \in T(a_2, a_1 - |B_1|)} }
\frac{
\nabla(A_1,B_2)\nabla(A_2, B_1) \pi_{\alpha_1}(A_1) \pi_{\widehat{\alpha_1}}(X_2) \pi_{\alpha_2}(A_2) \pi_{\widehat{\alpha_2}}(X_1) \Delta(A_1) \Delta(A_2)}{
(-1)^{|A_2 < A_1| + |\widehat{\alpha_1}| + |\widehat{\alpha_2}|}
}.
\end{align*}
And now the right-hand side:
\begin{align*}
&\nabla(B_1,B_2)\nabla(X_1, B_2) \nabla(B_1,X_2)\nabla(X_1,X_2) \Delta(A) \\ &=
\sum_{\substack{\beta_1 \in T(|B_1|, a_2 - |B_2|) \\ \beta_2 \in T(|B_2|, a_1 - |B_1|)  \\ \tau \in T(a_1 - |B_1|, a_2- |B_2|)}}
\frac{
\nabla(B_1,B_2)\pi_{\beta_1}(B_1) \pi_{\widehat{\beta_1}}(X_2) \pi_{\beta_2}(B_2) \pi_{\widehat{\beta_2}}(X_1) \pi_{\tau}(X_1) \pi_{\widehat{\tau}}(X_2) \Delta(A) 
}{
(-1)^{|\widehat{\beta_1}| + |\beta_2| + |\widehat{\tau}|}
}
\\ 
&= \sum_{\substack{\beta_1 \in T(|B_1|, a_2 - |B_2|) \\ \beta_2 \in T(|B_2|, a_1 - |B_1|)  \\ \tau \in T(a_1 - |B_1|, a_2- |B_2|) \\ \lambda, \mu}}
\frac{
c_{\widehat{\beta_2}\tau}^{\lambda} c_{\widehat{\beta_1}\widehat{\tau}}^{\mu}   \nabla(B_1,B_2)\pi_{\beta_1}(B_1)  \pi_{\beta_2}(B_2)  \pi_{\lambda}(X_1) \pi_{\mu}(X_2) \Delta(A)
}{
(-1)^{|\widehat{\beta_1}| + |\beta_2| + |\widehat{\tau}|}
}. 
\end{align*}
We can identify on both sides with $\widehat{\lambda} = \alpha_2$ in  $T(a_2, a_1- |B_1|)$ and $\widehat{\mu}= \alpha_1$ in $T(a_1, a_2- |B_2|)$. This gives:
\begin{align*}
&\sum_{(A_1, A_2) \in A(a_1, a_2)} (-1)^{|A_2 < A_1|} \nabla(A_1,B_2)\nabla(A_2, B_1) \pi_{\alpha_1}(A_1) \pi_{\alpha_2}(A_2) \Delta(A_1) \Delta(A_2) \\ &=
\sum_{\substack{\beta_1 \in T(|B_1|, a_2 - |B_2|) \\ \beta_2 \in T(|B_2|, a_1 - |B_1|)  \\ \tau \in T(a_1 - |B_1|, a_2- |B_2|) \\ }} (-1)^{|\tau| + |B_2|(a_1-|B_1|)} c_{\widehat{\beta_2}\tau}^{\widehat{\alpha_2}} c_{\widehat{\beta_1}\widehat{\tau}}^{\widehat{\alpha_1}}   \nabla(B_1,B_2)\pi_{\beta_1}(B_1)  \pi_{\beta_2}(B_2)  \Delta(A).
\end{align*}
Therefore:
\begin{align*}
&\sum_{(A_1, A_2) \in A(a_1, a_2)} (-1)^{|A_2 < A_1|}  \nabla(A_1,B_2)\nabla(A_2, B_1) \pi_{\alpha}(A_1B_2C) \pi_{\beta}(A_2B_1C) \Delta(A_1) \Delta(A_2) 
\\ &=
\left(\sum_{\substack{\gamma_1, \gamma_2 \\ \alpha_1\in T(a_1, a_2-|B_2|) \\ \alpha_2 \in T(a_2, a_1 -|B_1|) }}c_{\alpha_1 \gamma_1}^{\alpha} c_{\alpha_2 \gamma_2}^{\beta} \pi_{\gamma_1}(B_2C) \pi_{\gamma_2}(B_1C)\right) \\ & \qquad \qquad \cdot
\left(\sum_{\substack{\beta_1 \in T(|B_1|, a_2 - |B_2|) \\ \beta_2 \in T(|B_2|, a_1 - |B_1|)  \\ \tau \in T(a_1 - |B_1|, a_2- |B_2|) \\ }} 
\frac{c_{\widehat{\beta_2}\tau}^{\widehat{\alpha_2}} c_{\widehat{\beta_1}\widehat{\tau}}^{\widehat{\alpha_1}}   \nabla(B_1,B_2)\pi_{\beta_1}(B_1)  \pi_{\beta_2}(B_2)  \Delta(A)}{
(-1)^{|\tau| + |B_2|(a_1-|B_1|)}
}
\right)
\\ &=
\sum_{\substack{\gamma_1, \gamma_2 \\\alpha_1\in T(a_1, a_2-|B_2|) \\ \alpha_2 \in T(a_2, a_1 -|B_1|)  \\ \beta_1 \in T(|B_1|, a_2 - |B_2|) \\ \beta_2 \in T(|B_2|, a_1 - |B_1|)  \\ \tau \in T(a_1 - |B_1|, a_2- |B_2|) \\ }} 
\frac{
c_{\alpha_1 \gamma_1}^{\alpha} c_{\alpha_2 \gamma_2}^{\beta} c_{\beta_2\widehat{\tau}}^{\alpha_2} c_{\beta_1{\tau}}^{\alpha_1}   \pi_{\gamma_1}(B_2C) \pi_{\gamma_2}(B_1C)  \nabla(B_1,B_2)\pi_{\beta_1}(B_1)  \pi_{\beta_2}(B_2)  \Delta(A)
}{
(-1)^{|\tau| + |B_2|(a_1-|B_1|)|}
}
.
\end{align*}
Since $\alpha$ is in $T(a_1 -|B_1|, a_2-|B_2|)$ and $\beta$ in $T(a_2-|B_2|, a_1 -|B_1|)$ we can restrict the ranges of $\alpha_1$, $\gamma_1$, $\beta_2$ and $\gamma_2$. We use the associativity property on Littlewood-Richardson coefficients (namely $c_{\alpha_2 \gamma_2}^{\beta} c_{\beta_2\widehat{\tau}}^{\alpha_2} =c_{\alpha_2 \widehat{\tau}}^{\beta} c_{\beta_2\gamma_2}^{\alpha_2}$ and $c_{\alpha_1 \gamma_1}^{\alpha} c_{\beta_1{\tau}}^{\alpha_1} = c_{\tau \gamma_1}^{\alpha} c_{\beta_1 \gamma_1}^{\alpha_1}$). This gives:
\begin{align*}
 &=
\sum_{\substack{ \alpha_1, \gamma_1\in T(a_1-|B_1|, a_2-|B_2|) \\ \alpha_2, \gamma_1 \in T(a_2-|B_2|, a_1 -|B_1|)  \\ \beta_1 \in T(|B_1|, a_2 - |B_2|) \\ \beta_2 \in T(|B_2|, a_1 - |B_1|)  \\ \tau \in T(a_1 - |B_1|, a_2- |B_2|) \\ }} 
\frac{
c_{\tau \gamma_1}^{\alpha} c_{\alpha_2 \widehat{\tau}}^{\beta} c_{\beta_2\gamma_2}^{\alpha_2} c_{\beta_1 \gamma_1}^{\alpha_1}   \pi_{\gamma_1}(B_2C) \pi_{\gamma_2}(B_1C)  \nabla(B_1,B_2)\pi_{\beta_1}(B_1)  \pi_{\beta_2}(B_2)  \Delta(A)
}{
(-1)^{|\tau| + |B_2|(a_1-|B_1|)}
} 
\\ &=
\sum_{\substack{ \alpha_1, \gamma_1\in T(a_1-|B_1|, a_2-|B_2|) \\ \alpha_2, \gamma_1 \in T(a_2-|B_2|, a_1 -|B_1|)  \\ \beta_1 \in T(|B_1|, a_2 - |B_2|) \\ \beta_2 \in T(|B_2|, a_1 - |B_1|)  \\ \tau \in T(a_1 - |B_1|, a_2- |B_2|) \\ }} 
(-1)^{|\tau| + |B_2|(a_1-|B_1|)}
c_{\tau \alpha_1}^{\alpha} c_{\alpha_2 \widehat{\tau}}^{\beta}
 \pi_{\alpha_1}(B_{1}B_2C) \pi_{\alpha_2}(B_1B_2C)  \nabla(B_1,B_2)
\Delta(A)
\\
&=
\sum_{\substack{ \nu \\\alpha_1, \gamma_1\in T(a_1-|B_1|, a_2-|B_2|) \\ \alpha_2, \gamma_1 \in T(a_2-|B_2|, a_1 -|B_1|)  \\ \beta_1 \in T(|B_1|, a_2 - |B_2|) \\ \beta_2 \in T(|B_2|, a_1 - |B_1|)  \\ \tau \in T(a_1 - |B_1|, a_2- |B_2|) \\ }} 
(-1)^{|\tau| + |B_2|(a_1-|B_1|)}
c^{\nu}_{\alpha_1\alpha_2}c_{\tau \alpha_1}^{\alpha} c_{\alpha_2 \widehat{\tau}}^{\beta}
 \pi_{\nu}(B_{1}B_2C)  \nabla(B_1,B_2)
\Delta(A). 
\end{align*}
We now use Lemma~\ref{lem:LRpoincare} and change $\tau^c$ for $\tau$. This gives:
\begin{align*}
&\sum_{(A_1, A_2) \in A(a_1, a_2)} (-1)^{|A_2 < A_1|}\nabla(A_1,B_2)\nabla(A_2, B_1) \pi_{\alpha}(A_1B_2C) \pi_{\beta}(A_2B_1C) \Delta(A_1) \Delta(A_2) 
\\ &=
\sum_{\substack{ \nu \\\alpha_1, \gamma_1\in T(a_1-|B_1|, a_2-|B_2|) \\ \alpha_2, \gamma_1 \in T(a_2-|B_2|, a_1 -|B_1|)  \\ \beta_1 \in T(|B_1|, a_2 - |B_2|) \\ \beta_2 \in T(|B_2|, a_1 - |B_1|)  \\ \tau \in T(a_1 - |B_1|, a_2- |B_2|) \\ }} (-1)^{|\tau| + |B_2|(a_1-|B_1|)}
c^{\nu}_{\alpha_1\alpha_2}c_{\alpha^c \alpha_1}^{\tau} c_{\alpha_2 \beta^c}^{\widehat{\tau}}
 \pi_{\nu}(B_{1}B_2C)  \nabla(B_1,B_2)
\Delta(A).
\end{align*}
We conclude with Corollary~\ref{cor:LRorthog}.
\end{proof}